\title{On a Class of Global Solutions to 3D Free-Boundary Relativistic Euler Equations with a Physical Vacuum Boundary}
\author{Marcelo M. Disconzi\thanks{Department of Mathematics, Vanderbilt University, Nashville, TN, 37240, United States.\url{marcelo.disconzi@vanderbilt.edu}},
Zhongtian Hu\thanks{Mathematics Department, Princeton University, Princeton, NJ, 08544, United States. \url{zh1077@princeton.edu}},
and Chenyun Luo\thanks{Department of Mathematics, The Chinese University of Hong Kong. Shatin, NT, Hong Kong SAR. \url{cluo@math.cuhk.edu.hk}.}
}
\documentclass[11pt]{article}
\usepackage[margin=1in]{geometry}
\usepackage{comment,slashed,tensor}
\usepackage[T1]{fontenc}
\usepackage{ stmaryrd }
\usepackage{cancel}
\usepackage[symbol]{footmisc}
\usepackage[title]{appendix}

\usepackage[bookmarksnumbered, bookmarksopen, colorlinks, citecolor=red, linkcolor=black]{hyperref}
 \usepackage{amsfonts}
 \usepackage{comment}
  \usepackage{cancel}
\usepackage{graphicx}
 \usepackage{amsmath, amssymb}
 \usepackage{amsthm}
 \usepackage{mathtools}

 \numberwithin{equation}{section}
\newtheorem{thm}{Theorem}[section]
\newtheorem{lem}[thm]{Lemma}
\newtheorem{cor}[thm]{Corollary}

\newtheorem{prop}[thm]{Proposition}
\newtheorem{rmk}[thm]{Remark}
\newtheorem{defn}[thm]{Definition}

\allowdisplaybreaks

\newcommand{\R}{\mathbb{R}}
\newcommand{\N}{\mathbb{N}}

\newcommand{\calA}{\mathcal{A}}

\newcommand{\calF}{\mathcal{F}}

\newcommand{\mcalB}{\mathcal{B}}
\newcommand{\calG}{\mathcal{G}}
\newcommand{\calT}{\mathcal{T}}
\newcommand{\calC}{\mathcal{C}}
\newcommand{\calE}{\mathcal{E}}
\newcommand{\calD}{\mathcal{D}}
\newcommand{\calO}{\mathcal{O}}
\newcommand{\calL}{\mathcal{L}}
\newcommand{\calS}{\mathcal{S}}

\newcommand{\trho}{\tilde{\rho}}

\newcommand{\ze}{\zeta}

\newcommand{\tv}{\tilde{v}}

\newcommand{\mfE}{\mathfrak{E}}
\newcommand{\mfD}{\mathfrak{D}}

\newcommand{\mfI}{\mathfrak{I}}
\newcommand{\mfC}{\mathfrak{C}}
\newcommand{\mfR}{\mathfrak{R}}

\makeatletter
\newcommand*{\rom}[1]{\expandafter\@slowromancap\romannumeral #1@}
\makeatother

\newcommand{\p}{\partial}
\newcommand{\kk}{\kappa}
\newcommand{\eps}{\epsilon}
\newcommand{\lam}{\lambda}
\newcommand{\lamt}{\bar{\lambda}_{\tau}}
\newcommand{\dz}{\p_{\zeta}}
\newcommand{\pz}{\p_{\zeta}}
\newcommand{\pt}{\p_{\tau}}
\newcommand{\dt}{\p_{\tau}}
\newcommand{\Dz}{D_{\zeta}}
\newcommand{\barcalD}{\bar{\calD}}
\newcommand{\calP}{\mathcal{P}}
\newcommand{\barcalP}{\bar{\mathcal{P}}}
\newcommand{\barcalG}{\bar{\mathcal{G}}}
\newcommand{\dist}{\text{dist }}

\newcommand{\EulerianDomain}{\Omega}
\newcommand{\DefFunction}{\mathrm{d}}
\newcommand{\GenericFunction}{\mathsf{f}}
\newcommand{\NewSoundSpeedSq}{\mathsf{r}}
\newcommand{\NewVelocity}{\mathsf{v}}
\newcommand{\DITWeightedSpace}{\mathcal{H}}

\usepackage{tensor}

\begin{document}
\newpage
\maketitle
\begin{abstract}
    We consider the free-boundary relativistic Euler equations in Minkowski spacetime $\mathbb{M}^{1+3}$ equipped with a physical vacuum boundary, which models the motion of a relativistic gas. We concern ourselves with the family of isentropic, barotropic, and polytropic gas, with an equation of state $p = \rho^{1+\kk}, \kk \in (0,\frac23]$. We construct an open class of initial data that launches future-global solutions. Such solutions are spherically symmetric, have small initial density, and expand asymptotically linearly in time. In particular, the asymptotic rate of expansion is allowed to be arbitrarily close to the speed of light. Therefore, our main result is far from a perturbation of existing results concerning the classical isentropic Euler counterparts.
\end{abstract}
\tableofcontents
\section{Introduction}
In this article, we study the motion of a relativistic fluid in the Minkowski background $\mathbb{M}^{1+3}$. The fluid state is represented by the (energy) density $\rho\geq 0$, and the relativistic (four-) velocity $u=(u^0, u^1, u^2, u^3)$. The velocity is assumed to the a forward time-like vector field, normalized by
\begin{equation}\label{normalization general}
    m_{\mu\nu}u^{\mu}u^{\nu}=-1.
\end{equation}
where $m$ is the Minkowski metric. Throughout this article, we adopt the summation convention for repeated indices, where Greek indices range from $0$ to $3$ and Latin indices range from $1$ to $3$. Also, indices are raised and lowered in accordance with $m$. 

The equations of motion consist of 
\begin{equation}\label{rel Euler general form}
    \p_{\mu} \calT^{\mu\alpha} =0, 
\end{equation}
where $\calT$ is the energy-momentum tensor for a perfect fluid, given by
\begin{equation}\label{def: energy-mom tensor}
    \calT^{\alpha\beta} = (p+\rho)u^{\alpha}u^{\beta} + pm^{\alpha\beta}.
\end{equation}
Here, we denote by $p$ the pressure, which is related to $\rho$ through the equation of states
\begin{equation}
    p = p (\rho). 
\end{equation}
Projecting \eqref{rel Euler general form} onto the directions parallel and orthogonal to $u$, using \eqref{normalization general} and \eqref{def: energy-mom tensor}, yields the system of equations:
\begin{subequations}\label{eq:RE}
\begin{align}
    &u^\mu\p_\mu \rho + (p + \rho)\p_\mu u^\mu = 0, \label{eq:RE_mass}\\
    &(\rho + p) u^\mu\p_\mu u^\alpha + (m^{\mu \alpha } + u^\mu u^\alpha)\p_\mu p = 0, \label{eq:RE_mom}\\
    &m_{\mu\nu}u^\mu u^\nu = -1. \label{eq: RE_normalization}
\end{align}
\end{subequations}
This is a nonlinear hyperbolic system, which in the reference frame of the moving fluid has the propagation speed
\begin{align}\label{def: c_s}
    c_s^2(\rho)  = \frac{dp}{d\rho},
\end{align}
which is subject to $0\leq c_s<1$, implying that the speed of propagation of sound waves stays non-negative and always below the speed of light, which is equal to one in the units we adopted. 

The relativistic Euler equations are extensively used in modeling fluids in regimes when the laws of relativity cannot be neglected. They are ubiquitous in many areas of physics, ranging from high-energy physics \cite{Romatschke:2017ejr}, astrophysics \cite{Rezzolla-Zanotti-Book-2013}, and cosmology \cite{Weinberg-Book-2008}. They are also a fertile source of mathematical problems, see e.g., \cite{Choquet-Bruhat-Book-2009}.

In this article, we consider the physical situation in which vacuum states are allowed, i.e., $\rho$ may vanish. In this case, \eqref{eq:RE} models the motion of a \textit{gas}, which occupies the domain
\begin{align}\label{def: moving domain Omega_t}
    \Omega_t :=\{x\in \R^3 :  \rho(t,x)>0\},
\end{align}
whose boundary $\p \Omega_t$ is the \textit{vacuum boundary}, and is advected by the fluid's velocity. It is important to note that the distinguishing characteristic of a gas, as opposed to a liquid, is that the density vanishes on $\p\Omega_t$:
\begin{align}\label{rho is 0 on the free boundary}
    \rho =0,\quad \text{on } \p\Omega_t. 
\end{align}
Moreover, an appropriate equation of state to describe this situation is
\begin{equation}\label{equation of states}
    p(\rho) = \rho^{1+\kk},\quad \text{for some fixed }\kk>0. 
\end{equation}
Consequently, we infer from \eqref{rho is 0 on the free boundary} that $p$ and $c_s$ also vanish on $\p\Omega_t$. From a physical viewpoint, this indicates that the sound waves cannot propagate in the vacuum region across the moving boundary. However, it is important to note that there is essentially only one vanishing rate of $c_s$ near the vacuum boundary that is both mathematically and physically consistent with the evolution of the free boundary \cite{disconzi2024recent}:
\begin{align}\label{fallout rate}
    c_s^2(t,x) \approx \text{dist}(x, \p\Omega_t),\quad \text{in }\Omega_t,
\end{align}
where the denote by $\text{dist}(\cdot,\cdot)$ the distance function. Under \eqref{fallout rate}, the free boundary $\p\Omega_t$ is commonly referred to as the \textit{physical vacuum boundary}.

\subsection{History and Background}

    Local well-posedness for free-boundary relativistic Euler equations has been studied extensively in the past decade. In the case of a \textit{liquid}, i.e., where the pressure of the fluid vanishes on the free-boundary but the density remains strictly positive, the local well-posedness has been obtained by Ginsberg--Lindblad \cite{ginsberg2023local}, Miao--Shahshahani--Wu \cite{miao2021well}, and Oliynyk \cite{oliynyk2019dynamical}. Very recently, the result of \cite{miao2021well} has been extended to the local well-posedness for Einstein--Euler equations, where the effects of general relativity are taken into consideration, by Miao--Shahshahani \cite{miao2024well} and Hao--Huo--Miao \cite{hao2024well}. In contrast to the case of a \textit{liquid}, the present work focuses on the relativistic Euler equations with a \textit{physical vacuum boundary} modeling the motion of a \textit{gas}, where the density vanishes on the moving boundary. The a priori estimates for this model are addressed in Had{\v{z}}i{\'c}--Shkoller--Speck \cite{hadvzic2019priori} and Jang--LeFloch--Masmoudi \cite{jang2016lagrangian}, followed by the local well-posedness result addressed in a very recent work by Disconzi--Ifrim--Tataru \cite{disconzi2022relativistic}. When neglecting all relativistic effects, the system of equations \eqref{eq:RE} reduces to the compressible Euler equations with a vacuum boundary, whose local well-posedness is addressed in \cite{coutand2012well, ifrim2023compressible, jang2015well, liu2024physical, luo2014well}.
 
    Due to the quasilinear nature of the relativistic Euler equations, this system exhibits extremely rich dynamics beyond local well-posedness theory. Even if one considers the relativistic Euler equations on Minkowski spacetime (flat spacetime) and \textit{without} a free boundary, it is possible that various types of singular behaviors might occur, many of which being under active investigation. In a series of monographs by Christodoulou \cite{christodoulou2007formation,christodoulou2017shock} and a work by Christodoulou-Lisibach \cite{christodoulou2016shock}, they studied the shock formation and the restricted shock development problem for the relativistic Euler equations, especially for open sets of irrotational, isentropic data in three dimensions which are compactly supported perturbations of non-vacuum constant fluid background. We refer to the recent review article \cite{abbrescia2023relativistic} and references therein for a thorough discussion on shock-type singularities for relativistic Euler equations. In a recent breakthrough by Shao--Wei--Zhang \cite{shao2024self}, the authors constructed self-similar \textit{imploding} solutions to the relativistic Euler equations with the isothermal equation of state $p = \frac1l\rho$, $l > 1$, in three dimensions. The imploding singularity is, in some sense, more singular than shocks, as the hydrodynamic variables $\rho$ or $u$ themselves blowup. The imploding solutions found in \cite{shao2024self}, parallel to those discovered in the groundbreaking work \cite{MRRS1} studying the isentropic compressible Euler equations, are spacetime irrotational, spherically symmetrical, and featured by finite-time density blowup. Interestingly, these solutions play a critical role in the study of finite-time singularity formation to the supercritical defocusing nonlinear wave equation. See \cite{shao2025blow,buckmaster2024blowup} for recent developments.


    In contrast with possible singularity formation scenarios, there also have been exciting developments concerning scenarios which lead to global-in-time solutions to relativistic Euler equations. One mechanism of particular interest is the effect of \textit{expansion}. On an intuitive level, sufficiently rapid expansion not only should annihilate imploding singularities, as expansion discourages concentrating effects, but also should suppress shock formations, in that rapid expansion should ``spread out'' characteristics and prevent them from colliding. In the case of relativistic Euler equations without free boundary, the above intuition has been rigorously justified when the background spacetime expands. In particular, small perturbations to a quiescent, homogeneous fluid background with the Friedmann–Lemaître-Robertson–Walker (FLRW) metric that expands at least linearly are nonlinearly stable for all times, given the isothermal equation of state $p = K\rho$, $K \in (0,1/3)$. Such stability even holds in the full Einstein-Euler case, where the background metric is nonlinearly coupled to the fluid part via Einstein equations. See \cite{speck2012nonlinear,speck2013stabilizing,oliynyk2016future,fajman2021stabilizing,fajman2024stability,fajman2025stability} for recent advances in this direction.


    An important parallel to the stabilizing effects induced by the expansion of the spacetime background in the scenario of relativistic fluids is the expansion of the fluid bulk. With the presence of the physical vacuum boundary, the study of expanding solutions to the classical compressible Euler(-Poisson) equations has a rich history, which dates back to the works of \cite{makino1992blowing,fu1998critical}. In a more recent work \cite{sideris2017global}, the author discovered a more general class of explicit solutions, namely the ``affine motions", which feature linear expansion in an infinite-in-time and spatially inhomogeneous manner. In a series of remarkable works \cite{hadvzic2018expanding,hadvzic2018nonlinear,hadvzic2019class,parmeshwar2019global,shkoller2019global}, it is found that asymptotically linear expansions are nonlinear stable under small perturbations, thereby giving rise to a rather large class of initial data which lead to global solutions to compressible Euler(-Poisson) equations. The aforementioned results strongly suggest that an at least linear expansion of the fluid bulk alone is sufficient to suppress possible singularity scenarios despite the presence of vacuum and a free boundary. In the case of isentropic compressible Euler equations, linear expansion of the fluid support is actually necessary in guaranteeing the global existence of a regular solution. See \cite{sideris2014spreading}.

With the above discussions, while the local theory for the relativistic Euler equations with a physical vacuum boundary is relatively well-known, the understanding of large-time dynamics of this model is still limited. In fact, one may wonder whether there exist solutions describing expanding gas in the relativistic Euler scenario. Inspired by the remarkable works \cite{hadvzic2018nonlinear,hadvzic2018expanding}, we give a positive answer the the above question by constructing a class of future-global solutions to \eqref{eq:RE}, \eqref{def: moving domain Omega_t}--\eqref{fallout rate} featured by asymptotically linear expansion of the fluid support. To the best of the authors' knowledge, this article provides the first rigorous construction of an \textbf{open} class of initial data that lead to future-global solutions. A rough version of the main result is formulated as follows:

\begin{thm}(Main theorem: a Rough Version)
    Given $\kk \in (0,\frac23]$, $(\lam_0,\lam_1) \in \R^+ \times (0,1)$, there exists an open set of spherically symmetric initial data supported in $B(0,\lam_0)$ with initial expanding rate $\lam_1$ and initial mass $\delta^\frac1\kk \ll 1$ in a suitable high-order weighted Sobolev space which lead to future-global solutions to the system \eqref{eq:RE}, \eqref{def: moving domain Omega_t}--\eqref{fallout rate}. Moreover, the support of these solutions expands linearly (with respect to an Eulerian observer) in time with an asymptotic rate $\bar\lam \sim \lam_1$. 
\end{thm}

We refer the readers to Theorem \ref{thm:mainprecise} for a precise statement of this main result.

\begin{rmk}
\begin{enumerate}
    \item It should be emphasized that $\lam_1$ \textit{a priori} must verify the bound $\lam_1 < 1$ due to causality, as the speed of light is $1$ in our case. Therefore, our result saturates the regime of causality in this sense.

    \item We expect that our main result can be extended to the full regime $\kappa \in (0,1)$, which is of physical significance, after adapting a modified energy framework established in \cite{shkoller2019global}. For the sake of conciseness, we leave the details to interested readers.
\end{enumerate}
\end{rmk}

We finish this Section with some brief remarks about potential connections with physics. One problem of great interest, for which very little is known from a mathematical perspective, is that of the stability of stars. Stability here does not necessarily mean global existence; in fact, non-stellar end states such as a black-hole or a supernova are natural in many situations. However, stars are generally expected to be long lived, so that their mathematical description should go beyond simple local well-posedness, and some type of ``semi-global" result is thus desirable. The type of expanding solutions we study here do not correspond to a stellar model, but should be viewed within the broader context of long time results for relativistic fluids that can help lead to new ideas for the mathematical study of stars. In this regard, it is important to point out that our solutions satisfy standard physical requirements, such as non-negative energy density, subluminal sound speed, and they satisfy the dominant energy condition.

\subsection{Main Difficulties and Methodology}
We first detail the major difficulties in discovering stable (future-in-time) global scenarios for the relativistic Euler equations with physical vacuum boundary, especially compared with those arising in the study of global dynamics for the classical free-boundary compressible Euler equations.
\begin{enumerate}
    \item \textbf{Degeneracies due to coordinate singularity and the physical vacuum.}\\
    Similar to the free-boundary compressible Euler equations with physical vacuum and radial symmetry, the corresponding problem for the relativistic Euler equations has identical issues of degeneracies near the free interface as well as at the origin. Such degeneracies behoove us to work in a suitable \textit{reformulation} of the original problem. More precisely, we need to work with a set of suitable coordinates as well as a suitable functional framework which takes such degeneracies into account.

    \item \textbf{Necessity of a robust blowup suppressing mechanism.}\\
    Similar to many other compressible fluid equations, the relativistic Euler equations can also experience the formation of singularities in finite time, including those of shock type \cite{christodoulou2007formation} or those of implosion type \cite{shao2024self}. Therefore in hope of constructing any stable global-in-time solution for the relativistic Euler equations, it is necessary to exploit certain robust stabilizing mechanism which suppresses the aforementioned singularity formation scenarios.
    
    \item \textbf{Lack of scaling symmetry for the relativistic Euler equations.}\\
    In contrast to the classical (isentropic) compressible Euler equations with a polytropic equation of state, which enjoy a $2$-parameter family of scaling symmetry corresponding to the hydrodynamic variables $(\rho,u^i)$, the corresponding problem to the relativistic Euler equations do not possess any scaling symmetry in variables $(\rho, u^\mu)$. Specifically, there are two main obstacles that prevent scaling symmetries in the relativistic Euler equations.  First, in the evolution equation for fluid's density $\rho$, the terms $\rho \p_\mu u^\mu$ and $p \p_\mu u^\mu$ scale differently due to the polytropic equation of state $p = \rho^{1+\kk}$, $\kk > 0$. Second, due to the nonlinear constraint equation $m_{\mu\nu}u^\mu u^\nu = -1$, the component $u^0$ ``breaks the scaling,'' even under the assumption of spherical symmetry.

    In general, the absence of scaling symmetry in a nonlinear PDE poses immense difficulties for understanding possible global dynamics. From the point of view of discovering special global solutions, a lack of scaling symmetry generally prevents one from reducing the PDE to a more tractable ODE. Additionally, from the perspective of stability analysis, the lack of scaling symmetry and the presence of special solutions generally discourage attempts to construct a class of global solutions near a (stable) background via perturbative methods.

    \item \textbf{Source errors and potential derivative loss}\\
    However, it should be noted that proving stability results outside admissible scaling symmetry by the target PDE is possible, as long as, roughly speaking, the main decay mechanism is sufficiently robust, and the error terms outside the scaling are ``subcritical'' and enjoy a favorable (non)-linear structure. In fact, such overarching scheme has been classically deployed in the study of (asymptotically) stable phenomena in many nonlinear PDEs, particularly in those concerning fluid dynamics, e.g. the study of imploding-type singularities \cite{cao2023non,cao2024non,MRRS2,MRRSNLS,shao2025blow,buckmaster2024blowup}, gravitational collapse \cite{guo2021continued}, and global gas dynamics of compressible Euler(-Poisson) equations \cite{hadvzic2019class,parmeshwar2019global}. Motivated by this idea, we enforce a ``mass-critical'' scaling ansatz to the solution (see Section \ref{sect: setup} for a detailed discussion). As a consequence, two types of main errors emerge in the equations for the rescaled variables:
    \begin{itemize}
        \item \textbf{Type I:} errors with favorable decay but might lose derivatives;
        \item \textbf{Type II:} errors at low frequencies, but generally with large size and without decay.
    \end{itemize}
    Type I errors typically arise from the pressure $p$ and its space-time derivatives; Type II errors emerge from the component $u^0$ of the $4$-velocity. Both types of errors pose serious obstacles to the attempt to close a continuation argument based upon the energy method.
    \end{enumerate}

    \subsubsection{Methodology}
    In view of the study of classical compressible Euler equations (\cite{hadvzic2018nonlinear,hadvzic2018expanding,sideris2017global}), we expect that linear expansion of the fluid bulk is a viable mechanism that gives rise to global solutions. To capture the effects of expansion, a crucial insight revealed by the aforementioned works is the correspondence in between linear expansion with a \textit{mass-critical scaling}. We are thus motivated to make the following ansatz to the unknowns of the relativistic Euler equations with spherical symmetry:
    \begin{equation}\label{introansatz}
    \rho(t,r) = \lambda(s)^{-3}\trho(s,z),\quad v(t,r) = \lambda(s)^{-\frac32\kappa}\tv(s,z),\quad \frac{ds}{dt} = \lambda(t)^{-\frac32\kappa - 1},\quad z = \frac{r}{\lambda(t)},
    \end{equation}
    where $\lam(t)$ solves the following ODE:
    \begin{equation}\label{introODE}
    \p_t^2 \lam = \delta\lam^{-3\kk-1},
    \end{equation}
    for some parameter $\delta$, and has the asymptotics $\lam(t) \sim t$. Here, $r = |x|$ and $u^i(t,x) = v(t,r)\frac{x^i}{r}$. We remark that the ODE \eqref{introODE} is classical in the study of expansion/gravitational collapse of gaseous stars. See \cite{makino1992blowing,fu1998critical}. Also note that we can always solve for the component $u^0$ of the $4$-velocity via the following algebraic relation: $$u^0(t,r)^2 = 1+v(t,r)^2.$$

    The rescaled variables $(\trho, \tv)$ satisfy a rather complicated, modified relativistic Euler equation. To further recast it into a more convenient form, we introduce the modified velocity $V$ given by:
    $$
    \frac{\tv}{u^0} = V + \frac{\p_s\lam}{\lam}z,
    $$
    which is reminiscent to a similar definition in \cite{hadvzic2018expanding}. In the new variables $(\trho, V)$, the relativistic Euler equations can be written in the following schematic form:
    \begin{subequations}
        \label{schematic1}
        \begin{equation}\label{schematic1mass}
            u^0\left(\p_s\trho + D_z(\trho V)\right) = \trho \left(\p_s u^0 + V\p_z u^0\right) + \lam^{-3\kk} F_1(\trho, u^0, V, \p_z V),
        \end{equation}
        \begin{equation}
            \label{schematic1mom}
            \begin{split}
            &g(\trho)\left[\left(\p_s V + V\p_z V + \left(1-\frac32\kk\right)\frac{\p_s\lam}{\lam} V + \delta z\right) + \frac{V + \frac{\p_s\lam}{\lam}z}{u^0}(\p_s u^0 + V \p_z u^0)\right] + \frac{1+\kk}{\kk}\p_z(\trho^\kk)\\
            &\quad = \lam^{-3\kk}F_2(V,\trho, \p_s\trho, \p_z\trho),
            \end{split}
        \end{equation}
    \end{subequations}
where $D_z := \p_z + \frac2z$ is the divergence operator restricted to radial functions in three dimensions, and $g, F_1, F_2$ are explicit smooth functions in their arguments. We also recall that $$(u^0)^2 = 1+v^2 = 1 + \lam^{-3\kk}\tv^2.$$

Readers who are familiar with the expanding gas dynamics of the classical compressible Euler equations might have noticed that \eqref{schematic1} formally reduces to the compressible Euler equations in an expanding background after setting $u^0, g \equiv 1$, $F_1,F_2 \equiv 0$. However, we emphasize that \eqref{schematic1} should not be viewed as a trivial extension of the previous results \cite{hadvzic2018expanding,shkoller2019global}: firstly, in both \eqref{schematic1mass} and \eqref{schematic1mom}, the following term
$$
\p_s u^0 + V \p_z u^0,
$$
which is absent in the study of the classical compressible Euler equations, is not straightforward to treat. Secondly, from the perspective of derivative count, the forcing terms $F_1, F_2$ involve terms to the top order, despite the fact that there is a favorable decaying factor $\lam^{-3\kk}$ in front of them. In view of the transport nature of \eqref{schematic1mass} and \eqref{schematic1mom}, there might be a potential derivative loss. \textit{We remark that the above issues are fundamentally tied to the main difficulty concerning the lack of scaling symmetry for the relativistic Euler equations}.

To address the issues raised in the previous paragraph, we need to fully exploit the nonlinear structure of the relativistic Euler equations, particularly in the Lagrangian formulation. We now proceed to further elucidate this remark in greater details. Define the Lagrangian map $\eta$ generated by the modified velocity $V$ as
$$
\frac{d\eta}{ds}(s,\ze) = V(s,\eta(s,\ze)),
$$
where $\eta \equiv \ze$ corresponds to the expanding background. We first comment on the crucial role played by this Lagrangian formulation in the mass equation. In the study of classical compressible Euler(-Poisson) dynamics (e.g. \cite{hadvzic2018expanding,hadvzic2018nonlinear,hadvzic2019class}), the mass equation trivializes to the conservation of the quantity $f\calF$ in time, where $f = \trho\circ \eta$ and $\calF = \left(\frac{\eta}{\ze}\right)^2\pz \eta$, the latter of which being exactly the Jacobian of the Lagrangian gradient. In the case of the relativistic Euler equations, the quantity $f\calF$, nevertheless, satisfies a more complicated PDE, which corresponds to the fact that the relativistic Euler does not conserve mass! To address this difficulty, a novel observation is that the equation for $f\calF$ bears a remarkable nonlinear structure which only manifests itself in Lagrangian coordinates. In particular, we fundamentally exploit the following elementary yet important observation: let $U^0 = u^0 \circ \eta$. We have
$$
\p_s U^0 = (\p_s u^0 + V \p_z u^0)\circ \eta,
$$
which reduces the first forcing term on the RHS of \eqref{schematic1mass} (in Lagrangian coordinates) to a total time derivative. Moreover, after a careful design of a structural ansatz for the Lagrangian density $f$ (cf. \eqref{f ansatz}), the mass equation \eqref{schematic1mass} in the Lagrangian variable can be equivalently written as an evolution equation for the so-called \textit{relativistic correction} $\calG$. More importantly, this evolution equation can be viewed as a nonlinear \textbf{ODE} in $s$. This novel observation allows us to almost explicitly solve for the Lagrangian density $f$ in terms of fundamental quantities $\calF$ and $U^0$ (but not their derivatives), which are more amenable to analysis. See Subsection \ref{subsect:barG} for a detailed discussion.

As for the momentum equation in Lagrangian coordinates, we utilize another structure concerning $\p_s u^0 + V\p_z u^0$ in Lagrangian coordinates, namely:
$$
\left[\frac{V + \frac{\p_s\lam}{\lam}z}{u^0}(\p_s u^0 + V \p_z u^0)\right]\circ \eta = ((U^0)^2 - 1) \left(\p_s V + V\p_z V + \left(1-\frac32\kk\right)\frac{\p_s\lam}{\lam} V + \delta z\right)\circ \eta,
$$
which can be combined with other main terms in the square bracket in \eqref{schematic1mom}. Moreover, using the structural ansatz for $f\calF$ (cf. \eqref{f ansatz}), the pressure gradient $\p_\ze (\trho^\kk)$ in Lagrangian coordinates can be reduced to a degenerate elliptic operator parallel to the one appearing in the study of expanding Euler gas dynamics. See, for instance, \cite{hadvzic2018expanding,hadvzic2018nonlinear,hadvzic2019class}.

After introducing the perturbed Lagrangian $\Theta := \eta - \ze$ and a logarithmic time $\tau \sim \log(1+t)$, we may schematically rewrite \eqref{schematic1} as
\begin{equation}
    \label{schematic2}
    \begin{split}
    &g^0(\lam,w,\calF,\calG)\left[\delta^{-1}\lam^{3\kk} \left(\pt^2\Theta + \frac{\pt\lam}{\lam}\pt\Theta\right) + \Theta\right] + g^1(U^0,\Theta) \frac{1}{w^{\frac1\kk}}\pz \left(w^{1+\frac1\kk}(\calF^{-1-\kk} - 1)\right)\\
    &\quad= \bar{F}(\lam, w, U^0, \Theta, \calF, \calG, \pt\Theta,\p_{\tau,\ze}\calF, \p_{\tau,\ze}\calG),
    \end{split}
\end{equation}
where $g^0, g^1, \bar F$ are explicit smooth functions in their arguments. Moreover, $w: [0,1] \to \R^+$ is the admissible enthalpy profile satisfying $\delta w = (f\calF)^\kk|_{\tau = 0}$. See Subsection \ref{subsect:enthalpy} and \eqref{gauge fixing} for a detailed discussion. We also remark that the appearance of the favorable coefficient $\delta^{-1}\lam^{3\kk}$ on the LHS of \eqref{schematic2} reflects the sub-criticality of the problem under the mass-critical scaling ansatz, in the sense that such favorable factor originates from the fact that $p$ scales subcritically with respect to $\rho$ under scaling \eqref{introansatz} due to the polytropic equation of state.

Now, the main problem is transformed into proving the global existence for \eqref{schematic2} with small data. We accomplish this goal by adapting an energy method in a weighted framework introduced by Jang-Masmoudi \cite{jang2015well}. We briefly introduce this approach here. On the $i^{th}$ level energy estimate, the LHS of \eqref{schematic2} gives rise to the following energy with damping given $\kk \le \frac23$. Namely,
\begin{equation}\label{schematicenergy}
\begin{split}
&\frac12\frac{d}{d\tau}\left(\delta^{-1}\lam^{3\kk}\|w^{\frac{\frac1\kk + i}{2}}\ze D^i\pt\Theta\|_{L^2}^2 + \|w^{\frac{\frac1\kk + i}{2}}\ze D^i\Theta\|_{L^2}^2 + \|(U^0)^{-2}w^{\frac{\frac1\kk + i + 1}{2}}\ze D^{i+1}\Theta\|_{L^2}^2\right)\\
&\quad + \left(1-\frac32\kk\right)\delta^{-1}\lam^{3\kk}\|w^{\frac{\frac1\kk + i}{2}}\ze D^i\Theta\|_{L^2}^2,
\end{split}
\end{equation}
where $D^j$, $j \in \N$, should be viewed an appropriate $j^{th}$ order derivative in $\ze$. In contrast with the classical Euler equations, an additional factor $(U^0)^{-2}$ enters the highest order spatial derivative control. This factor essentially corresponds to the restriction on the choice of initial data $\lam_1 < 1$, since $U^0 \to \infty$ as $\lam_1 \to 1$, which captures the causal structure tied to the relativistic Euler equations.

Regarding the forcing term $\bar F$ in \eqref{schematic2}, it turns out that one may perform the following decomposition:
$$
D^i \bar F = F^i_{\text{Type I}} + F^i_{\text{Type II}}.
$$
In particular, $F^i_{\text{Type I}}$ can be written schematically as
$$
F^i_{\text{Type I}} = a_1 \left(w\ze  D^{i+1}\pt\Theta\right) + a_2 D^i\pt^2\Theta + a_3 D^{i+2}\Theta,
$$
where $a_j \in W^{1,\infty}_{\tau,\ze}$ are suitably bounded. Moreover, $F^i_{\text{Type II}}$ at best obeys the bound
$$
F^i_{\text{Type II}} \lesssim 1 + \calE^i,
$$
where $\calE^i$ is the natural energy induced by \eqref{schematicenergy}. Indeed, $F^i_{\text{Type I}}$ and $F^i_{\text{Type II}}$ correspond to the Type I and Type II errors mentioned in Main Difficulty 4 above. On the one hand, the error $F^i_{\text{Type I}}$ will be properly controlled after a careful study of its exact top order structure and the structural symmetry afforded by the main equation \eqref{schematic2}. On the other hand, $F^i_{\text{Type II}}$ does not lose derivative, but an $\calO(1)$ size error emerges precisely due to the low-order derivatives of $U^0$ and ``profile errors'' incurred by $w$. We refer the readers to Section \ref{sect:buildingblock} for a detailed discussion. Finally, in the treatment of the two types of errors above, a new technical aspect of our work, compared with the study of the classical compressible Euler equations, is the development of an additional control of the Lagrangian acceleration $\pt^2\Theta$ up to the $(i-1)^{th}$ order in an $L^2_{\tau,\ze}$ setting. Such control is necessary when we treat $\pt\calG$, from which $\pt U^0$ arises.

With the main high-order norm $E^N$ introduced, we show the following main energy inequality:
\begin{equation}
    \label{mainenergyintro}
    E^N(\tau) \lesssim (\eps + \delta^2) + \int_0^\tau e^{-\frac32\kk \bar\lam \tilde\tau} E^N(\tilde\tau) d\tilde\tau + \delta^\frac12 \int_0^\tau e^{-\frac32\kk \bar\lam \tilde\tau} (E^N(\tilde\tau))^\frac12 d\tilde\tau,
\end{equation}
where $\bar\lam \sim \lam_1$ and $\eps$ measures the size of perturbation $\Theta$. The structure revealed by \eqref{mainenergyintro} necessitates the parameter regime $\delta \le \eps \ll 1$ to complete the proof of global existence via a continuation argument.

\subsection{Organization of the Paper}
\begin{itemize}
    \item In Section \ref{sect: setup}, we introduce the relativistic Euler equations in spherical symmetry and facilitate the notion of mass-critical scaling ansatz. With this structural assumption, the introduction of a suitable Lagrangian framework, and an important notion of relativistic corrector which characterizes the evolution of the density of the gas, we derive a modified version of the relativistic Euler equations tailored to the stability analysis in later sections. 

    \item In Section \ref{sect: functional}, we lay out the main functional preliminaries deployed in this article, including a vector field method developed by \cite{guo2021continued} and a weighted framework introduced by \cite{jang2015well}. We also define a hierarchy of spacetime norms $E^N$ that will be crucial in the nonlinear stability analysis.

    \item In Section \ref{sec: elliptic operator}, we study the term contributed by the pressure gradient in the momentum equation under the Lagrangian framework and recast it into a form that is more amenable to the analysis at high orders.

    \item In Section \ref{sec: high-order equation}, we introduce the high-order differentiated main equation after commuting the ``good derivative'' $\calD_i$.

    \item In Section \ref{sect:buildingblock}, we record detailed expressions concerning various crucial quantities, highlighted by $U^0, \calF, \calG$, after the vector field $\calD_i$ is applied. Here, $U^0$ and $\calG$ are key quantities which keep track of the relativistic effects. In particular, these expressions reveal very delicate highest order structures which are compatible with the main energy estimate culminating in Section \ref{sect: Energy Estimate V}.

    \item In Section \ref{sect: energynorm}, we introduce the relevant norms and the energies for the main equation that we use in our global existence argument. We also set up the main bootstrap assumptions, which involve the smallness of a spacetime norm measuring the size of perturbed variables, and suitable boundedness of key quantities $U^0, \calF, \calG$. We conclude this section by showing several helpful consequences of the main bootstrap assumptions, one of which being the norm-energy equivalence.

    \item Section \ref{sect: Energy Estimate 1}--\ref{sect: Energy Estimate III} contain the majority of technical estimates: Section \ref{sect: Energy Estimate 1} addresses estimates of nonlinear errors arising from the pressure gradient as hinted in Section \ref{sec: elliptic operator}; Section \ref{sect: Energy Estimate II} addresses commutators generated by commuting $\pt$ across the main equation; Section \ref{sect: Energy Estimate III} estimates terms in the momentum equation which encode relativistic effects, and the analysis of which crucially rely on the structural features observed in Section \ref{sect:buildingblock}.

    \item Section \ref{sect: Energy Estimate 4} shows that a suitable $L^2$ spacetime norm for the Lagrangian acceleration can be bounded by the main norm $E^N$. In Section \ref{sect: Energy Estimate V}, we combine all estimates derived in Section \ref{sect: Energy Estimate 1}--\ref{sect: Energy Estimate 4} to conclude our main energy inequality. 

    \item In Section \ref{S:LWP}, we provide a basic local existence and uniqueness result for the main system of equations to be studied. We remark that this will not be a new construction but simply a translation of known results to our setting.

    \item In Section \ref{sect: mainthm}, we exploit the aforementioned main energy inequality to prove the main theorem, which is a small-data, global existence theorem for the modified equation derived in Section \ref{sect: setup}. Moreover, we prove that the convergence of the Lagrangian to end-state as the (conformal) time $\tau \to \infty$, from which we obtain a precise asymptotic description of the gas.
\end{itemize}

\subsection{List of Notations}
\begin{itemize}
\item $\bar\lam$ is the asymptotic expansion rate of the support of the fluid; $N = N(\kk)$ is the regularity threshold of the solution. 
    \item We denote by $C>0$ a generic constant, which is allowed to change from line to line, depends on $\kk, \bar\lam$, and $N$, but is independent of bootstrap parameters $M_*, \epsilon$, and $\delta$ (see Definition \ref{assump: bootstrap}).  On the other hand, we use $C_\alpha$, $\alpha=0,1,2,\cdots,$ to denote a generic constant satisfying the properties of $C$, but it is not allowed to change from line to line. 
    \item By $A\lesssim B$ we mean $A\leq CB$ for some constant $C>0$, which is independent from $M_*, \epsilon$, and $\delta$. Also, by $A\approx B$ we mean $B\lesssim A \lesssim B$. 
\end{itemize}
\subsection*{Acknowledgments}
MMD acknowledges support from NSF grants DMS-2406870 and NRT-2125764, and from DOE grant DE-SC0024711.
ZH acknowledges support from NSF grants DMS-2306726, DMS-2106528, and from a Simons Collaboration
Grant 601960. CL acknowledges support from Hong Kong RGC Grants CUHK--14302922, CUHK--14304424, and CUHK--14301225. 

\section{A Lagrangian Formulation of relativistic Euler Equations in Spherical Symmetry}\label{sect: setup}

\subsection{Relativistic Euler Equations in Spherical Symmetry}
We focus on spherically symmetric solutions to the relativistic Euler equations with a physical vacuum boundary \eqref{eq:RE}. Let $r=|x|$. We consider 
\begin{align}\label{ansatz radial symmetry}
\rho = \rho(t,r),\quad u^i = v(t,r) \frac{x^i}{r}, \quad i=1,2,3. 
\end{align}
Then, we infer from \eqref{eq: RE_normalization} that
\begin{equation}\label{u00}
(u^0)^2(t,r) = 1 + u_iu^i = 1 + v^2.
\end{equation}

\noindent Under the ansatz \eqref{ansatz radial symmetry}, the mass equation \eqref{eq:RE_mass} becomes
\begin{equation}\label{mass0}
u^0\p_t\rho + \left(v\p_r\rho + \rho (\p_r v + \frac2r v)\right) + (p+\rho)\p_t u^0 + p (\p_r v + \frac{2}{r}v) = 0, \quad \text{in }\Omega_t.
\end{equation}
Note that, under the spherically symmetric setting, $\Omega_t=B(0,r(t))$ for some $r(t)>0$, i.e., the ball centered at the origin with radius $r(t)$. 
Additionally, the momentum equation \eqref{eq:RE_mom} can be written as
\begin{equation}\label{momentum0}
(\rho + p)\left[u^0\p_tv + v\p_r v\right] + (u^0)^2\p_r p + vu^0\p_tp = 0, \quad \text{in }\Omega_t,
\end{equation}
after applying \eqref{u00} to \eqref{eq:RE_mom} with $\alpha=i$.
\subsection{Scaling Ansatz} 
Consider the following change of variables:
\begin{equation}\label{defn:scaling}
\rho(t,r) = \lambda(s)^{-3}\trho(s,z),\quad v(t,r) = \lambda(s)^{-\frac32\kappa}\tv(s,z),\quad \frac{ds}{dt} = \lambda(t)^{-\frac32\kappa - 1},\quad z = \frac{r}{\lambda(t)},
\end{equation}
where the scaling factor $\lam(t)$ satisfies the following ODE:
\begin{equation}
    \label{lambdaODEt}
    \p_t^2 \lambda = \delta \lambda^{-3\kk - 1},
\end{equation}
equipped with initial conditions 
\begin{align}\label{lambdaODEt_data}
\lam(0) = \lam_0,\quad \p_t \lam(0) = \lam_1,\quad (\lam_0,\lam_1)\in \R^+\times (0,1).
\end{align}
Here, $\lam_1$ is the initial expanding rate, and we choose $r(0)=\lam_0$ so that $\lam_0$ is the radius of the initial domain. Moreover, $\delta$ is a free parameter, which should be viewed as a measure of the initial mass of the fluid body. It is worth mentioning that the scaling assumption \eqref{defn:scaling} is classically referred as the \textit{mass-critical scaling} in the study of compressible Euler equations. In fact, such scaling corresponds to an exact scaling symmetry enjoyed by the compressible Euler equations, from which a large class of nonlinearly stable affine motions emerges (see \cite{sideris2017global,hadvzic2018expanding}).

Recalling the rescaled time $s$, we will also frequently use the ODE \eqref{lambdaODEt} in $s$-variable:
\begin{align}\label{ODE}
\left(\frac{\lambda'}{\lambda}\right)' - \frac32\kappa\left(\frac{\lambda'}{\lambda}\right)^2 = \delta,
\end{align}
where we used the convenient notation $\lambda' := \p_s\lambda$. We remark that the ODE \eqref{lambdaODEt} exactly reflects the asymptotically linear expansion of the domain, an important fact which we summarize in the following technical lemma:
\begin{lem}
    \label{lem:lambdaODE}
    The initial value problem \eqref{lambdaODEt}--\eqref{lambdaODEt_data} admits a unique, smooth solution $\lam(t)$ for all $t \ge 0$. Moreover, there exist constants $\delta_0 > 0, C>0$ such that for any $\delta \in (0,\delta_0)$, the unique solution $\lambda(t)$ to \eqref{lambdaODEt}--\eqref{lambdaODEt_data} can be written in the form
\begin{equation}
    \label{lambdaasymp}
    \lambda(t) = a^\delta(t) + t\bar\lambda^\delta,
\end{equation}
where $\bar\lambda^\delta$ is time independent and
$$
|\bar\lambda^\delta - \lambda_1| \le C\delta,\quad |\p_t^2 a^\delta| \le C\delta (1+t)^{-3\kk-1}.
$$
Furthermore, consider the following time rescaling
$$
\frac{d\tau}{dt}= \frac{1}{\lambda(t)}.
$$
Then the following bounds hold:
\begin{equation}\label{est:lambdaasym}
C^{-1}e^{\bar\lambda^\delta \tau} \le \lambda(\tau) \le C e^{\bar\lambda^\delta \tau},\quad\tau\ge 0,
\end{equation}
and for $\lamt^\delta := \frac{\pt\lam}{\lam}$, we have
\begin{equation}\label{est:ptlamt}
\left|\p_\tau \lamt^\delta\right| \le C\delta e^{-3\kk \bar\lambda^\delta \tau}.
\end{equation}
Finally, $\lim_{\tau \to \infty}\lamt^\delta = \bar\lam^\delta$ and the following estimate holds:
\begin{equation}
    \label{est:difflamt}
    |\lamt^\delta - \bar\lam^\delta| \le \frac{C}{\kk\bar\lam^\delta}\delta e^{-3\kk\lam^\delta \tau}.
\end{equation}
\end{lem}
\begin{proof}
    All estimates except for \eqref{est:ptlamt} and \eqref{est:difflamt} were proved in \cite[Lemma 2.1]{hadvzic2019class}. To prove \eqref{est:ptlamt}, we note that
    $$
    \p_\tau \left(\frac{\p_\tau\lambda}{\lambda}\right) = \lambda \p_t^2\lambda = \delta\lambda^{-3\kk},
    $$
    where we used the ODE for $\lambda$. Then the estimate follows immediately from \eqref{est:lambdaasym}.

    To prove \eqref{est:difflamt}, we note from the definition of $\tau$ and \eqref{lambdaasymp} that
    $$
    \lamt- \bar\lam = \p_t \lambda- \bar\lam = \p_t a^\delta,
    $$
    which converges to $0$ thanks to the integrable decay bound $|\p_t^2 a^\delta| \le C\delta (1+t)^{-3\kk-1}$ when $\kk > 0$. Moreover,
    $$
    |\lamt^\delta - \bar\lam^\delta| \le \int_\tau^\infty |\p_t^2 a^\delta(t(\tilde\tau))|\lam(\tilde \tau) d\tilde\tau \lesssim \delta \int_\tau^\infty e^{-3\kk\bar\lam^\delta \tilde\tau} d\tilde\tau \lesssim \frac{1}{\kk\bar\lam^\delta}\delta e^{-3\kk\lam^\delta \tau},
    $$
    which concludes the proof of \eqref{est:difflamt}.
\end{proof}
\begin{rmk}
    Since we will perform our analysis after fixing the parameter $\delta$, we will write $\lamt$ and $\bar\lambda$ instead of $\lamt^\delta$ and $\bar\lambda^\delta$ for the rest of this article.
\end{rmk}

With the structural assumptions established above, a straightforward application of chain rule yields the following equations for $(\trho, \tv)$:
\begin{equation}
    \label{mass1}
    \begin{split}
    u^0&\left(\p_s \trho - 3\frac{\lambda'}{\lambda}\trho - \frac{\lambda'}{\lambda}z\p_z \trho\right) + \left( \tv \p_z \trho + \trho(\p_z\tv + \frac{2}{z}\tv) \right)\\
    &= -\lambda^{-3\kappa}\trho^{1+\kappa}(\p_z\tv + \frac{2}{z}\tv)- \lambda^{\frac{3}{2}\kappa + 1}(\lam^{-3\kk}\trho^{1+\kk}+\trho)\p_t u^0,\quad \text{in } \widetilde{\Omega}_s, 
    \end{split}
\end{equation}
where $\widetilde{\Omega}_s = B\left(0, \frac{r(s)}{\lam(s)}\right)$, and 
\begin{equation}
    \label{momentum1}
    \begin{split}
        &\left(1 + \lam^{-3\kk}\trho^\kk\right)\left[\left(\p_s \tv - \frac32\kappa \frac{\lambda'}{\lambda}\tv - \frac{\lambda'}{\lambda}z\p_z\tv\right) + \frac{\tv}{u^0}\p_z\tv\right] + \frac{1+\kappa}{\kappa}u^0\p_z(\trho^\kappa)\\
        &= -\lambda^{-3\kappa}\frac{1+\kappa}{\kk}\tv \left(\p_s (\trho^\kk) - 3\kk\frac{\lambda'}{\lambda}\trho^\kk - \frac{\lambda'}{\lambda}z\p_z (\trho^\kk)\right),\quad \text{in } \widetilde{\Omega}_s. 
    \end{split}
\end{equation}
Note that in \eqref{mass1}, we keep the last term in its original form without expressing it in tilde variables until we introduce the modified velocity in the upcoming section.

\subsection{Modified Velocity}
We now introduce the modified velocity $V$ given by the following relation:
\begin{equation}
    \label{modvel}
    \frac{\tv}{u^0}  = V + \frac{\lambda'}{\lambda}z.
\end{equation}
It turns out that the system \eqref{mass1}--\eqref{momentum1} can be recast into a more tractable form after a rewriting in $(\trho, V)$ variables. To motivate the modified velocity defined in \eqref{modvel}, one should view $V$ as a perturbation around the background (asymptotically) linear expansion afforded by $\frac{\lam'}{\lam}z$. Moreover, the introduction of the weight $u^0$ in \eqref{modvel} is natural in view of the transport structure of the relativistic Euler equations: the material derivative associated with \eqref{eq:RE} is $u^\mu \p_\mu$, which is equivalent to $\p_t + \frac{u^i}{u^0}\p_i$ due to the lower bound $u^0 \ge 1$.

By combining \eqref{u00} and \eqref{modvel}, we obtain the following useful formula for $u^0$: 
\begin{equation}
    \label{u0}
    (u^0)^2(s,z) = \frac{1}{1-\lambda^{-3\kappa}\left(V + \frac{\lambda'}{\lambda}z\right)^2}.
\end{equation}
Hence, we are able to rewrite \eqref{mass1}, \eqref{momentum1} in terms of equations for $\trho, V$. First, we treat the mass equation \eqref{mass1}. By employing the identities 
\begin{align*}
    &\tv \p_z \trho = u^0 \left(V + \frac{\lambda'}{\lambda}z\right)\p_z \trho,\\
    &\trho(\p_z \tv + \frac{2}{z}\tv) = \trho \left(V + \frac{\lambda'}{\lambda}z\right) \p_zu^0 + u^0\trho(\p_z V + \frac{2}{z}V + 3\frac{\lambda'}{\lambda} ),
\end{align*}
and plugging these two identities in \eqref{mass1}, we obtain
\begin{equation}
    \label{mass1.5}
    \begin{split}
    u^0&\left(\p_s \trho  +  V \p_z \trho +\trho(\p_z V + \frac{2}{z}V )\right)\\
    &= -(1+\lam^{-3\kk}\trho^\kk)\trho\left[\left(V + \frac{\lambda'}{\lambda}z\right) \p_zu^0 +\lambda^{\frac{3}{2}\kappa + 1}\p_t u^0\right] - \lam^{-3\kk}\trho^{1+\kk}u^0\left(\p_z V + \frac{2}{z}V + 3\frac{\lam'}{\lam}\right).
    \end{split}
\end{equation}
Next, we derive an important reformulation of the expression in square brackets appearing in the first term on the RHS of \eqref{mass1.5}.
\begin{lem}
\label{lem:u^0 cancellation}
    The following identities hold:
    \begin{align}\label{eq: u^0 cancellation}
        \begin{aligned}
            \left(V + \frac{\lambda'}{\lambda}z\right) \p_zu^0 &+\lambda^{\frac{3}{2}\kappa + 1}\p_t u^0 = \p_s u^0 + V\p_z u^0\\
            &=\lam^{-3\kk}(u^0)^3 \left(V+\frac{\lam'}{\lam}z\right)\left[\p_s V+ V\p_z V + \delta z + \frac{\lam'}{\lam}\left(1-\frac32\kk\right)V\right].
        \end{aligned}
    \end{align}
\end{lem}
\begin{proof}
    We first address the first equality. Since $\frac{ds}{dt} = \lambda^{-\frac32\kappa - 1}$ and $\frac{dz}{dt} = -\frac{\lambda'}{\lambda}\lambda^{-\frac32\kappa - 1}z$, we have
    \begin{align*}
    \lam^{\frac32\kk + 1}\p_t u^0 &= \lam^{\frac32\kk + 1}\left(\lambda^{-\frac32\kappa - 1}\p_s u^0 - \frac{\lambda'}{\lambda}\lambda^{-\frac32\kappa - 1}z\p_z u^0\right)\\
    &= \p_s u^0 -\frac{\lambda'}{\lambda}z\p_z u^0,
    \end{align*}
    and the first equality follows from the identity above.
    
    To address the second equality, by \eqref{u0}, a straightforward computation gives: 
    \begin{align}\label{p_z u^0}
    \p_z u^0 &= (u^0)^3 \lambda^{-3\kappa}\left(V + \frac{\lambda'}{\lambda}z\right)\left(\p_zV + \frac{\lambda'}{\lambda}\right),
\end{align}
and 
\begin{align}\label{p_s u^0 pre}
\begin{aligned}
    \p_s u^0 &= -\frac12 (u^0)^3 \p_s (u^0)^{-2} = \frac12 (u^0)^3 \p_s\left(\lam^{-3\kk}\left(V + \frac{\lam'}{\lam}z\right)^2\right)\\
    &= (u^0)^3\lam^{-3\kk}\left(V + \frac{\lam'}{\lam}z\right)\left[\p_s V - \frac32\kk \frac{\lam'}{\lam} V + \left(\left(\frac{\lam'}{\lam}\right)' - \frac32\kk \left(\frac{\lam'}{\lam}\right)^2\right)z\right]\\
    &= (u^0)^3\lam^{-3\kk}\left(V + \frac{\lam'}{\lam}z\right)\left[\p_s V - \frac32\kk \frac{\lam'}{\lam} V + \delta z\right],
    \end{aligned}
\end{align}
where the final equality follows from \eqref{ODE}. Then the second equality in \eqref{eq: u^0 cancellation} is proved after we combine \eqref{p_z u^0} and \eqref{p_s u^0 pre}.
\end{proof}

\noindent Combining \eqref{eq: u^0 cancellation} and \eqref{mass1.5}, we obtain, after dividing $u^0$ on both sides, that
\begin{equation}
    \label{mass2}
    \begin{split}
        &\p_s\trho + V\p_z\trho + \trho (\p_z V + \frac{2}{z}V)\\
        &= -(1+\lam^{-3\kk}\trho^\kk)\trho \frac{\p_s u^0 + V\p_z u^0}{u^0}-\lam^{-3\kk}\trho^{1+\kk}\left(\p_z V + \frac{2}{z}V + 3\frac{\lam'}{\lam}\right).
    \end{split}
\end{equation}

Next, we rewrite the momentum equation \eqref{momentum1}. 
We begin by treating terms in the square brackets appearing on the LHS of \eqref{momentum1}. In light of \eqref{modvel} and \eqref{ODE}, we have
\begin{align*}
     \left(\p_s \tv - \frac32\kappa \frac{\lambda'}{\lambda}\tv - \frac{\lambda'}{\lambda}z\p_z\tv\right) &+ \frac{\tv}{u^0}\p_z\tv = u^0\left(\p_s V + (1-\frac32\kappa)\frac{\lambda'}{\lambda}V + V\p_z V + \delta z\right)\\
     &\quad + (V + \frac{\lambda'}{\lambda}z)(\p_s u^0 - \frac{\lambda'}{\lambda}z\p_z u^0 + (V + \frac{\lambda'}{\lambda}z)\p_zu^0)\\
     &= u^0\left(\p_s V + (1-\frac32\kappa)\frac{\lambda'}{\lambda}V + V\p_z V + \delta z\right) + (V + \frac{\lambda'}{\lambda}z)(\p_s u^0 + V\p_z u^0).
\end{align*}
By plugging this identity into \eqref{momentum1} and using \eqref{modvel} again and then dividing $u^0$ on both sides, we arrive at
\begin{equation}
    \label{momentum2}
    \begin{split}
    &\left(1+\lambda^{-3\kk}\trho^\kk\right)\bigg[\left(\p_s V + (1-\frac32\kappa)\frac{\lambda'}{\lambda}V + V\p_z V + \delta z\right) + \frac{1}{u^0}(V + \frac{\lambda'}{\lambda}z)(\p_s u^0 + V\p_z u^0)\bigg] + \frac{1+\kappa}{\kappa}\p_z (\trho^\kappa)\\
        &= -\lambda^{-3\kappa}\frac{1+\kappa}{\kk}(V+\frac{\lambda'}{\lambda}z) \left(\p_s (\trho^\kk) - 3\kk\frac{\lambda'}{\lambda}\trho^\kk - \frac{\lambda'}{\lambda}z\p_z (\trho^\kk)\right).
    \end{split}
\end{equation}
\begin{rmk}\label{rmk: u^0 cancellation}
    We note that in \eqref{mass2} and \eqref{momentum2}, the derivatives of $u^0$ arise in a very specific way; namely, only the material derivative induced by the modified velocity $V$ emerges. We will crucially take advantage of this observation in the Lagrangian formulation in the section afterwards.
\end{rmk}

\subsection{Lagrangian Formulation}
Our next aim is to express \eqref{mass2} and \eqref{momentum2} in suitable Lagrangian coordinates, thereby enabling us to formulate these equations on a fixed compact domain. More importantly, in spherical symmetry we can further exploit the structure of these equations (e.g., Remark \ref{rmk: u^0 cancellation}), which ultimately yields a non-homogeneous degenerate quasilinear wave equation for the Lagrangian flow map. 
Under the spherically symmetric setting, we introduce the Lagrangian flow map associated with $V$:
\begin{align}
\label{E:d_eta_d_s}
\frac{d\eta(s,\zeta)}{ds} = V(s,\eta(s,\zeta)), \quad\eta(0,\zeta) = \eta_0(\zeta),    
\end{align}
where $\eta_0:[0,1] \mapsto [0,1]$ is an orientation-preserving diffeomorphism (so in particular $\eta_0(0)=0$) 
to be determined. 
Also, we denote by
\begin{align}
\label{E:f_and_U_of_s}    
f(s,\zeta) = \trho\circ \eta(s,\zeta),\quad U(s,\ze) = V \circ \eta(s,\zeta), \quad U^0(s,\zeta) = u^0\circ \eta(s,\zeta),
\end{align}
respectively the Lagrangian variables of $\trho$, $V$, and $u^0$. Importantly, we remark that
\begin{equation}
    \label{materialU0}
    \p_s U^0 = (\p_s u^0 + V\p_z u^0) \circ \eta (s,\zeta).
\end{equation}
\begin{rmk}
    Thanks to our choice $r(0)=\lam(0)$, we have $\widetilde{\Omega}_s|_{s=0} = B(0,1)$. This implies that $\eta_0$ is a diffeomorphism that maps $[0,1]$ to itself. 
\end{rmk}

\subsubsection{Momentum Equation in Lagrangian Coordinates}
We begin with deriving the Lagrangian formulation for the momentum equation \eqref{momentum2}. By applying the identities \eqref{materialU0} and
\begin{equation}
    \label{U0}
    (U^0)^{2} = \frac{1}{1- \lambda^{-3\kappa}\left(U + \frac{\lambda'}{\lambda}\eta\right)^2} = 1 + \lambda^{-3\kappa}(U^0)^2\left(U + \frac{\lambda'}{\lambda}\eta\right)^2,
\end{equation}
which follows from \eqref{modvel} and \eqref{u0}, to \eqref{momentum2}, we obtain
\begin{equation}
    \label{momentumLag0}
    \begin{split}
    \left(1+ \lam^{-3\kk}f^\kk\right)\bigg[\p_s^2 \eta &+ (1-\frac32\kappa )\frac{\lambda'}{\lambda}\p_s\eta + \delta \eta + \frac{U + \frac{\lambda'}{\lambda}\eta}{U^0}\p_s U^0\bigg] + \frac{1+\kappa}{\kappa}\frac{1}{\p_\zeta \eta}\p_\zeta (f^\kappa)\\
    &= -\frac{1+\kappa}{\kk}\lambda^{-3\kappa}(U + \frac{\lambda'}{\lambda}\eta)(\p_s f^\kk - \frac{U}{\p_\zeta\eta}\p_\zeta f^\kk - 3\kk\frac{\lambda'}{\lambda}f^\kk - \frac{\lambda'}{\lambda}\eta \frac{\p_\zeta f^\kk}{\p_\zeta \eta})\\
    &= \frac{1+\kk}{\kk}\lambda^{-3\kk}(U + \frac{\lambda'}{\lambda}\eta)^2\frac{\p_\zeta f^\kk}{\p_\zeta \eta} - \frac{1+\kk}{\kk}\lambda^{-3\kappa}(U + \frac{\lambda'}{\lambda}\eta)(\p_s f^\kk  - 3\kk\frac{\lambda'}{\lambda}f^\kk).
    \end{split}
\end{equation}
In view of \eqref{materialU0} and \eqref{eq: u^0 cancellation}, we further note that
\begin{align*}
\p_s U^0 = \lambda^{-3\kappa}(U^0)^3 (U+\frac{\lambda'}{\lambda}\eta)\left[\p_s U + (1-\frac32\kappa)\frac{\lambda'}{\lambda}U + \delta\eta\right].
\end{align*}
Then, 
$$
\frac{U + \frac{\lambda'}{\lambda}\eta}{U^0}\p_s U^0 = \lambda^{-3\kappa}(U^0)^2 (U+\frac{\lambda'}{\lambda}\eta)^2\left[\p_s U + (1-\frac32\kappa)\frac{\lambda'}{\lambda}U + \delta\eta\right].
$$
Thanks to this identity, the terms in the square brackets on the LHS of \eqref{momentumLag0} become
\begin{align*}
     \p_s^2 \eta + (1-\frac32\kappa )\frac{\lambda'}{\lambda}\p_s\eta + \delta \eta + \frac{U + \frac{\lambda'}{\lambda}\eta}{U^0}\p_s U^0  &= \left(1+\lambda^{-3\kappa}(U^0)^2 (U+\frac{\lambda'}{\lambda}\eta)^2\right)\left[\p_s U + (1-\frac32\kappa)\frac{\lambda'}{\lambda}U + \delta\eta\right]\\
     &= (U^0)^2\left[\p_s U + (1-\frac32\kappa)\frac{\lambda'}{\lambda}U + \delta\eta\right],
\end{align*}
where \eqref{U0} is applied to obtain the final equality. Now, by moving the first term on the RHS of \eqref{momentumLag0} to the LHS and then using \eqref{U0} again, we arrive at 
\begin{equation}
    \label{momentumLag}
    \begin{split}
    &\left(1+\lam^{-3\kk}f^\kk\right)\left[\p_s U + (1-\frac32\kappa)\frac{\lambda'}{\lambda}U + \delta\eta\right] + \frac{1+\kappa}{\kappa}(U^0)^{-4}\frac{1}{\p_\zeta \eta}\p_\zeta (f^\kappa)\\
    &= - \frac{1+\kk}{\kk}\lambda^{-3\kappa}(U^0)^{-2}(U + \frac{\lambda'}{\lambda}\eta)(\p_s f^\kk  - 3\kk\frac{\lambda'}{\lambda}f^\kk).
    \end{split}
\end{equation}

\subsubsection{Mass Equation in Lagrangian Coordinates} \label{subsect:barG}
Next, we derive the Lagrangian formulation for the mass equation \eqref{mass2}. Denoting by
\begin{align}
    \label{E:Jac_det_def}
    \calF(s,\zeta):=\det(D\bar\eta), \quad \text{where }\bar\eta(s,y) = \eta(s,\zeta)\frac{y^i}{\zeta},
\end{align}
and since $$(D\bar\eta)_{ij} = \p_i \left(\eta(s,\zeta)\frac{y_j}{\zeta}\right)=\frac{\eta}{\zeta}\delta_{ij} + (\p_\zeta \eta-\frac{\eta}{\zeta})\frac{y_iy_j}{\zeta^2},
$$
we obtain, from a direct computation, that
$$
\calF(s,\zeta) := \left(\frac{\eta(s,\zeta)}{\zeta}\right)^2 \p_\zeta \eta(s,\zeta).
$$
Moreover, 
\begin{equation}\label{dsF}
\p_s \calF = \left(\frac{\eta}{\zeta}\right)^2\p_\zeta \eta \left(\frac{\p_\zeta U}{\p_\zeta \eta} + \frac{2U}{\eta}\right) = \calF \left(\frac{\p_\zeta U}{\p_\zeta \eta} + \frac{2U}{\eta}\right).
\end{equation}
Using the computations above, we note that also using \eqref{mass2}:
\begin{equation}\label{mass2.5}
\begin{split}
    \p_s (f\calF) &= \calF(\p_s \trho(s,\eta) + V(s,\eta)\p_z \trho(s,\eta)) + f\p_s\calF\\
    &= \underbrace{-\calF f(\p_z V(s,\eta) + \frac{2V(s,\eta)}{\eta}) + f\calF \left(\frac{\p_\zeta U}{\p_\zeta \eta} + \frac{2U}{\eta}\right)}_{=0}\\
    &\quad -(f\calF)\bigg[(1+\lam^{-3\kk} f^\kk)\frac{\p_s U^0}{U^0}+ \lam^{-3\kk} f^\kk \left(\frac{\pz U}{\pz \eta} + \frac{2}{\eta}U + 3\frac{\lam'}{\lam}\right) \bigg].
    \end{split}
\end{equation}

Using \eqref{dsF} and the following elementary identity
\begin{equation}
    \label{eq:logU0}
    -\frac12\p_s\log(U^0)^{-2} = \frac{\p_s U^0}{U^0},
\end{equation}
from \eqref{mass2.5} we obtain the following equation for $f\calF$:
\begin{equation}
    \label{massLag}
    \begin{split}
    \p_s(f\calF) &= -(f\calF)\left[-\frac{1}{2}(1+\lambda^{-3\kk}f^\kk)\p_s\log\left(U^0\right)^{-2} + \lambda^{-3\kk}f^\kk (\calF^{-1}\p_s \calF + 3\frac{\lambda'}{\lambda})\right].
    \end{split}
\end{equation}

Summing up, we have deduced the Lagrangian formulation for the relativistic Euler equations:
    \begin{subequations}\label{eq:releulerLag}
    \begin{align}
            &\p_s(f\calF) =-(f\calF)\left[-\frac{1}{2}(1+\lambda^{-3\kk}f^\kk)\p_s\log\left(U^0\right)^{-2} + \lambda^{-3\kk}f^\kk (\calF^{-1}\p_s \calF + 3\frac{\lambda'}{\lambda})\right],\label{eq:releulerLagmass}
            \end{align}
            \begin{align}
    &\left(1+\lam^{-3\kk}f^\kk\right)\left[\p_s U + (1-\frac32\kappa)\frac{\lambda'}{\lambda}U + \delta\eta\right] + \frac{1+\kappa}{\kappa}(U^0)^{-4}\frac{1}{\p_\zeta \eta}\p_\zeta (f^\kappa)\nonumber\\
    &\quad= - \frac{1+\kk}{\kk}\lambda^{-3\kappa}(U^0)^{-2}(U + \frac{\lambda'}{\lambda}\eta)(\p_s f^\kk  - 3\kk\frac{\lambda'}{\lambda}f^\kk).\label{eq:releulerLagmom}
    \end{align}
    \end{subequations}
\subsection{Admissible Weight and Enthalpy Profile}\label{subsect:enthalpy}
A crucial difficulty in the study of free-boundary fluids with physical vacuum boundary in spherical symmetry consists of both the degeneracy of the fluid density at the free interface and the coordinate singularity. To correctly keep track of such degeneracies, we define the following class of admissible weight functions:
\begin{defn}\label{def: w}
    A smooth function $w: [0,1] \to \R^+$ is said to be an admissible weight function if $w$ possesses the following properties:
    \begin{enumerate}
        \item (Absence of interior vacuum) $w > 0$ on $[0,1)$  and $w(1) = 0$;
        \item (Physical vacuum boundary) $-\infty <w'(1) < 0$. Or equivalently, there exists a universal constant $C > 0$ such that
        $$
        \frac1C |r-1| \le w(r) \le C|r-1|,
        $$
        when $r$ is sufficiently close to $1$;
        \item (Regularity at the origin) $w^{(2k+1)}(0) = 0$ for all nonnegative integers $k$.
    \end{enumerate}
    Given $\delta > 0$ and an admissible weight $w$, we also define the 1-parameter family $w_\delta = \delta w$.
\end{defn}
\begin{rmk}
    Given $\delta > 0$, the 1-parameter family $w_\delta$ also serves as the approximate enthalpy profile for the gas. We stress that the degree of freedom in choosing $\delta$ to be sufficiently small is indispensable for our analysis. This manifests that our result concerns the small density/sound speed regime.
\end{rmk}

\subsection{Equation for Perturbation Variables}
To facilitate our nonlinear stability analysis, we define the Lagrangian perturbation 
\begin{align}
    \label{E:Theta_def}
    \Theta = \eta - \zeta.
\end{align}
We further introduce another convenient rescaled time $\tau$ given by:
\begin{align}
\label{E:d_tau_d_s}
\frac{d\tau}{ds} = \lambda^{\frac32\kk}, \, \tau(s=0)=0
\quad \text{ or equivalently }\frac{d\tau}{dt} = \lam^{-1}.    
\end{align}

\subsubsection{Mass Equation in a Good Unknown}
To start with, we further rewrite \eqref{eq:releulerLagmass}. To motivate the necessity of this rewriting, we emphasize an outstanding difficulty of the relativistic Euler equations: in contrast with the classical compressible Euler equations, where the quantity $f\calF$ is invariant in time due to the conservation of mass, the evolution of this quantity is more involved for the relativistic Euler equations exactly due to a lack of such conservation. As a result, the evolution of $f\calF$ described by \eqref{eq:releulerLagmass} is a highly nonlinear equation that couples to \eqref{eq:releulerLagmom} to the top order, which is manifested by the term $\pt \log (U^0)^{-2} \sim \pt^2 \Theta$.

To remedy this issue, we first perform a gauge fixing as in the classical compressible Euler case. Secondly, we introduce a structural ansatz for $f$ which effectively keeps track of the forcing terms appearing in \eqref{eq:releulerLagmass}, and reduce the mass equation to an equivalent evolution equation for the relativistic corrector $\calG$ (cf. \eqref{eq:G0}). In the final step, we further introduce a good unknown $\barcalG$. See \eqref{barcalG}. Remarkably, the evolution equation for $\barcalG$ can be viewed as a linear ODE in $\tau$, from which we can explicitly solve $\barcalG$ in terms of other fundamental quantities of the relativistic Euler equations in Lagrangian coordinates. 

\medskip
\noindent\textbf{Step 1} (Gauge fixing).
We fix the gauge by picking $\eta_0$ so that (recall \eqref{E:Jac_det_def})

\begin{align*}
    \left(\frac{\eta_0(\zeta)}{\zeta}\right)^2 \p_\zeta \eta_0(\zeta) \trho(0, \eta_0(\zeta) ) = (\delta w(\zeta))^\frac{1}{\kappa},
\end{align*}
or equivalently
\begin{equation}\label{gauge fixing}
\calF_0 f_0 = (w_{\delta})^{\frac1\kk}, 
\end{equation}
where $w_{\delta}$ is given by Definiton \ref{def: w}, $\calF_0 = \calF|_{s=0}$, and $f_0 = \trho\circ \eta_0$. The existence of such $\eta_0$ follows from Dacorogna--Moser \cite{dacorogna1990partial}. 

\medskip
\noindent\textbf{Step 2} (Choice of ansatz). We introduce the following ansatz 
\begin{align}\label{f ansatz}
f(\tau,\ze) = w_{\delta}^{\frac1\kk}(\ze)\calF^{-1}(\tau, \ze)\calG(\tau, \ze), \quad \text{with }\calG(0,\zeta) = 1.
\end{align}
Here, the quantity $\calG$ is the so-called \textit{relativistic correction}, which plays a crucial role in the rest of this article. 
Thanks to \eqref{f ansatz}, we can equivalently write the mass equation \eqref{eq:releulerLagmass} in $(\tau,\ze)$ coordinates as an evolution equation for $\calG$:
\begin{equation}
    \label{eq:G0}
    \begin{split}
    \p_\tau \calG &= \calG\left(\frac12 \p_\tau \log (U^0)^{-2}\right) - \delta\calG^{1+\kk} \lambda^{-3\kk} w \calF^{-\kk}\left[-\frac12 \p_\tau \log (U^0)^{-2} + \calF^{-1}\p_\tau \calF + 3\lamt\right].
    \end{split}
\end{equation}

\medskip
\noindent\textbf{Step 3} (New variables). 
Define the following change of variable:
\begin{equation}\label{defn:barG}
\bar\calG := \frac{1}{\kk}\calG^{-\kk}.
\end{equation}
Then we note that
$$
\p_\tau \bar\calG = -\calG^{-1-\kk}\p_\tau \calG,
$$
and \eqref{eq:G0} can be equivalently written as
\begin{equation}
    \label{eq:barG}
    \begin{split}
    &\p_\tau \bar\calG = \left(-\frac\kk2 \p_\tau \log (U^0)^{-2}\right)\bar\calG + \delta \lambda^{-3\kk} w \calF^{-\kk}\left[-\frac12 \p_\tau \log (U^0)^{-2} + \calF^{-1}\p_\tau \calF + 3\lamt\right],\\
    &\bar\calG(0,\zeta) = \frac1\kk.
    \end{split}
\end{equation}

Multiplying \eqref{eq:barG} on both sides by $(U^0)^{-\kk}$, we have
$$
\frac{d}{d\tau}\left((U^0)^{-\kk}\bar\calG\right) = \delta (U^0)^{-\kk} \lambda^{-3\kk} w \calF^{-\kk}\left[-\frac12 \p_\tau \log (U^0)^{-2} + \calF^{-1}\p_\tau \calF + 3\lamt\right],
$$
Integrating in $\tau$ yields the following formula:
\begin{equation}
    \label{barcalG}
    \barcalG = (U^0)^{\kk}\left(\frac1\kk (U^0)^{-\kk}(0) + \delta\int_0^\tau (U^0)^{-\kk} \lambda^{-3\kk} w \calF^{-\kk}\left[-\frac12 \p_\tau \log (U^0)^{-2} + \calF^{-1}\p_\tau \calF + 3\lamt\right] d\tilde\tau\right)
\end{equation}
where we used that $\bar\calG(0,\zeta) = \frac1\kk$. 

To proceed, we note the following fact:
$$
(U^0)^{-\kk} \frac12\p_\tau\log(U^0)^{-2} = \frac1\kk \p_\tau (U^0)^{-\kk}.
$$
Then
\begin{equation}\label{barGcancel1}
    -\int_0^\tau (U^0)^{-\kk} \lambda^{-3\kk} w \calF^{-\kk}\left[\frac12 \p_\tau \log (U^0)^{-2}\right] d\tilde\tau = -\frac w\kk \int_0^\tau \p_\tau (U^0)^{-\kk} \lambda^{-3\kk}\calF^{-\kk} d\tilde\tau.
\end{equation}
On the other hand, it is straightforward to check that
$$
\lambda^{-3\kk}\left(\calF^{-\kk-1} \p_\tau \calF + 3\lamt \calF^{-\kk}\right) = -\frac1\kk \p_\tau \left(\lambda^{-3\kk}\calF^{-\kk}\right).
$$
Then after an integration by parts in $\tau$, we have:
\begin{equation}\label{barGcancel2}
\begin{split}
    \int_0^\tau (U^0)^{-\kk} \lambda^{-3\kk} w \calF^{-\kk}&\left[\calF^{-1}\p_\tau \calF + 3\lamt\right] d\tilde\tau = w\int_0^\tau (U^0)^{-\kk} \left(-\frac1\kk \p_\tau \left(\lambda^{-3\kk}\calF^{-\kk}\right)\right) d\tilde\tau\\
    &= -\frac{w}{\kk}\left[(U^0)^{-\kk} \lambda^{-3\kk}\calF^{-\kk}(\tau) - (U^0)^{-\kk} \lambda^{-3\kk}\calF^{-\kk}(0)\right]\\
    &\quad + \frac{w}{\kk}\int_0^\tau \lambda^{-3\kk} \p_\tau ((U^0)^{-\kk}) \calF^{-\kk} d\tilde\tau.
    \end{split}
\end{equation}
Combining \eqref{barcalG}, \eqref{barGcancel1}, and \eqref{barGcancel2}, we get the following simple expression for $\barcalG$:\begin{equation}
    \label{eq:barcalG}
    \begin{split}
    \barcalG &= (U^0)^{\kk}\left(\frac1\kk (U^0)^{-\kk}(0) - \frac{\delta w}{\kk}\left[(U^0)^{-\kk} \lambda^{-3\kk}\calF^{-\kk}(\tau) - (U^0)^{-\kk} \lambda^{-3\kk}\calF^{-\kk}(0)\right]\right)\\
    &= (U^0)^{\kk}\left(\frac1\kk (U^0)^{-\kk}(0)(1+ \delta\lambda^{-3\kk}(0)w\calF^{-\kk}(0)) \right) - \frac{\delta}{\kk} \lambda^{-3\kk} w\calF^{-\kk}.
    \end{split}
\end{equation}
Differentiating in $\tau$, we have
\begin{equation}\label{eq:dtbarcalG}
\begin{split}
\kk \p_\tau \barcalG &= \pt (U^0)^{\kk}\left((U^0)^{-\kk}(0)(1+ \delta\lambda^{-3\kk}(0)w\calF^{-\kk}(0)) \right) - \delta w\pt(\lam^{-3\kk}\calF^{-\kk})\\
&= -\frac{\kk}{2}\pt (U^0)^{-2} \cdot (U^0)^{\kk+2}\left(\frac1\kk (U^0)^{-\kk}(0)(1+ \delta\lambda^{-3\kk}(0)w\calF^{-\kk}(0)) \right)\\
&\quad + \kk\delta w \lam^{-3\kk} \calF^{-\kk-1}\pt\calF + 3\kk \lamt \delta w\lam^{-3\kk}\calF^{-\kk}.
\end{split}
\end{equation}
The explicit formulae \eqref{eq:barcalG} and \eqref{eq:dtbarcalG} will be instrumental when we perform nonlinear energy estimates.

\subsubsection{Equation for Lagrangian Perturbation $\Theta$}
In this section, we rewrite \eqref{eq:releulerLagmom} into an equation for the Lagrangian perturbation $\Theta$. Observe that
\begin{align}
    \label{E:U_0_Theta_and_bar_lambda}
    U = \p_s\eta = \lambda^{\frac32\kk} \p_\tau \Theta,\quad \p_s U  = \lambda^{3\kk}(\p_\tau^2 \Theta + \frac32\kk \lamt \p_\tau \Theta),\quad (U^0)^{-2} = 1 - (\p_\tau \Theta + \lamt (\Theta + \ze))^2
\end{align}
and recall that (cf. Lemma \ref{lem:lambdaODE})
$$
\lamt:= \frac{\p_\tau \lambda}{\lambda} = \lambda^{-\frac32\kk}\frac{\lambda'}{\lambda}.
$$
We then rewrite \eqref{eq:releulerLagmom} as:
\begin{equation}\label{momlagaux1}
\begin{split}
\left(1+\lam^{-3\kk}f^\kk\right)\bigg[\lambda^{3\kk}(\p_\tau^2 \Theta + \lamt \p_\tau \Theta) &+ \delta (\Theta + \ze)\bigg] +\frac{1+\kappa}{\kappa}(U^0)^{-4}\frac{1}{\p_\zeta \eta}\p_\zeta (f^\kappa)\\
&= -\frac{1+\kk}{\kk} (U^0)^{-2}(\p_\tau \Theta + \lamt(\Theta + \zeta))(\p_\tau f^\kk - 3\kk\lamt f^\kk)\\
&= -\frac{1+\kk}{\kk} (U^0)^{-2}(\p_\tau \Theta + \lamt(\Theta + \zeta))(\p_\tau f^\kk )\\
&\quad + 3(\kk+1) \lamt (U^0)^{-2}(\p_\tau \Theta + \lamt(\Theta + \zeta)) f^\kk.
\end{split}
\end{equation}
From the definition of $\calF$, i.e., $\calF = (\eta/\zeta)^2\p_\zeta \eta$, we have $\calF^{-1}=(1+\Theta/\zeta)^{-2}\frac{1}{\p_\zeta \eta}$, so that
\begin{align*}
   \frac{1+\kappa}{\kappa} w^{\frac1\kappa} \frac{1}{\p_\zeta \eta} \p_\zeta (w\calF^{-\kappa}) &= \frac{1+\kappa}{\kappa} \frac{1}{\p_\zeta \eta} w^{\frac1\kappa}\p_\zeta w \calF^{-\kappa} - (1+\kappa) w^{1+\frac1\kappa}\frac{1}{\p_\zeta \eta} \calF^{-\kappa - 1}\p_\zeta \calF\\
    &={\frac{1}{\p_\zeta\eta}\left(\p_{\zeta}w^{1+\frac1\kappa}\calF^{-\kappa}-(1+\kappa)w^{1+\frac1\kappa} \calF^{-\kappa-1}\p_\zeta\calF\right)}\\
    &={\left(1+\frac\Theta\zeta\right)^2\left(\p_\zeta w^{1+\frac1\kappa}\calF^{-\kappa-1}-(1+\kappa)w^{1+\frac1\kappa}\calF^{-\kappa-2}\p_\zeta\calF \right)}\\
    &=\left(1+\frac\Theta\zeta\right)^2\p_\zeta \left(w^{1+\frac1\kappa} \calF^{-1-\kappa}\right).
\end{align*}
Using the computations above, we may rewrite the pressure gradient term as follows:
\begin{align*}
    \frac{1+\kk}{\kk}(U^0)^{-4} &\frac{1}{\pz \eta} \pz f^\kk =  \delta\frac{1+\kk}{\kk}(U^0)^{-4} \frac{1}{\pz \eta} \pz (w\calF^{-\kk}\calG^\kk)\\
    &= \delta(U^0)^{-4}\left(\frac{1+\kk}{\kk} \frac{1}{\pz \eta} \pz (w\calF^{-\kk})\calG^\kk + \frac{1+\kk}{\kk} \frac{1}{\pz \eta} \pz(\calG^\kk) w\calF^{-\kk}\right)\\
    &= \delta(U^0)^{-4}\left(1+\frac\Theta\zeta\right)^2\bigg(\frac{1}{w^{\frac1\kk}}\p_\zeta \left(w^{1+\frac1\kappa} \calF^{-1-\kappa}\right)\\
    &\quad + \frac{1}{w^{\frac1\kk}}\p_\zeta \left(w^{1+\frac1\kappa} \calF^{-1-\kappa}\right)(\calG^\kk - 1)
    + \frac{1+\kk}{\kk} \calF^{-\kk-1}w\pz(\calG^\kk)\bigg).
\end{align*}

We then deduce that
\begin{align*}
    \delta \ze + \frac{1+\kk}{\kk}(U^0)^{-4} &\frac{1}{\pz \eta} \pz f^\kk =\delta(U^0)^{-4}\left(1+\frac\Theta\zeta\right)^2\bigg(\frac{1}{w^{\frac1\kk}}\p_\zeta \left(w^{1+\frac1\kappa} (\calF^{-1-\kappa}-1)\right)\\
    &\quad +\frac{1}{w^{\frac1\kk}}\p_\zeta \left(w^{1+\frac1\kappa} (\calF^{-1-\kappa}-1)\right)(\calG^\kk - 1) + \frac{1}{w^{\frac1\kk}}\pz (w^{\frac1\kk+1})\calG^\kk\\
    &\quad + \frac{1+\kk}{\kk} \calF^{-\kk-1}w\pz(\calG^\kk)\bigg) + \delta\ze
\end{align*}
With the above computation, the momentum equation can be written as
\begin{equation}
    \label{eq:momlagmain0}
    \begin{split}
    &\left(1+\delta \lam^{-3\kk}w\calF^{-\kk}\calG^\kk\right)\left[\delta^{-1}\lambda^{3\kk}(\p_\tau^2 \Theta + \lamt \p_\tau \Theta) + \Theta\right]\\
    &+ (U^0)^{-4}\left(1+\frac\Theta\zeta\right)^2\frac{1}{w^{\frac1\kk}}\p_\zeta \left(w^{1+\frac1\kappa} (\calF^{-1-\kappa}-1)\right) = -\sum_{j=1}^6 E_j,
    \end{split}
\end{equation}
where
\begin{equation}
    \label{mainerror}
    \begin{split}
    &E_1= (U^0)^{-4}\left(1+\frac\Theta\zeta\right)^2\frac{1}{w^{\frac1\kk}}\p_\zeta \left(w^{1+\frac1\kappa} (\calF^{-1-\kappa}-1)\right)(\calG^\kk - 1),\\
    &E_2= (U^0)^{-4}\left(1+\frac\Theta\zeta\right)^2\frac{1}{w^{\frac1\kk}}\pz (w^{\frac1\kk+1})\calG^\kk,\\
    &E_3 = \frac{1+\kk}{\kk}(U^0)^{-4}\left(1+\frac\Theta\zeta\right)^2 \calF^{-\kk-1}w\pz(\calG^\kk),\\
    &E_4 =\frac{1+\kk}{\kk} (U^0)^{-2}(\p_\tau \Theta + \lamt(\Theta + \zeta))(w\p_\tau (\calF^{-\kk}\calG^\kk)),\\
    &E_5 = -3(1+\kk) \lamt (U^0)^{-2}(\p_\tau \Theta + \lamt(\Theta + \zeta)) w\calF^{-\kk}\calG^\kk,\\
    &E_6 = -(1+\delta\lambda^{-3\kk}w\calF^{-\kk}\calG^\kk)\ze.
\end{split}
\end{equation}

\subsubsection{Main Perturbation Equations around an Expanding Background}
Summarizing the analysis of the above subsections, we record our main system of equations to study in this work:
\begin{equation}
\boxed{
    \label{eq:momlagmain}
    \begin{split}
    &\left(1+\delta \lam^{-3\kk}w\calF^{-\kk}\calG^\kk\right)\left[\delta^{-1}\lambda^{3\kk}(\p_\tau^2 \Theta + \lamt \p_\tau \Theta) + \Theta\right]\\
    &+ (U^0)^{-4}\left(1+\frac\Theta\zeta\right)^2\frac{1}{w^{\frac1\kk}}\p_\zeta \left(w^{1+\frac1\kappa} (\calF^{-1-\kappa}-1)\right) = -\sum_{j=1}^6 E_j,
    \end{split}
    }
\end{equation}
equipped with initial conditions
\begin{equation}
    \label{mainID}
    \Theta(0,\ze) = \Theta_0(\ze),\quad \pt\Theta(0,\zeta) = U_0(\ze),\quad \zeta \in [0,1].
\end{equation}
Here, $E_j$, $j = 1,\hdots, 6$ are defined by \eqref{mainerror}. Furthermore, $\calG = \left(\frac{1}{\kk}\barcalG^{-1}\right)^{\frac1\kk}$, and $\barcalG$ verifies \eqref{eq:barcalG}.

\section{Functional Setup}\label{sect: functional}
As we will see in later sections, the main equation \eqref{eq:momlagmain} manifests itself as a quasilinear wave equation with degeneracies in the spatial variable $\ze$ localized at $\ze = 0, 1$. To resolve such degeneracies and in view of performing high-order energy estimates for a quasilinear wave equation, we adopt a robust vector field method first introduced in Guo-Had{\v{z}}i{\'c}-Jang \cite{guo2021continued} as well as the classical weighted framework developed in Jang-Masmoudi \cite{jang2015well}.

To start with, we define
\begin{align}
    \label{E:Primitive_weighted_derivative}
    \Dz = \frac{1}{\ze^2}\pz (\ze^2\cdot) = \pz + \frac{2}{\ze}. 
\end{align}
We then define the differential operators formed by a concatenation of $\pz$ and $\Dz$:
\begin{equation}
    \label{calDj}
    \calD_j = \begin{cases}
        (\dz\Dz)^{j/2},& j\text{ is even},\\
        \Dz(\dz\Dz)^{\frac{j-1}{2}},& j\text{ is odd},
    \end{cases}
\end{equation}
with $\calD_0 = 1$. We also define auxiliary operators
\begin{equation}
    \label{barcalDj}
    \barcalD_j = \begin{cases}
        \calD_0,& j=0\\
        \calD_{j-1}\dz,& j \ge 1
    \end{cases}
\end{equation}

With the definition of good derivatives \eqref{calDj},  we introduce the following weighted norm:
\begin{defn}\label{def: inner product}
    Given $i \in \N$, $\kk > 0$, and $w$ an admissible weight (cf. Definition \ref{def: w}), we define the spaces $L^2_{\kk,i}$ as the completion of $C^\infty_c(0,1)$ by the norm $\|\cdot\|_{\kk,i}$ afforded by the following inner product:
    $$
    (f,g)_{\kk,i} := \int_0^1 w^{\frac1\kk + i}\ze^2 f(\ze)g(\ze)d\ze.
    $$
\end{defn}
\begin{rmk}
    Since we fix the value of $\kk > 0$ for the rest of this work, we will drop the $\kk$ dependence of the above notations. Namely, we will write $L^2_i$ instead of $L^2_{\kk,i}$, $\|\cdot\|_{i}$ instead of $\|\cdot\|_{\kk,i}$, and $(\cdot,\cdot)_{i}$ instead of $(\cdot,\cdot)_{\kk, i}$.
\end{rmk}
We then introduce the spacetime norms that we will use to perform our nonlinear analysis:
\begin{defn}
    Let $\tau \ge \tau_0 \ge 0$ and $N \in \N$. We define
    \begin{equation}
        \label{defn:EN}
        E^N(\tau;\tau_0):= \sum_{i = 0}^N \sup_{\tau_0 \le \tau' \le \tau} \left[\delta^{-1}\lam(\tau')^{3\kk}\|\calD_i \pt \Theta(\tau',\cdot)\|_i^2 + \|\calD_i \Theta(\tau',\cdot)\|_i^2 + \|\calD_{i+1}\Theta(\tau',\cdot)\|_{i+1}^2 \right].
    \end{equation}
    We also define
    \begin{equation}
        \label{defn:SN}
        S^N(\tau;\tau_0) := \sum_{i = 0}^{N-1}\int_{\tau_0}^{\tau} \|\calD_i\pt^2\Theta(\tau',\cdot)\|_i^2 d\tau'.
    \end{equation}
    We will also use the following shorthand notations:
    $$
    E^N(\tau) := E^N(\tau;0),\quad S^N(\tau) := S^N(\tau;0).
    $$
\end{defn}

Finally, we introduce the following classes of admissible vector fields:
\begin{align*}
    \calP_{2j+2} := \{\prod_{k=1}^{j+1} (\dz V_k):V_k \in \{\Dz, \ze^{-1}\}\},\\\calP_{2j+1} := \{ V_{j+1}\prod_{k=1}^{j} (\dz V_k):V_k \in \{\Dz, \ze^{-1}\}\},
\end{align*}
for $j \ge 0$. We also set $\calP_0 = Id$. We also define
$$
\barcalP_{2j+1} := \{W\dz:W\in \calP_{2j}\},\quad \barcalP_{2j+2} := \{W\dz:W\in \calP_{2j+1}\}.
$$

We conclude this section by remarking on admissible vector fields that can be applied to $w$ or its derivatives:
\begin{lem}
    \label{lem:wderivative}
    Let $w$ be an admissible weight function given by Definition \ref{def: w}. Then for any $i \ge 0$,
    $$
    \bar P_i w, \quad P_i w', \quad  \bar P_i\left(\frac{w'}{\ze}\right) \in C^\infty([0,1]),
    $$
    where $P_i \in \calP_i$, $\bar P_i \in \barcalP_i$.
\end{lem}
\begin{proof}
    It suffices to consider the case near the origin, and the desired statement follows straightforwardly from Definition \ref{def: w} (3).
\end{proof}

\section{Decomposition of the Pressure Gradient}\label{sec: elliptic operator}
In this section, we concern ourselves with the study of the following term:
\begin{align}\label{ellipticpre}
\frac{1}{w^{\frac1\kk}}\p_\zeta \left(w^{1+\frac1\kappa} (\calF^{-1-\kappa}-1)\right),
\end{align}
which originates from the pressure gradient in the momentum equation written in the original variables. Inspired by the study of compressible Euler equations in Lagrangian variables \cite{hadvzic2018nonlinear,hadvzic2018expanding,guo2021continued}, we similarly make a key observation that \eqref{ellipticpre} can be realized as a degenerate, linear elliptic operator in divergence form modulo higher order errors.

As a preparation, we introduce some notations which will be convenient for our analysis:
$$
\xi = \frac{\eta}{\ze}, \quad \theta = \frac{\Theta}{\ze} = \xi -1. 
$$
Under these new variables, $\calF = (\eta/\ze)^2\pz\eta$ becomes
\begin{align} \label{eq: calF in xi}
    \calF = \xi^2 (\xi + \ze \pz \xi). 
\end{align}
Since $\pz \xi  = \pz \theta$, a direct computation yields
\begin{align}\label{eq: pz calF in xi}
    \pz \calF = \xi^2 (\ze \pz^2\theta + 4\pz\theta) + 2\ze \xi (\pz \theta)^2. 
\end{align}
With the preparations above, we show the main result of this section:
\begin{lem} \label{lem:elliptic term in theta}
    The following identity holds:
    \begin{align}\label{eq: elliptic term in theta}
        \begin{aligned}
            w^{-\frac{1}{\kk}}\pz \left(w^{1+\frac1\kk}(\calF^{-\kk-1}-1)\right)=&
            -(1+\kk)\calF^{-\kk-2}\xi^2w^{-\frac1\kk}\pz\left(w^{1+\frac1\kk}\frac{1}{\ze^2}\pz(\ze^3\theta)\right)\\
            &\underbrace{-2(1+\kk)\calF^{-\kk-2}w\xi\ze (\pz\theta)^2}_{=R_1(\theta)}\\
            &+\frac{1+\kk}{\kk}w'\left(\calF^{-\kk-1}-1+(1+\kk)\calF^{-\kk-2}\xi^2(\ze\pz\theta+3\theta)\right). 
        \end{aligned}
    \end{align}
    Moreover, the term $\frac{1+\kk}{\kk}w'\left(\calF^{-\kk-1}-1+(1+\kk)\calF^{-\kk-2}\xi^2(\ze\pz\theta+3\theta)\right)$
    in \eqref{eq: elliptic term in theta} can be modified into
    \begin{align}\label{eq:cancellation pre}
        \begin{aligned}
            &\frac{1+\kk}{\kk}w'\left(\calF^{-\kk-1}-1+(1+\kk)\calF^{-\kk-2}\xi^2(\ze\pz\theta+3\theta)\right)\\
            =  &\underbrace{\frac{(1+\kk)^2}{\kk}w'(3\theta^2 + 2\theta^3)}_{=R_2(\theta)}\\
            &+\underbrace{\frac{1+\kk}{\kk}w'\left(\calF^{-\kk-1}-1+ (1+\kk)(\calF-1)+(1+\kk)(\calF^{-\kk-2}-1)\xi^2(\ze\pz\theta+3\theta)\right)}_{=R_3(\theta)}. 
        \end{aligned}
    \end{align}
\end{lem}
\begin{proof}
    We expand \eqref{ellipticpre} to obtain
    \begin{align*}
        \frac{1}{w^{\frac1\kk}}\p_\zeta \left(w^{1+\frac1\kappa} (\calF^{-1-\kappa}-1)\right) = - (1+\kk)w\calF^{-\kk-2}\pz \calF
        + \frac{1+\kk}{\kk}w'(\calF^{-\kk-1}-1). 
    \end{align*}
    From \eqref{eq: pz calF in xi}, we have
    \begin{align*}
        - (1+\kk)w\calF^{-\kk-2}\pz \calF =& -(1+\kk)\calF^{-\kk-2}\xi^2w^{-\frac1\kk}\pz \left(w^{1+\frac1\kk}\frac{1}{\ze^2}\pz (\ze^3 \theta)\right)\\
        -&2(1+\kk)\calF^{-\kk-2}w\xi\ze(\pz\theta)^2\\
        +&\frac{(1+\kk)^2}{\kk}w'\calF^{-\kk-2}\xi^2(\ze\pz\theta+3\theta),
    \end{align*}
    which leads to \eqref{eq: elliptic term in theta}, where we also used $\partial_\zeta(\frac{1}{\zeta^2} \partial_\zeta(\zeta^3 \theta))=\zeta \partial_\zeta^2\theta +4 \partial_\zeta \theta$. Moreover, \eqref{eq:cancellation pre} follows from the identity
    \begin{align*}
        (\calF-1) - \xi^2(\ze\pz\theta+3\theta)=-(3\theta^2+2\theta^3),
    \end{align*}
    which follows from a direct computation by setting $\xi=1+\theta$ in \eqref{eq: calF in xi}. 
\end{proof}

Next, we state an important decomposition that plays a crucial role in the $\calD_i$-differentiated high-order version of \eqref{eq:momlagmain} and the associated energy estimates.

\begin{lem}[Decomposition of $\Dz R_3(\theta)$]\label{lem:cancellation lemma theta}
Let $R_3(\theta)$ be defined in \eqref{eq:cancellation pre}. Then
    \begin{align}\label{eq:cancellation theta}
        \begin{aligned}
            \Dz R_3(\theta) &=\frac{1+\kk}{\kk}(\Dz w')\left(\calF^{-\kk-1}-1+ (1+\kk)(\calF-1)+(1+\kk)(\calF^{-\kk-2}-1)\xi^2(\ze\pz\theta+3\theta)\right)\\
            &+ \frac{1+\kk}{\kk} w' \pz \left(\calF^{-\kk-1}-1+ (1+\kk)(\calF-1)+(1+\kk)(\calF^{-\kk-2}-1)\xi^2(\ze\pz\theta+3\theta)\right)
            \\&=\frac{1+\kk}{\kk}(\Dz w')\left(\calF^{-\kk-1}-1+ (1+\kk)(\calF-1)+(1+\kk)(\calF^{-\kk-2}-1)\xi^2(\ze\pz\theta+3\theta)\right)\\
            &+\frac{1+\kk}{\kk}w'\bigg(6(1+\kk)(\calF^{-\kk-2}-1)\xi\theta\pz\theta
            -(\kk+1)(\kk+2)\calF^{-\kk-3}\xi^2\pz\calF(\ze\pz\theta+3\theta)\bigg). 
        \end{aligned}
    \end{align}
\end{lem}
\begin{proof}
The first equality follows from the product rule for $\Dz$ \eqref{eq:basicproduct}.  The second equality follows from $$   \pz \calF -\xi^2 (\ze \pz^2\theta + 4\pz\theta) =2\ze \xi (\pz \theta)^2,
$$ 
and
\begin{align*}
    &\pz \left(\calF^{-\kk-1}-1+ (1+\kk)(\calF-1)+(1+\kk)(\calF^{-\kk-2}-1)\xi^2(\ze\pz\theta+3\theta)\right)\\
    &= -(1+\kk)\calF^{-\kk-2}\left(\pz \calF -\xi^2 (\ze \pz^2\theta + 4\pz\theta)\right)+(1+\kk)\left(\pz \calF -\xi^2 (\ze \pz^2\theta + 4\pz\theta)\right)\\
    &+2(1+\kk)\xi(\calF^{-\kk-2}-1) \pz \theta(\ze\pz\theta+3\theta)-(\kk+1)(\kk+2)\calF^{-\kk-3}\xi^2\pz\calF(\ze\pz\theta+3\theta)\\
    &=
    6(1+\kk)(\calF^{-\kk-2}-1)\xi\theta\pz\theta-(\kk+1)(\kk+2)\calF^{-\kk-3}\xi^2\pz\calF(\ze\pz\theta+3\theta).
\end{align*}
\end{proof}
\begin{rmk}
In particular, we show in Lemma \ref{lem:cancellation lemma theta} that $\Dz R_3(\theta)$ consists of terms with at most one spatial derivative, i.e., $\zeta$ derivative, on $\theta$, except for the term  $-(\kappa+1)(\kappa+2)F^{-\kappa-3}\,\xi^{2}(\partial_\zeta F)(\zeta\,\partial_\zeta\theta+3\theta)$ in the last line of \eqref{eq:cancellation theta}, which has $\pz \calF$ that contributes to $\pz^2\theta$ at the highest order. This term can, in fact, be canceled by another leading-order term appearing when treating the commutators of the elliptic operator. See Lemma \ref{lem: cancellation lemma M}. 
\end{rmk}

From Lemma \ref{lem:elliptic term in theta} and since
$$
\ze\pz\theta = \Dz\Theta - 3\frac{\Theta}{\ze}, 
$$
we have
\begin{align}\label{eq:elltipic term original form}
    \begin{aligned}
        &(U^0)^{-4}\left(1+\frac\Theta\zeta\right)^2\frac{1}{w^{\frac1\kk}}\p_\zeta \left(w^{1+\frac1\kappa} (\calF^{-1-\kappa}-1)\right)\\
        &=-(1+\kk)\calF^{-\kk-2}(U^0)^{-4}\left(1+\frac{\Theta}{\ze}\right)^4w^{-\frac1\kk}\pz\left(w^{1+\frac1\kk}\Dz\Theta\right) +\mathfrak{R}_1(\Theta)+\mathfrak{R}_2(\Theta)+\mathfrak{R}_3(\Theta),
    \end{aligned}
\end{align}
where $\mathfrak{R}_i(\Theta) = (U^0)^{-4}\left(1+\frac{\Theta}{\ze}\right)^2 R_i(\Theta)$ with $i=1,2,3$. In particular, \begin{align}
    &\mathfrak{R}_1(\Theta)  = -2(1+\kk)\calF^{-\kk-2}(U^0)^{-4}\left(1+\frac{\Theta}{\ze}\right)^3 w \ze\left(\pz\left(\frac{\Theta}{\ze}\right)\right)^2,
    \label{def:R1}
    \\ 
    &\mathfrak{R}_2(\Theta) = \frac{(1+\kk)^2}{\kk}(U^0)^{-4}\left(1+\frac{\Theta}{\ze}\right)^2 w' \left(3\left(\frac{\Theta}{\zeta}\right)^2+2\left(\frac{\Theta}{\zeta}\right)^3\right),\label{def:R2}\\
    &\mathfrak{R}_3(\Theta) = \frac{1+\kk}{\kk}(U^0)^{-4}\left(1+\frac{\Theta}{\ze}\right)^2w'\bigg((\calF^{-\kk-1}-1)+(1+\kk)(\calF-1)\nonumber\\
    &\quad+(1+\kk)(\calF^{-\kk-2}-1)\left(1+\frac{\Theta}{\ze}\right)^2\Dz\Theta\bigg).\label{def:R3}
\end{align}
In summary, \eqref{eq:momlagmain} becomes
\begin{align}\label{eq:momlagmainfull}
    \begin{aligned}
            &\left(1+\delta \lam^{-3\kk}w\calF^{-\kk}\calG^\kk\right)\left[\delta^{-1}\lambda^{3\kk}(\p_\tau^2 \Theta + \lamt \p_\tau \Theta) + \Theta\right]\\
    &+(1+\kk)\calF^{-\kk-2}(U^0)^{-4}\left(1+\frac{\Theta}{\ze}\right)^4L_0(\Theta) + \sum_{j=1}^3\mathfrak{R}_j(\Theta) = -\sum_{i=1}^6 E_i.
    \end{aligned}
\end{align}
Here and throughout, we define, for each $k\geq 0$, the elliptic operator
\begin{equation}\label{def: elliptic Lk}
    L_{k}(f)=-\frac{1}{w^{\frac1\kk + k}}\pz\left(w^{1+\frac1\kk + k}\Dz f\right),
\end{equation}
where $f$ is a generic smooth function. 

\section{The $\calD_i$-differentiated High-Order Equations}\label{sec: high-order equation}
In this section, we aim to derive the high-order version of the equation \eqref{eq:momlagmainfull} by commuting $\calD_i$. We begin by studying the $\calD_i$-differentiated elliptic term, i.e., 
\begin{align}\label{Di elliptic operator pre}
\begin{aligned}
&\calD_i\left(\calF^{-\kk-2}(U^0)^{-4}\left(1+\frac{\Theta}{\ze}\right)^4L_0(\Theta)\right)=\bar\calD_{i-1}\left(\Dz \left(\calF^{-\kk-2}(U^0)^{-4}\left(1+\frac{\Theta}{\ze}\right)^4L_0(\Theta)\right)\right)\\
=&\barcalD_{i-1}\left(\pz\left(\calF^{-\kk-2}(U^0)^{-4}\left(1+\frac{\Theta}{\ze}\right)^4\right)L_0(\Theta) + \calF^{-\kk-2}(U^0)^{-4}\left(1+\frac{\Theta}{\ze}\right)^4\Dz L_0(\Theta)\right)\\
=&\calF^{-\kk-2}(U^0)^{-4}\left(1+\frac{\Theta}{\ze}\right)^4 \calD_i L_0(\Theta)
+\left[\barcalD_{i-1}, \calF^{-\kk-2}(U^0)^{-4}\left(1+\frac{\Theta}{\ze}\right)^4\right]\Dz L_0(\Theta) + \barcalD_{i-1}\mathfrak{R}_4(\Theta), 
\end{aligned}
\end{align}
where
\begin{equation}\label{def:R4}
\mathfrak{R}_4(\Theta)=\pz\left(\calF^{-\kk-2}(U^0)^{-4}\left(1+\frac{\Theta}{\ze}\right)^4\right)L_0(\Theta). 
\end{equation}
Taking $\calD_i$ on both sides of \eqref{eq:momlagmainfull}, we obtain 
\begin{equation}\label{eq:highorder momlagmain pre}
    \begin{aligned}
        &\left(1+\delta \lam^{-3\kk}w\calF^{-\kk}\calG^\kk\right)\left(\delta^{-1}\lambda^{3\kk}(\p_\tau^2 \calD_i\Theta + \lamt \p_\tau \calD_i\Theta) + \calD_i\Theta\right)\\
        &+ [\calD_i, 1+\delta \lam^{-3\kk}w\calF^{-\kk}\calG^\kk] \left(\delta^{-1}\lambda^{3\kk}(\p_\tau^2 \Theta + \lamt \p_\tau \Theta) + \Theta\right)\\
        &+(1+\kk)\calF^{-\kk-2}(U^0)^{-4}\left(1+\frac{\Theta}{\ze}\right)^4 \calD_i L_0(\Theta)\\
        &+(1+\kk)\left[\barcalD_{i-1}, \calF^{-\kk-2}(U^0)^{-4}\left(1+\frac{\Theta}{\ze}\right)^4\right]\Dz L_0(\Theta)\\
        &+\calD_i\sum_{j=1}^3\mathfrak{R}_j(\Theta) + (1+\kk)\barcalD_{i-1}\mathfrak{R}_4(\Theta) = -\calD_i\sum_{j=1}^6E_j,
    \end{aligned}
\end{equation}
where $\mathfrak{R}_1(\Theta)$, $\mathfrak{R}_2(\Theta)$, and $\mathfrak{R}_3(\Theta)$ are given by \eqref{def:R1}--\eqref{def:R3}, respectively. 
Note that we do not commute $\calD_i$ directly through 
$\calF^{-\kk-2}(U^0)^{-4}\left(1+\frac{\Theta}{\ze}\right)^4$ in \eqref{eq:highorder momlagmain pre}; instead, we write $\calD_i = \barcalD_{i-1}\Dz$ and then commute $\Dz$ first. The reason of doing this is that we need to further exploit another hidden cancellation in 
\begin{align}\label{def: M}
\mathfrak{M}(\Theta)=\Dz \mathfrak{R}_3(\Theta) + (1+\kk)\mathfrak{R}_4(\Theta),
\end{align} so that the terms involving two spatial derivatives on $\Theta$ without the corresponding multiple of $w$ will be cancelled. 

Invoking the commutator identity \eqref{commutator [calDi g]f}, we see that
\begin{align*}
    [\calD_i, 1+\delta \lam^{-3\kk}w\calF^{-\kk}\calG^\kk] \left(\delta^{-1}\lambda^{3\kk}(\p_\tau^2 \Theta + \lamt \p_\tau \Theta) + \Theta\right)\\
    = [\calD_i, w\calF^{-\kk}\calG^\kk]\left(\p_\tau^2 \Theta + \lamt \p_\tau \Theta + \delta \lam^{-3\kk}\Theta\right).
\end{align*}
Next, we treat the third term on the LHS of \eqref{eq:highorder momlagmain pre} involving $\calD_iL_0(\Theta)$. We denote by $L_k^*$ the dual operator of $L_k$ given by \eqref{def: elliptic Lk}, which is given by
\begin{equation}\label{def:Lk dual}
L_k^*f = -\frac{1}{w^{\frac{1}{\kk}+k}}\Dz(w^{1+\frac{1}{\kk}+k}\dz f).
\end{equation}
Then, through $L_k$ and $L_k^*$ we can define the high-order elliptic operators:
\begin{equation}\label{def:high-order elliptic op}
\calL_j\calD_j = \begin{cases}
        L_j\calD_j,& j\text{ is even},\\
        L_{j}^*\calD_j,& j\text{ is odd}.
    \end{cases}
\end{equation}
In view of Lemma \ref{lem:Lcommute}, we have
$$
\calD_i L_0(\Theta) = \calL_{i}\calD_i\Theta + \sum_{j=0}^{i-1}\xi_{i,j}\calD_{i-j}\Theta,
$$
where $\xi_{i,j}$ are smooth functions on $[0,1]$. Consequently, the $\calD_i$-differentiated \eqref{eq:momlagmainfull} reads
\begin{equation}\label{eq:highorder momlagmain}
\boxed{
    \begin{aligned}
        &\left(1+\delta \lam^{-3\kk}w\calF^{-\kk}\calG^\kk\right)\left(\delta^{-1}\lambda^{3\kk}(\p_\tau^2 \calD_i\Theta + \lamt \p_\tau \calD_i\Theta) + \calD_i\Theta\right)\\
        &\quad+(1+\kk)\calF^{-\kk-2}(U^0)^{-4}\left(1+\frac{\Theta}{\ze}\right)^4 \calL_{i}\calD_i\Theta\\
        &\quad + \sum_{j=1}^3 \mfC_j+\calD_i\sum_{j=1}^2\mathfrak{R}_j(\Theta) + \barcalD_{i-1}\mathfrak{M}(\Theta)
        = -\calD_i\sum_{j=1}^6E_j,
    \end{aligned}
    }
\end{equation}
where $\mfC_j$ are commutators defined as follows: 
\begin{align}
    &\mfC_1 = (1+\kk)\left[\barcalD_{i-1}, \calF^{-\kk-2}(U^0)^{-4}\left(1+\frac{\Theta}{\ze}\right)^4\right]\Dz L_0(\Theta),\label{def:C1}\\
    &\mfC_2 = (1+\kk)\calF^{-\kk-2}(U^0)^{-4}\left(1+\frac{\Theta}{\ze}\right)^4\sum_{j=0}^{i-1}\xi_{i,j}\calD_{i-j}\Theta,\label{def:C2}\\
    &\mfC_3= [\calD_i, w\calF^{-\kk}\calG^\kk]\left(\p_\tau^2 \Theta + \lamt \p_\tau \Theta + \delta \lam^{-3\kk}\Theta\right). \label{def:C3}
    \end{align}
    
Next, we present a crucial cancellation structure in $\mathfrak{M}(\Theta)$: 
\begin{lem}[Cancellation on $\mathfrak{M}(\Theta)$]\label{lem: cancellation lemma M}
    Let $\mathfrak{M}(\Theta)$ be given by \eqref{def: M}. Then
    \begin{align}\label{def: M decomp}
        \mathfrak{M}(\Theta) = \mathfrak{M}_1+ \mathfrak{M}_2 + \mathfrak{M}_3, 
    \end{align}
    where 
    \begin{align}
        \begin{aligned}
            \mathfrak{M}_1 =& \frac{1+\kk}{\kk}\Dz\left((U^0)^{-4}\left(1+\frac{\Theta}{\ze}\right)^2 w'\right)\Big(\calF^{-\kk-1}-1+ (1+\kk)(\calF-1)\\
            &+(1+\kk)(\calF^{-\kk-2}-1)\left(1+\frac{\Theta}{\ze}\right)^2\Dz\Theta\Big),\\
            \mathfrak{M}_2 =&\frac{6(1+\kk)^2}{\kk}\left((U^0)^{-4}\left(1+\frac{\Theta}{\ze}\right)^2 w'\right)\Bigg(
            (\calF^{-\kk-2}-1)\left(1+\frac{\Theta}{\ze}\right)\left(\frac{\Theta}{\ze}\right)\pz\left(\frac{\Theta}{\ze}\right)\Bigg),\\
            \mathfrak{M}_3 =& (1+\kk)w\left((\kk+2)\calF^{-\kk-3}(\pz\calF)(U^0)^{-4}\left(1+\frac{\Theta}{\ze}\right)^4\right)\pz\Dz\Theta\\
           &-(1+\kk)\left(\calF^{-\kk-2}\pz\left((U^0)^{-4}\left(1+\frac{\Theta}{\ze}\right)^4\right)\right)\left(\frac{1+\kk}{\kk}w'\Dz\Theta+w\pz\Dz\Theta\right). 
        \end{aligned}
    \end{align}
 In other words, terms in $\mathfrak{M}(\Theta)$ that involve two spatial derivatives of $\Theta$ must be multiplied by $w$. 
\end{lem}
\begin{proof}
    Invoking Lemma \ref{lem:cancellation lemma theta}, we have
        \begin{align}\label{eq: Dz frakR3Theta}
        \begin{aligned}
            \Dz \mathfrak{R}_3(\Theta) &=\frac{1+\kk}{\kk}\Dz\left((U^0)^{-4}\left(1+\frac{\Theta}{\ze}\right)^2 w'\right)\Big(\calF^{-\kk-1}-1+ (1+\kk)(\calF-1)\\
            &+(1+\kk)(\calF^{-\kk-2}-1)\left(1+\frac{\Theta}{\ze}\right)^2\Dz\Theta\Big)\\
            &+\frac{1+\kk}{\kk}\left((U^0)^{-4}\left(1+\frac{\Theta}{\ze}\right)^2 w'\right)\Bigg[
            6(1+\kk)(\calF^{-\kk-2}-1)\left(1+\frac{\Theta}{\ze}\right)\left(\frac{\Theta}{\ze}\right)\pz\left(\frac{\Theta}{\ze}\right)\\
            &-(\kk+1)(\kk+2)\calF^{-\kk-3}\left(1+\frac{\Theta}{\ze}\right)^2\pz\calF\Dz\Theta\Bigg]. 
            \end{aligned}
    \end{align}
   We will show that the term in the last line of \eqref{eq: Dz frakR3Theta} can be cancelled after expanding $(1+\kk)\mathfrak{R}_4(\Theta)$. From definition of $\mathfrak{R}_4(\Theta)$ \eqref{def:R4}, we see that 
   \begin{align}\label{eq:pz frakR4}
       \begin{aligned}
           (1+\kk)\mathfrak{R}_4(\Theta)=& (1+\kk)\pz\left(\calF^{-\kk-2}(U^0)^{-4}\left(1+\frac{\Theta}{\ze}\right)^4\right)\left(-\frac{1+\kk}{\kk}w'\Dz\Theta-w\pz\Dz\Theta\right)\\
           =&(1+\kk)\left(-(\kk+2)\calF^{-\kk-3}\pz\calF(U^0)^{-4}\left(1+\frac{\Theta}{\ze}\right)^4\right)\left(-\frac{1+\kk}{\kk}w'\Dz\Theta-w\pz\Dz\Theta\right)\\
           +&(1+\kk)\left(\calF^{-\kk-2}\pz\left((U^0)^{-4}\left(1+\frac{\Theta}{\ze}\right)^4\right)\right)\left(-\frac{1+\kk}{\kk}w'\Dz\Theta-w\pz\Dz\Theta\right).
       \end{aligned}
   \end{align}
   It can be seen that the first term of the second line of \eqref{eq:pz frakR4} cancels with the final term of \eqref{eq: Dz frakR3Theta}. 
\end{proof}

\section{Building Block Expressions}\label{sect:buildingblock}
In this section, we record several building block expressions. These terms involve crucial quantities appearing in error estimates when acted upon by vector fields from classes $\barcalP_i$ or $\calP_i$. 
We begin with the following identity for $\calF$, which will be used frequently throughout. 
\begin{lem}\label{lem:calF-1} We have
        \begin{align}\label{eq:calF-1}
        \calF-1 = \Dz\left(\Theta + \frac{\Theta^2}{\ze}\right) + \frac13\Dz\left(\frac{\Theta^3}{\ze^2}\right).
        \end{align}    
\end{lem}
\begin{proof}
    We recall 
        $
        \calF = \left(\frac{\eta}{\zeta}\right)^2 \pz \eta. 
        $
        Since $\eta = \Theta+\zeta$, 
        $$
        \calF -1 = \Dz\Theta +\left(\frac{2\Theta}{\zeta}\pz \Theta +\frac{\Theta^2}{\zeta^2}\right)+\frac{\Theta^2}{\zeta^2}\pz \Theta. 
        $$
        Here, 
        \begin{align*}
            \frac{2\Theta}{\zeta}\pz \Theta +\frac{\Theta^2}{\zeta^2} = \pz\left(\frac{\Theta^2}{\zeta}\right) + \frac{1}{\zeta}\frac{2\Theta^2}{\zeta} = \Dz\left(\frac{\Theta^2}{\zeta}\right),
        \end{align*}
        and
        \begin{align*}
            \frac{\Theta^2}{\zeta^2}\pz \Theta = \frac{1}{3}\pz\left(\frac{\Theta^3}{\zeta^2}\right) +\frac{2}{3}\frac{\Theta^3}{\zeta^3}=\frac{1}{3}\Dz\left(\frac{\Theta^3}{\zeta^2}\right). 
        \end{align*}
\end{proof}

With the preparatory expressions above, we state the following proposition which concerns spatial derivatives of $\calF$ and $(U^0)^{-2}$ \textit{without} time derivatives.
\begin{prop}\label{prop:buildingblock}
    Let $i \ge 1$ be an integer, and consider $\bar P_i \in \barcalP_{i}$. Then the following identities hold:
    \begin{enumerate}
        \item ($\barcalP$ Derivative of Jacobian $\calF$)
        \begin{equation}
            \label{PbarF}
            \bar P_i\calF = P_{i+1}\Theta + \sum_{\substack{A_{1,2} \in \calP_{l_1,l_2}\\ l_1 + l_2 = i+2,\\l_1,l_2 \le i+1}} c^{l_1,l_2}_i (A_1\Theta)(A_2\Theta) + \sum_{\substack{A_{1,2,3} \in \calP_{l_1,l_2, l_3}\\ l_1 + l_2 + l_3 = i+2,\\l_1,l_2,l_3 \le i+1}} c^{l_1,l_2,l_3}_i(A_1\Theta)(A_2\Theta)(A_3\Theta),
        \end{equation}
        where $P_{i+1} = \bar P_i \Dz \in \calP_{i+1}$ and $c^{l_1,l_2}_i,c^{l_1,l_2,l_3}_i$ are constants.

        \item ($\barcalP$ Derivative of the velocity component $U^0$)
        \begin{equation}
            \label{PbarU0}
            \begin{split}
            &\bar P_i (U^0)^{-2} = \sum_{\substack{A_{1,2}\in \calP_{l_1,l_2},\\l_1+l_2 =i,\\l_1 \le i-1 }} c^{l_1,l_2}_{i,1} (A_1\pt\Theta)(A_2\pt\Theta) + \lamt\sum_{\substack{A_{1,2}\in \calP_{l_1,l_2},\\l_1+l_2 =i,\\l_1 \le i-1}} c^{l_1,l_2}_{i,2} (A_1\pt\Theta)(A_2\Theta)\\
            &+ \lamt^2\sum_{\substack{A_{1,2}\in \calP_{l_1,l_2},\\l_1+l_2 =i,\\l_1 \le i-1}} c^{l_1,l_2}_{i,3} (A_1\Theta)(A_2\Theta)+ \lamt\sum_{\substack{A_{1,2}\in \calP_{l_1,l_2},\\l_1+l_2 =i,\\l_1 \le i-1}} c^{l_1,l_2}_{i,4} (A_1\ze)(A_2\pt\Theta + \lamt A_2\Theta) - \lamt^2 \bar P_{i}\ze^2,
            \end{split}
        \end{equation}
        where $c^{l_1,l_2}_{i,j}$ are constants.
        
        In particular, we also have:
        \begin{equation}
            \label{PbarU0highest}
            \begin{split}
            &\barcalD_{i+1}(U^0)^{-2} = -2(\pt \Theta + \lamt(\Theta + \ze)) \calD_{i+1}\pt \Theta +\sum_{\substack{A_{1,2}\in \calP_{l_1,l_2},\\l_1+l_2 =i+1,\\l_1,l_2 \le i }} c^{l_1,l_2}_{i,1} (A_1\pt\Theta)(A_2\pt\Theta)\\
            &\quad+ \lamt\sum_{\substack{A_{1,2}\in \calP_{l_1,l_2},\\l_1+l_2 =i+1,\\l_1 \le i}} c^{l_1,l_2}_{i,2} (A_1\pt\Theta)(A_2\Theta) + \lamt^2\sum_{\substack{A_{1,2}\in \calP_{l_1,l_2},\\l_1+l_2 =i+1,\\l_1 \le i}} c^{l_1,l_2}_{i,3} (A_1\Theta)(A_2\Theta)\\
            &\quad + \lamt\sum_{\substack{A_{1,2}\in \calP_{l_1,l_2},\\l_1+l_2 =i+1,\\l_1,l_2 \le i}} c^{l_1,l_2}_{i,4} (A_1\ze)(A_2\pt\Theta) + \lamt\sum_{\substack{A_{1,2}\in \calP_{l_1,l_2},\\l_1+l_2 =i+1,\\l_1 \le i}} c^{l_1,l_2}_{i,5} (A_1\ze)(A_2\Theta)- \lamt^2 \barcalD_{i+1} \ze^2,
            \end{split}
        \end{equation}
        where $c^{l_1,l_2}_{i,j}$ are constants. 
    \end{enumerate}
\end{prop}
\begin{proof}
Throughout the proof, all constants might change from line to line.
    \begin{enumerate}
        \item 
        Invoking the identity \eqref{eq:calF-1}, then
        \begin{align*}
            \bar P_i \calF &= \bar P_i \Dz \Theta + \bar P_i \Dz \left(\frac{\Theta^2}{\ze}\right) + \frac13\bar P_i \Dz \left(\frac{\Theta^3}{\ze^2}\right)\\
            &= P_{i+1} \Theta + P_{i+1} \left(\frac{\Theta^2}{\ze}\right) + \frac13 P_{i+1} \left(\frac{\Theta^3}{\ze^2}\right),
        \end{align*}
        where $P_{i+1} := \bar P_i \Dz \in \calP_i$. Then the desired identity \eqref{PbarF} follows from Lemma \ref{lem:Fnonlinear}.

        \item We first prove \eqref{PbarU0}. To begin with, we note the following computations:
        \begin{equation}\label{pzU0}
        \begin{split}
        \pz (U^0)^{-2} &= \pz\left(1-(\pt\Theta + \lamt(\Theta + \ze))^2\right)\\
        &= -2(\pt\Theta + \lamt(\Theta + \ze))(\dz\dt\Theta + \lamt(\dz\Theta + 1))\\
        &= -2(\pt\Theta + \lamt\Theta)(\dz\dt\Theta + \lamt(\dz\Theta + 1)) - 2\lamt\ze(\dz\dt\Theta + \lamt\dz\Theta) - 2\lamt^2 \ze\\
        &= -2 (\pt\Theta + \lamt\Theta)\left(\Dz\dt\Theta - \frac{2}{\ze}\dt\Theta + \lamt(\Dz\Theta - \frac2\ze \Theta + 1)\right)\\
        &\quad -2\lamt\ze\left(\Dz\dt\Theta - \frac{2}{\ze}\dt\Theta + \lamt(\Dz\Theta - \frac2\ze \Theta)\right) - 2\lamt^2\ze.
        \end{split}
        \end{equation}
        Writing $\bar P_i = P_{i-1}\dz$, $P_{i-1} \in \calP_{i-1}$, and applying Lemma \ref{lem:generalproductrule}, we arrive at
        \begin{align*}
            \bar P_i (U^0)^{-2} &= \sum_{\substack{A_{1,2}\in \calP_{l_1,l_2},\\l_1+l_2 =i,\\l_1 \le i-1 }} c^{l_1,l_2}_i(A_1\p_\tau \Theta + \lamt A_1\Theta)(A_2\p_\tau\Theta + \lamt A_2\Theta)\\
            &\quad + \sum_{\substack{B_{1,2}\in \calP_{l_1,l_2},\\l_1+l_2 =i,\\l_1 \le i-1 }}\bar c^{l_1,l_2}_i (B_1 \ze)(B_2\p_\tau\Theta + \lamt B_2\Theta) + \lamt^2 \bar P_{i}\ze^2\\
            &= \sum_{\substack{A_{1,2}\in \calP_{l_1,l_2},\\l_1+l_2 =i,\\l_1 \le i-1 }} c^{l_1,l_2}_i(A_1\p_\tau \Theta + \lamt A_1(\Theta+\ze))(A_2\p_\tau\Theta + \lamt A_2\Theta) - \lamt^2 \bar P_{i}\ze^2,
        \end{align*}
        which completes the proof of \eqref{PbarU0} after we expand the expression above.
        
        To prove \eqref{PbarU0highest}, we view $\barcalD_{i+1} = \calD_i\pz$. According to the second line of \eqref{pzU0}, most terms can be treated in the same way as the proof of \eqref{PbarU0} except for the following top order term:
        \begin{align*}
        -2\calD_{i}&\left[(\pt \Theta + \lamt(\Theta + \ze))(\pz\pt\Theta)\right] = -2\calD_i\left[(\pt \Theta + \lamt(\Theta + \ze))(\Dz\pt\Theta-\frac{2}{\ze}\pt\Theta)\right]\\
        &= -2\calD_i\left((\pt \Theta + \lamt(\Theta + \ze)) \Dz\pt\Theta - 2\frac{(\pt\Theta)^2}{\ze} - 2\lamt \left(1+\frac{\Theta}{\ze}\right)\pt\Theta\right)\\
        &=: -2(I_1 + I_2 + I_3).
        \end{align*}
        Using Lemma \ref{lem:highestspatial}, we can write $I_1$ as
        \begin{align*}
            I_1 &= (\pt \Theta + \lamt(\Theta + \ze))\calD_{i+1}\dt\Theta + \sum_{\substack{A_{l_1,l_2}\in \calP_{l_1,l_2},\\l_1 + l_2 = i+1,\\l_1,l_2 \le i}} c_i^{l_1,l_2}(A_1\dt\Theta)(A_2\dt\Theta)\\
            &\quad + \lamt\sum_{\substack{A_{l_1,l_2}\in \calP_{l_1,l_2},\\l_1 + l_2 = i+1,\\l_1,l_2 \le i}} c_i^{l_1,l_2} (A_1\Theta)(A_2\dt\Theta) + \lamt\sum_{\substack{A_{l_1,l_2}\in \calP_{l_1,l_2},\\l_1 + l_2 = i+1,\\l_1,l_2 \le i}} c_i^{l_1,l_2} (A_1\ze)(A_2\dt\Theta)
        \end{align*}
        Using product rule \ref{eq:thetasquare}, we can write $I_2$ as
        $$
        I_2 = \sum_{\substack{A_{l_1,l_2}\in \calP_{l_1,l_2},\\l_1 + l_2 = i+1,\\l_1,l_2 \le i}} c_i^{l_1,l_2}(A_1\dt\Theta)(A_2\dt\Theta).
        $$
        Finally using \eqref{Pproduct}, we can write $I_3$ as
        $$
        I_3 = \lamt \sum_{\substack{A_{l_1,l_2}\in \calP_{l_1,l_2},\\l_1 + l_2 = i+1,\\l_2 \le i}}c_i^{l_1,l_2}(A_1\Theta)(A_2\pt\Theta).
        $$
        The identity \eqref{PbarU0highest} is thus proved by combining all computations above.
    \end{enumerate}
\end{proof}
\begin{rmk}
    It is readily seen from formula \eqref{PbarU0} that a $\calO(1)$-size source term $\lamt^2 \bar P_i \ze^2$ appears, and it persists when $i \le 2$. As we will see later, this exactly reflects the fact that the largest error appearing in energy estimates comes from low frequencies. We point out that this phenomenon is fundamentally tied to the scaling ansatz that we impose on the original variable $\rho$ and $u$, in the sense that the relativistic Euler equations \textbf{does not} admit any scaling symmetry in the original variables.
\end{rmk}
The second proposition concerns derivatives for powers of $\calF$ and $(U^0)^{-2}$:
\begin{prop}
    \label{prop:buildingblockpowers}
    Let $i \ge 1$ be an integer, and consider $\bar P_i \in \barcalP_{i}$. Then the following statements hold:
    \begin{enumerate}
        \item For a given $\alpha > 0$, we can write $\bar P_i (\calF^{-\alpha})$ as a linear combination of the following three types of terms:
        \begin{equation}
            \label{PbarF-atype1}
            \calF^{-\alpha-k}\prod_{\substack{W_j \in \bar P_{i_j},\\\sum_{j=1}^k i_j = i}} (W_j\Dz\Theta),
        \end{equation}
        \begin{equation}
            \label{PbarF-atype2}
            \calF^{-\alpha-k}\prod_{\substack{W_j^a \in \bar P_{i_j^a},\\ \sum_{a = 1}^2 \sum_{j=1}^k i_j^a = i}} (W_j^a\Dz\Theta),
        \end{equation}
        and
        \begin{equation}
            \label{PbarF-atype3}
            \calF^{-\alpha-k}\prod_{\substack{W_j^b \in \bar P_{i_j^b},\\ \sum_{b = 1}^3 \sum_{j=1}^k i_j^b = i}} (W_j^b\Dz\Theta),
        \end{equation}
        for $k = 1,\hdots, i$.

        \item For a given $\alpha \neq 0$, $\bar P_i \left((U^0)^{\alpha}\right)$ can be written as a linear combination of terms in the following types:
        \begin{equation}
            \label{PbarU0type1}
            (U^0)^{\alpha + 2k} \prod_{\substack{W_j^a \in \barcalP_{i_j^a},\\\sum_{a=1}^2 \sum_{j=1}^k i_j^a = i}}(W_j^aH_j^a(\Theta)),
        \end{equation}
        \begin{equation}
            \label{PbarU0type2}
            (U^0)^{\alpha + 2k} \prod_{\substack{W_j^b \in \barcalP_{i_j^b},\\\sum_{b=1}^2 \sum_{j=1}^k i_j^b = i}} (W_j^1\ze)(W_j^2H_j(\Theta)),
        \end{equation}
        and
        \begin{equation}
            \label{PbarU0type3}
            (U^0)^{\alpha + 2k} \prod_{\substack{W_j \in \barcalP_{i_j},\\\sum_{j=1}^k i_j = i}} \lamt^2(W_j\ze^2),
        \end{equation}
        where $k = 1,\hdots, i$, and $H_j^a(\Theta), H_j(\Theta)$ can either be $\pt\Theta$ or $\Theta$.
    \end{enumerate}
\end{prop}
\begin{proof}
        The representation of typical terms \eqref{PbarF-atype1}, \eqref{PbarF-atype2}, and \eqref{PbarF-atype3} follow from the chain rule \eqref{eq:chainrule} and then \eqref{PbarF}. As for expressions concerning $(U^0)^\alpha$, we note that the chain rule tells us that
        \begin{align*}
        \bar P_i((U^0)^\alpha) &= \bar P_i((U^0)^{-2})^{-\frac{\alpha}{2}} = \sum_{k=1}^i (U^0)^{-2\cdot (-\frac\alpha2 - k)}\prod_{\substack{W_j \in \barcalP_{i_j},\\\sum_{j=1}^k i_j = i}} c^{i_1,\hdots,i_k}_i (W_j(U^0)^{-2})\\
        &= \sum_{k=1}^i (U^0)^{\alpha + 2k}\prod_{\substack{W_j \in \barcalP_{i_j},\\\sum_{j=1}^k i_j = i}} c^{i_1,\hdots,i_k}_i (W_j(U^0)^{-2}).
        \end{align*}
        Then \eqref{PbarU0type1}, \eqref{PbarU0type2}, and \eqref{PbarU0type3} follow from \eqref{PbarU0}.
\end{proof}

With the preparations done in Proposition \ref{prop:buildingblockpowers}, we show in the following proposition a detailed decomposition of the key quantity $\barcalG$:

\begin{prop} \label{prop:buildingblockpowersG}
    Let $i \ge 1$ be an integer, and consider $\bar P_i \in \barcalP_i$. The quantity $\bar P_i (\kk \barcalG)$ can be represented as a linear combination of the following types of terms: 
    \begin{subequations}\label{PbarGtype1}
    \begin{equation}
            \label{PbarGtype1a}
            (U^0)^{\kk + 2k}\cdot\bar P\left((U^0)^{-\kk}(0)(1+\delta \lam^{-3\kk}(0)w\calF^{-\kk}(0)\right) \prod_{\substack{V_j^a \in \calP_{i_j^a},\\\sum_{a=1}^2 \sum_{j=1}^k i_j^a = n_2}}(V_j^aH_j^a(\Theta)),
    \end{equation}
    \begin{equation}
            \label{PbarGtype1b}
            (U^0)^{\kk + 2k}\cdot\bar P\left((U^0)^{-\kk}(0)(1+\delta \lam^{-3\kk}(0)w\calF^{-\kk}(0)\right) \prod_{\substack{V_j^b \in \calP_{i_j^b},\\\sum_{b=1}^2 \sum_{j=1}^k i_j^b = n_2}} (V_j^b\ze)(V_j^bH_j(\Theta)),
    \end{equation}
    and
    \begin{equation}
            \label{PbarGtype1c}
            (U^0)^{\kk + 2k}\cdot\bar P\left((U^0)^{-\kk}(0)(1+\delta \lam^{-3\kk}(0)w\calF^{-\kk}(0)\right) \prod_{\substack{W_j \in \barcalP_{i_j},\\\sum_{j=1}^k i_j = n_2}} \lamt^2(W_j\ze^2),
    \end{equation}
    \end{subequations}
    as well as
    \begin{subequations}\label{PbarGtype2}
    \begin{equation}
            \label{PbarGtype2a}
            \delta \lam^{-3\kk}\calF^{-\kk - 2k}\cdot\bar Pw \prod_{\substack{W_j \in \bar P_{i_j},\\\sum_{j=1}^k i_j = n_2}} (W_j\Dz\Theta),
    \end{equation}
    \begin{equation}
            \label{PbarGtype2b}
            \delta \lam^{-3\kk}\calF^{-\kk - 2k}\cdot\bar Pw \prod_{\substack{W_j^a \in \bar P_{i_j^a},\\ \sum_{a = 1}^2 \sum_{j=1}^k i_j^a = n_2}} (W_j^a\Dz\Theta),
    \end{equation}
    and
    \begin{equation}
            \label{PbarGtype2c}
            \delta \lam^{-3\kk}\calF^{-\kk - 2k}\cdot\bar Pw \prod_{\substack{W_j^b \in \bar P_{i_j^b},\\ \sum_{b = 1}^3 \sum_{j=1}^k i_j^b = n_2}} (W_j^b\Dz\Theta),
    \end{equation}
    \end{subequations}
    where $\bar P \in \barcalP_{n_1}$ and $k = 1,\hdots, n_2$. Here, $n_1, n_2$ are nonnegative integers such that $n_1 + n_2 = i$.
\end{prop}
\begin{proof}
    The proposition follows from a straightforward computation using chain rule, product rule \eqref{Pbarproduct}, and Proposition \ref{prop:buildingblockpowers}.
\end{proof}

Next, we present several key computations that will be instrumental in studying the nonlinear terms arising from the highest order spatial derivative of $\p_\tau \calF$. Such expressions will be useful in treating $E_4$.
\begin{prop}\label{prop:highestdtF}
    Let $i \ge 1$ be an integer. Then the following set of identities hold:\begin{equation}
        \label{eq:dtFsqpttheta}
        \calD_i \left(\pt\Theta \Dz\p_\tau\left(\frac{\Theta^2}{\ze}\right) \right)= 2\pt\Theta \frac{\Theta}{\ze}\calD_{i+1} \p_\tau \Theta + \sum_{\substack{A_{1,2,3} \in \calP_{l_1,l_2,l_3}\\l_1 + l_2 + l_3 = i+1\\l_1, l_2 \le i}} c_i^{l_1,l_2,l_3}(A_1\pt\Theta)(A_2\p_\tau\Theta)(A_3 \Theta),
    \end{equation}
    \begin{equation}
        \label{eq:dtFsqtheta}
        \calD_i \left(\Theta \Dz\p_\tau\left(\frac{\Theta^2}{\ze}\right) \right)= 2\frac{\Theta^2}{\ze}\calD_{i+1} \p_\tau \Theta + \sum_{\substack{A_{1,2,3} \in \calP_{l_1,l_2,l_3}\\l_1 + l_2 + l_3 = i+1\\l_1 \le i}} c_i^{l_1,l_2,l_3}(A_1\pt\Theta)(A_2\Theta)(A_3 \Theta),
    \end{equation}
    \begin{equation}
        \label{eq:dtFsqzeta}
        \calD_i \left(\ze \Dz\p_\tau\left(\frac{\Theta^2}{\ze}\right) \right)= 2\Theta\calD_{i+1} \p_\tau \Theta + \sum_{\substack{A_{1,2} \in \calP_{l_1,l_2}\\l_1 + l_2 = i+1\\l_1 \le i}} c_i^{l_1,l_2}(A_1\pt\Theta)(A_2\Theta).
    \end{equation}
    Similarly, the following identities hold:
    \begin{equation}
        \label{eq:dtFcupttheta}
        \begin{split}
        \calD_i \left(\pt\Theta \Dz\p_\tau\left(\frac{\Theta^3}{\ze^2}\right)\right) &= 3\pt\Theta\left(\frac{\Theta}{\ze}\right)^2\calD_{i+1} \p_\tau \Theta\\
        &\quad+ \sum_{\substack{A_{1,\hdots,4} \in \calP_{l_1,\hdots,l_4}\\l_1 +\hdots+ l_4 = i+1\\l_1,l_2\le i}} c_i^{l_1,l_2,l_3,l_4}(A_1\pt\Theta)(A_2\pt\Theta)(A_3\Theta)(A_4\Theta),
        \end{split}
    \end{equation}
    \begin{equation}
        \label{eq:dtFcutheta}
        \begin{split}
        \calD_i \left(\Theta \Dz\p_\tau\left(\frac{\Theta^3}{\ze^2}\right)\right) &= 3\Theta\left(\frac{\Theta}{\ze}\right)^2\calD_{i+1} \p_\tau \Theta\\
        &\quad+ \sum_{\substack{A_{1,\hdots,4} \in \calP_{l_1,\hdots,l_4}\\l_1 +\hdots+ l_4 = i+1\\l_1\le i}} c_i^{l_1,l_2,l_3,l_4}(A_1\pt\Theta)(A_2\Theta)(A_3\Theta)(A_4\Theta),
        \end{split}
    \end{equation}
    \begin{equation}
        \label{eq:dtFcuzeta}
        \calD_i \left(\ze \Dz\p_\tau\left(\frac{\Theta^3}{\ze^2}\right) \right)= 3\frac{\Theta^2}{\ze}\calD_{i+1} \p_\tau \Theta + \sum_{\substack{A_{1,2,3} \in \calP_{l_1,l_2,l_3}\\l_1 + l_2 + l_3 = i+1\\l_1 \le i}} c_i^{l_1,l_2,l_3}(A_1\pt\Theta)(A_2\Theta)(A_3 \Theta),
    \end{equation}
    where $c_i^{l_1,l_2,l_3}, c_i^{l_1,l_2,l_3,l_4}$ are constants.
\end{prop}
\begin{proof}
    We start with the proof of identity \eqref{eq:dtFsqpttheta}. First, we note that
    $$
    \Dz\pt\left(\frac{\Theta^2}{\ze}\right) = 2\Dz\left(\pt\Theta\frac  \Theta\ze\right) = 2\left(\Dz\pt\Theta \frac{\Theta}{\ze} + \pt\Theta \pz\left(\frac{\Theta}{\ze}\right)\right).
    $$
    Using this piece of computation, we may write
    \begin{equation}\label{dtFaux1}
        \calD_i \left(\pt\Theta \Dz\p_\tau\left(\frac{\Theta^2}{\ze}\right) \right) = 2\calD_i\left(\pt\Theta \frac{\Theta}{\ze} \Dz\pt\Theta\right) + 2\calD_i\left((\pt\Theta)^2 \pz\left(\frac{\Theta}{\ze}\right)\right).
    \end{equation}
    For the first term in \eqref{dtFaux1}, we use Lemma \ref{lem:highestspatial} and the product rule \eqref{Pproduct} to write
    \begin{equation}
        \label{dtFaux2}
        \begin{split}
            2\calD_i\left(\pt\Theta \frac{\Theta}{\ze} \Dz\pt\Theta\right) &= 2\pt\Theta \frac{\Theta}{\ze} \calD_{i+1}\pt\Theta + \sum_{\substack{V_{1,2} \in \calP_{l_1,l_2},\\l_1 + l_2 = i+1,\\l_1\le i}}V_1\left(\pt\Theta\frac{\Theta}{\ze}\right) V_2\pt\Theta\\
            &= 2\pt\Theta \frac{\Theta}{\ze} \calD_{i+1}\pt\Theta + \sum_{\substack{A_{1,2,3} \in \calP_{l_1,l_2,l_3}\\l_1 + l_2 + l_3 = i+1\\l_1, l_2 \le i}} c_i^{l_1,l_2,l_3}(A_1\pt\Theta)(A_2\p_\tau\Theta)(A_3 \Theta)
        \end{split}
    \end{equation}
    For the second term in \eqref{dtFaux1}, we make the following crucial observation:
    $$
    (\pt\Theta)^2 \pz\left(\frac{\Theta}{\ze}\right) = \frac{(\pt\Theta)^2}{\ze}\cdot \ze\pz\left(\frac{\Theta}{\ze}\right) = \frac{(\pt\Theta)^2}{\ze}\cdot\left(\Dz\Theta -3\frac{\Theta}{\ze}\right).
    $$
    Hence, we may use the product rule \eqref{Pproduct} and subsequently \eqref{eq:thetasquare} to obtain
    \begin{equation}
        \label{dtFaux3}
        \begin{split}
            2\calD_i\left((\pt\Theta)^2 \pz\left(\frac{\Theta}{\ze}\right)\right) &= \sum_{\substack{V_1\in\calP_{l_1}, V_2 \in \barcalP_{l_2},\\l_1 + l_2 = i}} c^{l_1,l_2} V_1\left(\frac{(\pt\Theta)^2}{\ze}\right)V_2\left(\Dz\Theta -3\frac{\Theta}{\ze}\right)\\
            &= \sum_{\substack{A_{1,2,3} \in \calP_{l_1,l_2,l_3}\\l_1 + l_2 + l_3 = i+1\\l_1, l_2 \le i}} c_i^{l_1,l_2,l_3}(A_1\pt\Theta)(A_2\p_\tau\Theta)(A_3 \Theta).
        \end{split}
    \end{equation}
    Then \eqref{eq:dtFsqpttheta} follows from \eqref{dtFaux1}, \eqref{dtFaux2}, and \eqref{dtFaux3}. The identities \eqref{eq:dtFsqtheta} and \eqref{eq:dtFsqzeta} follow from very similar arguments. To prove \eqref{eq:dtFcupttheta}, \eqref{eq:dtFcutheta}, and \eqref{eq:dtFcuzeta}, we use the following identity:
    \begin{align*}
        \Dz \pt \left(\frac{\Theta^3}{\ze^2}\right) &= 3\left[\left(\frac{\Theta}{\ze}\right)^2 \Dz\pt\Theta + 2\frac{\Theta}{\ze} \pt\Theta \pz\left(\frac{\Theta}{\ze}\right)\right].
    \end{align*}
    Then a similar argument to that proving \eqref{eq:dtFsqpttheta}, \eqref{eq:dtFsqtheta}, and \eqref{eq:dtFsqzeta} would yield the desired identities. We omit the tedious details here.
\end{proof}

Finally, we remark on several computations which concern the scenario where a $\calP_i$ vectorfield acts on $\pt (U^0)^{-2}$. 
\begin{prop}\label{prop:dtU0}
Let $i \ge 1$ be an integer and $P_i \in \calP_i$. Then the following identities hold:  
\begin{enumerate}
\item (Lower order formula) Let $H(\Theta) = \Theta$ or $\pt\Theta$. Then
{\allowdisplaybreaks
\begin{align}
    \label{eq:dtU0H}
    P_i&\left(H(\Theta) \pt (U^0)^{-2}\right) = \sum_{\substack{A_{1,2,3}\in\calP_{l_1,l_2,l_3},\\l_1+l_2 + l_3 = i}} c^{l_1,l_2,l_3}_{i,1}(A_1H(\Theta))(A_2\pt^2\Theta)\left[\lamt(A_3\ze) + A_3(\pt\Theta + \lamt\Theta)\right] \nonumber\\
    &+ \lamt\pt\lamt\sum_{\substack{A_{1}\in\calP_{l_1},A_2 \in\barcalP_{l_2}\\l_1+l_2= i}} c^{l_1,l_2}_{i,2}(A_1H(\Theta))(A_2\ze^2)\nonumber\\
    &+ \lamt\pt\lamt \sum_{\substack{A_{1,2,3}\in\calP_{l_1,l_2,l_3},\\l_1+l_2 + l_3 = i}} c^{l_1,l_2,l_3}_{i,3}(A_1H(\Theta))(A_2\Theta)(A_3\ze)\nonumber\\
    &+ (\lamt^2 + \pt\lamt)\sum_{\substack{A_{1,2,3}\in\calP_{l_1,l_2,l_3},\\l_1+l_2 + l_3 = i}} c^{l_1,l_2,l_3}_{i,4}(A_1H(\Theta))(A_2\pt\Theta)(A_3\ze)\nonumber\\
    &+ \sum_{\substack{A_{1,2,3}\in\calP_{l_1,l_2,l_3},\\l_1+l_2 + l_3 = i}} c^{l_1,l_2,l_3}_{i,5}(A_1H(\Theta))(A_2(\pt\Theta + \lamt\Theta))(A_3(\pt\lamt \Theta + \lamt \pt\Theta)).
\end{align}}
We also have the following:
\begin{equation}
    \label{eq:dtU0zeta}
    \begin{split}
    P_i&\left(\ze \pt (U^0)^{-2}\right) = \sum_{\substack{A_{1,2,3}\in\calP_{l_1,l_2,l_3},\\l_1+l_2 + l_3 = i}} c^{l_1,l_2,l_3}_{i,1}(A_1\ze)(A_2\pt^2\Theta)\left[\lamt(A_3\ze) + A_3(\pt\Theta + \lamt\Theta)\right] \\
    &+ \lamt\pt\lamt (P_i\ze^3) + \lamt\pt\lamt \sum_{\substack{A_1\in\calP_{l_1}, A_2 \in \barcalP_{l_2}\\l_1+l_2 = i}} c^{l_1,l_2}_{i,3}(A_1\Theta)(A_2\ze^2)\\
    &+ (\lamt^2 + \pt\lamt)\sum_{\substack{A_1\in\calP_{l_1}, A_2 \in \barcalP_{l_2}\\l_1+l_2 = i}} c^{l_1,l_2}_{i,4}(A_1\pt\Theta)(A_2\ze^2)\\
    &+ \sum_{\substack{A_{1,2,3}\in\calP_{l_1,l_2,l_3},\\l_1+l_2 + l_3 = i}} c^{l_1,l_2,l_3}_{i,5}(A_1\ze)(A_2(\pt\Theta + \lamt\Theta))(A_3(\pt\lamt \Theta + \lamt \pt\Theta)).
    \end{split}
\end{equation}
Here, $c_i^{l_1,l_2},c_i^{l_1,l_2,l_3}$ are constants.
\item (Highest order formula) Let $H(\Theta) = \Theta$ or $\pt\Theta$. Then
\begin{equation}
    \label{eq:dtU0Hhighest}
    \begin{split}
    \calD_i\left(H(\Theta) \pt (U^0)^{-2}\right) &= -2H(\Theta)(\pt\Theta + \lamt(\Theta +\ze))\calD_i\pt^2\Theta\\
    &+\sum_{\substack{A_{1,2,3}\in\calP_{l_1,l_2,l_3},\\l_1+l_2 + l_3 = i,\\l_2 \le i-1}} c^{l_1,l_2,l_3}_{i,1}(A_1H(\Theta))(A_2\pt^2\Theta)\left[\lamt(A_3\ze) + A_3(\pt\Theta + \lamt\Theta)\right] \\
    &+ \lamt\pt\lamt\sum_{\substack{A_{1}\in\calP_{l_1},A_2 \in\barcalP_{l_2}\\l_1+l_2= i}} c^{l_1,l_2}_{i,2}(A_1H(\Theta))(A_2\ze^2)\\
    &+ \lamt\pt\lamt \sum_{\substack{A_{1,2,3}\in\calP_{l_1,l_2,l_3},\\l_1+l_2 + l_3 = i}} c^{l_1,l_2,l_3}_{i,3}(A_1H(\Theta))(A_2\Theta)(A_3\ze)\\
    &+ (\lamt^2 + \pt\lamt)\sum_{\substack{A_{1,2,3}\in\calP_{l_1,l_2,l_3},\\l_1+l_2 + l_3 = i}} c^{l_1,l_2,l_3}_{i,4}(A_1H(\Theta))(A_2\pt\Theta)(A_3\ze)\\
    &+ \sum_{\substack{A_{1,2,3}\in\calP_{l_1,l_2,l_3},\\l_1+l_2 + l_3 = i}} c^{l_1,l_2,l_3}_{i,5}(A_1H(\Theta))(A_2(\pt\Theta + \lamt\Theta))(A_3(\pt\lamt \Theta + \lamt \pt\Theta)).
    \end{split}
\end{equation}
We also have the following:
\begin{align}
    \label{eq:dtU0zetahighest}
    \calD_i&\left(\ze \pt (U^0)^{-2}\right) = -2\ze(\pt\Theta + \lamt(\Theta +\ze))\calD_i\pt^2\Theta\nonumber\\
    &+\sum_{\substack{A_{1,2,3}\in\calP_{l_1,l_2,l_3},\\l_1+l_2 + l_3 = i,\\l_2\le i-1}} c^{l_1,l_2,l_3}_{i,1}(A_1\ze)(A_2\pt^2\Theta)\left[\lamt(A_3\ze) + A_3(\pt\Theta + \lamt\Theta)\right] \nonumber\\
    &+ \lamt\pt\lamt (P_i\ze^3) + \lamt\pt\lamt \sum_{\substack{A_1\in\calP_{l_1}, A_2 \in \barcalP_{l_2}\\l_1+l_2 = i}} c^{l_1,l_2}_{i,3}(A_1\Theta)(A_2\ze^2)\nonumber\\
    &+ (\lamt^2 + \pt\lamt)\sum_{\substack{A_1\in\calP_{l_1}, A_2 \in \barcalP_{l_2}\\l_1+l_2 = i}} c^{l_1,l_2}_{i,4}(A_1\pt\Theta)(A_2\ze^2)\nonumber\\
    &+ \sum_{\substack{A_{1,2,3}\in\calP_{l_1,l_2,l_3},\\l_1+l_2 + l_3 = i}} c^{l_1,l_2,l_3}_{i,5}(A_1\ze)(A_2(\pt\Theta + \lamt\Theta))(A_3(\pt\lamt \Theta + \lamt \pt\Theta)).
\end{align}
Here, $c_i^{l_1,l_2},c_i^{l_1,l_2,l_3}$ are constants.
\end{enumerate}
\end{prop}
\begin{proof}
    We first prove \eqref{eq:dtU0H}. Note that
    \begin{align*}
    H(\Theta)\pt (U^0)^{-2} &= -2H(\Theta)\left(\pt\Theta + \lamt(\Theta + \ze)\right)\left(\pt^2\Theta + \pt\lamt(\Theta + \ze) + \lamt \pt\Theta\right)\\
    &= -2H(\Theta)\bigg[\pt^2\Theta (\lamt\ze) + \pt^2\Theta (\pt\Theta + \lamt\Theta)\\
    &\quad + \lamt\pt\lamt \ze^2 + (\lamt^2+\pt\lamt) \pt\Theta\ze + 2\lamt \pt\lamt \Theta \ze\\
    &\quad + (\pt\Theta + \lamt\Theta)(\pt\lamt\Theta + \lamt \pt\Theta)\bigg]
    \end{align*}
    Then a direct application of product rule \eqref{Pproduct} and \eqref{Pbarproduct2} yields the desired identity \eqref{eq:dtU0H}. The identity \eqref{eq:dtU0zeta} follows from a similar argument. The proof for \eqref{eq:dtU0Hhighest} and \eqref{eq:dtU0zetahighest} follows from almost identical arguments to those proving \eqref{eq:dtU0H} and \eqref{eq:dtU0zeta}. The only modification to be made is to use \eqref{eq: leading order symmetry} instead of \eqref{Pproduct} when computing terms involving $\pt^2\Theta$.
\end{proof}
\begin{rmk}
    The importance of all identities appearing in Proposition \ref{prop:highestdtF} and \ref{prop:dtU0} lies in the precise form of the terms where most derivatives fall on $\Dz\pt\Theta$ (or $\pt^2\Theta$). We require such derivatives to be exactly $\calD_{i+1}\pt\Theta$ (or $\calD_{i}\pt^2\Theta$) instead of a general derivative $P_{i+1}\pt\Theta$ (or $P_{i}\pt^2\Theta$), so that one can combine this term with $\calD_i\pt\Theta$ and avoid derivative loss upon integration by parts (See Section \ref{sect: Energy Estimate V}).
\end{rmk}

\section{Energy, Norm, and Preliminary Bounds}\label{sect: energynorm}
In this section, we introduce a hierarchy of natural energies corresponding to \eqref{eq:momlagmainfull} and equivalent norms that facilitate our analysis. With these fundamental quantities introduced, we provide several other technical statements that prepare us for the nonlinear estimates in the forthcoming sections. Throughout the rest of this article, let $N \ge \lfloor \frac1\kk\rfloor + 10$ and $\kk\in (0,\frac{2}{3}]$ be fixed. 

\subsection{High-order Energies and the Associated Energy Identity}
We start by motivating the high-order energies and deriving the corresponding energy identity, obtained by testing equation \eqref{eq:highorder momlagmain} against $\p_\tau \calD_i\Theta$ with respect to the inner product $\left( \cdot,\cdot \right)_i$ (See Definition \ref{def: inner product}). The first term in the first line of \eqref{eq:highorder momlagmain} contributes to the following terms:
\begin{align}\label{Energy contribution 1st line 1}
    &\left((\delta^{-1}\lam^{3\kk}+w\calF^{-\kk}\calG^{\kk})\p_{\tau}^2\calD_i\Theta, \p_{\tau}\calD_i\Theta\right)_i \nonumber \\
    &= \frac{1}{2}\frac{d}{d\tau}\int_0^1 w^{\frac1\kk+i}\ze^2(\delta^{-1}\lam^{3\kk}+w\calF^{-\kk}\calG^{\kk})|\p_{\tau}\calD_i\Theta|^2-\frac12\int_0^1 w^{\frac1\kk+i}\ze^2\p_{\tau}(\delta^{-1}\lam^{3\kk}+w\calF^{-\kk}\calG^{\kk})|\p_{\tau}\calD_i\Theta|^2\nonumber\\
    &=\frac{1}{2}\frac{d}{d\tau}\int_0^1 w^{\frac1\kk+i}\ze^2(\delta^{-1}\lam^{3\kk}+w\calF^{-\kk}\calG^{\kk})|\p_{\tau}\calD_i\Theta|^2\underbrace{-\frac{3\kk}{2}\int_0^1 w^{\frac1\kk+i}\ze^2\delta^{-1}\lam^{3\kk}\lamt|\p_{\tau}\calD_i\Theta|^2}_{=J_1}\nonumber\\
    &-\frac12\int_0^1 w^{\frac1\kk+i}\ze^2\p_{\tau}(w\calF^{-\kk}\calG^{\kk})|\p_{\tau}\calD_i\Theta|^2.
\end{align}
Moreover, the second and third terms in the first line of \eqref{eq:highorder momlagmain} contribute, respectively, to the following two expressions. 
\begin{align}\label{Energy contribution 1st line 2}
\begin{aligned}
    &\left((\delta^{-1}\lam^{3\kk}+w\calF^{-\kk}\calG^{\kk})\lamt\p_{\tau}\calD_i\Theta, \p_{\tau}\calD_i\Theta\right)_i \\
    &= \underbrace{\int_0^1 w^{\frac1\kk+i}\ze^2\delta^{-1}\lam^{3\kk}\lamt|\p_{\tau}\calD_i\Theta|^2}_{=J_2}
    +\int_0^1 w^{\frac1\kk+i}\ze^2(w\calF^{-\kk}\calG^{\kk})\lamt|\p_{\tau}\calD_i\Theta|^2,
\end{aligned}
\end{align}
and
\begin{align}\label{Energy contribution 1st line 3}
\begin{aligned}
    &\left((1+\delta\lam^{-3\kk}w\calF^{-\kk}\calG^{\kk})\calD_i\Theta, \p_{\tau}\calD_i\Theta\right)_i\\
    &=\frac{1}{2}\frac{d}{d\tau}\int_0^1w^{\frac1\kk+i}\ze^2(1+\delta\lam^{-3\kk}w\calF^{-\kk}\calG^{\kk})|\calD_i\Theta|^2-\frac{1}{2}\int_0^1w^{\frac1\kk+i}\ze^2\p_{\tau}(\delta\lam^{-3\kk}w\calF^{-\kk}\calG^{\kk})|\calD_i\Theta|^2.
\end{aligned}
\end{align}
Here, it is important to note that one can combine $J_1$ and $J_2$ together, i.e., 
$$
J_1+J_2 = \left(1-\frac32\kk\right)\int_0^1 w^{\frac1\kk+i}\ze^2\delta^{-1}\lam^{3\kk}\lamt|\p_{\tau}\calD_i\Theta|^2. 
$$
This term contributes to the $N$-th order damping functional. See Definition \ref{def: damping} below. 

Additionally, we study the contribution of the third line of \eqref{eq:highorder momlagmain}. For any generic smooth function $g$ defined on $[0,1]$, we have 
\begin{align}\label{Energy contribution elliptic'}
    \begin{aligned}
        (g\calL_i\calD_i\Theta, \p_{\tau}\calD_i\Theta)_i=
        \begin{cases}
            -\int_0^1\ze^2 g \pz (w^{1+\frac1\kk+i}\calD_{i+1}\Theta)\p_{\tau}\calD_i\Theta,\quad \text{when }i\text{ is even},\\
            -\int_0^1\ze^2 g \Dz(w^{1+\frac1\kk+i}\calD_{i+1}\Theta)\p_{\tau}\calD_i\Theta,\quad \text{when }i\text{ is odd},
        \end{cases}
    \end{aligned}
\end{align}
where, we recall $\calL_i$ is given by \eqref{def:high-order elliptic op}. When $i$ is even, we integrate by parts in $\pz$ and get
\begin{equation}\label{ellipticibp}
\begin{split}
    (g\calL_i\calD_i\Theta, \p_{\tau}\calD_i\Theta)_i=&\int_0^1 w^{1+\frac1\kk+i}\pz(g \ze^2  \p_{\tau}\calD_i\Theta)(\calD_{i+1}\Theta)\\
    &=\int_0^1w^{1+\frac1\kk+i}\ze^2g(\p_{\tau}\calD_{i+1}\Theta)(\calD_{i+1}\Theta)+\int_0^1w^{1+\frac1\kk+i}\ze^2(\pz g)(\p_{\tau}\calD_i\Theta)(\calD_{i+1}\Theta)\\
    &=\frac{1}{2}\frac{d}{d\tau}\int_0^1w^{1+\frac1\kk+i}\ze^2g|\calD_{i+1}\Theta|^2 - \frac{1}{2}\int_0^1w^{1+\frac1\kk+i}\ze^2(\p_{\tau}g)|\calD_{i+1}\Theta|^2\\
    &+\int_0^1w^{1+\frac1\kk+i}\ze^2(\pz g)(\p_{\tau}\calD_i\Theta)(\calD_{i+1}\Theta). 
    \end{split}
\end{equation}
Similarly, when $i$ is odd, 
\begin{align*}
    (g\calL_i\calD_i\Theta, \p_{\tau}\calD_i\Theta)_i=& -\int_0^1\ze^2 g \pz(w^{1+\frac1\kk+i}\calD_{i+1}\Theta)\p_{\tau}\calD_i\Theta-\int_0^1\ze^2 g \frac{2}{\ze}(w^{1+\frac1\kk+i}\calD_{i+1}\Theta)\p_{\tau}\calD_i\Theta\\
    &=\int_0^1w^{1+\frac1\kk+i}\pz (g\ze^2\p_{\tau}\calD_i\Theta)\calD_{i+1}\Theta-\int_0^12\ze g (w^{1+\frac1\kk+i}\calD_{i+1}\Theta)\p_{\tau}\calD_i\Theta\\
    &=\int_0^1 w^{1+\frac1\kk+i}\ze^2g(\p_{\tau}\calD_{i+1}\Theta)(\calD_{i+1}\Theta) + \int_0^1w^{1+\frac1\kk+i}\ze^2(\pz g)(\p_{\tau}\calD_i\Theta)(\calD_{i+1}\Theta)\\
    &=\frac{1}{2}\frac{d}{d\tau}\int_0^1w^{1+\frac1\kk+i}\ze^2g|\calD_{i+1}\Theta|^2 - \frac{1}{2}\int_0^1w^{1+\frac1\kk+i}\ze^2(\p_{\tau}g)|\calD_{i+1}\Theta|^2\\
    &+\int_0^1w^{1+\frac1\kk+i}\ze^2(\pz g)(\p_{\tau}\calD_i\Theta)(\calD_{i+1}\Theta).
\end{align*}
Setting $g=(1+\kk)\calF^{-\kk-2}(U^0)^{-4}\left(1+\frac{\Theta}{\ze}\right)^4$, \eqref{Energy contribution elliptic'} turns into
\begin{align}\label{Energy contribution elliptic}
\begin{aligned}
    &\left((1+\kk)\calF^{-\kk-2}(U^0)^{-4}\left(1+\frac{\Theta}{\ze}\right)^4\calL_i\calD_i\Theta, \p_{\tau}\calD_i\Theta\right)_i\\
    &=\frac{1+\kk}{2}\frac{d}{d\tau}\int_0^1w^{1+\frac1\kk+i}\ze^2\left(\calF^{-\kk-2}(U^0)^{-4}\left(1+\frac{\Theta}{\ze}\right)^4\right)|\calD_{i+1}\Theta|^2 \\
    &- \frac{1+\kk}{2}\int_0^1w^{1+\frac1\kk+i}\ze^2\p_{\tau}\left(\calF^{-\kk-2}(U^0)^{-4}\left(1+\frac{\Theta}{\ze}\right)^4\right)|\calD_{i+1}\Theta|^2\\
    &+(1+\kk)\int_0^1w^{1+\frac1\kk+i}\ze^2\pz \left(\calF^{-\kk-2}(U^0)^{-4}\left(1+\frac{\Theta}{\ze}\right)^4\right)(\p_{\tau}\calD_i\Theta)(\calD_{i+1}\Theta).
    \end{aligned}
\end{align}
In view of \eqref{Energy contribution 1st line 1},  \eqref{Energy contribution 1st line 3}, and \eqref{Energy contribution elliptic}, we define our high-order energy functional as follows:
\begin{defn}[$N$-th order energy functional]
    Let $N \in \N$ and $\tau \ge 0$. We define
    \begin{equation}
        \label{defn:mfEN}
        \mfE^N(\tau) := \sum_{j = 0}^N \mfE_j(\tau),
    \end{equation}
    where
    \begin{equation}
        \label{defn:mfEj}
\begin{aligned}
    \mfE_j(\tau):=&\frac{1}{2}\delta^{-1}\lam^{3\kk}\int_0^1 w^{\frac1\kk+j}\ze^2(1+\delta\lam^{-3\kk}w\calF^{-\kk}\calG^{\kk})|\calD_j\p_{\tau}\Theta|^2d\ze \\
    &+ \frac{1}{2}\int_0^1 w^{\frac1\kk+j}\ze^2(1+\delta\lam^{-3\kk}w\calF^{-\kk}\calG^{\kk})|\calD_j\Theta|^2d\ze\\
    &+\frac{1+\kk}{2}\int_0^1w^{1+\frac1\kk+j}\ze^2\left(\calF^{-\kk-2}(U^0)^{-4}\left(1+\frac{\Theta}{\ze}\right)^4\right)|\calD_{j+1}\Theta|^2d\ze.
\end{aligned}
    \end{equation}
\end{defn}
Apart from the $N$-th order energy functional, we also introduce the following high-order damping functional from $J_1+J_2$ in \eqref{Energy contribution 1st line 1} and \eqref{Energy contribution 1st line 2}: 
\begin{defn}[$N$-th order damping functional]\label{def: damping}
    Let $N \in \N$ and $\tau \ge 0$. We define
    \begin{equation}\label{defn:damping}
        \mfD^N(\tau) := \sum_{j=0}^N\mfD_j(\tau), 
    \end{equation}
    where
    \begin{equation}
    \mfD_j(\tau) := \left(1-\frac32\kk\right)\lamt\delta^{-1}\lam^{3\kk}\int_0^1 w^{\frac1\kk+j}\ze^2|\p_{\tau}\calD_j\Theta|^2d\ze.
\end{equation}
\end{defn}
\begin{rmk}
    It is worth noting that the only property of $\mfD^N$ which we will take advantage of in order to close the nonlinear estimate is its coercivity whenever $\kk\leq \frac23$. 
\end{rmk}

Next, we will derive the energy identity associated with \eqref{defn:mfEN} and \eqref{defn:damping}. Let $1\leq i\leq N$ be fixed. By invoking \eqref{eq:highorder momlagmain} and summing up \eqref{Energy contribution 1st line 1}, \eqref{Energy contribution 1st line 2}, \eqref{Energy contribution 1st line 3}, and \eqref{Energy contribution elliptic}, we conclude the following fundamental energy balance:
\begin{align}\label{eq: high-order energy identity}
    \begin{aligned}
        \frac{d}{d\tau}\mfE_i(\tau) + \mfD_i(\tau) &=  -(\barcalD_{i-1}\mathfrak{M}, \p_{\tau}\calD_i\Theta)_i - \sum_{j=1}^2 (\calD_i\mathfrak{R}_j(\Theta), \p_\tau\calD_i\Theta)_i-\sum_{j=1}^6(\calD_iE_j, \p_{\tau}\calD_i\Theta)_i\\
        &-\sum_{j=1}^3(\mfC_j, \p_{\tau}\calD_i\Theta)_i -\sum_{j=1}^5\mfI_j.
    \end{aligned}
\end{align}
Here, $\barcalD_{i-1}\mathfrak{M}$ is given by \eqref{def: M}, $\mathfrak{R}_1(\Theta)$ and $\mathfrak{R}_2(\Theta)$ are given respectively by \eqref{def:R1} and \eqref{def:R2}, $E_1,\cdots, E_6$ are defined as \eqref{mainerror}, and the high-order commutators $\mfC_1, \mfC_2$, and $\mfC_3$ are given by \eqref{def:C1}--\eqref{def:C3}. 
Moreover, we denote by $\mathfrak{I}_1,\cdots,\mathfrak{I}_5$ the high-order errors generated when deriving the energy identity \eqref{eq: high-order energy identity}, which are taken from \eqref{Energy contribution 1st line 1}, \eqref{Energy contribution 1st line 2}, \eqref{Energy contribution 1st line 3}, and \eqref{Energy contribution elliptic}:
\begin{subequations}\label{def:I1toI5}
    \begin{align}
        &\mfI_1 = -\frac12\int_0^1 w^{\frac1\kk+i}\ze^2\p_{\tau}(w\calF^{-\kk}\calG^{\kk})|\p_{\tau}\calD_i\Theta|^2, \label{def:I3}\\
        &\mfI_2=\int_0^1 w^{\frac1\kk+i}\ze^2(w\calF^{-\kk}\calG^{\kk})\lamt|\p_{\tau}\calD_i\Theta|^2,\label{def:I4}\\
        &\mfI_3=-\frac{1}{2}\int_0^1w^{\frac1\kk+i}\ze^2\p_{\tau}(\delta\lam^{-3\kk}w\calF^{-\kk}\calG^{\kk})|\calD_i\Theta|^2,\label{def:I5}\\
        &\mfI_4=- \frac{1+\kk}{2}\int_0^1w^{1+\frac1\kk+i}\ze^2\p_{\tau}\left(\calF^{-\kk-2}(U^0)^{-4}\left(1+\frac{\Theta}{\ze}\right)^4\right)|\calD_{i+1}\Theta|^2,\label{def:I6}\\
        &\mfI_5 = (1+\kk)\int_0^1w^{1+\frac1\kk+i}\ze^2\pz \left(\calF^{-\kk-2}(U^0)^{-4}\left(1+\frac{\Theta}{\ze}\right)^4\right)(\p_{\tau}\calD_i\Theta)(\calD_{i+1}\Theta).\label{def:I7}
    \end{align}
\end{subequations}
We shall study the RHS estimates of \eqref{eq: high-order energy identity} in Sections \ref{sect: Energy Estimate 1}--\ref{sect: Energy Estimate III}. 

We conclude this subsection by introducing the following auxiliary norms, which will be convenient in the nonlinear analysis. 
\begin{defn}[Auxiliary Norms]
Fixing $\tau \ge 0$, we define: 
\begin{equation}
    \label{defn:calEi}
    \calE_i(\tau) :=\delta^{-1}\lam(\tau)^{3\kk}\|\calD_i \pt \Theta(\tau,\cdot)\|_i^2 + \|\calD_i \Theta(\tau,\cdot)\|_i^2 + \|\calD_{i+1}\Theta(\tau,\cdot)\|_{i+1}^2,
\end{equation}
and
\begin{equation}
    \label{defn:calEN}
    \calE^N(\tau) :=\sum_{i = 0}^N \calE_i(\tau).
\end{equation}
We also define
\begin{equation}
    \label{defn:calSN}
    \calS^N(\tau) := \sum_{j = 0}^{N-1} \|\calD_j\pt^2\Theta(\tau,\cdot)\|_j^2.
\end{equation}
\end{defn}
\begin{rmk}
    With appropriate bootstrap assumptions in place, we can show that $\calE^N(\tau)$ and $\mfE^N(\tau)$ are in fact equivalent. See Corollary \ref{cor: energy norm equiva}. 
\end{rmk}

The quantities $E^N, S^N$ and $\calE^N,\calS^N$ are related as follows: given $\tau \ge \tau_0 \ge 0$,
$$
E^N(\tau;\tau_0) = \sup_{\tau_0 \le \tau' \le \tau}\calE^N(\tau'),\quad S^N(\tau;\tau_0) = \int_{\tau_0}^\tau \calS^N(\tau') d\tau'.
$$

\subsection{Bootstrap Assumptions}
The main theorem of this work will be proved via a bootstrap argument. In this section, we state our main bootstrap assumptions and present several immediate consequences.

\begin{defn}[Bootstrap Assumptions]\label{assump: bootstrap}
Set $M_* > 0$ and $\eps \in (0,1)$ such that $M_* \eps < 1$. Let $\tau_* > 0$ be the maximal time such that the solution to \eqref{eq:momlagmainfull} exists on $[0,\tau_*]$ and verifies the following bounds:
\begin{subequations}
\begin{align}\label{bootstrap EN}
    E^N(\tau_*) \leq M_*\epsilon,
\end{align}
\begin{align}\label{bootstrap U0}
    \frac12 (1-\bar\lam^2\ze^2)\leq |U^0(\tau,\ze)|^{-2} \leq \frac32 (1- \bar\lam^2 \ze^2), \quad (\tau,\ze) \in [0,\tau_*]\times [0,1], 
\end{align}
and
    \begin{align} 
        \|\calF(\tau)-1\|_{L^{\infty}}+\|\kk\barcalG(\tau)-1\|_{L^{\infty}} \leq \frac{1}{100}, \quad \tau \in [0,\tau_*]. \label{bootstrap FG}
    \end{align}
\end{subequations}
\end{defn}
\begin{rmk}
    From now on, all appearances of ``bootstrap assumptions'' mean precisely the statements in Definition \ref{assump: bootstrap}.
\end{rmk}

We then discuss immediate consequences of the bootstrap assumptions in the following lemma. It should be noted that the bounds developed for quantities stated in the following lemma will play a crucial role when we bound the lower-order coefficients appearing in the top-order terms of the main energy estimate.
\begin{lem}\label{lem: induced bootstrap}
    Let $(\Theta,\pt\Theta)$ be a solution to \eqref{eq:momlagmainfull} on $[0,\tau_*]$ verifying the bootstrap assumptions. There exists a constant $\eps' > 0$, a constant $C > 0$ independent of $\eps'$ and $\delta$, such that for any $\eps \le \eps'$, $\delta \le \eps^\frac12$, and $\tau \in [0,\tau_*]$, the following estimates hold:

    \begin{enumerate}
    \item (Uniform bound for $U^0$) For any $k\neq 0$, there exists a constant $C(k,\bar\lam) > 0$ such that
    \begin{equation}
        \label{U0k bd}
        \frac{1}{C(k,\bar\lam)} \le (U^0)^k(\tau,\ze) \le C(k,\bar\lam),
    \end{equation}
    for all $\ze \in [0,1]$.
    \item ($P_j$-differentiated $\Theta$) 
        \begin{align}\label{bootstrap Theta}
   \sum_{j=0}^{2}\|P_j \Theta(\tau)\|_{L^{\infty}}\leq \frac{1}{100}, 
    \end{align}
    where $P_j\in \calP_j$ for $j=0,1,2$, 
    
    \item ($\tau$--derivatives of $\Theta$)     
        \begin{align} \sum_{j=0}^1\|\p_{\tau}P_j\Theta(\tau)\|_{L^{\infty}}\leq \frac{1}{100}\lam^{-\frac32 \kk}, \quad \|\p_{\tau}^2\Theta(\tau)\|_{L^{\infty}}\leq \frac{1}{50}\lam^{-\frac32 \kk}, \label{bootstrap ptau Theta}
                \end{align}

   \item ($\tau$-- and $\ze$--derivatives of $U^0$)
      \begin{align}\label{est: induced bootstrap U^0}
     \|\p_{\tau}(U^0(\tau))^{-2}\|_{L^{\infty}}\leq \frac{1}{10}\lam^{-\frac{3}{2}\kk},\quad \|\pz (U^0(\tau))^{-2}\|_{L^{\infty}} \leq 100, 
    \end{align}

    \item ($\tau$-- and $\ze$--derivatives of $\calF$)
        \begin{align}\label{est: induced bootstrap F}
        \|\p_{\tau} \calF(\tau)\|_{L^{\infty}} \leq \frac{1}{10}\lam^{-\frac32 \kk},\quad \|\pz \calF(\tau)\|_{L^{\infty}}\leq \frac{1}{10},
    \end{align}

    \item ($\tau$-- and $\ze$--derivative of $\calG$)
        \begin{align}\label{est: induced bootstrap G}
    \|\p_{\tau} \calG(\tau)\|_{L^{\infty}}\leq \frac{1}{10}\kk^{-1}\lam^{-\frac{3}{2}\kk},\quad
    \|\pz \calG(\tau)\|_{L^{\infty}}\leq 100\kk^{-1}, 
    \end{align}

    \item (Smallness of $\calF^{\kk}-1$ and $\calG^{\kk}-1$)
    \begin{align}\label{smallness of calFGkk-1}
        \|\calF^{\kk}(\tau)-1\|_{L^{\infty}} \leq \frac14, \quad \|\calG^{\kk}(\tau)-1\|_{L^{\infty}} \leq \frac14.
    \end{align}
    \end{enumerate}
\end{lem}
\begin{proof} We shall prove \eqref{U0k bd}--\eqref{smallness of calFGkk-1} in order. 
\begin{enumerate}
\item \textit{Proof of \eqref{U0k bd}.} This follows from \eqref{bootstrap U0} because, for each fixed $\bar\lam$, we have $1 - \bar\lam^2\ze^2 \leq 1$, and it is also bounded away from $0$ since $\bar\lam < 1$.

\item \textit{Proof of \eqref{bootstrap Theta}.} Since $0\leq j\leq 2$, from the Hardy-type embedding \eqref{est:LinftyHardy1} we conclude that $\|P_j \Theta(\tau)\|_{L^{\infty}}\leq C(\calE^N(\tau))^{\frac12}$ for $j=0,1,2$, where $C>0$ is a constant independent of $M_*$ and $\epsilon'$. Now, by choosing $\epsilon'$ so that 
$$
M_*\epsilon \leq 10^{-4}C^{-2},\quad\text{whenever }\epsilon\leq \epsilon',
$$ 
we obtain \eqref{bootstrap Theta} from \eqref{bootstrap EN}. 

\item \textit{Proof of \eqref{bootstrap ptau Theta}.} Similar to the proof of \eqref{bootstrap Theta}, the estimate for $\|\pt \Theta(\tau)\|_{L^{\infty}}$ follows also from \eqref{bootstrap EN} and \eqref{est:LinftyHardy1}. To bound $\|\pt^2 \Theta(\tau)\|_{L^{\infty}}$, our strategy is to express $\pt^2 \Theta$ as a linear combination of terms involving at most one $\tau$-derivative of $\Theta$. From \eqref{eq:momlagmainfull}, we obtain
\begin{align}\label{pt2Theta bootstrap eq}
\begin{aligned}
    (1+\delta\lam^{-3\kk}w\calF^{-\kk}\calG^{\kk})\pt^2\Theta  =& -(1+\delta\lam^{-3\kk}w\calF^{-\kk}\calG^{\kk})\lamt\pt\Theta-\delta\lam^{-3\kk}\bigg((1+\delta\lam^{-3\kk}w\calF^{-\kk}\calG^{\kk})\Theta\\
    &+(1+\kk)\calF^{-\kk-2}\left(1+\frac{\Theta}{\ze}\right)^4L_0\Theta
    + \sum_{j=1}^3\mathfrak{R}_j(\Theta) + \sum_{j=1}^6E_j\bigg).
    \end{aligned}
\end{align}
Now, by choosing $\epsilon'$ sufficiently small so that $\delta=\epsilon^{\frac12}< \frac{1}{100}$, we obtain from \eqref{bootstrap U0} and \eqref{bootstrap FG} that $(1+\delta\lam^{-3\kk}w\calF^{-\kk}\calG^{\kk})\approx 1$. 
In view of the definitions of $\mathfrak{R}_j(\Theta)$ \eqref{def:R1}--\eqref{def:R3} and $E_j$ \eqref{mainerror}, all terms on the RHS have at most one $\tau$--derivative on $\Theta$, apart from a term generated by $\delta\lam^{-3\kk}E_4$ when $\pt$ lands on $\calG^{\kk}$. 
However, this term can be absorbed by the LHS thanks to the smallness of $\delta$. Additionally, we infer from the first term on the RHS of \eqref{pt2Theta bootstrap eq}, whose coefficient is of $\calO(1)$ size, that both $\|\pt^2\Theta\|_{L^{\infty}}$ and $\|\pt\Theta\|_{L^{\infty}}$ should be of $\calO(\lam^{-\frac32\kk})$ size. 
This confirms the desired bound for $\|\pt^2\Theta\|_{L^{\infty}}$ in \eqref{bootstrap ptau Theta}.

\item \textit{Proof of \eqref{est: induced bootstrap U^0}.} 
    Recall that $(U^0)^{-2} = 1-(\p_\tau\Theta+\lamt (\Theta+\ze))^2$, we have
    $$
    \p_{\tau} (U^0)^{-2} = -2 (\p_\tau\Theta + \lamt (\Theta+\ze))\left(\p_\tau^2\Theta + (\p_{\tau}\lamt)(\Theta+\ze) +\lamt\p_{\tau}\Theta\right),
    $$
and
    \begin{align*}
        \pz (U^0)^{-2} = -2(\p_\tau\Theta + \lamt(\Theta+\ze))\left(\Dz\p_{\tau}\Theta+ \lamt(\Dz\Theta+3)\right) + 4\ze^{-1}(\p_\tau\Theta + \lamt(\Theta+\ze))^2. 
    \end{align*}
    These imply \eqref{est: induced bootstrap U^0} after invoking \eqref{est:lambdaasym}, \eqref{est:ptlamt}, \eqref{bootstrap Theta}, and \eqref{bootstrap ptau Theta}.  
    
\item  \textit{Proof of \eqref{est: induced bootstrap F}.} Recall that we have shown in Lemma \ref{lem:calF-1} that
    $$
    \calF-1= \Dz\left(\Theta + \frac{\Theta^2}{\ze}\right) + \frac13\Dz\left(\frac{\Theta^3}{\ze^2}\right). 
    $$
    In light of Lemma \ref{lem:Fnonlinear}, this contributes to two types of terms:
    \begin{align*}
        \Dz \Theta, \quad \prod_{\substack{A_j\in \calP_{l_j},\,l_j\leq 1\\
        \sum_{j=1}^{2 \text{ or } 3}l_j=2}}(A_j\Theta).
    \end{align*}
    Thus, both estimates in \eqref{est: induced bootstrap F} follow from \eqref{bootstrap Theta} and \eqref{bootstrap ptau Theta}, respectively, since $\p_{\tau}^k\calF = \p_{\tau}^k(\calF-1)$. 
    
\item    \textit{Proof of \eqref{est: induced bootstrap G}.} Since $\kk \barcalG = \calG^{-\kk}$,  
    we have 
    $$
    |\pz \calG(\tau)|\leq \kk^{-1}|G(\tau)^{\kk+1}\pz(\kk\barcalG(\tau))|=\kk^{-1}|\kk^{-1}\barcalG^{-1}|^{\frac{1+\kk}{\kk}}|\pz(\kk\barcalG(\tau))|.
    $$
    Because \eqref{bootstrap FG} indicates that $|\kk^{-1}\barcalG^{-1}|^{\frac{1+\kk}{\kk}} \approx 1$, hence it suffices to control $\|\pz (\kk\barcalG(\tau))\|_{L^{\infty}}$. We infer from Proposition \ref{prop:buildingblockpowersG}, \eqref{bootstrap FG}, \eqref{bootstrap Theta}, and \eqref{bootstrap ptau Theta} that $\|\pz (\kk\barcalG(\tau))\|_{L^{\infty}}\leq 10$. Additionally, $\|\p_{\tau} \calG(\tau)\|_{L^{\infty}}\leq \frac{1}{10}\kk^{-1}\lam^{-\frac{3}{2}\kk}$ follows from a similar argument together with \eqref{eq:dtbarcalG}, \eqref{bootstrap FG}, \eqref{est: induced bootstrap U^0}, and \eqref{est: induced bootstrap F}.  

\item \textit{Proof of \eqref{smallness of calFGkk-1}.} We infer from the Taylor expansion that 
$$
\calF^{\kk}-1 = \kk(\calF-1) + \calO(|\calF-1|^2). 
$$
Since $\|\calF(\tau)-1\|_{L^{\infty}}\leq 10^{-2}$ from \eqref{bootstrap FG} and $\kk<1$, $\|\calF(\tau)^{\kk}-1\|_{L^{\infty}}\leq \frac14$ follows. Additionally, $\|\calG(\tau)^{\kk}-1\|_{L^{\infty}}\leq \frac14$ is a direct consequence of \eqref{bootstrap FG} as $\frac{1}{\kk}\barcalG^{-1}=\calG^{\kk}$.
\end{enumerate}  
\end{proof}

We conclude this section by showing the following equivalence in between our main energy $\mfE^N(\tau)$ and the auxiliary norm $\calE^N(\tau)$:
\begin{cor}\label{cor: energy norm equiva}
Let $(\Theta,\pt\Theta)$ be a solution to \eqref{eq:momlagmainfull} in $[0,\tau_*]$ verifying the bootstrap assumptions. Suppose that $\eps < \eps'$ (where $\eps'$ is introduced in Lemma \ref{lem: induced bootstrap}) and $\delta \le \eps^\frac12$. Then there exists a universal constant $C' > 0$ such that
$$
\frac{1}{C'} \mfE_i(\tau) \le \calE_i(\tau) \le C' \mfE_i(\tau),\quad i = 0,\hdots, N,
$$
and
$$
\frac{1}{C'} \mfE^N(\tau) \le \calE_i(\tau) \le C' \mfE^N(\tau)
$$
hold for all $\tau \in [0,\tau_*]$.
\end{cor}
\begin{proof}
    In view of \eqref{defn:mfEN}, it suffices to show 
    $$
    (1+\delta\lam^{-3\kk}w\calF^{-\kk}\calG^{\kk})\approx 1, 
    $$
    and 
    $$
    \calF^{-\kk-2}(U^0)^{-4}\left(1+\frac{\Theta}{\ze}\right)^4 \approx 1. 
    $$
    Indeed, the former follows from \eqref{bootstrap U0} and \eqref{bootstrap FG}; the latter follows from \eqref{bootstrap Theta}, \eqref{bootstrap U0}, and \eqref{bootstrap FG}. 
\end{proof}

\subsection{Technical Bounds}
In this subsection, we prove two technical lemmas that will be useful in the nonlinear energy estimates. The first technical lemma below provides a general recipe to treat the error terms when only $\Theta$, $\pt\Theta$, and their spatial derivatives appear.

\begin{lem}\label{lem:nonlinear}
    Let $(\Theta,\pt\Theta)$ be a solution to \eqref{eq:momlagmainfull} in $[0,\tau_*]$ verifying the bootstrap assumptions. Let $i \le N+1$, $i_1, i_2$ be integers such that $i_1 + i_2 \le i$, and $J_1, J_2 \ge 0$ be integers such that $J_1 + J_2 \ge 1$. Assume further that $H_k(\Theta) = \Theta$ or $\pt\Theta$. Then the following inequality holds for all $\tau \in [0,\tau_*]$:
    \begin{equation}\label{est:nonlinear}
        \int_0^1 w^{\frac1\kk + i}\ze^2 \prod_{\substack{A_k \in \calP_{a_k},\\\sum_{k=1}^{J_1}a_k \le i_1}}(A_kH_k(\Theta))^2\cdot \prod_{\substack{B_k \in \barcalP_{b_k},\\\sum_{k=1}^{J_2}b_k \le i_2 -1}}(B_k P_1\Theta)^2 d\ze \lesssim (\delta\lam^{-3\kk})^K (\calE^N)^{J_1+J_2},
    \end{equation}
    where $a_k \le N$ whenever $H_{k}(\Theta) = \pt\Theta$, $P_1 \in \calP_1$, and $K = |\{b\;|\; 1 \le b \le J_1, H_b(\Theta) = \pt\Theta\}|$.
\end{lem}
\begin{proof}
    Fixing $i \le N+1$, we prove by inducting on $J := J_1 + J_2$. 
    \begin{enumerate}
        \item $J = 1$. We have either $J_1 = 1$ or $J_2 = 1$. In either case, \eqref{est:nonlinear} is reduced to
        \begin{equation}
            \label{nonlinaux1}
            \int_0^1 w^{\frac1\kk + i}\ze^2 |PH(\Theta)|^2d\ze,
        \end{equation}
        where $P \in \calP_{l}$, $l \le i$, and $H(\Theta) = \Theta$ or $\pt\Theta$. We first demonstrate the desired estimate for $H(\Theta) = \Theta$. If $l \le 4$, we invoke \eqref{est:LinftyHardy1} to bound \eqref{nonlinaux1} by $\|w^{\frac1\kk + i}\|_{L^\infty} \|P\Theta\|_{L^\infty}^2 \lesssim \calE^N$. If $l \ge 5$, we first consider $i \le N$. In this case, we invoke the $L^2$ embedding \eqref{est:L2Hardy} to obtain       
        $$
        \eqref{nonlinaux1} \le \|w^{\frac{1}{\kk}+i-(\frac1\kk + 2l - N)}\|_{L^\infty} \int_0^1 w^{\frac1\kk + 2l - N}|P\Theta|^2 d\ze \lesssim \calE^N,
        $$
        where we used the fact that for $l \le i \le N$:
        $$
        \frac{1}{\kk}+i-(\frac1\kk + 2l - N) \ge N- i \ge 0.
        $$
For the reader's convenience, we verify here the applicability of \eqref{est:L2Hardy}. It requires
$\frac{N-\alpha}{2} \le k \equiv l \le N$, $P_k \in \calP_k$. Since $N = \lfloor \frac1\kk\rfloor + 10$, we have
$ l \equiv k  \geq \frac{N}{2} - \frac{\alpha}{2} = \frac{1}{2}(N - \frac{1}{\kappa})  = \frac{1}{2}( \lfloor \frac1\kk\rfloor + 10 - \frac{1}{\kappa} )
        \geq \frac{9}{2}
$.
        
        If $i = N+1$, we use \eqref{est:L2Hardy2} instead of \eqref{est:L2Hardy} to obtain
        $$
        \eqref{nonlinaux1} \le \|w^{\frac{1}{\kk}+i-(\frac1\kk + 2(l-1)+1 - N)}\|_{L^\infty} \int_0^1 w^{\frac1\kk + 2(l-1)+1 - N}|P\Theta|^2 d\ze \lesssim \calE^N,
        $$
        where we used that
        $$
        \frac{1}{\kk}+i-(\frac1\kk + 2(l-1)+1 - N) \ge -i+1+N \ge 0.
        $$
        The scenario where $H(\Theta) = \pt\Theta$ is similar, where we only need to consider $i \le N$ due to the assumption of the lemma on the maximum number of derivatives falling on $\pt\Theta$. This ends the proof of the base case.
        
        \item $J \ge 2$. 
        We proceed by assuming the inductive hypothesis that \eqref{est:nonlinear} holds up to $J-1$. 
        Let us also without loss assume that $a_1 \le \hdots \le a_{J_1}$, $b_1 \le \hdots \le b_{J_2}$. We discuss the following cases:
        \begin{enumerate}
            \item $\min\{a_1, b_1 + 1\} \le 4$: in this case, either $A_1 \in \calP_l$ or $B_1P_1 \in \calP_l$, where $l \le 4$. Without loss of generality, we assume $a_1$ is the minimum and $H_1(\Theta) = \Theta$. Then we may apply \eqref{est:LinftyHardy1} to $A_1H_1(\Theta)$ to obtain the bound:
            \begin{align*}
                \eqref{est:nonlinear} &\lesssim \calE^N\cdot\int_0^1 w^{\frac1\kk + i}\ze^2 \prod_{\substack{A_k \in \calP_{a_k},\\\sum_{k=2}^{J_1}a_k \le i_1}}(A_kH_k(\Theta))^2\cdot \prod_{\substack{B_k \in \barcalP_{b_k},\\\sum_{k=1}^{J_2}b_k \le i_2 -1}}(B_k P_1\Theta)^2 d\ze\\
                &\lesssim (\delta \lam^{-3\kk})^K (\calE^N)^{J_1 + J_2},
            \end{align*}
            where we used the induction hypothesis in the final inequality above.
            \item $\min\{a_1, b_1 + 1\} \ge 5$: we consider $\max\{a_{J_1}, b_{J_2}\}$. Then by definition of $a_k, b_k$, it is clear that the maximum, and thus all $a_k, b_k$, are less or equal to $N - 4$. We first consider the case where $b_{J_2} = \max\{a_{J_1}, b_{J_2}\}$. Our plan is to bound $B_{{b_{J_2}}}P_1\Theta$ in $L^2$ using \eqref{est:L2Hardy} and all other terms using $L^\infty$ embedding \eqref{est:LinftyHardyweighted1N-4}. Indeed,
            \begin{align*}
                \eqref{est:nonlinear}&\lesssim \|w^{\frac1\kk + i - (\frac1\kk + 2b_{J_2}+ 2 - N)-2\sum_{k=1}^{J_2 - 1}(b_k - 1) - 2\sum_{l=1}^{J_1}(a_l - 2)}\|_{L^\infty}\\
                &\quad \times \prod_{\substack{A_k \in \calP_{a_k},\\\sum_{k=1}^{J_1}a_k \le i_1}}\|w^{a_k - 2}A_kH_k(\Theta)\|_{L^\infty}^2 \cdot \prod_{\substack{B_k \in \barcalP_{b_k},\\\sum_{k=1}^{J_2-1}b_k \le i_2 -1}}\|w^{b_k - 1} B_kP_1\Theta\|_{L^\infty}^2\\
                &\quad \times \int_0^1 w^{\frac1\kk + 2b_{J_2}+ 2 - N}\ze^2 |B_{b_{J_2}}P_1\Theta|^2 d\ze\\
                &\lesssim (\delta \lam^{-3\kk})^K (\calE^N)^{J_1 + J_2},
            \end{align*}
            where we used that
            \begin{align*}
                \frac1\kk + &i - (\frac1\kk + 2b_{J_2}+ 2 - N)-2\sum_{k=1}^{J_2 - 1}(b_k - 1) - 2\sum_{l=1}^{J_1}(a_l - 2)\\
                &= i - 2(\sum_{k=1}^{J_1}a_k + \sum_{l=1}^{J_2}b_l) -2+N + 2(J_2 - 1) + 4J_1\\
                &\ge -i + N  + 2(J_1 + J_2) -2+2J_1 \ge 1+2J_1 > 0.
            \end{align*}
            Here, we used that $\sum a_k + \sum b_l \le i - 1$, $J_1 + J_2 \ge 2$, and $i \le N+1$.
            If $a_{J_1} = \max\{a_{J_1}, b_{J_2}\}$, we instead will bound $A_{{a_{J_1}}}H_{J_1}(\Theta)$ in $L^2$ using \eqref{est:L2Hardy} (or \eqref{est:L2Hardy2} when $i = N+1$) and all other terms using $L^\infty$ embedding \eqref{est:LinftyHardyweighted1N-4}. The argument follows similarly to the previous case and we omit the details here.
        \end{enumerate}
    \end{enumerate}
\end{proof}

The following variant of Lemma \ref{lem:nonlinear} will also be instrumental when $\pt^2\Theta$ or its derivatives appear:
\begin{lem}\label{lem:nonlinear2}
    Let $(\Theta,\pt\Theta)$ be a solution to \eqref{eq:momlagmainfull} in $[0,\tau_*]$ verifying the bootstrap assumptions. Let $i \le N+1$, $i_1, i_2$ be integers such that $i_1 + i_2 \le i$, and $J_1, J_2, J_3 \ge 0$ be integers such that $J_1 + J_2 + J_3 \ge 1$. Assume further that $\bar H_k(\Theta) = \Theta$, $\pt\Theta$, or $\pt^2\Theta$. Then the following inequality holds for all $\tau \in [0,\tau_*]$:
    \begin{equation}\label{est:nonlinear2}
        \int_0^1 w^{\frac1\kk + i}\ze^2 \prod_{\substack{A_k \in \calP_{a_k},\\\sum_{k=1}^{J_1+J_2}a_k \le i_1}}(A_k \bar H_k(\Theta))^2\cdot \prod_{\substack{B_k \in \barcalP_{b_k},\\\sum_{k=1}^{J_3}b_k \le i_2 -1}}(B_k P_1\Theta)^2 d\ze \lesssim (\delta\lam^{-3\kk})^K (\calE^N)^{J_1+J_3}(\calS^N)^{J_2},
    \end{equation}
    where $a_k \le N$ whenever $\bar H_{k}(\Theta) = \pt\Theta$ and $a_k \le N-1$ whenever $\bar H_k(\Theta) = \pt^2\Theta$. Moreover, $P_1 \in \calP_1$, $J_2 = |\{b\;|\; \bar H_b(\Theta) = \pt^2\Theta\}|$, and $K = |\{b\;|\; \bar H_b(\Theta) = \pt\Theta\}|$.
\end{lem}
\begin{proof}
    The proof of this lemma follows an almost identical argument to that of Lemma \ref{lem:nonlinear}, except that we invoke suitable Hardy inequalities \eqref{est:L2Hardypt2}, \eqref{est:LinftyHardypt2}, and \eqref{est:LinftyHardyweightedpt2} and use the definition of $\calS^N$ whenever we treat terms involving $\pt^2\Theta$.
\end{proof}

\section{Energy Estimates I: Terms Arising From the Pressure Gradient}\label{sect: Energy Estimate 1}
From this section onward, we initiate the investigation of various types of error terms appearing in the $\calD_i$-differentiated equation \eqref{eq:highorder momlagmain}. For clarity, we assert that in Sections \ref{sect: Energy Estimate 1}--\ref{sect: Energy Estimate III}, we will always assume that the solution $(\Theta, \pt\Theta)$ of \eqref{eq:momlagmain} on $[0, \tau_*]$ verifies the bootstrap assumptions given in Definition \ref{assump: bootstrap}. Also, we will assume that $\epsilon\leq \epsilon'$ and $\delta\leq \epsilon^{\frac12}$, where $\epsilon'$ is defined in Lemma \ref{lem: induced bootstrap}. 

We start our analysis by first studying $\barcalD_{i-1}\mathfrak{M}$ \eqref{def: M}, 
$\calD_i\mathfrak{R}_1$ \eqref{def:R1} and $\calD_i\mathfrak{R}_2$ \eqref{def:R2}, as well as the commutators $\mfC_1$ \eqref{def:C1} and $\mfC_2$ \eqref{def:C2}, which emerge from the derivation of the elliptic operator in Sections \ref{sec: elliptic operator} and \ref{sec: high-order equation}. 
\subsection{Estimates for $\barcalD_{i-1}\mathfrak{M}$}
In this subsection, we aim to prove an estimate for $\|\barcalD_{i-1}\mathfrak{M}\|_i^2$.
\begin{prop}\label{prop:est M}
    Let $\mathfrak{M}$ be given by \eqref{def: M decomp}. Then
    \begin{equation}
        \|\barcalD_{i-1}\mathfrak{M}\|_i^2 \lesssim \calE^N. 
    \end{equation}
\end{prop}
\subsubsection{Estimates for $\barcalD_{i-1}\mathfrak{M}_1$}
We begin by further decomposing $\mathfrak{M}_1 = \mathfrak{M}_{11}+\mathfrak{M}_{12}+\mathfrak{M}_{13}$, where
\begin{align*}
&\mathfrak{M}_{11} = \frac{1+\kk}{\kk}\Dz\left((U^0)^{-4}\left(1+\frac{\Theta}{\ze}\right)^2 w'\right)\Big(\calF^{-\kk-1}-1\Big),\\
&\mathfrak{M}_{12} = \frac{(1+\kk)^2}{\kk}\Dz\left((U^0)^{-4}\left(1+\frac{\Theta}{\ze}\right)^2 w'\right)(\calF-1),\\
&\mathfrak{M}_{13} = \frac{(1+\kk)^2}{\kk}\Dz\left((U^0)^{-4}\left(1+\frac{\Theta}{\ze}\right)^2 w'\right)\Big((\calF^{-\kk-2}-1)\left(1+\frac{\Theta}{\ze}\right)^2\Dz\Theta\Big).
\end{align*}
\paragraph{Analysis of $\barcalD_{i-1}\mathfrak{M}_{11}$.} 
Since product rule \eqref{Pbarproduct} implies 
\begin{align*}
    \barcalD_{i-1}\mathfrak{M}_{11} =& \frac{1+\kk}{\kk}\sum_{\substack{A,B\in \barcalP_{l_{A,B}}\\l_A+l_B=i-1}}c^{AB}A\Dz\left((U^0)^{-4}\left(1+\frac{\Theta}{\ze}\right)^2 w'\right) B\Big(\calF^{-\kk-1}-1\Big),
\end{align*}
where $c^{AB}$ are real constants, 
we begin by studying $\Dz\left((U^0)^{-4}\left(1+\frac{\Theta}{\ze}\right)^2 w'\right)$. Writing 
$$
\Dz\left((U^0)^{-4}\left(1+\frac{\Theta}{\ze}\right)^2 w'\right)=\Dz\left((U^0)^{-4}\ze\left(1+\frac{\Theta}{\ze}\right)^2 \frac{w'}{\ze}\right),
$$
and then invoking the product rule \eqref{eq:basicproduct}, we obtain
$$
\Dz\left((U^0)^{-4}\ze\left(1+\frac{\Theta}{\ze}\right)^2 \frac{w'}{\ze}\right) = \Dz\left(\ze\left(1+\frac{\Theta}{\ze}\right)^2\right)(U^0)^{-4}\frac{w'}{\ze} + \pz \left((U^0)^{-4}\frac{w'}{\ze}\right)\ze\left(1+\frac{\Theta}{\ze}\right)^2.
$$
Since 
\begin{align}\label{eq:Dz ze f}
\Dz (\ze f) = 3f + \ze\pz f
\end{align}
holds for any generic $f\in C^1$, we have
\begin{align*}
    \Dz\left(\ze\left(1+\frac{\Theta}{\ze}\right)^2\right) =& 3\left(1+\frac{\Theta}{\ze}\right)^2 + 2\ze \left(1+\frac{\Theta}{\ze}\right)\pz \left(\frac{\Theta}{\ze}\right)\\
    =&3\left(1+\frac{\Theta}{\ze}\right)^2+2\left(1+\frac{\Theta}{\ze}\right)\left(\Dz \Theta - 3\frac{\Theta}{\ze}\right). 
\end{align*}
Thus, we obtain
\begin{align*}
    \Dz\left((U^0)^{-4}\left(1+\frac{\Theta}{\ze}\right)^2 w'\right) =& 3\left(1+\frac{\Theta}{\ze}\right)^2(U^0)^{-4}\frac{w'}{\ze}+2\left(1+\frac{\Theta}{\ze}\right)\left(\Dz \Theta - 3\frac{\Theta}{\ze}\right)(U^0)^{-4}\frac{w'}{\ze}\\
    &+\pz \left((U^0)^{-4}\frac{w'}{\ze}\right)\ze\left(1+\frac{\Theta}{\ze}\right)^2.
\end{align*}
Now, we compute $A\Dz\left((U^0)^{-4}\left(1+\frac{\Theta}{\ze}\right)^2 w'\right)$ with $A\in\barcalP_{l_A}$, $l_A\leq i-1$ by applying \eqref{Pbarproduct} to the first two terms on the RHS, while applying \eqref{Pbarproduct2} to the third term. Consequently, $A\Dz\left((U^0)^{-4}\left(1+\frac{\Theta}{\ze}\right)^2 w'\right)$ is a linear combination of terms of the following form:
\begin{equation*}
 A_1\frac{w'}{\ze} A_2(U^0)^{-4}\prod_{\substack{A_3^j\in\calP_{i_j}\\\sum_{j=1}^{m_3}=l_3}}A_3^j\Theta, \quad\text{with } A_{1,2}\in \barcalP_{l_{1,2}},\quad l_1+l_2+l_3\leq i,\quad m_3=0,1,2.  
\end{equation*}
In particular, this reduces to $A_1\frac{w'}{\ze} A_2(U^0)^{-4}$ whenever $m_3=0$. 

Moreover, invoking Proposition \ref{prop:buildingblockpowers}, we see that $A_2(U^0)^{-4}$ is a linear combination of the following terms:
\begin{align}\label{eq:A(U0)^-4}
\begin{aligned}
       &(U^0)^{-4+2k_2}\prod_{\substack{A_2^j\in \barcalP_{i_j}\\\sum_{j=1}^{m_2}i_j=l_2}}A_2^jH_j(\Theta),\quad k_2=1,2, \quad m_2 = k_2, 2k_2,\\
       &(U^0)^{-4+2k_2}\prod_{\substack{A_2^j\in \barcalP_{i_j}\\\sum_{j=1}^{k_2}i_j=l_2}}\lamt^2 A_2^j\ze,\quad k_2=1,2,
\end{aligned}
\end{align}
where $H_j(\Theta)=\p_\tau\Theta$ or $\Theta$. Also, since $B\in \barcalP_{l_B}$ with $l_B\leq i-1$, $B(\calF^{-\kk-1}-1)$ is a linear combination of the following terms: 
\begin{align*}
    \calF^{-\kk-1-k_B}\prod_{\substack{B^j\in \barcalP_{i_j}\\\sum_{j=1}^{m_B}i_j=l_B}}(B^j\Dz\Theta), \quad k_B=1,\cdots, l_B,\quad m_B = k_B, 2k_B, 3k_B,
\end{align*}
provided that $l_B>0$.
In summary, $\barcalD_{i-1}\mathfrak{M}_{11}$ is a linear combination of terms of the following forms:
\begin{subequations}\label{barDi-1M11lb0}
\begin{align}
    (U^0)^{-4+2k_2}(\calF^{-\kk-1}-1) \left(A_1\frac{w'}{\ze}\right)\prod_{\substack{A_2^j\in \barcalP_{i_j}\\\sum_{j=1}^{m_2}i_j=l_2}}\left(A_2^jH_j(\Theta)\right)\prod_{\substack{A_3^j\in\calP_{i_j}\\\sum_{j=1}^{m_3}=l_3}}A_3^j\Theta,\label{barDi-1M111stlb0}\\
    (U^0)^{-4+2k_2}(\calF^{-\kk-1}-1) \left(A_1\frac{w'}{\ze}\right)\prod_{\substack{A_2^j\in \barcalP_{i_j}\\\sum_{j=1}^{k_2}i_j=l_2}}(\lamt^2 A_2^j\ze)\prod_{\substack{A_3^j\in\calP_{i_j}\\\sum_{j=1}^{m_3}=l_3}}A_3^j\Theta,\label{barDi-1M112ndlb0}
\end{align}
\end{subequations}
whenever $l_B=0$, and
\begin{subequations}\label{barDi-1M11}
\begin{align}
    (U^0)^{-4+2k_2}\calF^{-\kk-1-k_B} \left(A_1\frac{w'}{\ze}\right)\prod_{\substack{A_2^j\in \barcalP_{i_j}\\\sum_{j=1}^{m_2}i_j=l_2}}\left(A_2^jH_j(\Theta)\right)\prod_{\substack{A_3^j\in\calP_{i_j}\\\sum_{j=1}^{m_3}=l_3}}A_3^j\Theta \prod_{\substack{B^j\in \barcalP_{i_j}\\\sum_{j=1}^{m_B}i_j=l_B}}(B^j\Dz\Theta),\label{barDi-1M111st}\\
    (U^0)^{-4+2k_2}\calF^{-\kk-1-k_B} \left(A_1\frac{w'}{\ze}\right)\prod_{\substack{A_2^j\in \barcalP_{i_j}\\\sum_{j=1}^{k_2}i_j=l_2}}(\lamt^2 A_2^j\ze)\prod_{\substack{A_3^j\in\calP_{i_j}\\\sum_{j=1}^{m_3}=l_3}}A_3^j\Theta \prod_{\substack{B^j\in \barcalP_{i_j}\\\sum_{j=1}^{m_B}i_j=l_B}}(B^j\Dz\Theta),\label{barDi-1M112nd}
\end{align}
\end{subequations}
whenever $l_B\geq 1$. It is worth noting that in \eqref{barDi-1M11lb0}, we obtain from \eqref{eq:calF-1} and Lemma \ref{lem:Fnonlinear}, as well as the expansion
\begin{align*}
\calF^{-\kk-1} -1 = - (\kk+1)(\calF-1) + \calO(|\calF-1|^2)
\end{align*}
that $\calF^{-\kk-1}$ is at least linear in $\Theta$. 
From the bootstrap assumption $U^0, \calF\approx 1$, and 
thanks to the smoothness of $A_1\frac{w'}{\ze}$ (Lemma \ref{lem:wderivative}), we have
\begin{align}\label{control:barDi-1M111stlb0}
    \|\eqref{barDi-1M111stlb0}\|_{i}^2 \lesssim \int_0^1 w^{1+\frac1\kk+i}|\calF^{-\kk-1}-1|^2\prod_{\substack{A_2^j\in \barcalP_{i_j}\\\sum_{j=1}^{m_2}i_j=l_2}}\left|A_2^jH_j(\Theta)\right|^2\prod_{\substack{A_3^j\in\calP_{i_j}\\\sum_{j=1}^{m_3}i_j=l_3}}|A_3^j\Theta|^2\lesssim \calE^N,
\end{align}
and
\begin{align}\label{control:barDi-1M111st}
    \|\eqref{barDi-1M111st}\|_{i}^2 \lesssim \int_0^1 w^{1+\frac1\kk+i}\prod_{\substack{A_2^j\in \barcalP_{i_j}\\\sum_{j=1}^{m_2}i_j=l_2}}\left|A_2^jH_j(\Theta)\right|^2\prod_{\substack{A_3^j\in\calP_{i_j}\\\sum_{j=1}^{m_3}i_j=l_3}}|A_3^j\Theta|^2 \prod_{\substack{B^j\in \barcalP_{i_j}\\\sum_{j=1}^{m_B}i_j=l_B}}|B^j\Dz\Theta|^2\lesssim \calE^N.
\end{align}
Here, we applied Lemma \ref{lem:nonlinear}  
and note that $l_2, l_3\leq i$ and $l_B\leq i-1$. In addition, a similar argument yields:
\begin{align}\label{control:barDi-1M112ndlb0}
    \|\eqref{barDi-1M112ndlb0}\|_{i}^2 \lesssim \int_0^1 w^{1+\frac1\kk+i}|\calF^{-\kk-1}-1|^2\prod_{\substack{A_3^j\in\calP_{i_j}\\\sum_{j=1}^{m_3}i_j=l_3}}|A_3^j\Theta|^2 \lesssim \calE^N,
\end{align}
and
\begin{align}\label{control:barDi-1M112nd}
    \|\eqref{barDi-1M112nd}\|_{i}^2 \lesssim \int_0^1 w^{1+\frac1\kk+i}\prod_{\substack{A_3^j\in\calP_{i_j}\\\sum_{j=1}^{m_3}i_j=l_3}}|A_3^j\Theta|^2 \prod_{\substack{B^j\in \barcalP_{i_j}\\\sum_{j=1}^{m_B}i_j=l_B}}|B^j\Dz\Theta|^2\lesssim \calE^N.
\end{align}

\paragraph{Analysis of $\barcalD_{i-1}\mathfrak{M}_{12}$ and $\barcalD_{i-1}\mathfrak{M}_{13}$.}
Our analysis of $\barcalD_{i-1}\mathfrak{M}_{11}$ can be directly adapted to $\barcalD_{i-1}\mathfrak{M}_{12}$. The only exception is that we require only Proposition \ref{prop:buildingblock}, instead of Proposition \ref{prop:buildingblockpowers}, to treat $\calF-1$. Additionally, we need to expand $B\Big((\calF^{-\kk-2}-1)\left(1+\frac{\Theta}{\ze}\right)^2\Dz\Theta\Big)$, $B\in \barcalP_{l_B}$, $l_B\leq i-1$ when treating $\barcalD_{i-1}\mathfrak{M}_{13}$. Consequently, $\barcalD_{i-1}\mathfrak{M}_{13}$ is also a linear combination of terms given by \eqref{barDi-1M11lb0} and \eqref{barDi-1M11}, with $\calF^{-\kk-1-k_B}$ replaced by $\calF^{-\kk-2-k_B}$ in \eqref{barDi-1M11}. In summary, 
\begin{align}\label{est: barDi-1M1}
    \|\barcalD_{i-1}\mathfrak{M}_{1}\|_i^2 \lesssim \calE^N. 
\end{align}

\subsubsection{Estimates for $\barcalD_{i-1}\mathfrak{M}_2$}
Writing 
\begin{align}\label{M2}
\begin{aligned}
\mathfrak{M}_2 &= \frac{6(1+\kk)^2}{\kk}\left((U^0)^{-4}\left(1+\frac{\Theta}{\ze}\right)^2 \frac{w'}{\ze}\right)\Bigg((\calF^{-\kk-2}-1)\left(1+\frac{\Theta}{\ze}\right)\left(\frac{\Theta}{\ze}\right)\ze\pz\left(\frac{\Theta}{\ze}\right)\Bigg)\\
&=\frac{6(1+\kk)^2}{\kk}\left((U^0)^{-4}\left(1+\frac{\Theta}{\ze}\right)^2 \frac{w'}{\ze}\right)\Bigg((\calF^{-\kk-2}-1)\left(1+\frac{\Theta}{\ze}\right)\left(\frac{\Theta}{\ze}\right)\left(\Dz\Theta-3\frac{\Theta}{\ze}\right)\Bigg),
\end{aligned}
\end{align}
and using Proposition \ref{prop:buildingblockpowers}, one can see that $\barcalD_{i-1}\mathfrak{M}_{2}$ is a linear combination of terms of the forms \eqref{barDi-1M111st} and \eqref{barDi-1M112nd}, with $\calF^{-\kk-1-k_B}$ replaced by $\calF^{-\kk-2-k_B}$. Therefore, 
\begin{align}\label{est: barDi-1M2}
    \|\barcalD_{i-1}\mathfrak{M}_{2}\|_i^2 \lesssim (\calE^N)^2. 
\end{align}
Note that this estimate must be (at least) quadratic in $\calE^N$ due to the appearance of  $\left(\frac{\Theta}{\ze}\right)\left(\Dz\Theta-3\frac{\Theta}{\ze}\right)$ in the second line of \eqref{M2}. 

\subsubsection{Estimates for $\barcalD_{i-1}\mathfrak{M}_3$}
We decompose $\mathfrak{M}_3 = \mathfrak{M}_{31}+\mathfrak{M}_{32}+\mathfrak{M}_{33}$, where
\begin{align*}
&\mathfrak{M}_{31} = (1+\kk)(2+\kk)w\left(\calF^{-\kk-3}(\pz\calF)(U^0)^{-4}\left(1+\frac{\Theta}{\ze}\right)^4\right)\pz\Dz\Theta,\\
&\mathfrak{M}_{32} =-\frac{(1+\kk)^2}{\kk}w'\left(\calF^{-\kk-2}\pz\left((U^0)^{-4}\left(1+\frac{\Theta}{\ze}\right)^4\right)\right)\Dz\Theta,\\
&\mathfrak{M}_{33} =-(1+\kk)w\left(\calF^{-\kk-2}\pz\left((U^0)^{-4}\left(1+\frac{\Theta}{\ze}\right)^4\right)\right)\pz\Dz\Theta. 
\end{align*}
We can observe that $\mathfrak{M}_{31}$ includes terms with two $\ze$-derivatives acting on $\Theta$, specifically $\pz\Dz\Theta$, while the other emerges from $\pz\mathcal{F}$. This also applies to $\mathfrak{M}_{33}$. 
We require the additional $w$-weight to control the corresponding terms in $\|\cdot\|_{i}$. 
\paragraph{Analysis of $\barcalD_{i-1}\mathfrak{M}_{31}$.}
From the product rule \eqref{Pbarproduct}, 
\begin{align}\label{barcalDM31 pre}
\begin{aligned}
    \barcalD_{i-1} \mathfrak{M}_{31} =& (1+\kk)(2+\kk)\sum_{\substack{A,B\in \barcalP_{l_A, l_B}\\l_A+l_B=i-1}} c_{i-1}^{l_Al_B} A\left(w\calF^{-\kk-3}(U^0)^{-4}\left(1+\frac{\Theta}{\ze}\right)\right)B\left((\pz\calF)(\pz\Dz\Theta)\right)\\
    =&(1+\kk)(2+\kk)\sum_{\substack{A_{1,2,3,4}\in \barcalP_{l_{1,2,3,4}}\\B_{1,2}\in \calP_{\ell_{1,2}}\\\sum_{j=1}^4l_j+\sum_{j=1}^2\ell_j=i-1}} c_{i-1,l_B}^{l_{1,2,3,4}\ell_{1,2}} A_1w(A_2\calF^{-\kk-3})A_3(U^0)^{-4}A_4\left(1+\frac{\Theta}{\ze}\right)\\
    &\cdot B_1(\pz\calF)B_2(\pz\Dz\Theta),
    \end{aligned}
\end{align}
where we applied \eqref{Pbarproduct2} on $B\left((\pz\calF)(\pz\Dz\Theta)\right)$. We discuss the following cases:
\begin{itemize}
    \item [(a)] When either $\ell_1=i-1$ or $\ell_2=i-1$. It suffices to treat the sub-case when $\ell_1=i-1$, which is more complex compared to the other sub-case.  If $\ell_1=i-1$, the RHS of \eqref{barcalDM31 pre} reduces to a linear combination of terms of the following form:
    $$
    w\calF^{-\kk-3}(U^0)^{-4}\left(1+\frac{\Theta}{\ze}\right)(\pz\Dz\Theta)B_1(\pz \calF). 
    $$  
  From Proposition \ref{prop:buildingblock}, this term can be expressed as a linear combination of the following terms: 
  \begin{equation}\label{M31case(a)}
      w\calF^{-\kk-3}(U^0)^{-4}\left(1+\frac{\Theta}{\ze}\right)(\pz\Dz\Theta)\prod_{\substack{B_1^j\in \calP_{i_j}\\\sum_{j=1}^{m_1}i_j=i}}(B_1^j\Dz\Theta),\quad m_1=1,2,3. 
  \end{equation}
  If there exists some $j\in\{1,\cdots, m_1\}$ such that $i_j=i$, then we use the bootstrap assumptions $U^0, \calF \approx 1$, as well as Lemma \ref{lem:nonlinear} (after replacing $i$ by $i+1$) to have
  \begin{align*}
      \|\eqref{M31case(a)}\|_{i}^2 \lesssim \int_0^1 w^{\frac1\kk+i+2}\ze^2 \left|1+\frac{\Theta}{\ze}\right|^2|\pz\Dz\Theta|^2\prod_{\substack{B_1^j\in \calP_{i_j}\\\sum_{j=1}^{m_1}i_j=i}}|B_1^j\Dz\Theta|^2 \lesssim (\calE^N)^2. 
  \end{align*}
 On the other hand, if $i_j<i$ for all $j\in\{1,\cdots, m_1\}$, then the extra $w$-weight plays no role when applying Lemma \ref{lem:nonlinear}. Consequently, 
   \begin{align*}
      \|\eqref{M31case(a)}\|_{i}^2 \lesssim |w|^2\int_0^1 w^{\frac1\kk+i}\ze^2 \left|1+\frac{\Theta}{\ze}\right|^2|\pz\Dz\Theta|^2\prod_{\substack{B_1^j\in \calP_{i_j}\\\sum_{j=1}^{m_1}i_j=i,\,\,i_j\leq i-1}}|B_1^j\Dz\Theta|^2 \lesssim (\calE^N)^2. 
  \end{align*}
 
  \item [(b)] When $\ell_1,\ell_2\leq i-2$. The RHS of \eqref{barcalDM31 pre} is a linear combination of terms of the following form:
  \begin{align}\label{M31case(b)'}
      A_1w(A_2\calF^{-\kk-3})A_3(U^0)^{-4}A_4\left(1+\frac{\Theta}{\ze}\right)
    B_1(\pz\calF)B_2(\pz\Dz\Theta).
  \end{align}
   In the above, we apply Proposition \ref{prop:buildingblockpowers} to express $A_2\calF^{-\kk-3}$ as a linear combination of
   $$
   \calF^{-\kk-3-k_2}\prod_{\substack{A_2^j\in \barcalP_{i_j}\\\sum_{j=1}^{m_2}i_j=l_2}}(A_2^j\Dz\Theta), \quad k_2=0,\cdots, l_2,\quad m_2 = k_2, 2k_2, 3k_2, 
   $$
   provided that $l_2>0$, 
   and $A_3(U^0)^{-4}$ as a linear combination of terms similar to \eqref{eq:A(U0)^-4}.
   Moreover, from Proposition \ref{prop:buildingblock} we see that $B_1\pz \calF$ is a linear combination of 
   $$
   \prod_{\substack{B^j\in \barcalP_{i_j}\\\sum_{j=1}^{m}i_j=\ell_1+1\\\ell_1\leq i-2}}(B^j\Dz\Theta),\quad m=1,2,3.
   $$
   Then, \eqref{M31case(b)'} reduces to a linear combination of the following terms:
       \begin{subequations}\label{M31case(b)l20}
    \begin{align}
      (U^0)^{-4+2k_3}\calF^{-\kk-3}A_1w\prod_{\substack{A_3^j\in \barcalP_{i_j}\\\sum_{j=1}^{m_3}i_j=l_3}}A_3^jH_j(\Theta)A_4\left(1+\frac{\Theta}{\ze}\right)
    \left(\prod_{\substack{B^j\in \barcalP_{i_j}\\\sum_{j=1}^{m}i_j=\ell_1+1}}(B^j\Dz\Theta)\right)B_2(\pz\Dz\Theta),\\
       (U^0)^{-4+2k_3}\calF^{-\kk-3}A_1w\prod_{\substack{A_3^j\in \barcalP_{i_j}\\\sum_{j=1}^{k_3}i_j=l_3}}(\lamt^2A_3^j\ze) A_4\left(1+\frac{\Theta}{\ze}\right)
    \left(\prod_{\substack{B^j\in \barcalP_{i_j}\\\sum_{j=1}^{m}i_j=\ell_1+1}}(B^j\Dz\Theta)\right)B_2(\pz\Dz\Theta),
    \end{align}
  \end{subequations}
  whenever $l_2=0$, and
    \begin{subequations}\label{M31case(b)}
    \begin{align}
      (U^0)^{-4+2k_3}\calF^{-\kk-3-k_2}A_1w\prod_{\substack{A_2^j\in \barcalP_{i_j}\\\sum_{j=1}^{m_2}i_j=l_2}}(A_2^j\Dz\Theta)\prod_{\substack{A_3^j\in \barcalP_{i_j}\\\sum_{j=1}^{m_3}i_j=l_3}}A_3^jH_j(\Theta)A_4\left(1+\frac{\Theta}{\ze}\right)\nonumber\\
    \cdot \left(\prod_{\substack{B^j\in \barcalP_{i_j}\\\sum_{j=1}^{m}i_j=\ell_1+1}}(B^j\Dz\Theta)\right)B_2(\pz\Dz\Theta),\\
       (U^0)^{-4+2k_3}\calF^{-\kk-3-k_2}A_1w\prod_{\substack{A_2^j\in \barcalP_{i_j}\\\sum_{j=1}^{m_2}i_j=l_2}}(A_2^j\Dz\Theta)\prod_{\substack{A_3^j\in \barcalP_{i_j}\\\sum_{j=1}^{k_3}i_j=l_3}}(\lamt^2A_3^j\ze) A_4\left(1+\frac{\Theta}{\ze}\right)\nonumber\\
    \cdot \left(\prod_{\substack{B^j\in \barcalP_{i_j}\\\sum_{j=1}^{m}i_j=\ell_1+1}}(B^j\Dz\Theta)\right)B_2(\pz\Dz\Theta),
    \end{align}
  \end{subequations}
  whenever $l_2\geq 1$, 
  where $k_3=1,2$ and $m_3=k_3, 2k_3$. Since $A_1w$ is smooth by Lemma \ref{lem:wderivative} and $\ell_1,\ell_2\leq i-2$, terms in \eqref{M31case(b)l20} and \eqref{M31case(b)} can be controlled in the norm $\|\cdot\|_{i}^2$ similar to \eqref{control:barDi-1M111stlb0}--\eqref{control:barDi-1M112nd}, respectively. Therefore, 
  \begin{align}
      \|\barcalD_{i-1}\mathfrak{M}_{31}\|_i^2 \lesssim (\calE^N)^2. 
  \end{align}
  \end{itemize}
\paragraph{Analysis of $\barcalD_{i-1}\mathfrak{M}_{32}$.}
By decomposing
\begin{align}\label{M32 pre}
\mathfrak{M}_{32} &=-\frac{(1+\kk)^2}{\kk}w'\left(\calF^{-\kk-2}\left(\pz(U^0)^{-4}\right)\left(1+\frac{\Theta}{\ze}\right)^4\right)\Dz\Theta\\
&-\frac{(1+\kk)^2}{\kk}w'\left(\calF^{-\kk-2}(U^0)^{-4}\pz\left(1+\frac{\Theta}{\ze}\right)^4\right)\Dz\Theta,
\end{align}
we see that $\barcalD_{i-1}\mathfrak{M}_{32}$ can be controlled in $\|\cdot\|_{i}^2$ by an argument similar to the control of $\barcalD_{i-1}\mathfrak{M}_{11}$.
In particular, the key strategy is to write the second term as
\begin{equation}\label{M322nd}
\begin{aligned}
-4\frac{(1+\kk)^2}{\kk}\frac{w'}{\ze}\left(\calF^{-\kk-2}(U^0)^{-4}\left(1+\frac{\Theta}{\ze}\right)^3\ze\pz\left(1+\frac{\Theta}{\ze}\right)\right)\Dz\Theta\\
=-4\frac{(1+\kk)^2}{\kk}\frac{w'}{\ze}\left(\calF^{-\kk-2}(U^0)^{-4}\left(1+\frac{\Theta}{\ze}\right)^3\left(\Dz\Theta-3\frac{\Theta}{\ze}\right)\right)\Dz\Theta.
\end{aligned}
\end{equation}
This allows us to bound $\barcalD_{i-1}$-differentiate \eqref{M322nd} via \eqref{Pbarproduct} without encountering issues of derivative loss. In the end, we have
\begin{align}\label{est barDi-1 M32}
    \|\barcalD_{i-1}\mathfrak{M}_{32}\|_i^2\lesssim \calE^N. 
\end{align}
It is important to note that, unlike the estimate for $\barcalD_{i-1}\mathfrak{M}_{31}$, this estimate is only linear in $\calE^N$. This is due to the appearance of $\pz (U^0)^{-4}$ in the first term of \eqref{M32 pre}, which leads to contributions of terms of type \eqref{PbarU0type3} found in Proposition \ref{prop:buildingblockpowers}. 

\paragraph{Analysis of $\barcalD_{i-1}\mathfrak{M}_{33}$.} From \eqref{Pbarproduct} and \eqref{Pbarproduct2}, we see that 
\begin{align*}
    \barcalD_{i-1} \mathfrak{M}_{33} =&-(1+\kk)\sum_{\substack{A_{1,2}\in \barcalP_{l_{1,2}}\,\,A_{3,4}\in \calP_{l_{3,4}}\\\sum_{j=1}^4l_j=i-1}}c_{i-1}^{l_{1,2,3,4}}A_1w A_2\calF^{-\kk-2}A_3\pz\left((U^0)^{-4}\left(1+\frac{\Theta}{\ze}\right)^4\right)A_4\pz\Dz\Theta\\
    =&-(1+\kk)\sum_{\substack{A_{1,2}\in \barcalP_{l_{1,2}},\,\,A_3^{1,2}\in \barcalP_{i_{1,2}}\\A_{4}\in \calP_{l_{4}}\\i_1+i_2=l_3+1\\\sum_{j=1}^4l_j=i-1}}c^{l_{1,3,4}i_{1,2}}_{i-1,l_3}A_1w A_2\calF^{-\kk-2}A_3^1(U^0)^{-4}A_3^2\left(1+\frac{\Theta}{\ze}\right)^4A_4\pz\Dz\Theta. 
\end{align*}
This then follows from an analysis similar to that of $\barcalD_{i-1}\mathfrak{M}_{31}$, where we need two sub-cases: (a) when either $l_2=i$ or $l_4=i-1$, and (b) when $l_2<i$ and $l_4<i-1$. Additionally, in a manner analogous to the discussion following \eqref{est barDi-1 M32}, the estimate for $\barcalD_{i-1}\mathfrak{M}_{33}$ is also expected to be linear in $\calE^N$, i.e., 
\begin{align}
    \|\barcalD_{i-1}\mathfrak{M}_{33}\|_i^2\lesssim \calE^N. 
\end{align}
Summing up the estimates for $\barcalD_{i-1}\mathfrak{M}_{31}$, $\barcalD_{i-1}\mathfrak{M}_{32}$, and $\barcalD_{i-1}\mathfrak{M}_{33}$, we conclude
\begin{align}\label{est: barDi-1M3}
    \|\barcalD_{i-1}\mathfrak{M}_{3}\|_i^2\lesssim \calE^N +(\calE^N)^2. 
\end{align}

In conclusion, we infer from the estimates \eqref{est: barDi-1M1}, \eqref{est: barDi-1M2}, and \eqref{est: barDi-1M3} that 
\begin{align*}
    \|\barcalD_{i-1}\mathfrak{M}\|_i^2 \lesssim \calE^N + (\calE^N)^2.
\end{align*}
This implies Proposition \ref{prop:est M} after invoking the bootstrap assumption \eqref{bootstrap EN}.

\subsection{Estimates for $\calD_{i}\mathfrak{R}_1$ and $\calD_{i}\mathfrak{R}_2$}
We devote this subsection to estimating $\calD_i$-differentiated $\mathfrak{R}_1(\Theta)$ and $\mathfrak{R}_2(\Theta)$. 
\begin{prop}\label{prop:est R}
    Let $\mathfrak{R}_1(\Theta)$ and $\mathfrak{R}_2(\Theta)$ be given by \eqref{def:R1} and \eqref{def:R2}, respectively. Then
    \begin{equation}
        \sum_{k=1}^2\|\calD_{i}\mathfrak{R}_k\|_i^2 \lesssim (\calE^N)^2. 
    \end{equation}
\end{prop}
\subsubsection{Estimates for $\calD_{i}\mathfrak{R}_1$}
We begin by writing 
\begin{align*}
\mathfrak{R}_1(\Theta) = -2(1+\kk)w\calF^{-\kk-2}(U^0)^{-4}\left(1+\frac{\Theta}{\ze}\right)^3 Z,
\end{align*}
where $Z=\frac{1}{\ze}\left(\ze\pz\left(\frac{\Theta}{\ze}\right)\right)^2$. From the product rule \eqref{Pproduct}, $\calD_i\mathfrak{R}_1(\Theta)$ is a linear combination of the following terms:
\begin{equation}\label{R1 pre}
  A_1w A_2\calF^{-\kk-2}  A_3(U^0)^{-4} A_4 \left(1+\frac{\Theta}{\ze}\right)^3 BZ, \quad A_{1,2,3,4}\in \barcalP_{l_{1,2,3,4}},\quad B\in \calP_{l_B}. 
\end{equation}
We shall study \eqref{R1 pre} by dividing into the following two cases:
\begin{itemize}
    \item [(a)] When either $l_4$ or $l_B$ equal to $i$. We first study the sub-case where $l_B=i$. Invoking Lemma \ref{lem: DiZ}, \eqref{R1 pre} becomes a linear combination of the following terms:
    \begin{equation}\label{R1case(a)}
        w\calF^{\kk-2}(U^0)^{-4}\left(1+\frac{\Theta}{\ze}\right)^3(B_1\Theta)(B_2\Theta),\quad B_{1,2}\in \calP_{\ell_{1,2}},\quad \ell_1,\ell_2\leq i+1.  
    \end{equation}
    Thanks to the extra $w$-weight, $\|\eqref{R1case(a)}\|_i^2\lesssim (\calE^N)^2$ can be obtained similarly to the control of $\barcalD_{i-1}\mathfrak{M}_{31}$. The sub-case where $l_4=i$ is treated by a parallel argument. 
    \item [(b)] When both $l_4$ and $l_B<i$.
   From Proposition \ref{prop:buildingblockpowers} and \eqref{eq:DiZ}, we see that \eqref{R1 pre} is a linear combination of $\prod_j P_{l_j}H^j(\Theta)$ where $P_{l_j}\in \calP_{l_j}$ and $l_j\leq i$. We use Lemma \ref{lem:nonlinear} to conclude $\|\eqref{R1case(a)}\|_i^2\lesssim (\calE^N)^2$. 
\end{itemize}

\subsubsection{Estimates for $\calD_i\mathfrak{R}_2$}
Since $\left(1+\frac{\Theta}{\ze}\right)^2 = 1 + 2\left(\frac{\Theta}{\ze}\right)+\left(\frac{\Theta}{\ze}\right)^2$, it suffices to bound
$$
 w'(U^0)^{-4}\left(\frac{\Theta}{\ze}\right)^n, \quad n=2,3,4,5,
$$
in $\|\cdot\|_{i}$. By writing 
$$
 w'(U^0)^{-4}\left(\frac{\Theta}{\ze}\right)^n = \frac{w'}{\ze}(U^0)^{-4}\ze\left(\frac{\Theta}{\ze}\right)^n, \quad n=2,3,4,5,
$$
we see that $\calD_i\left(\frac{w'}{\ze}(U^0)^{-4}\ze\left(\frac{\Theta}{\ze}\right)^n\right)$ is a linear combination of the follow terms:
\begin{equation}\label{R2 pre}
    A_1\frac{w'}{\ze}A_2(U^0)^{-4}B\left(\ze\left(\frac{\Theta}{\ze}\right)^n\right),\quad A_{1,2}\in\barcalP_{l_{1,2}},\quad B\in \calP_{l_B}. 
\end{equation}
If $B=A_3\Dz$ where $A_3\in \barcalP_{l_B-1}$, we invoke \eqref{eq:Dz ze f} to obtain
\begin{align*}
  B \left(\ze\left(\frac{\Theta}{\ze}\right)^n\right)= A_3\left(3\left(\frac{\Theta}{\ze}\right)^n + n\left(\frac{\Theta}{\ze}\right)^{n-1}\left(\Dz\Theta-3\frac{\Theta}{\ze}\right)\right). 
\end{align*}
Moreover, if $B=A_3\frac{1}{\ze}$ where $A_3\in \barcalP_{l_B-1}$, 
\begin{align*}
B\left(\ze\left(\frac{\Theta}{\ze}\right)^n\right)=A_3\left(\frac{\Theta}{\ze}\right)^n. 
\end{align*}
Thus, we conclude that \eqref{R2 pre} is a linear combination of $\prod_j P_{l_j}H^j(\Theta)$ where $P_{l_j}\in \calP_{l_j}$ and $l_j\leq i$. We use Lemma \ref{lem:nonlinear} to conclude $\|\eqref{R2 pre}\|_i^2\lesssim (\calE^N)^2$. 

\subsection{Estimates for $\mfC_1$ and $\mfC_2$}
\begin{lem}\label{prop:est C12}
    Let $\mfC_k$, $k=1,2$ be given as \eqref{def:C1} and \eqref{def:C2}, respectively. Then
    \begin{align}\label{est: est C12}
        \sum_{k=1}^2\|\mfC_k\|_i^2 \lesssim \calE^N.
    \end{align}
\end{lem}
\subsubsection{Estimates for $\mfC_1$.}
We begin by writing 
$L_0(\Theta) = -w\pz\Dz\Theta - \frac{1+\kk}{\kk}w'\Dz \Theta$. From this, we decompose $\mfC_1=\mfC_{11}+\mfC_{12}$, where
\begin{align*}
    \mfC_{11} = -(1+\kk)\left[\barcalD_{i-1}, \calF^{-\kk-2}(U^0)^{-4}\left(1+\frac{\Theta}{\ze}\right)^4\right]\Dz(w\pz\Dz\Theta),\\
    \mfC_{12} = -\frac{(1+\kk)^2}{\kk}\left[\barcalD_{i-1}, \calF^{-\kk-2}(U^0)^{-4}\left(1+\frac{\Theta}{\ze}\right)^4\right]\Dz(w'\Dz\Theta). 
\end{align*}
Since for each $j\in \mathbb{N}$, 
\begin{align}\label{commutator:barcalDi}
    [\barcalD_j,g] f = j\pz g\barcalD_{j-1}f + \sum_{\substack{2\leq k\leq j\\ A\in\barcalP_k,\,\,B\in\barcalP_{i-k}}}C_i^{AB}AgBf,
\end{align}
$C_{11}$ reduces to 
\begin{align}\label{C11 pre}
    C_{11} =& -(1+\kk)(i-1) \pz\left(\calF^{-\kk-2}(U^0)^{-4}\left(1+\frac{\Theta}{\ze}\right)^4\right)\calD_{i-1}(w\pz\Dz\Theta)\nonumber\\
    -&(1+\kk)\sum_{\substack{2\leq k\leq i-1\\A\in\barcalP_k,\,\,B'\in\barcalP_{i-k-1}}}c_{i-1}^{AB}A\left(\calF^{-\kk-2}(U^0)^{-4}\left(1+\frac{\Theta}{\ze}\right)^4\right) B'\Dz(w\pz\Dz\Theta). 
\end{align}
Since
$$
\calD_{i-1}(w\pz\Dz\Theta) = \sum_{\substack{B_1\in \barcalP_{j_1},\,\,B_2\in\barcalP_{j_2}\\j_1+j_2=i-1}}c_{i-1}^{j_1j_2}B_1w B_2\pz\Dz\Theta, 
$$
the first line of \eqref{C11 pre} reduces to a sum of terms of the following forms:
\begin{align}
w\pz\left(\calF^{-\kk-2}(U^0)^{-4}\left(1+\frac{\Theta}{\ze}\right)^4\right)\calD_{i+1}\Theta, \label{C11 1st}\\
\pz\left(\calF^{-\kk-2}(U^0)^{-4}\left(1+\frac{\Theta}{\ze}\right)^4\right)\sum_{\substack{B_1\in \barcalP_{j_1},\,\,B_2\in\barcalP_{j_2}\\j_1+j_2=i-1,\,\,j_1\geq 1}}B_1w B_2\pz\Dz\Theta.\label{C11 2nd}
\end{align}
Note that \eqref{C11 1st} involves $\calD_{i+1}\Theta$ together with an additional $w$-weight. Moreover, \eqref{C11 2nd} is a linear combination of $\prod_j P_{l_j}H^j(\Theta)$ where $P_{l_j}\in \calP_{l_j}$ and $l_j\leq i$. This also applies to the second line of \eqref{C11 pre} as well as $C_{12}$. Thus, we obtain 
\begin{align}\label{est: C1}
\|\mfC_1\|_{i}^2 \lesssim \calE^N,
\end{align}
after invoking Lemma \ref{lem:nonlinear}. In particular, this estimate appears to be linear in $\calE^N$, as opposed to the one in Proposition \ref{prop:est M} which appears to be quadratic in $\calE^N$. This is caused by the appearance of the final term in \eqref{PbarU0} in Proposition \ref{prop:buildingblock}. 

\subsubsection{Estimates for $\mfC_2$.} Since $\xi_{i,j}$ are smooth functions, we see directly from Lemma \ref{lem:nonlinear} that $\|\mfC_2\|_{i}^2 \lesssim \calE^N$. 

\subsection{Decomposition of the $\calD_i$-Differentiated Pressure Gradient}
In summary of all the estimates presented in this section, we draw the following conclusions:
\begin{thm}\label{thm: pressure estimate}
Let 
\begin{align}\label{def:E7}
    E_{\text{pressure}} = \sum_{j=1}^2(\mfC_j+\calD_i\mathfrak{R}_j(\Theta)) + \barcalD_{i-1}\mathfrak{M}. 
\end{align}
Then, 
\begin{equation}\label{elliptic derivation}
\begin{aligned}
\calD_i\left((U^0)^{-4}\left(1+\frac\Theta\zeta\right)^2\frac{1}{w^{\frac1\kk}}\p_\zeta \left(w^{1+\frac1\kappa} (\calF^{-1-\kappa}-1)\right)\right)\\
= (1+\kk)\calF^{-\kk-2}(U^0)^{-4}\left(1+\frac{\Theta}{\ze}\right)^4\calL_i\calD_i\Theta + E_{\text{pressure}},
\end{aligned}
    \end{equation}
    where
    \begin{equation}
        E_{\text{pressure}} \lesssim \calE^N.
    \end{equation}
\end{thm}
\begin{proof}
    The identity \eqref{elliptic derivation} comes from \eqref{eq:elltipic term original form} and \eqref{eq:highorder momlagmain}. Additionally, from Proposition \ref{prop:est M}, Proposition \ref{prop:est R}, and the estimate \eqref{est: est C12} in Proposition \ref{prop:est C}, we obtain
    $$
    E_{\text{pressure}} \lesssim \calE^N(1+\calE^N) \lesssim \calE^N,
    $$
    where the second inequality comes from the $\calE^N \lesssim 1$ in the bootstrap regime. 
\end{proof}

\section{Energy Estimates II: High-Order Errors Arising from the Energy Identity}\label{sect: Energy Estimate II}
In this section, we will estimate $\mfC_3$ \eqref{def:C3} and $\mfI_1,\cdots, \mfI_5$ \eqref{def:I1toI5}, all of which arise from commuting $\pt$ through coefficients in the main equation \eqref{eq:highorder momlagmain}.
\begin{prop}\label{prop:est C}
    The high-order commutator $\mfC_3$ verifies the following estimate: 
    \begin{align}\label{est: est C3}
        \|\mfC_3\|_{i}^2 \lesssim \calS^N+ \calE^N\calS^N+\delta\lam^{-3\kk}\calE^N. 
    \end{align}
    Also, we have
    \begin{align}\label{est: mfI}
        \sum_{j=1}^5 \mfI_j \lesssim \lam^{-\frac32\kk}\calE^N. 
    \end{align}
\end{prop}
\subsection{Estimates for $\mfC_3$.}
Invoking \eqref{commutator [calDi g]f}, 
\begin{align}\label{C3_pre1}
    \begin{aligned}
        \mfC_3 = \sum_{\substack{A\in \calP_{l_A}, B\in\barcalP_{l_B}\\l_A+l_B=i,l_A\leq i-1}}c_k^{AB} A\left(\p_\tau^2 \Theta + \lamt \p_\tau \Theta + \delta \lam^{-3\kk}\Theta\right)B(w\calF^{-\kk}\calG^{\kk}),
    \end{aligned}
\end{align}
where, from the product rule \eqref{Pbarproduct} and since $\calG^{\kk}=\frac{1}{\kk}\barcalG^{-1}$, $B(w\calF^{-\kk}\calG^{\kk})$ is linear combination of the following terms:
\begin{align}\label{C3_pre2}
     \frac{1}{\kk}B_1w B_2\calF^{-\kk}B_3\barcalG^{-1}, \quad B_j\in \barcalP_{l_j},\quad \sum_{j=1}^3l_j = l_B. 
\end{align}
We are ready to bound $\|\mfC_3\|_i^2$, which will be discussed in the following cases:
\begin{itemize}
    \item [(a)] When $l_3=0$. We need to discuss the following two sub-cases, depending on the value of $l_2$:
    \begin{itemize}
        \item [(a1)] When $l_2=i$.  In light of Proposition \ref{prop:buildingblockpowers}, $B_2\calF^{-\kk}$ is a linear combination of the following terms: 
\begin{align*}
    \calF^{-\kk-k_2}\prod_{\substack{B_2^j\in \barcalP_{i_j}\\\sum_{j=1}^{m_2}i_j=l_2}}(B_2^j\Dz\Theta), \quad k_2=1,\cdots, l_2,\quad m_2 = k_2, 2k_2, 3k_2.
\end{align*}
Thus, from \eqref{C3_pre1} and \eqref{C3_pre2}, we see that $\mfC_3$ is a linear combination of terms of the following types:
\begin{align}\label{C3 case(a)1}
    w\calF^{-\kk-k_2}\barcalG^{-1}\left(\p_\tau^2 \Theta + \lamt \p_\tau \Theta + \delta \lam^{-3\kk}\Theta\right)\prod_{\substack{B_2^j\in \barcalP_{i_j}\\\sum_{j=1}^{m_2}i_j=l_2}}(B_2^j\Dz\Theta).
\end{align}
If there exists some $j\in \{1,\cdots, m_2\}$ such that $i_j = i$, 
then, parallel to the control of $\barcalD_{i-1}\mathfrak{M}_{31}$, we require the extra $w$-weight to control $B^j\Dz\Theta$ whenever $i_j=i$. 
Since $\calF, \barcalG\approx 1$ from the bootstrap assumptions, we have
\begin{align}
    \|\eqref{C3 case(a)1}\|_i^2 \lesssim \calE^N\calS^N+\delta\lam^{-3\kk}\calE^N, 
\end{align}
where we used Lemma \ref{lem:nonlinear2} with $i$ replaced by $i+1$.  On the other hand, if $i_j\leq i-1$ for all $j\in \{1,\cdots, m_2\}$, then we obtain the same bound as above using Lemma \ref{lem:nonlinear2}. 
\item [(a2)] When $l_2\leq i-1$. In this case, $\mfC_3$ is a linear combination of the following terms:
\begin{subequations}\label{C3 case(a)2}
    \begin{align}
        \calF^{-\kk}\barcalG^{-1}(B_1w) A\left(\p_\tau^2 \Theta + \lamt \p_\tau \Theta + \delta \lam^{-3\kk}\Theta\right), \label{C3 case(a)2l20}\\
    \calF^{-\kk-k_2}\barcalG^{-1}A\left(\p_\tau^2 \Theta + \lamt \p_\tau \Theta + \delta \lam^{-3\kk}\Theta\right)B_1w\prod_{\substack{B_2^j\in \barcalP_{i_j}\\\sum_{j=1}^{m_2}i_j=l_2}}(B_2^j\Dz\Theta). \label{C3 case(a)2l2greater0}
\end{align}
\end{subequations}
Here, \eqref{C3 case(a)2l20} comes from the case when $l_2=0$, while \eqref{C3 case(a)2l2greater0} comes from the case when $1\leq l_2\leq i-1$. 
Since $B_1w$ is a smooth function in view of Lemma \ref{lem:wderivative}, and since $l_A\leq i-1$, we invoke Lemma \ref{lem:nonlinear2} as well as the bootstrap assumptions $\calF,\barcalG\approx 1$ to obtain
\begin{align}
    \|\eqref{C3 case(a)2}\|_i^2 \lesssim  \calS^N+ \calE^N\calS^N+\delta\lam^{-3\kk}\calE^N.
\end{align}
    \end{itemize}

\item[(b)] When $l_3\geq 1$. Using the chain rule \eqref{eq:chainrule}, 
$$
B_3\barcalG^{-1} = \barcalG^{-1-k_3}\prod_{\substack{B_3^j\in \barcalP_{\ell_j}\\\sum_{j=1}^{k_3}\ell_j=l_3}}B_3^j \barcalG, \quad k_3=1,\cdots, l_3.
$$
Thus, $\mfC_3$ is a linear combination of the following terms:
\begin{subequations}\label{C3 case(b)}
    \begin{align}\label{C3 case(b) l20}
    \calF^{-\kk}\barcalG^{-1-k_3}A\left(\p_\tau^2 \Theta + \lamt \p_\tau \Theta + \delta \lam^{-3\kk}\Theta\right)B_1w \prod_{\substack{B_3^j\in \barcalP_{\ell_j}\\\sum_{j=1}^{k_3}\ell_j=l_3}}B_3^j \barcalG,
\end{align}
    \begin{align}\label{C3 case(b) l2greater0}
    \calF^{-\kk-k_2}\barcalG^{-1-k_3}A\left(\p_\tau^2 \Theta + \lamt \p_\tau \Theta + \delta \lam^{-3\kk}\Theta\right)B_1w\prod_{\substack{B_2^j\in \barcalP_{i_j}\\\sum_{j=1}^{m_2}i_j=l_2}}(B_2^j\Dz\Theta)\prod_{\substack{B_3^j\in \barcalP_{\ell_j}\\\sum_{j=1}^{k_3}\ell_j=l_3}}B_3^j \barcalG,
\end{align}
\end{subequations}
depending on whether $l_2=0$ or $l_2>0$. 
Also, invoking Proposition \ref{prop:buildingblockpowersG}, it can be seen that $B_3^j\barcalG$ is a linear combination of the following three types of terms:
\begin{subequations}\label{PbarGtype1 contribution}
\begin{equation}
        (U^0)^{\kk+2k_3'}\bar{P}\left((U^0)^{-\kk}(0)(1+\delta\lam^{-3\kk}{(0)}w\calF^{-\kk}(0) )\right)\prod_{\substack{B_3^{jk}\in\calP_{i_k}\\\sum_{k=1}^{m_3}i_k=\ell_{j2}}}(B_3^{jk}H_k(\Theta)),
        \end{equation}
        \begin{equation}\label{PbarGtype1 contribution2}
        (U^0)^{\kk+2k_3'}\bar{P}\left((U^0)^{-\kk}(0)(1+\delta\lam^{-3\kk}(0)w\calF^{-\kk}(0))\right)\prod_{\substack{B_3^{jk}\in\calP_{i_k}\\\sum_{k=1}^{k_3'}i_k=\ell_{j2}}}(B_3^{jk}(\lamt^2\zeta^2)).
\end{equation}
\end{subequations}
where $\bar{P}\in \barcalP_{\ell_{j1}}$, $\ell_{j1}+\ell_{j2}=\ell_j$, $k_3'=1,\cdots, \ell_{j2}$ and $m_3=k_3', 2k_3'$, as well as 
\begin{equation}\label{PbarGtype2 contribution}
        \delta\lam^{-3\kk}\calF^{-\kk-2k_3'}\cdot \bar{P}w \prod_{\substack{B_3^{jk}\in \calP_{i_k}\\\sum_{k=1}^{m_3'}i_k =\ell_{j2}}}(B_3^{jk}H_k(\Theta)), 
    \end{equation}
    where $m_3'=k_3',2k_3',3k_3'$. 
By plugging \eqref{PbarGtype1 contribution} and \eqref{PbarGtype2 contribution} into \eqref{C3 case(b)}, respectively, we can use Lemma \ref{lem:nonlinear2} to conclude
\begin{align}
    \|\eqref{C3 case(b)}\|_i^2 \lesssim  \calE^N(0)\left(\calS^N+ \calE^N\calS^N+\delta\lam^{-3\kk}\calE^N\right).
\end{align}
Lastly, since $\calE^N(0)\lesssim 1$, we sum up all estimates in both cases and get
\begin{align}
    \|\mfC_3\|_i^2 \lesssim \calS^N+ \calE^N\calS^N+\delta\lam^{-3\kk}\calE^N.
\end{align}
\end{itemize}

\subsection{Estimates for $\mfI_j$, $j=1,\cdots,5$}
These estimates are simple thanks to the bootstrap assumptions in Definition \ref{assump: bootstrap} and Lemma \ref{lem: induced bootstrap}. Invoking Lemma \ref{lem: induced bootstrap}, one sees that
$
\mfI_1 + \mfI_2 \lesssim \delta\lam^{-3\kk}\calE^N.
$
Similarly, 
$
\mfI_5 \lesssim (\delta\lam^{-3\kk})^{\frac{1}{2}}\calE^N.
$
Next, we study the estimates for $\mfI_3$. Decomposing $\mfI_3$ into $\mfI_{31}+\mfI_{32}$, where
\begin{align*}
    \mfI_{31} = \frac{3\kk}{2}\delta \lam^{-3\kk}\lamt \int_0^1w^{\frac1\kk+i}\ze^2 w\calF^{-\kk}\calG^{\kk}|\calD_i\Theta|^2, \\
    \mfI_{32} =-\frac12\delta\lam^{-3\kk}\int_0^1w^{\frac1\kk+i}\ze^2\p_{\tau}(w\calF^{-\kk}\calG^{\kk})|\calD_i\Theta|^2, 
\end{align*}
we invoke the bootstrap assumptions in Definition \ref{assump: bootstrap} to obtain $\mfI_{31}\lesssim \delta\lam^{-3\kk}\calE^N$. Additionally, we use \eqref{est: induced bootstrap F} and \eqref{est: induced bootstrap G} to obtain $\mfI_{32}\lesssim \delta\lam^{-3\kk}\calE^N$.

Finally, we study the estimates for $\mfI_4$. Invoking Lemma \ref{lem: induced bootstrap}, we infer
$$
\left\|\p_{\tau}\left(\calF^{-\kk-2}(U^0)^{-4}\left(1+\frac{\Theta}{\ze}\right)^4\right)\right\|_{L^{\infty}} \lesssim \lam^{-\frac32\kk}.
$$
Therefore, $\mfI_4\lesssim \lam^{-\frac32\kk}\calE^N$. 
In summary, $\sum_{j=1}^5\mfI_j \lesssim \left(\delta\lam^{-3\kk}+(\delta\lam^{-3\kk})^{\frac12}+\lam^{-\frac32\kk}\right)\calE^N$, which implies \eqref{est: mfI} as $\delta<1$.  

\section{Energy Estimates III: $E_j$, $j=1,\cdots, 6$}\label{sect: Energy Estimate III}
In this section, we estimate the terms $E_j$, $j = 1,\hdots, 6$. Such terms encode relativistic effects on the fluid velocity and are coupled to the evolution of $\barcalG$ in a highly nontrivial way. It turns out that the major difficulty comes from $E_1, E_3,$ and $E_4$, which has potential derivative loss. We resolve this issue by fundamentally utilizing the precise structures of $\calF$ and $\barcalG$, from which a symmetric structure in the highest order term emerges. We are thus able to avoid a potential derivative loss upon integrating by parts in time or space respectively (See Section \ref{sect: Energy Estimate V} for a detailed exposition).

We organize this section as follows: we will first treat $E_4$, $E_1$, and $E_3$, as those terms are not only the most dangerous ones, but also already contain the majority of the technical difficulties. We then estimate $E_2, E_5,$ and $E_6$, which have no risk of derivative loss but generate large errors at low frequencies.

\subsection{Estimates for $\calD_i E_4$}
In this section, we aim to prove the following Proposition:
\begin{prop}
    \label{prop:DiE4}
    The following decomposition holds for any $\tau \in [0,\tau_*]$:
    \begin{equation}
        \label{est:DiE4}
        \begin{split}
        \calD_i E_4 &= (\calC^{E_4}_1+\calC^{E_4}_2)w\ze\calD_{i+1}\pt\Theta +\calC^{E_4}_3w\calD_i\pt^2\Theta + R_{4}^i,
        \end{split}
    \end{equation}
    where
    \begin{equation}\label{E4coeff}
    \begin{split}
    \calC^{E_4}_1(\tau,\ze) &:=-\frac{1+\kk}{\kk}(U^0)^{-2}\calF^{-\kk-1}\barcalG^{-1}\frac{\pt\Theta + \lamt(\Theta +\ze)}{\ze}\left(1+\frac{\Theta}{\ze}\right)^2,\\
    \calC^{E_4}_2(\tau,\ze) &:= -\frac{1+\kk}{\kk}\delta \lam^{-3\kk} w(U^0)^{-2}\calF^{-2\kk-1}\barcalG^{-2}\frac{\pt\Theta + \lamt(\Theta +\ze)}{\ze}\left(1+\frac{\Theta}{\ze}\right)^2,\\
    \calC^{E_4}_3(\tau,\ze) &:= -\frac{1+\kk}{\kk}\lamt (U^0)^\kk \calF^{-\kk}\barcalG^{-2}\left((U^0)^{-\kk}(0)(1+ \delta\lambda^{-3\kk}(0)w\calF^{-\kk}(0)) \right)(\pt\Theta + \lamt(\Theta +\ze))^2.
    \end{split}
    \end{equation}
    Moreover,
    \begin{equation}
        \label{est:R4i}
        \|R_{4}^i\|_i^2 \lesssim (\delta \lam^{-3\kk})^2+\delta \lam^{-3\kk}\calE^N(\tau) + \calS^N(\tau).
    \end{equation}
\end{prop}
Note that we may write $E_4$ in the following way:
\begin{align*}
    E_4 &= \frac{1+\kk}{\kk} (U^0)^{-2}(\p_\tau \Theta + \lamt(\Theta + \zeta))(w\p_\tau (\calF^{-\kk}\calG^\kk))\\
    &= -\frac{1+\kk}{\kk} w(U^0)^{-2}(\p_\tau \Theta + \lamt(\Theta + \zeta)) \calF^{-\kk-1} \barcalG^{-1}\pt\calF\\
    &\quad -\frac{1+\kk}{\kk^2} w(U^0)^{-2}(\p_\tau \Theta + \lamt(\Theta + \zeta)) \calF^{-\kk}\barcalG^{-2}\pt\barcalG\\
    &= -E_{41} - E_{42},
\end{align*}
where we used the definition of $\barcalG$ in the above identities. We will study both $E_{41}$ and $E_{42}$ in detail below. It turns out that both of them can be decomposed into a leading order term and a remainder which is controlled by the energy.
\subsubsection{Estimates for $\calD_i E_{41}$}
In this subsection, we prove the following lemma concerning the estimate for $\calD_i E_{41}$:
\begin{lem}\label{lem:DiE41}
    The following decomposition holds for any $\tau \in [0,\tau_*]$:
    \begin{equation}
        \label{est:DiE41}
        \calD_i E_{41} = \frac{1+\kk}{\kk}w(U^0)^{-2}\calF^{-\kk-1}\barcalG^{-1}(\pt\Theta + \lamt(\Theta +\ze))\left(1+\frac{\Theta}{\ze}\right)^2\calD_{i+1}\pt\Theta + R_{41}^i,
    \end{equation}
    where $R_{41}^i$ obeys the following bound:
    \begin{equation}
        \label{est:R41i}
        \|R_{41}^i\|_i^2 \lesssim \delta \lam^{-3\kk}\calE^N(\tau).
    \end{equation}
\end{lem}
To prove the above lemma, it is helpful to further decompose $E_{41}$ into the following:
\begin{align*}
    \frac{\kk}{1+\kk}E_{41} &= w(U^0)^{-2}(\lamt\ze)\calF^{-\kk-1} \barcalG^{-1}\pt\calF\\
    &\quad + w(U^0)^{-2} (\pt\Theta + \lamt\Theta) \calF^{-\kk-1} \barcalG^{-1}\pt\calF\\
    &=: E_{411} + E_{412}.
\end{align*}
We will show in detail how to treat $E_{411}$, which contains the lowest order of nonlinearities. Then we comment on how to adapt the argument to $E_{412}$. 

\noindent\textbf{Estimates for $\calD_i E_{411}$.} Recalling the expression of $\calF$ \eqref{eq:calF-1}, we immediately have that
$$
\pt\calF = \pt\Dz\Theta + \pt\Dz\left(\frac{\Theta^2}{\ze}\right) + \frac13\pt\Dz\left(\frac{\Theta^3}{\ze^2}\right).
$$
Then after an application of the product rule \eqref{eq: leading order symmetry}, we obtain
\begin{equation}
    \label{E411aux1}
    \begin{split}
    \calD_i E_{411} &= \lamt w(U^0)^{-2}\calF^{-\kk-1}\barcalG^{-1}\cdot \calD_i\left[\ze\left(\Dz\pt\Theta + \Dz\pt\left(\frac{\Theta^2}{\ze}\right) + \frac13\Dz\pt\left(\frac{\Theta^3}{\ze^2}\right)\right)\right]\\
    &\quad + \sum_{\substack{A_{1,\hdots,4}\in\barcalP_{l_1,\hdots,l_4},\\
    A_5 \in \calP_{l_5},\\
    l_1 + \hdots + l_5 = i, l_5 \le i-1}} (A_1w)(A_2(U^0)^{-2})(A_3\calF^{-\kk-1})(A_4\barcalG^{-1})\\
    &\quad \times \left(A_5\left[\ze\left(\Dz\pt\Theta + \Dz\pt\left(\frac{\Theta^2}{\ze}\right) + \frac13\Dz\pt\left(\frac{\Theta^3}{\ze^2}\right)\right)\right]\right)
    \end{split}
\end{equation}
Applying \eqref{eq:dtFsqzeta} and \eqref{eq:dtFcuzeta} to the first term on the right-hand-side of \eqref{E411aux1}, we may further decompose it as
\begin{align*}
    &w(U^0)^{-2}\calF^{-\kk-1}\barcalG^{-1}\cdot \calD_i\left[\ze\left(\Dz\pt\Theta + \Dz\pt\left(\frac{\Theta^2}{\ze}\right) + \frac13\Dz\pt\left(\frac{\Theta^3}{\ze^2}\right)\right)\right]\\
    &= w(U^0)^{-2}\calF^{-\kk-1}\barcalG^{-1}\left(\ze+2\Theta + \frac{\Theta^2}{\ze}\right)\calD_{i+1}\pt\Theta + w(U^0)^{-2}\calF^{-\kk-1}\barcalG^{-1}\\
    &\quad \times \left(\sum_{\substack{B \in \calP_{l_1}, l_1 \le i}} c_i^{l_1}(B\pt\Theta)+ \sum_{\substack{A_{1,2} \in \calP_{l_1,l_2}\\l_1 + l_2 = i+1\\l_1 \le i}} c_i^{l_1,l_2}(B_1\pt\Theta)(B_2\Theta)+\sum_{\substack{B_{1,2,3} \in \calP_{l_1,l_2,l_3}\\l_1 + l_2 + l_3 = i+1\\l_1 \le i}} c_i^{l_1,l_2,l_3}(B_1\pt\Theta)(B_2\Theta)(B_3\Theta)\right)
\end{align*}
The above computation combining with \eqref{E411aux1} yields:
\begin{equation}
    \label{E411decomp}
    \begin{split}
        &\calD_i E_{411} = \lamt w(U^0)^{-2}\calF^{-\kk-1}\barcalG^{-1}\left(\ze+2\Theta + \frac{\Theta^2}{\ze}\right)\calD_{i+1}\pt\Theta + (U^0)^{-2}\calF^{-\kk-1}\barcalG^{-1}\\
    &\quad \times w\left(\sum_{\substack{B \in \calP_{l_1}, l_1 \le i}} c_i^{l_1}(B\pt\Theta)+ \sum_{\substack{A_{1,2} \in \calP_{l_1,l_2}\\l_1 + l_2 = i+1\\l_1 \le i}} c_i^{l_1,l_2}(B_1\pt\Theta)(B_2\Theta)+\sum_{\substack{B_{1,2,3} \in \calP_{l_1,l_2,l_3}\\l_1 + l_2 + l_3 = i+1\\l_1 \le i}} c_i^{l_1,l_2,l_3}(B_1\pt\Theta)(B_2\Theta)(B_3\Theta)\right)\\
    &\quad + \sum_{\substack{A_{1,\hdots,4}\in\barcalP_{l_1,\hdots,l_4},\\
    A_5 \in \calP_{l_5},\\
    l_1 + \hdots + l_5 = i, l_5 \le i-1}} (A_1w)(A_2(U^0)^{-2})(A_3\calF^{-\kk-1})(A_4\barcalG^{-1})\\
    &\quad \times \left(A_5\left[\ze\left(\Dz\pt\Theta + \Dz\pt\left(\frac{\Theta^2}{\ze}\right) + \frac13\Dz\pt\left(\frac{\Theta^3}{\ze^2}\right)\right)\right]\right)\\
    &= \lamt w(U^0)^{-2}\calF^{-\kk-1}\barcalG^{-1}\left(\ze+2\Theta + \frac{\Theta^2}{\ze}\right)\calD_{i+1}\pt\Theta + R_{411}^1 + R_{411}^2.
    \end{split}
\end{equation}
Despite the complexity of $R_{411}^1$ and $R_{411}^2$, we remark that they are readily controlled by our energy $\calE^N$. As we will demonstrate later, no more than $i$ derivatives fall on $\pt\Theta$ and no more than $i+1$ derivatives fall on $\Theta$. In particular, whenever an operator in class $\calP_{i+1}$ acts on $\Theta$, there is always an extra weight $w$ accompanying in that term. We present a detailed analysis of these two error terms as follows:
\begin{enumerate}
    \item \textit{Analysis of $R_{411}^1$.} By bootstrap assumptions and Lemma \ref{lem: induced bootstrap}, we readily see that
    $$
    (U^0)^{-2}\calF^{-\kk-1}\barcalG^{-1} \lesssim 1.
    $$
    Then in order to control $\|R_{411}^1\|_{i}^2$, it suffices to control the integrals assuming one of the following forms:
    \begin{align*}
        \int_0^1 w^{\frac1\kk + i}\ze^2 \cdot w^2|W\pt\Theta|^2 d\ze,\quad W \in \calP_{l},\;l\le i,\\
        \int_0^1 w^{\frac1\kk + i}\ze^2 \cdot w^2 |(W_1\pt\Theta)(W_2\Theta)|^2 d\ze,\quad W_{1,2} \in \calP_{l_1,l_2},\;l_1+l_2 = i+1,\; l_1 \le i,\\
        \int_0^1 w^{\frac1\kk + i}\ze^2 \cdot w^2 |(W_1\pt\Theta)(W_2\Theta)(W_3\Theta)|^2 d\ze,\quad W_{1,2,3} \in \calP_{l_1,l_2,l_3},\;l_1+l_2+l_3 = i+1,\; l_1 \le i.
    \end{align*}
    Since $w^2 \lesssim 1$, the first integral above is bounded by $\delta\lambda^{-3\kk}\calE^N$ due to Lemma \ref{lem:nonlinear} with $(J_1,J_2,i_1,i_2) = (1,0,i,0)$. Similarly, the rest of the two integrals can be bounded by
    $\delta\lambda^{-3\kk}(\calE^N)^2$ and $\delta\lambda^{-3\kk}(\calE^N)^3$ respectively, after invoking Lemma \ref{lem:nonlinear} with $(J_1,J_2,i_1,i_2) = (k,0,i+1,0)$, $k = 2,3$. Summarizing, we obtain
    \begin{equation}
        \label{est:R_411^1}
        \|R_{411}^1\|_{i}^2 \lesssim \delta\lambda^{-3\kk}\calE^N,
    \end{equation}
    where we also used the bootstrap assumption that $E^N < 1$.

    \item \textit{Analysis of $R_{411}^2$.} To study this term, we recall the following decompositions: by \eqref{PbarU0}, we can write $A_2(U^0)^{-2}$ as a linear combination of the following terms:
    $$
    \prod_{\substack{A_2^j \in \calP_{i_j},\\\sum_{j=1}^{m_2}i_j = l_2}}(A_2^j H^j(\Theta)),\quad m_2=1,2,\text{ and } \lamt^2 A_2\ze^2,
    $$
    where $H^j(\Theta) = \pt\Theta$ or $\Theta.$ According to Proposition \ref{prop:buildingblockpowers}, we can also write $A_3\calF^{-\kk-1}$ as a linear combination of the following terms:
    $$
    \calF^{-\kk-1-k_3}\prod_{\substack{A_3^j \in \barcalP_{i_j}\\\sum_{j=1}^{m_3}i_j = l_3}}(A_3^j\Dz\Theta),\quad k_3 = 0,\hdots,l_3,\quad  m_3 = k_3, 2k_3, 3k_3.
    $$
    Lastly, by following a similar argument to Proposition \ref{prop:dtU0}, we may write $$A_5\left[\ze\left(\Dz\pt\Theta + \Dz\pt\left(\frac{\Theta^2}{\ze}\right) + \frac13\Dz\pt\left(\frac{\Theta^3}{\ze^2}\right)\right)\right]$$ as a linear combination of the following terms:
    $$
    \prod_{\substack{A_5^j \in \calP_{i_j},\\\sum_{j=1}^{m_5}i_j = l_5}} (A_5^j H^j(\Theta)),\quad m_5 = 1,2,3,
    $$
    where $H^j(\Theta) = \pt\Theta$ or $\Theta$, and exactly one $H^j(\Theta)$ in each product is $\pt\Theta$. Summarizing the above computations and using chain rule \eqref{eq:chainrule} on $A_4 \barcalG^{-1}$, we finally write $R^2_{411}$ as a linear combination of the following terms:
    \begin{equation}
        \label{R4112type1}
        (A_1w)\calF^{-\kk-1-k_3}\barcalG^{-1-k_4}\prod_{\substack{\calA_j \in \calP_{i_j},\\ \sum_{j=1}^{m_2 + m_5}i_j = l_2 + l_5}}(\calA_j H^j(\Theta))\cdot \prod_{\substack{A_3^j \in \barcalP_{i_j}\\\sum_{j=1}^{m_3}i_j = l_3}}(A_3^j\Dz\Theta)\cdot \prod_{\substack{A_4^j \in \barcalP_{i_j},\\\sum_{j=1}^{k_4}i_j = l_4}}(A_4^j\barcalG),
    \end{equation}
    and 
    \begin{equation}
        \label{R4112type2}
        (\lamt^2 A_2\ze^2)(A_1w)\calF^{-\kk-1-k_3}\barcalG^{-1-k_4}\prod_{\substack{\calA_j \in \calP_{i_j},\\ \sum_{j=1}^{m_5}i_j = l_5}}(\calA_j H^j(\Theta))\cdot \prod_{\substack{A_3^j \in \barcalP_{i_j}\\\sum_{j=1}^{m_3}i_j = l_3}}(A_3^j\Dz\Theta)\cdot \prod_{\substack{A_4^j \in \barcalP_{i_j},\\\sum_{j=1}^{k_4}i_j = l_4}}(A_4^j\barcalG),
    \end{equation}
    where $k_3 = 1,\hdots, l_3$, and $k_4 = 1,\hdots, l_4$ when $l_3,l_4 \ge 1$.

    We first explain how to bound \eqref{R4112type1} in $\|\cdot\|_i$ norm, which can be split into the following cases:
    \begin{enumerate}
        \item $l_4 = 0.$ In this case, \eqref{R4112type1} can be reduced to
        \begin{equation}
            \label{R4112type1l4=0}
            (A_1w)\calF^{-\kk-1-k_3}\barcalG^{-1}\prod_{\substack{\calA_j \in \calP_{i_j},\\ \sum_{j=1}^{m_2 + m_5}i_j = l_2 + l_5}}(\calA_j H^j(\Theta))\cdot \prod_{\substack{A_3^j \in \barcalP_{i_j}\\\sum_{j=1}^{m_3}i_j = l_3}}(A_3^j\Dz\Theta).
        \end{equation}
        If $l_1 > 0$, we note that $\sum_{j=2}^5 l_j \le i-1$. Then after using $A_1w$ being a smooth function, and the bootstrap assumption that $\calF,\barcalG \approx 1$, we can bound:
        \begin{align*}
            \|\eqref{R4112type1l4=0}\|_i^2 &\lesssim \int_0^1 \ze^2 w^{\frac1\kk + i}\prod_{\substack{\calA_j \in \calP_{i_j},\\ \sum_{j=1}^{m_2 + m_5}i_j = l_2 + l_5}}|\calA_j H^j(\Theta)|^2\cdot \prod_{\substack{A_3^j \in \barcalP_{i_j}\\\sum_{j=1}^{m_3}i_j = l_3}}|A_3^j\Dz\Theta|^2 d\ze\\
            &\lesssim \delta \lam^{-3\kk} (\calE^N)^2,
        \end{align*}
        where we applied Lemma \ref{lem:nonlinear} with $(J_1,J_2, i_1, i_2) = (m_2 + m_5, m_3, l_2+l_5, l_3 + 1)$, as well as the facts that $m_2 + m_5 \ge 2$, $l_2 + l_3 + l_5 + 1 \le i$, and that at least one of $H^j(\Theta)$ is $\pt\Theta$.

        If $l_1 = 0,$ after using the bootstrap assumption that $\calF,\barcalG \approx 1$, we have
        \begin{align*}
            \|\eqref{R4112type1l4=0}\|_i^2 &\lesssim \int_0^1 \ze^2 w^{\frac1\kk + i + 2}\prod_{\substack{\calA_j \in \calP_{i_j},\\ \sum_{j=1}^{m_2 + m_5}i_j = l_2 + l_5}}|\calA_j H^j(\Theta)|^2\cdot \prod_{\substack{A_3^j \in \barcalP_{i_j}\\\sum_{j=1}^{m_3}i_j = l_3}}|A_3^j\Dz\Theta|^2 d\ze\\
            &\lesssim \delta \lam^{-3\kk} (\calE^N)^2,
        \end{align*}
        where in the final inequality we used Lemma \ref{lem:nonlinear} in the $i+1$ case, and $(J_1, J_2, i_1, i_2) = (m_2 + m_5, m_3, l_2 + l_5, l_3 + 1)$.

        \item $l_4 > 0$. In this case, we restrict our discussion to the case of $k_4 = 1$, which already encompasses all essential difficulties. The cases of $k_4 > 1$ will follow from a similar argument after a systematic use of Lemma \ref{lem:nonlinear}. For $k_4 = 1$, we can reduce \eqref{R4112type1} to 
        \begin{equation}
            \label{R4112type1l4>0}
            (A_1w)\calF^{-\kk-1-k_3}\barcalG^{-2}\prod_{\substack{\calA_j \in \calP_{i_j},\\ \sum_{j=1}^{m_2 + m_5}i_j = l_2 + l_5}}(\calA_j H^j(\Theta))\cdot \prod_{\substack{A_3^j \in \barcalP_{i_j}\\\sum_{j=1}^{m_3}i_j = l_3}}(A_3^j\Dz\Theta)\cdot \bar{A}_4 \barcalG,
        \end{equation}
        where $\bar A_4 \in \barcalP_{l_4}$, and $\bar A_4 \barcalG$ can be further decomposed into terms listed in \eqref{PbarGtype1} and \eqref{PbarGtype2}. We explain how to bound them on a term-by-term basis.

        \textbf{Contribution from \eqref{PbarGtype2}.} In this scenario, \eqref{R4112type1l4>0} is reduced to
        \begin{equation}\label{R4112type1l4>0aux1}
        \begin{split}
        &\delta \lam^{-3\kk}\calF^{-2\kk - 1 - k_3 - 2k_4}\barcalG^{-2}(A_1w)(\bar Pw) \prod_{\substack{\calA_j \in \calP_{i_j},\\ \sum_{j=1}^{m_2 + m_5}i_j = l_2 + l_5}}(\calA_j H^j(\Theta))\\
        &\cdot \prod_{\substack{A_3^j \in \barcalP_{i_j}\\\sum_{j=1}^{m_3}i_j = l_3}}(A_3^j\Dz\Theta)\cdot\prod_{\substack{W_j \in \barcalP_{i_j}\\\sum_{j=1}^{m_4}i_j = n_2}}(W_j\Dz\Theta)\\
        &= \delta \lam^{-3\kk}\calF^{-2\kk - 1 - k_3 - 2k_4}\barcalG^{-2}(A_1w)(\bar Pw)\prod_{\substack{\calA_j \in \calP_{i_j},\\ \sum_{j=1}^{m_2 + m_5}i_j = l_2 + l_5}}(\calA_j H^j(\Theta))\cdot\prod_{\substack{\mcalB_j \in \barcalP_{i_j}\\\sum_{j=1}^{m_3 +m_4}i_j = l_3 + n_2}}(\mcalB_j\Dz\Theta)
        \end{split}
        \end{equation}
        where $\bar P \in \barcalP_{n_1}$, $n_1 + n_2 = l_4$, $k_4 = 1,\hdots, n_2$, and $m_4 = k_4, 2k_4,$ or $3k_4$. A very similar argument to the estimate of \eqref{R4112type1l4=0} would yield the following bound:
        \begin{equation}
            \label{R4112type1l4>0aux2}
            \|\eqref{R4112type1l4>0aux1}\|_i^2 \lesssim (\delta \lam^{-3\kk})^3 (\calE^N)^2.
        \end{equation}

        \textbf{Contribution from \eqref{PbarGtype1}.} Note that we can always write a term of type \eqref{PbarGtype1} as
        $$
        g(\ze)(U^0)^{\kk + 2k}\bar P\left((U^0)^{-\kk}(0)(1+\delta \lam^{-3\kk}(0)w\calF^{-\kk}(0)\right)\prod_{\substack{V_j \in \calP_{i_j},\\ \sum_{j=1}^M i_j \le n_2}}(V_j H^j(\Theta)), k\ge 1.
        $$
        where $\bar P \in \barcalP_{n_1}$, $n_1 + n_2 = l_4$, $M \ge 0$, and $g(\ze)$ is a smooth function.
        
        Hence from a routine use of chain rule and product rule, for $\bar P \in \barcalP_{n_1}$, we may write
        $$
        \bar P\left((U^0)^{-\kk}(0)(1+\delta \lam^{-3\kk}(0)w\calF^{-\kk}(0)\right)
        $$
        as a linear combination of terms in form of:
        \begin{align*}
            g(\ze) (U^0)^{-\kk + 2\bar k_1}(0)\prod_{\substack{W_j \in \calP_{i_j},\\\sum_{j=1}^{\bar m_1} i_j = n_1^1}}(W_j H^j(\Theta)(0)) \cdot \left(1+\delta \lam^{-3\kk}(0)\calF^{-\kk - \bar k_2}(\tilde Pw)\prod_{\substack{V_j \in \barcalP_{i_j},\\\sum_{j=1}^{\bar m_2}i_j = n_1^3}}(V_j\Dz\Theta(0))\right),
        \end{align*}
        where $\bar m_1, \bar m_2 \ge 0$, $\bar k_1, \bar k_2 \ge 0$, $\tilde P \in \barcalP_{n_1^2}$, and $n_1^1 + n_1^2 + n_1^3 = n_1$. Moreover, $g(\zeta)$ is a smooth function. Therefore, after invoking bootstrap assumptions and bounds on initial data, it suffices to bound:
        \begin{align*}
            \int_0^1 w^{\frac1\kk + i} &\ze^2 |A_1w|^2|\tilde Pw|^2\prod_{\substack{W_j \in \calP_{i_j},\\\sum_{j=1}^{\bar m_1} i_j = n_1^1}}|W_j H^j(\Theta)(0)|^2\cdot \prod_{\substack{V_j \in \barcalP_{i_j},\\\sum_{j=1}^{\bar m_2}i_j = n_1^3}}|V_j\Dz\Theta(0)|^2\\
            &\times\prod_{\substack{\calA_j \in \calP_{i_j},\\ \sum_{j=1}^{m_2 + m_5 + M}i_j \le l_2 + l_5 + n_2}}|\calA_j H^j(\Theta)|^2\cdot \prod_{\substack{A_3^j \in \barcalP_{i_j}\\\sum_{j=1}^{m_3}i_j = l_3}}|A_3^j\Dz\Theta|^2d\ze\\
            &\lesssim \delta \lam^{-3\kk} (\calE^N(\tau))^2\left(1+\calE^N(0)\right),
        \end{align*}
        where we have used Lemma \ref{lem:nonlinear} together with the facts that $m_2 + m_5 \ge 2$ and $\calE^N(0) \lesssim \eps \ll 1$ in the final inequality. Summarizing the above estimates, we arrive at the following bound:
        \begin{equation}\label{est:R4112type1}
            \|\eqref{R4112type1}\|_i^2 \lesssim \delta \lam^{-3\kk} (\calE^N(\tau))^2\left(1+\calE^N(0)\right).
        \end{equation}
    \end{enumerate}
    Through an almost identical argument to the analysis of \eqref{R4112type1}, we are also able to obtain the corresponding estimate for \eqref{R4112type2}, namely:
    \begin{equation}\label{est:R4112type2}
            \|\eqref{R4112type2}\|_i^2 \lesssim \delta \lam^{-3\kk} \calE^N(\tau)\left(1+\calE^N(0)\right).
    \end{equation}
    Note that the bound is only linear in $\calE^N(\tau)$ due to the source term contribution $\lamt^2 A_2\ze^2$ which originates from $A_2 (U^0)^{-2}$. Combining \eqref{est:R4112type1} and \eqref{est:R4112type2}, we deduce the following bound for $R_{411}^2$:
    \begin{equation}
        \label{est:R4112}
        \|R_{411}^2\|_i^2 \lesssim \delta \lam^{-3\kk} \calE^N(\tau)\left(1+\calE^N(0)\right).
    \end{equation}
\end{enumerate}

\noindent\textbf{Estimates for $\calD_i E_{412}$.} We only sketch how to estimate this term, in that the only difference in between $E_{412}$ and $E_{411}$ is that $E_{412}$ contains $\pt\Theta + \lamt \Theta$ in place of $\lamt \ze$. Hence, using \eqref{eq:dtFsqpttheta}, \eqref{eq:dtFsqtheta} in place of \eqref{eq:dtFsqzeta}, and \eqref{eq:dtFcupttheta}, \eqref{eq:dtFcutheta} in place of \eqref{eq:dtFcuzeta}, we may write
\begin{equation}
    \label{E412decomp}
    \begin{split}
    \calD_i E_{412} &= \lamt w(U^0)^{-2}\calF^{-\kk-1}\barcalG^{-1}(\pt\Theta + \lamt\Theta)\left(1+ 2\frac{\Theta}{\ze} + \frac{\Theta^2}{\ze^2}\right)\calD_{i+1}\pt\Theta \\
    &\quad + R_{412}^1 + R_{412}^2,
    \end{split}
\end{equation}
where 
\begin{equation}
    \label{R4121}
    \begin{split}
    R_{412}^1 &= (U^0)^{-2}\calF^{-\kk - 1}\barcalG^{-1}\\
    &\quad \times \Bigg(\sum_{\substack{B_{1,2} \in \calP_{l_1,l_2}\\l_1 + l_2 = i+1\\l_1 \le i}} c_i^{l_1,l_2}(B_1\pt\Theta)(B_2H(\Theta))+\sum_{\substack{B_{1,2,3} \in \calP_{l_1,l_2,l_3}\\l_1 + l_2 + l_3 = i+1\\l_1 \le i}} c_i^{l_1,l_2,l_3}(B_1\pt\Theta)(B_2H(\Theta))(B_3H(\Theta))\\
    &\quad + \sum_{\substack{B_{1,\hdots,4} \in \calP_{l_1,\hdots,l_4}\\l_1 + \hdots + l_4 = i+1\\l_1 \le i}} c_i^{l_1,l_2,l_3,l_4}(B_1\pt\Theta)(B_2H(\Theta))(B_3H(\Theta))(B_4H(\Theta))\Bigg),
    \end{split}
\end{equation}
and
\begin{equation}
    \label{R4122}
    \begin{split}
    R_{412}^2 &= \sum_{\substack{A_{1,\hdots,4}\in\barcalP_{l_1,\hdots,l_4},\\
    A_5 \in \calP_{l_5},\\
    l_1 + \hdots + l_5 = i, l_5 \le i-1}} (A_1w)(A_2(U^0)^{-2})(A_3\calF^{-\kk-1})(A_4\barcalG^{-1})\\
    &\quad \times \left(A_5\left[H(\Theta)\left(\Dz\pt\Theta + \Dz\pt\left(\frac{\Theta^2}{\ze}\right) + \frac13\Dz\pt\left(\frac{\Theta^3}{\ze^2}\right)\right)\right]\right),
    \end{split}
\end{equation}
where $H(\Theta) = \Theta$ or $\pt\Theta$ in the above expressions. Estimating \eqref{R4121} (resp. \eqref{R4122}) in a similar way to $R_{411}^1$ (resp. $R_{411}^2$), we are able to obtain the following bound:
\begin{equation}
    \label{R4121+R4122}
    \|R_{411}^1\|_{i}^2 + \|R_{411}^2\|_{i}^2 \lesssim \delta \lam^{-3\kk}\calE^N(\tau)(1+\calE^N(0)).
\end{equation}
We have thus proved Lemma \ref{lem:DiE41} after combining \eqref{E411decomp}, \eqref{E412decomp}, \eqref{est:R_411^1}, \eqref{est:R4112}, and \eqref{R4121+R4122}, as well as using the bootstrap assumption that $E_N < 1$.

\subsubsection{Estimates for $\calD_i E_{42}$}
In this subsection, we prove the following lemma concerning the estimate for $\calD_i E_{42}$:
\begin{lem}\label{lem:DiE42}
    Assuming bootstrap assumptions, we may write
    \begin{equation}
        \label{est:DiE42}
        \begin{split}
        \calD_i E_{42} &= \calC^{E_{42}}_1\calD_{i+1}\pt\Theta +\calC^{E_{42}}_2\calD_i\pt^2\Theta + R_{42}^i,
        \end{split}
    \end{equation}
    where
    \begin{align*}
    \calC^{E_{42}}_1(\tau,\ze) &:= \frac{1+\kk}{\kk}\delta \lam^{-3\kk} w^2(U^0)^{-2}\calF^{-2\kk-1}\barcalG^{-2}(\pt\Theta + \lamt(\Theta +\ze))\left(1+\frac{\Theta}{\ze}\right)^2,\\
    \calC^{E_{42}}_2(\tau,\ze) &:= \frac{1+\kk}{\kk}\lamt w(U^0)^\kk \calF^{-\kk}\barcalG^{-2}\left((U^0)^{-\kk}(0)(1+ \delta\lambda^{-3\kk}(0)w\calF^{-\kk}(0)) \right)(\pt\Theta + \lamt(\Theta +\ze))^2.
    \end{align*}
    Moreover,
    \begin{equation}
        \label{est:R42i}
        \|R_{42}^i\|_i^2 \lesssim (\delta \lam^{-3\kk})^2+\delta \lam^{-3\kk}\calE^N(\tau) + S^N(\tau).
    \end{equation}
\end{lem}
Using \eqref{eq:dtbarcalG}, we decompose $E_{42}$ into the following:
\begin{align*}
    \frac{\kk}{1+\kk}E_{42} &= -\frac{1}{2}w(U^0)^\kk \calF^{-\kk}\barcalG^{-2}(\pt\Theta + \lamt(\ze + \Theta))\cdot \left((U^0)^{-\kk}(0)(1+ \delta\lambda^{-3\kk}(0)w\calF^{-\kk}(0)) \right) \cdot \pt(U^0)^{-2}\\
    &\quad + \delta \lam^{-3\kk} w^2(U^0)^{-2}\calF^{-2\kk - 1}\barcalG^{-2}(\pt\Theta + \lamt(\ze + \Theta)) \pt\calF \\
    &\quad + 3\delta \lam^{-3\kk}\lamt w^2(U^0)^{-2} \calF^{-2\kk}\barcalG^{-2}(\pt\Theta + \lamt(\ze + \Theta))\\
    &=: E_{421} + E_{422} + E_{423}.
\end{align*}
We arrange this subsection in the following way: we first study $E_{422}$ and $E_{423}$, which can be estimated by deploying a similar strategy to the one used for estimating $\calD_i E_{41}$. We then focus on $E_{421}$, which can be further decomposed into a leading order term and controllable errors.\\

\noindent\textbf{Estimates for $\calD_i E_{422}$.} Observe that $E_{422}$ has an identical structure to $E_{41}$. Namely, both of the terms can be written in the form of $w^a(U^0)^{-b}\calF^{-c}\barcalG^{-d}(\pt\Theta + \lamt(\ze + \Theta)) \pt\calF$ with some positive exponents $a, b, c, d$. In view of the bootstrap assumptions and Lemma \ref{lem: induced bootstrap} that $U_0,\calF, \kk\barcalG \approx 1$ as well as the number of weights, the exact values of $a, b, c, d$ do not affect our analysis as long as $a \ge 1$. Hence, $\calD_i E_{422}$ can be analyzed in a similar fashion to that concerning $\calD_i E_{41}$. In particular, we have:
\begin{equation}
    \label{E422}
    \calD_i E_{422} = \delta \lam^{-3\kk} w^2(U^0)^{-2}\calF^{-2\kk-1}\barcalG^{-2}(\pt\Theta + \lamt(\Theta +\ze))\left(1+\frac{\Theta}{\ze}\right)^2\calD_{i+1}\pt\Theta + R_{422}^i,
\end{equation}
where
\begin{equation}
    \label{est:R422}
    \|R_{422}^i\|_{i}^2 \lesssim (\delta \lam^{-3\kk})^3\calE^N(\tau)(1+\calE^N(0)).
\end{equation}

\noindent\textbf{Estimates for $\calD_i E_{423}$.} We further decompose $E_{423}$ into
\begin{align*}
    E_{423} &= 3\delta \lam^{-3\kk}\lamt w^2(U^0)^{-2} \calF^{-2\kk}\barcalG^{-2}(\lamt \zeta)\\
    &\quad +3\delta \lam^{-3\kk}\lamt w^2(U^0)^{-2} \calF^{-2\kk}\barcalG^{-2}(\pt\Theta + \lamt\Theta)\\
    &=: E_{423}^1 + E_{423}^2.
\end{align*}
We focus on the study of $\calD_i E_{423}^1$, which has the least amount of nonlinearities while also encompasses all essential difficulties. Invoking the product rule, chain rule, and Proposition \ref{prop:buildingblockpowers}, we argue in a similar way to the deduction of \eqref{R4112type1} and \eqref{R4112type2} to show that $\calD_i E_{423}^1$ is a linear combination of the following two types of terms:
\begin{equation}
        \label{E4231type1}
        \delta\lam^{-3\kk}(A_1(w^2 \ze))\calF^{-2\kk-k_3}\barcalG^{-2-k_4}\prod_{\substack{\calA_j \in \calP_{i_j},\\ \sum_{j=1}^{m_2}i_j = l_2}}(\calA_j H^j(\Theta))\cdot \prod_{\substack{A_3^j \in \barcalP_{i_j}\\\sum_{j=1}^{m_3}i_j = l_3}}(A_3^j\Dz\Theta)\cdot \prod_{\substack{A_4^j \in \barcalP_{i_j},\\\sum_{j=1}^{k_4}i_j = l_4}}(A_4^j\barcalG),
    \end{equation}
    for $m_2 = 1,2$, and 
    \begin{equation}
        \label{E4231type2}
        \delta\lam^{-3\kk}(\lamt^2 A_2\ze^2)(A_1(w^2\ze))\calF^{-2\kk-k_3}\barcalG^{-2-k_4}\prod_{\substack{A_3^j \in \barcalP_{i_j}\\\sum_{j=1}^{m_3}i_j = l_3}}(A_3^j\Dz\Theta)\cdot \prod_{\substack{A_4^j \in \barcalP_{i_j},\\\sum_{j=1}^{k_4}i_j = l_4}}(A_4^j\barcalG),
    \end{equation}
    where $A_1 \in \calP_{l_1}$, $A_2 \in \calP_{l_2}$, $\sum_{j = 1}^4 l_j = i$, $k_3 = 1,\hdots, l_3$, and $k_4 = 1,\hdots, l_4$ when $l_3,l_4 \ge 1$. Finally, $H^j(\Theta) = \pt\Theta$ or $\Theta$. The term \eqref{E4231type1} can be treated in the same way as \eqref{R4112type1}. In particular,
    \begin{equation}
        \label{est:E4231type1}
        \|\eqref{E4231type1}\|_i^2 \lesssim (\delta \lam^{-3\kk})^2 \calE^N(\tau)(1+\calE^N(0)).
    \end{equation}
    The term \eqref{E4231type2} can be treated in a similar way, except for the case where $l_2, l_4 \le 2$ and $l_3 = 0$. In this case, while the derivative estimates are of lower order, the source term involved in $A_2 \ze^2$ and $A_4^j \barcalG$ does not vanish. Therefore, the best bound which one can obtain is:
    \begin{equation}
        \label{est:E4231type2}
        \|\eqref{E4231type2}\|_i^2 \lesssim (\delta \lam^{-3\kk})^2 (1+\calE^N(0)).
    \end{equation}
    The estimate for $\calD_i E_{423}^2$ follows from a similar argument. Because it involves perturbation variables $\pt\Theta$ and $\Theta$, no source errors appear and we may obtain the following bound:
    \begin{equation}
        \label{est:E4232}
        \|\calD_i E_{423}^2 \|_i^2 \lesssim (\delta \lam^{-3\kk})^2 \calE^N(\tau)(1+\calE^N(0)).
    \end{equation}
    Combining \eqref{est:E4231type1}, \eqref{est:E4231type2}, and \eqref{est:E4232}, we conclude that
    \begin{equation}
        \label{est:E423}
        \|\calD_i E_{423} \|_i^2 \lesssim (\delta \lam^{-3\kk})^2(1+\calE^N(0)) +(\delta \lam^{-3\kk})^2 \calE^N(\tau)(1+\calE^N(0)).
    \end{equation}

\noindent\textbf{Estimates for $\calD_i E_{421}$.} This term possesses a different structure from all other terms which we have discussed earlier in this section. Therefore, we study $\calD_i E_{421}$ in detail. As before, we split $E_{421}$ into two parts:
\begin{align*}
    E_{421} &= -\frac{1}{2}w(U^0)^\kk \calF^{-\kk}\barcalG^{-2}(\lamt \ze)\cdot \left((U^0)^{-\kk}(0)(1+ \delta\lambda^{-3\kk}(0)w\calF^{-\kk}(0)) \right) \cdot \pt(U^0)^{-2}\\
    &\quad -\frac{1}{2}w(U^0)^\kk \calF^{-\kk}\barcalG^{-2}(\pt\Theta +\lamt \Theta)\cdot \left((U^0)^{-\kk}(0)(1+ \delta\lambda^{-3\kk}(0)w\calF^{-\kk}(0)) \right) \cdot \pt(U^0)^{-2}\\
    &=: E_{421}^1 + E_{421}^2.
\end{align*}
Again, we focus on $E_{421}^1$ since it contains the least amount nonlinearities while captures essential difficulties in the highest order estimates.
\begin{enumerate}
    \item Estimates for $\calD_i E_{421}^1$. Applying product rule \eqref{eq: leading order symmetry} and then \eqref{eq:dtU0zetahighest}, we obtain the following:
    \begin{equation}
        \label{DiE4211decomp}
        \begin{split}
        \calD_i E_{421}^1 &= -\frac{\lamt}{2}w(U^0)^\kk \calF^{-\kk}\barcalG^{-2}\left((U^0)^{-\kk}(0)(1+ \delta\lambda^{-3\kk}(0)w\calF^{-\kk}(0)) \right)\cdot \calD_i(\ze \pt(U^0)^{-2})\\
        &\quad -\frac{\lamt}{2}\sum_{\substack{A_{1,\hdots,5} \in \barcalP_{l_1,\hdots, l_5},\\A_6 \in \calP_{l_6}, l_6 \le i-1,\\ \sum_{j=1}^6 l_j = i}}c^{l_1,\hdots,l_6}_i(A_1w)(A_2 (U^0)^{\kk})(A_3\calF^{-\kk})(A_4\barcalG^{-2})\\
        &\quad \times \left(A_5\left((U^0)^{-\kk}(0)(1+ \delta\lambda^{-3\kk}(0)w\calF^{-\kk}(0)) \right)\right)(A_6(\ze\pt (U^0)^{-2}))\\
        &= \lamt w(U^0)^\kk \calF^{-\kk}\barcalG^{-2}\left((U^0)^{-\kk}(0)(1+ \delta\lambda^{-3\kk}(0)w\calF^{-\kk}(0)) \right)\cdot \ze(\pt\Theta + \lamt(\Theta +\ze))\calD_i\pt^2\Theta\\
        &\quad + R_{421,1}^1 + R_{421,2}^1,
        \end{split}
    \end{equation}
    where
    \begin{equation}
        \label{R4211}
        \begin{split}
            R_{421,1}^1 &:= -\frac{\lamt}{2}w(U^0)^\kk \calF^{-\kk}\barcalG^{-2}\left((U^0)^{-\kk}(0)(1+ \delta\lambda^{-3\kk}(0)w\calF^{-\kk}(0)) \right)\\
            &\times \bigg[\sum_{\substack{B_{1,2,3}\in\calP_{l_1,l_2,l_3},\\l_1+l_2 + l_3 = i,\\l_2 \le i-1}} c^{l_1,l_2,l_3}_{i,1}(B_1\ze)(B_2\pt^2\Theta)\left[\lamt(B_3\ze) + B_3(\pt\Theta + \lamt\Theta)\right]\\
    &+ \lamt\pt\lamt (P_i\ze^3) + \lamt\pt\lamt \sum_{\substack{B_1\in\calP_{l_1}, A_2 \in \barcalP_{l_2}\\l_1+l_2 = i}} c^{l_1,l_2}_{i,3}(B_1\Theta)(B_2\ze^2)\\
    &+ (\lamt^2 + \pt\lamt)\sum_{\substack{B_1\in\calP_{l_1}, B_2 \in \barcalP_{l_2}\\l_1+l_2 = i}} c^{l_1,l_2}_{i,4}(B_1\pt\Theta)(B_2\ze^2)\\
    &+ \sum_{\substack{B_{1,2,3}\in\calP_{l_1,l_2,l_3},\\l_1+l_2 + l_3 = i}} c^{l_1,l_2,l_3}_{i,5}(B_1\ze)(B_2(\pt\Theta + \lamt\Theta))(B_3(\pt\lamt \Theta + \lamt \pt\Theta))\bigg],
        \end{split}
    \end{equation}
    and
    \begin{equation}
        \label{R4212}
        \begin{split}
        R_{421,2}^1 &:= -\frac{\lamt}{2}\sum_{\substack{A_{1,\hdots,5} \in \barcalP_{l_1,\hdots, l_5},\\A_6 \in \calP_{l_6}, l_6 \le i-1,\\ \sum_{j=1}^6 l_j = i}}c^{l_1,\hdots,l_6}_i(A_1w)(A_2 (U^0)^{\kk})(A_3\calF^{-\kk})(A_4\barcalG^{-2})\\
        &\quad \times \left(A_5\left((U^0)^{-\kk}(0)(1+ \delta\lambda^{-3\kk}(0)w\calF^{-\kk}(0)) \right)\right)(A_6(\ze\pt (U^0)^{-2}))
        \end{split}
    \end{equation}

    \begin{enumerate}
        \item \textit{Estimates for $R_{421,1}^1$.} Using the bootstrap assumptions on $U^0$, $\calF, \kk\barcalG$, bounds on initial data, as well as the fact that $P\ze, \bar P\ze^2$ are smooth functions on $[0,1]$ given $P \in \calP_j$, $\bar P \in \barcalP_j$ for any $j \ge 0$, the $\|\cdot\|_i^2$ norm corresponding to the contribution from the last three terms in the square bracket of \eqref{R4211} can be bounded by integrals of the following two types:
        $$
        \int_0^1 w^{\frac1\kk + i}\ze^2 \prod_{\substack{\calA_j \in \calP_{i_j},\\\sum_{j=1}^{J} i_j \le i}}|\calA_j H^j(\Theta)|^2 d\ze,
        $$
        where $J \ge 1$ and at least one of $H^j(\Theta) = \pt\Theta$, or
        $$
        |\pt\lamt|^2 \int_0^1 w^{\frac1\kk + i}\ze^2 \prod_{\substack{\calA_j \in \calP_{i_j},\\\sum_{j=1}^{J} i_j \le i}}|\calA_j H^j(\Theta)|^2 d\ze,
        $$
        where $J \ge 1$. Then it follows from Lemma \ref{lem:nonlinear} and the bootstrap assumption \eqref{bootstrap EN} that the first integral above can be bounded by $\delta\lam^{-3\kk}\calE^N(\tau)$. By combining Lemma \ref{lem:nonlinear} and the property for $\pt\lamt$ (cf. \eqref{est:ptlamt}), we may bound the second integral above by $(\delta \lam^{-3\kk})^2 \calE^N(\tau)$.

        Thus, what are left to be estimated involve the first and second term in the square bracket of \eqref{R4211}. The second term is a source term: since $P_i \ze^3$ is a smooth function, an application of \eqref{est:ptlamt} yields $|\lamt \pt\lamt (P_i\ze^3)|^2 \lesssim (\delta \lam^{-3\kk})^2$.

        To estimate the contribution due to the first term in the square bracket of \eqref{R4211}, we note that, again after a use of bootstrap assumptions and bounds on initial data, its $\|\cdot\|_i^2$ norm can be bounded by
        \begin{align*}
            \int_0^1 w^{\frac1\kk + i}\ze^2 \prod_{\substack{\calA_j \in \calP_{i_j},\\\sum_{j=1}^{J_1 + 1}i_j \le i}}(\calA_j\bar H^j(\Theta))^2 d\ze,
        \end{align*}
        for some $J_1 \ge 0$. Here, $\bar H^j(\Theta) = \Theta, \pt\Theta$, or $\pt^2\Theta$. Moreover, exactly one of $\bar H^j(\Theta)$ is $\pt^2\Theta$, and there is no more than $i-1$ derivatives falling on that term. Therefore, we invoke Lemma \ref{lem:nonlinear2} with $(J_1,J_2,J_3, i_1, i_2) = (J_1, 1, 0, i,0)$ to deduce that the above integral can be bounded by $S^N(\tau) (1+\calE^N(\tau))(1+\calE^N(0))$. 

        Thus, based on the above analysis, we conclude the following bound:
        \begin{equation}
            \label{est:R42111}
            \|R_{421,1}^1\|_i^2 \lesssim (\delta \lam^{-3\kk})^2 + (\delta \lam^{-3\kk}) \calE^N(\tau) + \calS^N(\tau) (1+\calE^N(\tau))(1+\calE^N(0)),
        \end{equation}
        where we used that $\delta \lam^{-3\kk} \le 1$.

        \item \textit{Estimates for $R_{421,2}^1$.} To study this term, we need the following decomposition of $A_6(\ze \pt(U^0)^{-2})$ by invoking \eqref{eq:dtU0zeta}:
        \begin{equation}
        \label{eq:A6dtU0}
        \begin{split}
    A_6&\left(\ze \pt (U^0)^{-2}\right) = \sum_{\substack{B_{1,2,3}\in\calP_{a_1,a_2,a_3},\\a_1+a_2 + a_3 = l_6}} c^{a_1,a_2,a_3}_{i,1}(B_1\ze)(B_2\pt^2\Theta)\left[\lamt(B_3\ze) + B_3(\pt\Theta + \lamt\Theta)\right] \\
    &+ \lamt\pt\lamt (A_6\ze^3) + \lamt\pt\lamt \sum_{\substack{B_1\in\calP_{a_1}, B_2 \in \barcalP_{a_2}\\a_1+a_2 = l_6}} c^{a_1,a_2}_{i,3}(B_1\Theta)(B_2\ze^2)\\
    &+ (\lamt^2 + \pt\lamt)\sum_{\substack{B_1\in\calP_{a_1}, B_2 \in \barcalP_{a_2}\\a_1+a_2 = l_6}} c^{a_1,a_2}_{i,4}(B_1\pt\Theta)(B_2\ze^2)\\
    &+ \sum_{\substack{B_{1,2,3}\in\calP_{a_1,a_2,a_3},\\a_1+a_2 + a_3 = l_6}} c^{a_1,a_2,a_3}_{i,5}(B_1\ze)(B_2(\pt\Theta + \lamt\Theta))(B_3(\pt\lamt \Theta + \lamt \pt\Theta)),
        \end{split}
        \end{equation}
        where we used that $A_6 \in \calP_{l_6}$. We also remind ourselves that $l_6 \le i-1$. 

        We explain how to bound $R_{421,2}^1$ by discussing each term appearing in \eqref{eq:A6dtU0}. For terms involving the last three terms on the right-hand-side of \eqref{eq:A6dtU0}, via an analysis which is similar to that for $R_{421,1}^1$, it suffices to bound the following two types of terms in $\|\cdot\|_i^2$ norm:
        \begin{equation}\label{R42121type1}
        (A_1w)(A_2 (U^0)^{\kk})(A_3\calF^{-\kk})(A_4\barcalG^{-2})\prod_{\substack{\calA_j \in \calP_{i_j},\\\sum_{j=1}^{J} i_j \le l_6}}(\calA_j H^j(\Theta)),
        \end{equation}
        where $J \ge 1$ and at at least one of $H^j(\Theta)$ is $\pt\Theta$, or
        \begin{equation}\label{R42121type2}
        \pt\lamt (A_1w)(A_2 (U^0)^{\kk})(A_3\calF^{-\kk})(A_4\barcalG^{-2})\prod_{\substack{\calA_j \in \calP_{i_j},\\\sum_{j=1}^{J} i_j \le l_6}}(\calA_j H^j(\Theta)),
        \end{equation}
        where $J \ge 1$. We observe that both terms bear the same structure as that of $R_{411}^2$ (cf. \eqref{E411decomp}). Hence by following a similar argument, we may bound
        $$
        \|\eqref{R42121type1}\|_i^2 \lesssim \delta \lam^{-3\kk}\calE^N(\tau)(1+ \calE^N(0)).
        $$
        Moreover by the property for $\pt\lamt$ \eqref{est:ptlamt}, we can also bound
        $$
        \|\eqref{R42121type2}\|_i^2 \lesssim (\delta \lam^{-3\kk})^2\calE^N(\tau)(1+ \calE^N(0)).
        $$
        For the term which involves the second term appearing in \eqref{eq:A6dtU0}, we may write such term as a linear combination of:
        $$
        \lamt \pt\lamt(A_1w)(A_2 (U^0)^{\kk})(A_3\calF^{-\kk})(A_4\barcalG^{-2})\left(A_5\left((U^0)^{-\kk}(0)(1+ \delta\lambda^{-3\kk}(0)w\calF^{-\kk}(0)) \right)\right)(A_6\ze^3).
        $$
        The above expression can be further decomposed into a linear combination of terms resembling \eqref{E4231type1} and \eqref{E4231type2}. Then a similar analysis would imply that the aforementioned term is bounded by $(\delta \lam^{-3\kk})^2 (1+ \calE^N(\tau))(1+\calE^N(0))$. We note that this error accounts for several source terms. For example, one scenario for which such source term arises is when $l_2, l_4, l_5 \le 2,$ $l_3 = 0,$ and $l_6 \le 3.$

        Finally, as for the contribution afforded by the first term of \eqref{eq:A6dtU0}, it can be represented by
        \begin{equation}\label{R4212term1}
        \begin{split}
        (A_1w)(A_2 (U^0)^{\kk})(A_3\calF^{-\kk})(A_4\barcalG^{-2})\left(A_5\left((U^0)^{-\kk}(0)(1+ \delta\lambda^{-3\kk}(0)w\calF^{-\kk}(0)) \right)\right)\\
        \times g(\ze)\prod_{\substack{B_j \in \calP_{i_j},\\\sum_{j=1}^{J+1}i_j \le l_6}}(B_j\bar H^j(\Theta)),
        \end{split}
        \end{equation}
        for some $J \ge 0$ and some smooth function $g(\ze)$. Here, $\bar H^j(\Theta) = \Theta, \pt\Theta$, or $\pt^2\Theta$. Moreover, exactly one of $\bar H^j(\Theta)$ is $\pt^2\Theta$. Recalling Proposition \ref{prop:buildingblockpowers}, chain rule \eqref{eq:chainrule}, and product rule \eqref{Pbarproduct}, plus a use of bootstrap assumptions on $U^0, \calF,\kk\barcalG$, the $\|\cdot\|_i^2$ norm of \eqref{R4212term1} can be bounded by
        \begin{align*}
            \int_0^1 &w^{\frac1\kk + i}\ze^2 (A_1 w)^2 \prod_{\substack{\calA_j \in \calP_{i_j},\\\sum_{j=1}^{m_2 + J + 1}i_j \le l_2 + l_6}}|\calA_j \bar H^j(\Theta)|^2\cdot \prod_{\substack{A_3^j \in \calP_{i_j},\\ \sum_{j=1}^{m_3} i_j = l_3}}|A_3^j \Dz\Theta|^2 \cdot \prod_{\substack{A_4^j \in \calP_{i_j},\\ \sum_{j=1}^{k_4}i_j = l_4}}|A_4^j \barcalG|^2\\
            &\times \prod_{\substack{W_j \in \calP_{i_j},\\\sum_{j=1}^{\bar m_1}i_j = l_5^1}}|W_j H^j(\Theta)(0)|^2 \left(1+\delta \lam^{-3\kk}(0)|\tilde P w|\prod_{\substack{V_j \in \barcalP_{i_j},\\\sum_{j=1}^{\bar m_2} i_j = l_5^3}}|V_j\Dz\Theta(0)|\right)^2 d\ze,
        \end{align*}
        where we recall that $A_1 \in \barcalP_{l_1}$ with $l_1 + \hdots + l_6 = i$, $l_6 \le i-1$. Moreover, we have $m_2, \bar m_1, \bar m_2 \ge 0$, $\tilde P \in \barcalP_{l_5^2}$, and $l_5^1 + l_5^2 + l_5^3 = l_5$. Then, we may follow a similar argument to that analyzing \eqref{R4112type1} and \eqref{R4112type2}, with the only modification where we apply Lemma \ref{est:nonlinear2} instead of Lemma \ref{est:nonlinear}. Note that indeed Lemma \ref{est:nonlinear2} is eligible here, because $l_6 \le i-1 \le N-1$ in the integral above. We thus conclude that
        $$
        \|\eqref{R4212term1}\|_{i}^2 \lesssim \calS^N(1+\calE^N(0)).
        $$
        Summarizing all above estimates, we arrive at the outcome that
        \begin{equation}
            \label{est:R42121}
            \|R_{421,2}^1\|_i^2 \lesssim \left((\delta \lam^{-3\kk})^2 +\delta \lam^{-3\kk}\calE^N(\tau) + \calS^N(\tau)\right)(1+\calE^N(\tau))(1+\calE^N(0)).
        \end{equation}
        Finally, we combine \eqref{DiE4211decomp}, \eqref{est:R42111}, and \eqref{est:R42121} to conclude that
        \begin{equation}
            \label{DiE4211decomp2}
            \begin{split}
                \calD_i E_{421}^1 &= \lamt w(U^0)^\kk \calF^{-\kk}\barcalG^{-2}\left((U^0)^{-\kk}(0)(1+ \delta\lambda^{-3\kk}(0)w\calF^{-\kk}(0)) \right)\cdot \ze(\pt\Theta + \lamt(\Theta +\ze))\calD_i\pt^2\Theta\\
        &\quad + R_{421,1}^1 + R_{421,2}^1,
            \end{split}
        \end{equation}
        where
        \begin{equation}\label{est:R4211}
        \|R_{421,1}^1\|_i^2+ \|R_{421,2}^1\|_i^2 \lesssim \left((\delta \lam^{-3\kk})^2 +\delta \lam^{-3\kk}\calE^N(\tau) + \calS^N(\tau)\right)(1+\calE^N(\tau))(1+\calE^N(0)).
        \end{equation}
    \end{enumerate}
            \item Estimates for $\calD_i E_{421}^2.$ Through an almost identical argument to that analyzing $\calD_i E_{421}^1$ (with the exception that we use \eqref{eq:dtU0Hhighest}, \eqref{eq:dtU0H} in place of \eqref{eq:dtU0zetahighest}, \eqref{eq:dtU0zeta}), we have the following:
        \begin{equation}
            \label{DiE4212decomp}
            \begin{split}
                \calD_i E_{421}^2 &= w(U^0)^\kk \calF^{-\kk}\barcalG^{-2}\left((U^0)^{-\kk}(0)(1+ \delta\lambda^{-3\kk}(0)w\calF^{-\kk}(0)) \right)\\
                &\quad \times (\pt\Theta + \lamt\Theta)(\pt\Theta + \lamt(\Theta +\ze))\calD_i\pt^2\Theta + R_{421}^2,
            \end{split}
        \end{equation}
        where
        \begin{equation}
            \label{est:R4212}
            \|R_{421}^2\|_i^2 \lesssim \left((\delta \lam^{-3\kk})^2 +\delta \lam^{-3\kk}\calE^N(\tau) + \calS^N(\tau)\right)\calE^N(\tau)(1+\calE^N(\tau))(1+\calE^N(0)).
        \end{equation}
\end{enumerate}
Finally, we use the bootstrap assumption \eqref{bootstrap EN}, and combine \eqref{E422}, \eqref{est:R422}, \eqref{est:E423}, \eqref{DiE4211decomp2}, \eqref{DiE4212decomp}, \eqref{est:R4211}, and \eqref{est:R4212} to obtain that
\begin{align*}
    \frac{\kk}{1+\kk}\calD_i E_{42} &= \delta \lam^{-3\kk} w^2(U^0)^{-2}\calF^{-2\kk-1}\barcalG^{-2}(\pt\Theta + \lamt(\Theta +\ze))\left(1+\frac{\Theta}{\ze}\right)^2\calD_{i+1}\pt\Theta\\
    &\quad + \lamt w(U^0)^\kk \calF^{-\kk}\barcalG^{-2}\left((U^0)^{-\kk}(0)(1+ \delta\lambda^{-3\kk}(0)w\calF^{-\kk}(0)) \right)(\pt\Theta + \lamt(\Theta +\ze))^2\calD_i\pt^2\Theta\\
    &\quad + \frac{\kk}{1+\kk} R_{42}^i,
\end{align*}
where
$$
\|R_{42}^i\|_i^2 \lesssim (\delta \lam^{-3\kk})^2+\delta \lam^{-3\kk}\calE^N(\tau) + \calS^N(\tau),
$$
where we further used $\delta \lam^{-3\kk} <1$. This concludes the proof of Lemma \ref{lem:DiE42}.

We conclude the section by proving the main Proposition \ref{prop:DiE4}.
\begin{proof}[Proof of Proposition \ref{prop:DiE4}]
    The proposition holds after combining Lemma \ref{lem:DiE41} and Lemma \ref{lem:DiE42}.
\end{proof}
\subsection{Estimates for $\calD_i E_1$}
\begin{prop}\label{prop:DiE1}
    Let $E_1$ be given as in \eqref{mainerror}. Then for any $\tau \in [0,\tau_*]$,
    \begin{align}
        \calD_i E_1 = \left((1+\kk)\calF^{-\kk-2}(U^0)^{-4}\left(1+\frac{\Theta}{\ze}\right)^4\calL_i\calD_i\Theta\right)(\calG^{\kk}-1) + R_1^i,
    \end{align}
    where 
    \begin{align*}
    R_1^i := \sum_{\substack{A\in \calP_{l_A},\,B\in\barcalP_{l_B}\\l_A+l_B=i,\,l_B\geq 1}}A\left((U^0)^{-4}\left(1+\frac\Theta\zeta\right)^2\frac{1}{w^{\frac1\kk}}\p_\zeta \left(w^{1+\frac1\kappa} (\calF^{-1-\kappa}-1)\right) \right) B(\calG^\kk - 1) + E_{\text{pressure}}(\calG^{\kk}-1), 
    \end{align*}
    where $E_{\text{pressure}}$ is defined in Theorem \ref{thm: pressure estimate}. Also,  $R_1^i$ verifies the following estimate:
    \begin{align}\label{est:R1i}
        \|R_1^i\|_i^2\lesssim (1+\calE^N(0))\calE^N(\tau). 
    \end{align}
\end{prop}
Expanding $\calD_i E_1$ aligning with Theorem \ref{thm: pressure estimate}:
\begin{align}\label{DiE1 pre}
    \begin{aligned}
        \calD_iE_1 = \sum_{\substack{A\in \calP_{l_A},\,B\in\barcalP_{l_B}\\l_A+l_B=i}}A\left((U^0)^{-4}\left(1+\frac\Theta\zeta\right)^2\frac{1}{w^{\frac1\kk}}\p_\zeta \left(w^{1+\frac1\kappa} (\calF^{-1-\kappa}-1)\right) \right) B(\calG^\kk - 1),
    \end{aligned}
\end{align}
we discuss the control of \eqref{DiE1 pre} by separating into the following cases:
\begin{itemize}
    \item [(a)] When $l_A=i$. In this case, $A=\calD_i$, and thus \eqref{DiE1 pre} turns into
\begin{align}\label{DiE1Case1}
    \begin{aligned}
        \calD_iE_1 &= \calD_i\left((U^0)^{-4}\left(1+\frac\Theta\zeta\right)^2\frac{1}{w^{\frac1\kk}}\p_\zeta \left(w^{1+\frac1\kappa} (\calF^{-1-\kappa}-1)\right) \right) (\calG^\kk - 1)\\
        =& \left((1+\kk)\calF^{-\kk-2}(U^0)^{-4}\left(1+\frac{\Theta}{\ze}\right)^4\calL_i\calD_i\Theta\right)(\calG^{\kk}-1) + E_{\text{pressure}}(\calG^{\kk}-1),
    \end{aligned}
\end{align}
where we use Theorem \ref{thm: pressure estimate} in the second equality. Since $|\calG^{\kk}-1|\leq \frac{1}{4}$ as shown in \eqref{smallness of calFGkk-1} in Lemma \ref{lem: induced bootstrap}, the first term on the RHS can be absorbed into the corresponding elliptic term during the energy estimate process. 
Additionally, by virtue of Theorem \ref{thm: pressure estimate}, 
$$
\|E_{\text{pressure}}(\calG^{\kk}-1)\|_i^2 \lesssim \calE^N. 
$$
\item [(b)] When $l_A\leq i-1$. From \eqref{eq:elltipic term original form}, we see that
\begin{align}\label{E1case(b) pre}
\begin{aligned}
    &A\left((U^0)^{-4}\left(1+\frac\Theta\zeta\right)^2\frac{1}{w^{\frac1\kk}}\p_\zeta \left(w^{1+\frac1\kappa} (\calF^{-1-\kappa}-1)\right) \right)\\
    &=-(1+\kk)A\left(\calF^{-\kk-2}(U^0)^{-4}\left(1+\frac{\Theta}{\ze}\right)^4\left(w\pz\Dz\Theta+\frac{1+\kk}{\kk}w'\Dz\Theta\right)\right)+ \sum_{j=1}^3 A\mathfrak{R}_j(\Theta).
    \end{aligned}
\end{align}
Since $l_A\leq i-1$, we note that all terms in $A\mathfrak{R}_j(\Theta)$, $j=1,2,3$, are linear combinations of $\prod_j P_{l_j}H^j(\Theta)$ with $l_j\leq i$. Furthermore, the first term on the right-hand side of \eqref{E1case(b) pre} can be analyzed similarly to $\barcalD_{i-1}\mathfrak{M}_{31}$ in Section \ref{sect: Energy Estimate 1}. This is because these terms share a similar structure regarding the count of derivatives and the $w$-weight count at the leading order. In addition to this, since $\barcalG^{-1}=\kk\calG^{\kk}$ and $l_B\geq 1$, we write 
$$
B(\calG^{\kk}-1) = B(\kk^{-1}\barcalG^{-1}-1) = \kk^{-1} B\barcalG^{-1}. 
$$
The contribution of $B\barcalG^{-1}$ has already been discussed in Section \ref{sect: Energy Estimate II} when estimating $\mathfrak{C}_3$. Particularly, see \eqref{PbarGtype1 contribution} and \eqref{PbarGtype2 contribution}.
\end{itemize} 
\subsection{Estimates for $\calD_i E_3$}
This section is dedicated to proving the following result:
\begin{prop}
    \label{prop:DiE3}
    $\calD_i E_3$ possesses the following representation: for any $\tau \in [0,\tau_*]$,
    \begin{equation}
        \label{est:DiE3}
        \begin{split}
        \calD_iE_3 &=\calC^{E_3}_1w\ze\calD_{i+1}\pt\Theta + \calC_2^{E_3}\calL_i\calD_i\Theta + \calC^{E_3}_3w\ze\calD_{i+1}\pt\Theta(0) + \calC_4^{E_3}\calL_i\calD_i\Theta(0) + R_{3}^i,
        \end{split}
    \end{equation}
    where
    \begin{equation}\label{E3coeff}
    \begin{split}
    \calC^{E_3}_1(\tau,\ze) &:=-\frac{1+\kk}{\kk^2}(U^0)^{\kk-2}\left(1+\frac\Theta\zeta\right)^2 \calF^{-\kk-1} \barcalG^{-2}\left[(U^0)^{-\kk}(0) \left(1+\delta \lam^{-3\kk}(0) w\calF^{-\kk}(0)\right)\right]\\
        &\quad \times \frac{\pt\Theta + \lamt(\Theta + \ze)}{\ze},\\
    \calC^{E_3}_2(\tau,\ze) &:= \frac{1+\kk}{\kk^2}\delta \lam^{-3\kk} w\calF^{-2\kk-2}\barcalG^{-2}\left[(U^0)^{-4}\left(1+\frac{\Theta}{\ze}\right)^4\right],\\
    \calC^{E_3}_3(\tau,\ze) &:= \frac{1+\kk}{\kk^2} (U^0)^{\kk-4}(U^0(0))^{-\kk+2}\left(1+\frac{\Theta}{\ze}\right)^2 \calF^{-\kk-1}\barcalG^{-2}\left(1+\delta\lam^{-3\kk}(0)w\calF^{-\kk}(0)\right)\\
            &\quad \times \frac{\pt\Theta + \lamt(\Theta + \ze)}{\ze},\\
    \calC^{E_3}_4(\tau,\ze) &:= -  \frac{1+\kk}{\kk^2} \delta \lam(0)^{-3\kk} w \calF^{-\kk-1}\calF(0)^{-\kk-1} (U^0)^{\kk-4}(U^0)^{-\kk}(0) \left(1+ \frac{\Theta}{\ze}\right)^2\left(1+ \frac{\Theta(0)}{\ze}\right)^2\barcalG^{-2}.
    \end{split}
    \end{equation}
    Moreover,
    \begin{equation}
        \label{est:R3i}
        \|R_3^i\|_i^2 \lesssim 1 + \calE^N(\tau) + \calE^N(0)
    \end{equation}
\end{prop}
In order to obtain the desired representation as in Proposition \ref{prop:DiE3}, we need to rewrite $E_3$ substantially before commuting derivatives in a reckless way. The reason is as follows: it is rather clear that the most dangerous term (in the sense of derivative count) arises when $\calD_i$ is applied to $\pz \barcalG$. In view of the formula \eqref{eq:barcalG}, our main enemies are precisely $\barcalD_{i+1} (U^0)^\kk$ as well as $\barcalD_{i+1} (w\calF^{-\kk})$. The difficulty concerning $\barcalD_{i+1} (U^0)^\kk$ can be resolved by invoking the representation \eqref{PbarU0highest}, which captures a symmetry at the top order. For the term $\barcalD_{i+1} (w\calF^{-\kk})$, we face a difficulty which is very similar to the scenario when we tackle the pressure gradient, in which we encounter a danger of losing weights. To remedy this issue, we need to rewrite the contribution by $\barcalD_{i+1} (w\calF^{-\kk})$ in a way which recovers a similar structure to the pressure gradient.

\subsubsection{Rewriting of $E_3$}
In this section, we rewrite $E_3$ in a way that is more amenable to our analysis. For readers' convenience, we first recall the definition of $E_3$:
\begin{equation}
    \label{eq:E3}
    E_3 = \frac{1+\kk}{\kk}(U^0)^{-4}\left(1+\frac\Theta\zeta\right)^2 \calF^{-\kk-1}w\pz(\calG^\kk).
\end{equation}
Recalling that $\calG^\kk = \frac{1}{\kk}\barcalG^{-1}$, we have
\begin{align*}
    \pz \calG^\kk &= -\frac1\kk \barcalG^{-2}\pz \barcalG\\
    &= -\frac{1}{\kk^2} \barcalG^{-2}\pz\left[(U^0)^\kk (U^0)^{-\kk}(0) \left(1+\delta \lam^{-3\kk}(0) w\calF^{-\kk}(0)\right)\right] + \frac{1}{\kk^2} \delta \lam^{-3\kk}\barcalG^{-2} \pz(w\calF^{-\kk})\\
    &= \frac{1}{2\kk}\barcalG^{-2}(U^0)^{\kk+2}\pz(U^0)^{-2}\left[(U^0)^{-\kk}(0) \left(1+\delta \lam^{-3\kk}(0) w\calF^{-\kk}(0)\right)\right]\\
    &\quad -\frac{1}{\kk^2}\barcalG^{-2}(U^0)^{\kk}\pz\left[(U^0)^{-\kk}(0) \left(1+\delta \lam^{-3\kk}(0) w\calF^{-\kk}(0)\right)\right]\\
    &\quad + \frac{1}{\kk(1+\kk)}\delta \lam^{-3\kk}\barcalG^{-2} \calF \frac{1}{w^{\frac1\kk}}\pz\left(w^{\frac1\kk + 1}\calF^{-\kk-1}\right),
\end{align*}
where we used the following identities in the last equality above:
$$
\pz (U^0)^{\kk} = -\frac{\kk}{2}(U^0)^{\kk+2}\pz (U^0)^{-2},
$$
and
\begin{equation}\label{eq:dzwF}
\pz(w\calF^{-\kk}) = \frac{\kk}{1+\kk}\calF \frac{1}{w^{\frac1\kk}}\pz\left(w^{\frac1\kk + 1}\calF^{-\kk-1}\right).
\end{equation}
Hence via a straightforward computation, it holds that
\begin{align*}
    E_3 &= \frac{1+\kk}{2\kk^2}w(U^0)^{\kk-2}\left(1+\frac\Theta\zeta\right)^2 \calF^{-\kk-1} \barcalG^{-2}\pz(U^0)^{-2}\cdot\left[(U^0)^{-\kk}(0) \left(1+\delta \lam^{-3\kk}(0) w\calF^{-\kk}(0)\right)\right]\\
    &\quad - \frac{1+\kk}{\kk^3}w(U^0)^{\kk-4}\left(1+\frac\Theta\zeta\right)^2 \calF^{-\kk-1} \barcalG^{-2}\cdot \pz\left[(U^0)^{-\kk}(0) \left(1+\delta \lam^{-3\kk}(0) w\calF^{-\kk}(0)\right)\right]\\
    &\quad + \frac{1}{\kk^2}\delta \lam^{-3\kk}w \calF^{-\kk} \barcalG^{-2} \left[(U^0)^{-4}\left(1+\frac\Theta\zeta\right)^2\frac{1}{w^{\frac1\kk}}\pz\left(w^{\frac1\kk + 1}(\calF^{-\kk-1}-1)\right)\right]\\
    &\quad +\frac{1+\kk}{\kk^3}\delta \lam^{-3\kk}w w' (U^0)^{-4}\left(1+\frac\Theta\zeta\right)^2 \calF^{-\kk} \barcalG^{-2}\\
    &=: \sum_{j=1}^4 E_{3j}.
\end{align*}
In the following subsections, we first show how to estimate $E_{31}$ and $E_{33}$, which involve the most difficult terms. We then consider $E_{32}$ and $E_{34}$, which are slightly easier.

\subsubsection{Estimates for $\calD_i E_{31}$}
In this subsection, we aim to prove the following lemma:
\begin{lem}
    \label{lem:DiE31}
    For any $\tau \in [0,\tau_*]$, it holds that
    \begin{equation}
        \label{eq:DiE31}
        \begin{split}
        \calD_i E_{31} &= -\frac{1+\kk}{\kk^2}w(U^0)^{\kk-2}\left(1+\frac\Theta\zeta\right)^2 \calF^{-\kk-1} \barcalG^{-2}\left[(U^0)^{-\kk}(0) \left(1+\delta \lam^{-3\kk}(0) w\calF^{-\kk}(0)\right)\right]\\
        &\quad \times (\pt\Theta + \lamt(\Theta + \ze))\calD_{i+1}\pt\Theta + R_{31}^i,
        \end{split}
    \end{equation}
    where $R_{31}^i$ obeys the following bound:
    \begin{equation}
        \label{est:R31}
        \|R_{31}^i\|_i^2 \lesssim 1 + \calE^N(\tau)
    \end{equation}
\end{lem}
Applying Lemma \eqref{lem:leadingorder} and recalling that $\barcalD_{i+1} = \calD_i\pz$, we write:
\begin{equation}
    \label{DiE31aux0}
    \begin{split}
        \calD_i E_{31}&:= \frac{1+\kk}{2\kk^2}w(U^0)^{\kk-2}\left(1+\frac\Theta\zeta\right)^2 \calF^{-\kk-1} \barcalG^{-2}\left[(U^0)^{-\kk}(0) \left(1+\delta \lam^{-3\kk}(0) w\calF^{-\kk}(0)\right)\right]\barcalD_{i+1}(U^0)^{-2}\\
        &\quad + \sum_{\substack{A_{1,\hdots,6} \in \barcalP_{l_{1,\hdots,6}},\\A_7 \in \calP_{l_7},\\\sum_{j=1}^7 l_j = i, l_7 \le i-1}} c^{l_1,\hdots,l_7}_i (A_1 w)(A_2 (U^0)^{\kk-2})\left(A_3\left(1+\frac\Theta\zeta\right)^2\right) (A_4 \calF^{-\kk-1}) (A_5 \barcalG^{-2})\\
        &\quad \times (A_6\left[(U^0)^{-\kk}(0) \left(1+\delta \lam^{-3\kk}(0) w\calF^{-\kk}(0)\right)\right]) \left(A_7\pz (U^0)^{-2}\right).
    \end{split}
\end{equation}
After a further use of \eqref{PbarU0highest}, we may further write
\begin{equation}
    \label{DiE31}
    \begin{split}
        \calD_i E_{31}&:= -\frac{1+\kk}{\kk^2}w(U^0)^{\kk-2}\left(1+\frac\Theta\zeta\right)^2 \calF^{-\kk-1} \barcalG^{-2}\left[(U^0)^{-\kk}(0) \left(1+\delta \lam^{-3\kk}(0) w\calF^{-\kk}(0)\right)\right]\\
        &\quad \times (\pt\Theta + \lamt(\Theta + \ze))\calD_{i+1}\pt\Theta\\
        &\quad + w(U^0)^{\kk-2}\left(1+\frac\Theta\zeta\right)^2 \calF^{-\kk-1} \barcalG^{-2}\left[(U^0)^{-\kk}(0) \left(1+\delta \lam^{-3\kk}(0) w\calF^{-\kk}(0)\right)\right]\sum_{\substack{B_{j} \in \calP_{l_j},\\\sum_{j=1}^{k}l_j \le i+1}}c^{l_j}_i(B_jH^j(\Theta))\\
        &\quad + \frac{1+\kk}{2\kk^2}w(U^0)^{\kk-2}\left(1+\frac\Theta\zeta\right)^2 \calF^{-\kk-1} \barcalG^{-2}\left[(U^0)^{-\kk}(0) \left(1+\delta \lam^{-3\kk}(0) w\calF^{-\kk}(0)\right)\right]\lamt^2 \barcalD_{i+1}\ze^2\\
        &\quad + \sum_{\substack{A_{1,\hdots,6} \in \barcalP_{l_{1,\hdots,6}},\\A_7 \in \calP_{l_7},\\\sum_{j=1}^7 l_j = i, l_7 \le i-1}} c^{l_1,\hdots,l_7}_i (A_1 w)(A_2 (U^0)^{\kk-2})\left(A_3\left(1+\frac\Theta\zeta\right)^2\right) (A_4 \calF^{-\kk-1}) (A_5 \barcalG^{-2})\\
        &\quad \times (A_6\left[(U^0)^{-\kk}(0) \left(1+\delta \lam^{-3\kk}(0) w\calF^{-\kk}(0)\right)\right]) \left(A_7\pz (U^0)^{-2}\right),
    \end{split}
\end{equation}
where in the second term above $k = 1, 2$, and $H^j(\Theta) = \pt\Theta$ or $\Theta$. In particular, $l_j \le i$ whenever $H^j(\Theta) = \pt\Theta$. 

We note that the first term in \eqref{DiE31} has the desired highest order structure. Using bootstrap assumptions, Lemma \ref{lem: induced bootstrap}, and bounds on initial data, we note that
$$
\left|(U^0)^{\kk-2}\left(1+\frac\Theta\zeta\right)^2 \calF^{-\kk-1} \barcalG^{-2}\left[(U^0)^{-\kk}(0) \left(1+\delta \lam^{-3\kk}(0) w\calF^{-\kk}(0)\right)\right]\right| \lesssim 1.
$$
Thanks to an extra weight $w$, the $\|\cdot\|_i^2$ norm of the second term in \eqref{DiE31} can be bounded by $\calE^N(\tau)$ after a routine application of Lemma \ref{lem:nonlinear}. The third term in \eqref{DiE31}, which is a source term, can be bounded by $1$ due to the fact that $\barcalD_{i+1}\ze^2$ is a smooth function. We are then left to bound the last term in \eqref{DiE31}, which can be further written as a linear combination of terms in the form of:
\begin{equation}
    \label{DiE31typical}
    \begin{split}
        &(A_1 w)(A_2 (U^0)^{\kk-2})\left(A_3\left(1+\frac\Theta\zeta\right)^2\right) (A_4 \calF^{-\kk-1}) (A_5 \barcalG^{-2})\\
        &\quad \times (A_6\left[(U^0)^{-\kk}(0) \left(1+\delta \lam^{-3\kk}(0) w\calF^{-\kk}(0)\right)\right]) \left(A_7\pz (U^0)^{-2}\right)
    \end{split}
\end{equation}
We remark that one can further write $A_7\pz = \bar A_7 \in \barcalP_{l_7+1}$. Note that the term $\bar A_7 (U^0)^{-2}$ is not harmful as $l_7 + 1 \le i$. To bound $\|\eqref{DiE31typical}\|_i$, we split into the following two cases:
\begin{enumerate}
    \item $l_3 = 0$. In this case, by \eqref{bootstrap Theta}, \eqref{DiE31typical} can be bounded by:
    $$
    \left|(A_1 w)(A_2 (U^0)^{\kk-2})(A_4 \calF^{-\kk-1}) (A_5 \barcalG^{-2})(A_6\left[(U^0)^{-\kk}(0) \left(1+\delta \lam^{-3\kk}(0) w\calF^{-\kk}(0)\right)\right]) \left(\bar A_7 (U^0)^{-2}\right)\right|.
    $$
    After applying chain rule to $A_5 \barcalG^{-2}$ and using Proposition \ref{prop:buildingblockpowers}, the expression above can be further written as a linear combination of terms that bear a very similar structure to \eqref{R4112type1} and \eqref{R4112type2}. Then by using a similar argument to that analyzing \eqref{R4112type1} and \eqref{R4112type2}, we obtain the bound
    \begin{equation}
        \label{est:DiE31typicall3=0}
        \|\eqref{DiE31typical}\|_i^2 \lesssim \left(1 + \calE^N(\tau)\right)\left(1 + \calE^N(0)\right).
    \end{equation}
    We remark that the above estimate involves a low-frequency source error of size $\calO(1)$, which is generated when $l_2,l_5,l_6 \le 2$, $l_4 =0$, and $l_7 \le 1$.
    \item $l_3 \ge 1$. We first note that by product rule \eqref{Pbarproduct},
    $$
    A_3 \left(1+\frac{\Theta}{\ze}\right)^2 = 2A_3\left(\frac{\Theta}{\ze}\right) + A_3\left(\frac{\Theta}{\ze}\right)^2 = 2A_3\left(\frac{\Theta}{\ze}\right) + \sum_{\substack{A_3^{1,2} \in \calP_{l_3}^{1,2},\\ l_3^1 + l_3^2 =l_3+2,\\l_3^1,l_3^2 \le l_3+1}} c^{l_3^1,l_3^2}_{l_3} (A_3^1\Theta)(A_3^2\Theta).
    $$
    Let us restrict ourselves to the top order case $l_3 = i$: we invoke bootstrap assumptions and bounds on initial data to see that
    \begin{align*}
        \|\eqref{DiE31typical}\|_i^2 \lesssim \int_0^1 w^{\frac1\kk + i}\ze^2 \cdot w^2\left(\left|A_3\left(\frac{\Theta}{\ze}\right)\right|^2 + \left|(A_3^1\Theta)(A_3^2\Theta)\right|^2\right)d\ze
        \lesssim \calE^N(\tau),
    \end{align*}
    after an application of Lemma \ref{lem:nonlinear}. Here, we also used the fact that $A_3\left(\frac1\ze \cdot\right) \in \calP_{l_3 + 1}$. 
    
    The intermediate cases $1 \le l_3 \le i-1$ would follow from a similar treatment to the $l_3 = 0$ case, and we omit the details here.
\end{enumerate}
Combining the estimates resulting from the two cases above, we conclude that $\|\eqref{DiE31typical}\|_i^2 \lesssim 1 + \calE^N(\tau)$. This concludes the proof for Lemma \ref{lem:DiE31}.

\subsubsection{Estimates for $\calD_i E_{33}$}
In this subsection, we prove the following:
\begin{lem}
    \label{lem:DiE33}
    For any $\tau \in [0,\tau_*]$, it holds that
    \begin{equation}
        \label{est:DiE33}
        \begin{split}
        \calD_i E_{33} &= \frac{1+\kk}{\kk^2}\delta \lam^{-3\kk} w\calF^{-2\kk-2}\barcalG^{-2}\left[(U^0)^{-4}\left(1+\frac{\Theta}{\ze}\right)^4\calL_i\calD_i\Theta\right] + R_{33}^i,
        \end{split}
    \end{equation}
    where $R_{33}^i$ conforms to the following bound:
    \begin{equation}
        \label{est:R33}
        \|R_{33}^i\|_i^2 \lesssim (\delta \lam^{-3\kk})^2\calE^N(\tau).
    \end{equation}
\end{lem}
Using the commutator representation \eqref{eq: leading order symmetry}, we can immediately write
\begin{equation}
    \label{DiE33}
    \begin{split}
    \calD_i E_{33} &= \frac{1}{\kk^2}\delta \lam^{-3\kk} w\calF^{-\kk}\barcalG^{-2}\left[\calD_i \left((U^0)^{-4}\left(1+\frac\Theta\zeta\right)^2\frac{1}{w^{\frac1\kk}}\pz\left(w^{\frac1\kk + 1}(\calF^{-\kk-1}-1)\right)\right)\right]\\
    &\quad + \delta \lam^{-3\kk}\sum_{\substack{A_{1,2,3}\in \barcalP_{l_{1,2,3}}, A_4 \in \calP_{l_4},\\\sum_{j=1}^4 l_j = i,\\l_4 \le i - 1}} (A_1 w)(A_2 \calF^{-\kk})(A_3 \barcalG^{-2})\\
    &\quad \times\left[A_4\left((U^0)^{-4}\left(1+\frac\Theta\zeta\right)^2\frac{1}{w^{\frac1\kk}}\pz\left(w^{\frac1\kk + 1}(\calF^{-\kk-1}-1)\right)\right)\right] =: E_{33}^1 + E_{33}^2.
    \end{split}
\end{equation}
To treat $E_{33}^1$, we further invoke Theorem \ref{thm: pressure estimate} to decompose it as:
\begin{equation}\label{DiE33aux1}
\begin{split}
    E_{33}^1 &= \frac{1}{\kk^2}\delta \lam^{-3\kk} w\calF^{-\kk}\barcalG^{-2}\left[(1+\kk)\calF^{-\kk-2}(U^0)^{-4}\left(1+\frac{\Theta}{\ze}\right)^4\calL_i\calD_i\Theta\right] \\
    &\quad + \frac{1}{\kk^2}\delta \lam^{-3\kk} w\calF^{-\kk}\barcalG^{-2}E_{\text{pressure}}
\end{split}
\end{equation}
The first term in the above expression is the desired leading order expression appearing in Lemma \ref{lem:DiE33}. By invoking the bootstrap assumption $\calF, \barcalG \approx 1$ as well as Theorem \ref{thm: pressure estimate}, the second term in the above expression can be easily bounded by:
\begin{equation}\label{DiE33aux2}
\|\frac{1}{\kk^2}\delta \lam^{-3\kk} w\calF^{-\kk}\barcalG^{-2}E_{\text{pressure}}\|_i^2 \lesssim (\delta\lam^{-3\kk})^2 \calE^N(\tau).
\end{equation}

To treat $E_{33}^2$, it suffices to understand the term in the following form:
$$
\delta \lam^{-3\kk} (A_1 w)(A_2 \calF^{-\kk})(A_3 \barcalG^{-2})\left[A_4\left((U^0)^{-4}\left(1+\frac\Theta\zeta\right)^2\frac{1}{w^{\frac1\kk}}\pz\left(w^{\frac1\kk + 1}(\calF^{-\kk-1}-1)\right)\right)\right],
$$
where $A_{1,2,3} \in \barcalP_{l_{1,2,3}}$, $A_4 \in \calP_{l_4}$, $l_1 + \hdots + l_4 = i$, and $l_4 \le i-1$. To proceed, we use the computation \eqref{E1case(b) pre} to further decompose $A_4\left((U^0)^{-4}\left(1+\frac\Theta\zeta\right)^2\frac{1}{w^{\frac1\kk}}\pz\left(w^{\frac1\kk + 1}(\calF^{-\kk-1}-1)\right)\right)$. Then a routine argument mimicking the one estimating \eqref{E4231type1} by using Lemma \ref{lem:nonlinear} would yield the following bound:
\begin{equation}
    \label{DiE33aux3}
    \|E_{33}^2\|_i^2 \lesssim (\delta \lam^{-3\kk})^2\calE^N(\tau).
\end{equation}
Then Lemma \ref{lem:DiE33} follows from \eqref{DiE33aux1}--\eqref{DiE33aux3}.

\subsubsection{Estimates for $\calD_i E_{34}$}
Using the product rule \eqref{Pproduct}, we may write $\calD_i E_{34}$ as a linear combination of the following type of terms:
\begin{equation}
    \label{DiE34typical}
    \delta\lam^{-3\kk} (A_1 w')(A_2 w)(A_3 (U^0)^{-4})(A_4\calF^{-\kk})(A_5 \barcalG^{-2})\left(A_6\left(1+\frac{\Theta}{\ze}\right)^2\right),
\end{equation}
where $A_1 \in \calP_{l_1}$, $A_{2,\hdots, 6} \in \barcalP_{l_{2,\hdots,6}}$, and $\sum_{j=1}^6 l_j = i$. Similar to the analysis of $\calD_i E_{31}$, we first consider the case where $l_6 = 0$. In this case, after an application of Proposition \ref{prop:buildingblockpowers}, \eqref{DiE34typical} can be further written as a linear combination of terms with similar structures to \eqref{E4231type1} and \eqref{E4231type2}. Hence via a similar analysis, we can deduce the bound $\|\eqref{DiE34typical}\|_i^2 \lesssim (\delta \lam^{-3\kk})^2 (1+ \calE^N(\tau))$. We remark that the source error arises when $l_3, l_5 \le 2$. If $l_6 \ge 1$, we may instead use a similar argument to that treating the $l_3 \ge 1$ case for \eqref{DiE31typical} to deduce that $\|\eqref{DiE34typical}\|_i^2 \lesssim (\delta \lam^{-3\kk})^2  \calE^N(\tau)$. We omit the details here. Summarizing, we have proved the following lemma:
\begin{lem}\label{lem:DiE34}
    For any $\tau \in [0,\tau_*]$, the following estimate holds:
    \begin{equation}
    \label{est:DiE34}
    \|\calD_i E_{34}\|_i^2 \lesssim (\delta \lam^{-3\kk})^2 (1+ \calE^N(\tau)).
    \end{equation}
\end{lem}

\subsubsection{Estimates for $\calD_i E_{32}$}
The treatment of $\calD_i E_{32}$ is a bit subtle: if one directly distribute derivatives, there will be a derivative loss when all $i$ derivatives are applied to $\pz\left[(U^0)^{-\kk}(0) \left(1+\delta \lam^{-3\kk}(0) w\calF^{-\kk}(0)\right)\right]$. Indeed, such derivative loss will be incurred at the level of initial data, and one can close the bootstrap if one grants extra regularity to the initial data. Nevertheless, we will see that such loss is unnecessary as long as we write $\calD_i E_{32}$ by using a similar strategy to those treating $\calD_i E_{31}, \calD_i E_{33}$, and $\calD_i E_{34}$. In particular, we will show the following lemma:
\begin{lem}
    \label{lem:DiE32}
    For any $\tau \in [0,\tau_*]$, it holds that
    \begin{equation}
        \label{est:DiE32}
        \begin{split}
            \calD_i E_{32} &= \frac{1+\kk}{\kk^2} w(U^0)^{\kk-4}(U^0(0))^{-\kk+2}\left(1+\frac{\Theta}{\ze}\right)^2 \calF^{-\kk-1}\barcalG^{-2}\left(1+\delta\lam^{-3\kk}(0)w\calF^{-\kk}(0)\right)\\
            &\quad \times (\pt\Theta + \lamt(\Theta + \ze))\calD_{i+1}\pt\Theta(0)\\
            & -  \frac{1+\kk}{\kk^2} \delta \lam(0)^{-3\kk} w \calF^{-\kk-1}\calF(0)^{-\kk-1} (U^0)^{\kk-4}(U^0)^{-\kk}(0) \left(1+ \frac{\Theta}{\ze}\right)^2\left(1+ \frac{\Theta(0)}{\ze}\right)^2\barcalG^{-2}\calL_i\calD_i \Theta(0)\\
            &+ R_{32}^i,            
        \end{split}
    \end{equation}
    where $R_{32}^i$ obeys the following estimate:
    \begin{equation}
        \label{est:R32}
        \|R_{32}^i\|_i^2 \lesssim 1 + \calE^N(\tau) + \calE^N(0).
    \end{equation}
\end{lem}
The proof of the lemma highly overlaps the strategies used in the study of $\calD_i E_{31}$, $\calD_i E_{33}$, and $\calD_i E_{34}$, so we are content with only sketching the argument here. Note that we can further write
\begin{align*}
    E_{32} &= - \frac{1+\kk}{\kk^3}w(U^0)^{\kk-4}\left(1+\frac\Theta\zeta\right)^2 \calF^{-\kk-1} \barcalG^{-2}\left(1+\delta \lam^{-3\kk}(0) w\calF^{-\kk}(0)\right) \pz\left[(U^0)^{-\kk}(0) \right]\\
    &\quad -\frac{1+\kk}{\kk^3}\delta \lam^{-3\kk}(0)w(U^0)^{\kk-4}(U^0)^{-\kk}(0)\left(1+\frac\Theta\zeta\right)^2 \calF^{-\kk-1} \barcalG^{-2} \pz(w\calF^{-\kk}(0))\\
    &= - \frac{1+\kk}{\kk^3}w(U^0)^{\kk-4}\left(1+\frac\Theta\zeta\right)^2 \calF^{-\kk-1} \barcalG^{-2}\left(1+\delta \lam^{-3\kk}(0) w\calF^{-\kk}(0)\right) \pz\left[(U^0)^{-\kk}(0) \right]\\
    &\quad -\frac{\delta \lam^{-3\kk}(0)}{\kk^2}w(U^0)^{\kk-4}(U^0)^{-\kk+4}(0)\left(1+\frac\Theta\zeta\right)^2\left(1+\frac{\Theta(0)}{\zeta}\right)^{-2} \calF^{-\kk-1} \calF(0) \barcalG^{-2}\\
    &\quad\times\left[(U^0)^{-4}(0)\left(1+\frac{\Theta(0)}{\zeta}\right)^{2}\frac{1}{w^{\frac1\kk}}\pz\left(w^{\frac1\kk + 1}(\calF^{-\kk-1}(0) - 1)\right)\right]\\
    &\quad - \frac{\delta \lam^{-3\kk}(0)}{\kk^2}ww' (U^0)^{\kk-4}(U^0)^{-\kk}(0)\left(1+\frac{\Theta}{\zeta}\right)^{2} \calF^{-\kk-1} \calF(0)\barcalG^{-2}\\
    &=: \sum_{j=1}^3 E_{32}^j.
\end{align*}
Note that we used the identity \eqref{eq:dzwF} in the second equality above. Then we decompose $E_{32}^1$ (respectively $E_{32}^2$) in a similar way to $E_{31}$ (respectively $E_{33}$). This procedure gives the first two lines of \eqref{est:DiE32} plus an error term obeying a bound identical to the right-hand-side of \eqref{est:R32}. The term $E_{32}^3$ can be treated in almost the same way as $E_{34}$, and is also bounded by the right-hand-side of \eqref{est:R32}. Lemma \ref{lem:DiE32} then follows.

We conclude this section by proving the main Proposition \ref{prop:DiE3}.
\begin{proof}[Proof of Proposition \ref{prop:DiE3}]
    The proposition holds after combining Lemma \ref{lem:DiE31}, Lemma \ref{lem:DiE32}, Lemma \ref{lem:DiE33}, and Lemma \ref{lem:DiE34}.
\end{proof}

\subsection{Estimates for $\calD_i E_2$}
In this section, we aim to prove the following estimate:
\begin{prop}
    \label{prop:DiE2}
    The following estimate holds for any $\tau \in [0,\tau_*]$:
    \begin{equation}
        \label{est:DiE2}
        \|\calD_iE_2\|_i^2 \lesssim 1 + \calE^N(\tau).
    \end{equation}
\end{prop}
From a straightforward computation, we observe that 
\begin{align*}
  E_2 &= \frac{1+\kk}{\kk^2}w'(U^0)^{-4}\left(1+ 2\frac{\Theta}{\ze} + \frac{\Theta^2}{\ze^2}\right)\barcalG^{-1}\\
    &=  \frac{1+\kk}{\kk^2}\left[w'(U^0)^{-4}\barcalG^{-1} + 2\frac{w'}{\ze}(U^0)^{-4}\Theta \barcalG^{-1} + \frac{w'}{\ze}(U^0)^{-4}\frac{\Theta^2}{\ze} \barcalG^{-1}\right]\\
    &=:\frac{1+\kk}{\kk^2}(E_{21} + E_{22} + E_{23}).
\end{align*}
Then it suffices for us to estimate $E_{2j}$, $j = 1,2,3$ separately. 
\subsubsection{Estimates for $\calD_i E_{21}$}
Upon invoking the product rule \eqref{Pproduct}, we may represent $\calD_i E_{21}$ as a linear combination of:
\begin{equation}
    \label{DiE21}
    (A_1 w')(A_2 (U^0)^{-4})(A_3 \barcalG^{-1}),
\end{equation}
where $A_1 \in \calP_{l_1}$, $A_{2,3} \in \barcalP_{l_{2,3}}$, $l_1 + l_2 + l_3 = i$. After a further application of chain rule and \eqref{PbarU0}, we see that \eqref{DiE21} can be bounded by
\begin{equation}
    \label{DiE21aux1}
    \left|(A_1 w') (U^0)^{-4+2k_2}\barcalG^{-1-k_3}\prod_{\substack{\calA_j \in \calP_{i_j},\\ \sum_{j=1}^{m_2} i_j \le l_2}}(\calA_j H^j(\Theta))\prod_{\substack{A^{j}_3 \in \barcalP_{i_j},\\\sum_{j=1}^{k_3}i_j = l_3}} A^j_3\barcalG\right|,
\end{equation}
where $k_2, m_2, k_3 \ge 0$, $H^j(\Theta) = \pt\Theta$ or $\Theta$, and we may further decompose $A_3^j \barcalG$ according to \eqref{PbarGtype1} and \eqref{PbarGtype2}. Note that if $l_3 = 0$ (thus $k_3 = 0$) and $m_2 = 0$, \eqref{DiE21aux1} produces a source error after invoking the bootstrap assumptions, namely $|\eqref{DiE21aux1}| \lesssim 1$. If $l_3 = 0$ and $m_2 > 0$, then a routine application of Lemma \ref{lem:nonlinear} yields the bound $\|\eqref{DiE21aux1}\|_i^2 \lesssim \calE(\tau)$. 

Thus, it suffices for us to consider $l_3 > 0$. For the sake of simplicity, we only consider the case where $k_3 = 1$, as other case would follow from a systematic use of Lemma \ref{lem:nonlinear}. If $A_3^1\barcalG$ contributes to \eqref{DiE21aux1} by terms of type \eqref{PbarGtype1}, then it follows that we can obtain the following bound:
\begin{align*}
    \|\eqref{DiE21aux1}\|_i^2 &\lesssim \int_0^1 w^{\frac1\kk + i} \ze^2 |A_1w'|^2 |\tilde P w|^2 \prod_{\substack{W_j \in \calP_{i_j},\\\sum_{j=1}^{\bar m_1} i_j = n_1^1}}|W_j H^j(\Theta)(0)|^2\cdot \prod_{\substack{V_j \in \barcalP_{i_j},\\\sum_{j=1}^{\bar m_2}i_j = n_1^3}}|V_j\Dz\Theta(0)|^2\\
            &\times\prod_{\substack{\calA_j \in \calP_{i_j},\\ \sum_{j=1}^{M}i_j \le l_2 + n_2}}|\calA_j H^j(\Theta)|^2 d\ze,
\end{align*}
where $M,\bar m_1, \bar m_2 \ge 0$, $n_1 + n_2 = l_3$, $n_1^1 + n_1^2 + n_1^3 = n_1$, $\tilde P \in \barcalP_{n_1^2}$. By Lemma \ref{lem:nonlinear}, the above integral can be bounded by $(1 + \calE^N(\tau))(1+\calE^N(0)) \lesssim 1 + \calE^N(\tau)$. Note that the $\calO(1)$ source error is generated, for example, when $M, \bar m_1, \bar m_2 = 0$.

In the case where $A_3^1 \barcalG$ contributes to \eqref{DiE21aux1} by terms of type \eqref{PbarGtype2}, we focus on the term of type \eqref{PbarGtype2a}. In this scenario, after invoking bootstrap assumptions, we see that
\begin{equation}\label{DiE21aux2}
    \|\eqref{DiE21aux1}\|_i^2 \lesssim  |\delta \lam^{-3\kk}|^2\int_0^1 w^{\frac1\kk + i}\ze^2|\bar Pw|^2\prod_{\substack{\calA_j \in \calP_{i_j},\\ \sum_{j=1}^{m_2} i_j \le l_2}}|\calA_j H^j(\Theta)|^2 \cdot \prod_{\substack{W_j \in \bar P_{i_j},\\\sum_{j=1}^k i_j = n_2}} |W_j\Dz\Theta|^2 d\ze,
\end{equation}
where $\bar P \in \barcalP_{n_1}$, $n_1 + n_2 = l_3$, and $k = 1,\hdots, l_3$. We distinguish \eqref{DiE21aux2} into two separate cases. If $n_1 = 0$, then $w^{\frac1\kk + i}|\bar P w|^2 = w^{\frac1\kk + i +2}$. We thus have enough weights and can invoke Lemma \ref{lem:nonlinear} with $(J_1,J_2,i_1,i_2) = (m_2, k, l_2, n_2+1)$ to obtain the bound $\|\eqref{DiE21aux1}\|_i^2 \lesssim |\delta \lam^{-3\kk}|^2\calE^N(\tau)$. If $n_1 > 0$, we have $n_2 \le l_3 \le i-1$. Then we can again use Lemma \ref{lem:nonlinear} with $(J_1,J_2,i_1,i_2) = (m_2, k, l_2, n_2+1)$ to obtain the bound $\|\eqref{DiE21aux1}\|_i^2 \lesssim |\delta \lam^{-3\kk}|^2\calE^N(\tau)$. 

Based on the discussions above, we conclude that
\begin{equation}
    \label{est:DiE21}
    \|D_i E_{21}\|_{i}^2 \lesssim 1 + \calE^N(\tau).
\end{equation}

\subsubsection{Estimates for $\calD_i E_{22}$ and $\calD_i E_{23}$}
Upon invoking the product rule \eqref{Pproduct}, we may write $\calD_i\calD_{22}$ 
as a linear combination of:
\begin{equation}
    \label{DiE22}
    \left(A_1\frac{w'}{\ze}\right)(A_2 (U^0)^{-4})(A_3 \barcalG^{-1})(A_4 \Theta),
\end{equation}
where $A_j \in \barcalP_{l_j}$, $j = 1,2,3$, $A_4 \in \calP_{l_4}$, and $\sum_{j=1}^4 l_j = i$. Then by further application of \eqref{PbarU0} and chain rule, we have that \eqref{DiE22} is bounded by
$$
\left|\left(A_1 \frac{w'}{\ze}\right) (U^0)^{-4+2k_2}\barcalG^{-1-k_3}\prod_{\substack{\calA_j \in \calP_{i_j},\\ \sum_{j=1}^{m_2+1} i_j \le l_2+l_4}}(\calA_j H^j(\Theta))\prod_{\substack{A^{j}_3 \in \barcalP_{i_j},\\\sum_{j=1}^{k_3}i_j = l_3}} A^j_3\barcalG\right|,
$$
where $k_2, m_2, k_3 \ge 0$, $H^j(\Theta) = \pt\Theta$ or $\Theta$, and we may further decompose $A_3^j \barcalG$ according to \eqref{PbarGtype1} and \eqref{PbarGtype2}. Note that the above expression bears the same structure as \eqref{DiE21aux1}, except that the product of $\calA_j H^j(\Theta)$ is at least linear in $\Theta$. Then by a similar analysis to that the previous subsection, we may obtain the bound
\begin{equation}
    \label{est:DiE22}
    \|D_i E_{22}\|_{i}^2 \lesssim \calE^N(\tau).
\end{equation}
As for $\calD_iE_{23}$, we use the product rule \eqref{Pproduct} in conjunction with the special product rule \eqref{eq:thetasquare} to write $\calD_{i} E_{23}$ as a linear combination of:
\begin{equation}
    \label{DiE23}
    \left(A_1\frac{w'}{\ze}\right)(A_2 (U^0)^{-4})(A_3 \barcalG^{-1})(A_4 \Theta)(A_5 \Theta),
\end{equation}
where $A_j \in \barcalP_{l_j}$, $j = 1,2,3$, $A_{4,5} \in \calP_{l_{4,5}}$, $\sum_{j=1}^5 l_j = i + 1$, and $l_j \le i$ for all $j$. 
Note that the highest amount of derivatives in the above expression does not exceed $i$, and it is quadratic in $\Theta$. Thus, we may still invoke the same argument as the one in the previous subsection to deduce the following bound:
\begin{equation}
    \label{est:DiE23}
    \|D_i E_{23}\|_{i}^2 \lesssim (\calE^N(\tau))^2.
\end{equation}
The main result of this section (cf. Proposition \ref{prop:DiE2}) then follows from \eqref{est:DiE21}, \eqref{est:DiE22}, and \eqref{est:DiE23}.
\subsection{Estimates for $\calD_iE_5$ and $\calD_iE_6$}
The errors $E_5$ and $E_6$ share almost identical structure, so we estimate them together. The main goal of this section is proving the following Proposition:
\begin{prop}
    \label{prop:DiE5+6}
    The following estimate holds for any $\tau \in [0,\tau_*]$:
    \begin{equation}
        \label{est:DiE5+6}
        \|\calD_iE_5\|_i^2 + \|\calD_iE_6\|_i^2 \lesssim 1 + \calE^N(\tau).
    \end{equation}
\end{prop}
To prove this proposition, it is helpful to decompose $E_5$ in the following fashion:
\begin{equation}
    \label{E5decomp}
    \begin{split}
    E_5 &= -3(1+\kk)\lamt \left[(\lamt\ze)w(U^0)^{-2}\calF^{-\kk}\calG^\kk + (\pt\Theta + \lamt \Theta)w(U^0)^{-2}\calF^{-\kk}\calG^\kk\right]\\
    &= -3(1+\kk)\lamt (E_{51} + E_{52}).
    \end{split}
\end{equation}
\subsubsection{Estimates for $\calD_i E_{51}$}
After an application of the product rule \eqref{Pproduct}, we may realize $\calD_i E_{51}$ as a linear combination of terms in the following form:
\begin{equation}
    \label{DiE51}
    A_1(w\ze)A_2 (U^0)^{-2}A_3 \calF^{-\kk} A_4\barcalG^{-1},
\end{equation}
where $A_1 \in \calP_{l_1}$, $A_{2,\hdots, 4} \in \barcalP_{l_2,\hdots,l_4}$, and $l_1 + \hdots + l_4 = i$. After further invoking Proposition \eqref{prop:buildingblock} and Proposition \eqref{prop:buildingblockpowers}, we may further break down \eqref{DiE51} into the following two types of terms:
    \begin{equation}
        \label{DiE51type1}
        (A_1(w\ze))\calF^{-\kk-k_3}\barcalG^{-1-k_4}\prod_{\substack{\calA_j \in \calP_{i_j},\\ \sum_{j=1}^{m_2}i_j = l_2}}(\calA_j H^j(\Theta))\cdot \prod_{\substack{A_3^j \in \barcalP_{i_j}\\\sum_{j=1}^{m_3}i_j = l_3}}(A_3^j\Dz\Theta)\cdot \prod_{\substack{A_4^j \in \barcalP_{i_j},\\\sum_{j=1}^{k_4}i_j = l_4}}(A_4^j\barcalG),
    \end{equation}
    for $m_2 = 1, 2$, and 
    \begin{equation}
        \label{DiE51type2}
        (\lamt^2 A_2\ze^2)(A_1(w\ze))\calF^{-\kk-1-k_3}\barcalG^{-1-k_4} \prod_{\substack{A_3^j \in \barcalP_{i_j}\\\sum_{j=1}^{m_3}i_j = l_3}}(A_3^j\Dz\Theta)\cdot \prod_{\substack{A_4^j \in \barcalP_{i_j},\\\sum_{j=1}^{k_4}i_j = l_4}}(A_4^j\barcalG).
    \end{equation}
    Here, $k_3 = 1,\hdots, l_3$ and $k_4 = 1,\hdots, l_4$ when $l_3, l_4 \ge 1$. Note that \eqref{DiE51type1} and \eqref{DiE51type2} bear a very similar structure to \eqref{R4112type1} and \eqref{R4112type2}. In fact, an almost identical analysis to that treating \eqref{R4112type1} would yield the following bound:
    \begin{equation}
        \label{est:DiE51type1}
        \|\eqref{DiE51type1}\|_i^2 \lesssim \calE^N(\tau)(1+\calE^N(0)).
    \end{equation}
    As for the estimate for \eqref{DiE51type2}, we note that low frequency source terms appear (for example, when $l_2 \le 2$, $l_3 = l_4 = 0$). Therefore, we are only able to obtain the following bound:
    \begin{equation}
        \label{est:DiE51type2}
        \|\eqref{DiE51type2}\|_i^2 \lesssim (1+\calE^N(\tau))(1+\calE^N(0)).
    \end{equation}
    Summarizing \eqref{est:DiE51type1}, \eqref{est:DiE51type2}, \eqref{DiE51}, and appealing to the bootstrap assumption \eqref{bootstrap EN} as well as the bound on initial data that $\calE^N(0) < 1$ in the bootstrap horizon, we conclude that
    \begin{equation}
        \label{est:DiE51}
        \|\calD_iE_{51}\|_{i}^2 \lesssim 1 + \calE^N(\tau).
    \end{equation}
\subsubsection{Estimates for $\calD_i E_{52}$}
Similarly, we may apply product rule \eqref{Pproduct} to write $\calD_iE_{52}$ as a linear combination of 
\begin{equation}
    \label{DiE52}
    A_1(w)A_2 (U^0)^{-2}A_3 \calF^{-\kk} A_4\barcalG^{-1} A_5H(\Theta),
\end{equation}
where $A_{1,\hdots, 4} \in \barcalP_{l_1,\hdots,l_4}$, $A_5 \in \calP_{l_5}$, and $l_1 + \hdots + l_5 = i$. Moreover, $H(\Theta) = \pt\Theta$ or $\Theta$. Again by a similar argument to that estimating \eqref{R4112type1} and \eqref{R4112type2}, we conclude the following bound:
\begin{equation}
    \label{est:DiE52}
    \|\calD_iE_{52}\|_{i}^2 \lesssim \calE^N(\tau).
\end{equation}

\subsubsection{Estimates for $\calD_i E_6$}
Using the product rule \eqref{Pproduct}, we may write
\begin{equation}
    \label{DiE6}
    \calD_i E_6 = -\calD_i\ze - \delta \lam^{-3\kk}\sum_{\substack{A_1 \in \calP_{l_1}, A_{2,3} \in \barcalP_{l_{2,3}},\\\sum_{j=1}^3 l_j = i}}c^{l_1,l_2,l_3}_i(A_1 (w\ze))(A_2\calF^{-\kk})(A_3 \calG^\kk).
\end{equation}
Then by the fact that $\calD_i\ze = 3$ for $i = 1$ and $\calD_i\ze = 0$ for $i \ge 2$, and via a similar analysis to \eqref{DiE51} to treat the second term in \eqref{DiE6}, we obtain the following bound:
\begin{equation}
    \label{est:DiE6}
    \|\calD_i E_6\|_i^2 \lesssim 1 + \calE^N(\tau),
\end{equation}
where we also used that $\calE^N(0) < 1$ and $\delta \lam^{-3\kk} < 1$ for all $\tau \ge 0$.

Therefore, Proposition \ref{prop:DiE5+6} follows from \eqref{est:DiE51}, \eqref{est:DiE52}, and \eqref{est:DiE6}.

\section{Energy Estimates IV: Estimates for Lagrangian Acceleration} \label{sect: Energy Estimate 4}
In this section, we develop suitable estimates for $S^N$, which control spatial derivatives of the Lagrangian acceleration $\pt^2\Theta$ up to the $(N-1)$--th order. This quantity will be used in conjunction with the norm $E^N$ to close our bootstrap argument. The key result that we prove in this section is the following:
\begin{prop}
    \label{prop:SN}
    Let $(\Theta,\pt\Theta)$ be the solution to \eqref{eq:momlagmainfull} on $[0,\tau_*]$ verifying the bootstrap assumptions. Suppose that $\eps \le \eps'$ (where $\eps'$ is chosen in Lemma \ref{lem: induced bootstrap}), $\delta \le \eps^\frac12$, and time instances $\tau_1,\tau_2$ verifying $0\le \tau_1 \le \tau_2 \le \tau_*$. Then the following estimate holds:
    \begin{equation}
        \label{est:SN}
        S^N(\tau_2;\tau_1) \lesssim \delta^{\frac32}(\calE^N(\tau_1) + \calE^N(\tau_2) + \calE^N(0)) + \delta^2 + \delta\int_{\tau_1}^{\tau_2} \lam^{-3\kk} \calE^N(\tau) d\tau.
    \end{equation}
\end{prop}
\begin{proof}
    For fixed $i\in \{0,\cdots, N-1\}$, we obtain the following identity by testing \eqref{eq:highorder momlagmain} with $\delta \lam^{-3\kk}\pt^2\calD_i\Theta$ with respect to the inner product $(\cdot,\cdot)_i$:
    \begin{equation}
        \label{eq:Sienergy}
        \begin{split}
        &\int_0^1 w^{\frac1\kk + i}\ze^2 \left(1+\delta \lam^{-3\kk}w\calF^{-\kk}\calG^\kk\right)|\p_\tau^2 \calD_i\Theta|^2 d\ze\\
        &\quad + \lamt\int_0^1 w^{\frac1\kk + i}\ze^2 \left(1+\delta \lam^{-3\kk}w\calF^{-\kk}\calG^\kk\right) (\p_\tau \calD_i\Theta)(\pt^2 \calD_i \Theta) d\ze\\
        &\quad + \delta \lam^{-3\kk}\int_0^1 w^{\frac1\kk + i}\ze^2\left(1+\delta \lam^{-3\kk}w\calF^{-\kk}\calG^\kk\right) (\calD_i\Theta)(\pt^2 \calD_i\Theta)d\ze\\
        &\quad+(1+\kk)\delta \lam^{-3\kk}\int_0^1 w^{\frac1\kk + i}\ze^2\calF^{-\kk-2}(U^0)^{-4}\left(1+\frac{\Theta}{\ze}\right)^4 (\calL_{i}\calD_i\Theta)(\pt^2 \calD_i \Theta)d\ze\\
        &\quad + \delta \lam^{-3\kk}\bigg[\sum_{j=1}^3 (\mfC_i,\pt^2\calD_i \Theta)_i+\sum_{j=1}^2(\calD_i \mathfrak{R}_j(\Theta),\pt^2\calD_i \Theta)_i+ (\barcalD_{i-1}\mathfrak{M}(\Theta),\pt^2\calD_i \Theta)_i\\
        &\quad
        + \sum_{j=1}^6(\calD_iE_j,\pt^2\calD_i \Theta)_i\bigg] = 0,
        \end{split}
    \end{equation}
    We then estimate \eqref{eq:Sienergy} on a term-by-term basis:
    \begin{enumerate}
            \item \textbf{First term on the LHS of \eqref{eq:Sienergy}.} By bootstrap assumptions on $\calF, \barcalG$, and the fact that $|\lamt-\bar\lam| \lesssim \delta$ in view of the estimate \eqref{est:difflamt}, we immediately have $|\delta \lam^{-3\kk}w\calF^{-\kk}\calG^\kk|_{L^\infty} \le \frac12$ after choosing $\delta \le \eps^{\frac12} \ll 1$. We then obtain the following coercivity estimate:
        \begin{equation}
            \label{est:line1}
            \int_0^1 w^{\frac1\kk + i}\ze^2 \left(1+\delta \lam^{-3\kk}w\calF^{-\kk}\calG^\kk\right)|\p_\tau^2 \calD_i\Theta|^2 d\ze \ge \frac12 \|\pt^2 \calD_i \Theta\|_i^2.
        \end{equation}

        \item \textbf{Second and third terms on the LHS of \eqref{eq:Sienergy}.} Once again, by using the bounds $|\delta \lam^{-3\kk}w\calF^{-\kk}\calG^\kk|_{L^\infty} \le \frac12$ and $|\lamt| \lesssim 1$, a straightforward application of Cauchy-Schwarz inequality yields that
        \begin{equation}
            \label{est:line2+3}
            \begin{split}
                &\left|\lamt\int_0^1 w^{\frac1\kk + i}\ze^2 \left(1+\delta \lam^{-3\kk}w\calF^{-\kk}\calG^\kk\right) (\p_\tau \calD_i\Theta)(\pt^2 \calD_i \Theta) d\ze\right|\\
        &\quad + \left|\delta \lam^{-3\kk}\int_0^1 w^{\frac1\kk + i}\ze^2\left(1+\delta \lam^{-3\kk}w\calF^{-\kk}\calG^\kk\right) (\calD_i\Theta)(\pt^2 \calD_i\Theta)d\ze\right|\\
        &\lesssim \|\pt\calD_i \Theta\|_i \|\pt^2 \calD_i \Theta\|_i + \delta \lam^{-3\kk} \|\calD_i\Theta\|_i \|\pt^2\calD_i\Theta\|_i\\
        &\lesssim \mathring{\eps}\|\pt^2\calD_i\Theta\|_i^2 + C(\mathring{\eps})\delta\lam^{-3\kk}\calE^N(\tau),
            \end{split}
        \end{equation}
        where $\mathring{\eps} > 0$ is a small constant to be determined. We also used $\delta \lam^{-3\kk} \ll 1$ in the final inequality above.

        \item \textbf{Fourth term on the LHS of \eqref{eq:Sienergy}.} Let us denote
        $$
        g(\tau,\ze) := \calF^{-\kk-2}(U^0)^{-4}\left(1+\frac{\Theta}{\ze}\right)^4 .
        $$
        By definition of $\calL_i\calD_i$, we integrate by parts and obtain the following:
        \begin{equation}\label{eq:line4}
        \begin{split}
            &(1+\kk)\delta \lam^{-3\kk}\int_0^1 w^{\frac1\kk + i}\ze^2\calF^{-\kk-2}(U^0)^{-4}\left(1+\frac{\Theta}{\ze}\right)^4 (\calL_{i}\calD_i\Theta)(\pt^2 \calD_i \Theta)d\ze\\
            &\quad = (1+\kk)\delta \lam^{-3\kk}\int_0^1 w^{\frac1\kk + i + 1}\ze^2 g \calD_{i+1}\Theta \pt^2 \calD_{i+1}\Theta d\ze\\
            &\quad\quad +(1+\kk)\delta \lam^{-3\kk}\int_0^1 w^{\frac1\kk + i + 1}\ze^2 \pz g \calD_{i+1}\Theta \pt^2 \calD_i\Theta d\ze\\
            &\quad = (1+\kk)\frac{d}{d\tau}\left[\delta \lam^{-3\kk}\int_0^1 w^{\frac1\kk + i + 1}\ze^2 g \calD_{i+1}\Theta \pt\calD_{i+1}\Theta d\ze\right]\\
            &\quad\quad +3\kk(1+\kk)\delta \lam^{-3\kk}\lamt \int_0^1 w^{\frac1\kk + i + 1}\ze^2 g \calD_{i+1}\Theta \pt\calD_{i+1}\Theta d\ze\\
            &\quad\quad -(1+\kk)\delta \lam^{-3\kk}\int_0^1 w^{\frac1\kk + i + 1}\ze^2 (\pt g) \calD_{i+1}\Theta \pt\calD_{i+1}\Theta d\ze\\
            &\quad\quad -(1+\kk)\delta \lam^{-3\kk}\int_0^1 w^{\frac1\kk + i + 1}\ze^2 g |\pt\calD_{i+1}\Theta|^2 d\ze\\
            &\quad\quad +(1+\kk)\delta \lam^{-3\kk}\int_0^1 w^{\frac1\kk + i + 1}\ze^2 \pz g \calD_{i+1}\Theta \pt^2 \calD_i\Theta d\ze.
            \end{split}
        \end{equation}
        By bootstrap assumptions, it is clear that
        $$
        \|g\|_{L^\infty} + \|\pt g\|_{L^\infty} + \|\pz g\|_{L^\infty} \lesssim 1.
        $$
        Since $i \le N-1$, we have the following bounds:
        \begin{equation}
            \label{est:line4}
            \begin{split}
                \left|3\kk(1+\kk)\delta \lam^{-3\kk}\lamt \int_0^1 w^{\frac1\kk + i + 1}\ze^2 g \calD_{i+1}\Theta \pt\calD_{i+1}\Theta d\ze\right| &\lesssim (\delta\lam^{-3\kk})^{\frac32} \calE^N(\tau),\\
                \left|(1+\kk)\delta \lam^{-3\kk}\int_0^1 w^{\frac1\kk + i + 1}\ze^2 (\pt g) \calD_{i+1}\Theta \pt\calD_{i+1}\Theta d\ze\right| &\lesssim (\delta\lam^{-3\kk})^{\frac32} \calE^N(\tau),\\
                \left|(1+\kk)\delta \lam^{-3\kk}\int_0^1 w^{\frac1\kk + i + 1}\ze^2 g |\pt\calD_{i+1}\Theta|^2 d\ze\right| &\lesssim (\delta\lam^{-3\kk})^2 \calE^N(\tau),\\
                \left|(1+\kk)\delta \lam^{-3\kk}\int_0^1 w^{\frac1\kk + i + 1}\ze^2 \pz g \calD_{i+1}\Theta \pt^2 \calD_i\Theta d\ze\right| &\lesssim \mathring{\eps} \|\pt^2 \calD_i\Theta\|_i^2 + C(\mathring{\eps})(\delta \lam^{-3\kk})^2\calE^N(\tau),
            \end{split}
        \end{equation}
        where $\mathring{\eps} > 0$ is a small constant to be determined later.

        \item \textbf{Fifth term on the LHS of \eqref{eq:Sienergy}.} Invoking Theorem \ref{thm: pressure estimate} as well as the estimate \eqref{est: est C3}, a direct application of Cauchy-Schwarz inequality gives:
        \begin{equation}
            \label{est:line5}
            \begin{split}
            &\left|\delta \lam^{-3\kk}\bigg[\sum_{j=1}^3 (\mfC_i,\pt^2\calD_i \Theta)_i+\calD_i\sum_{j=1}^2(\mathfrak{R}_j(\Theta),\pt^2\calD_i \Theta)_i+ (\barcalD_{i-1}\mathfrak{M}(\Theta),\pt^2\calD_i \Theta)_i\bigg]\right|\\
            &\quad \le \delta \lam^{-3\kk}\left(\sum_{j=1}^3 \|\mfC_j\|_i + \sum_{j=1}^2 \|\calD_i \mfR_j(\Theta)\|_i + \|\barcalD_{i-1}\mathfrak{M}(\Theta)\|_i\right)\|\pt^2 \calD_i\Theta\|_i\\
            &\quad \lesssim \delta \lam^{-3\kk} (\calE^N(\tau)^\frac12 + \calS^N(\tau)^\frac12)\|\pt^2 \calD_i\Theta\|_i\\
            &\quad \lesssim \delta\lam^{-3\kk} S^N(\tau)^\frac12\|\pt^2 \calD_i\Theta\|_i + \mathring{\eps}\|\pt^2 \calD_i\Theta\|_i^2 + C(\mathring{\eps})(\delta\lam^{-3\kk})^2 \calE^N(\tau).
            \end{split}
        \end{equation}

        \item \textbf{Sixth term on the LHS of \eqref{eq:Sienergy}.} When $j = 2,5,6$, we invoke Proposition \ref{prop:DiE2}, Proposition \ref{prop:DiE5+6} and directly apply Cauchy-Schwarz inequality to obtain:
        \begin{equation}\label{line6aux1}
        \begin{split}
        |\delta\lam^{-3\kk}(\calD_i E_j , \pt^2\calD_i\Theta)_i| &\le \delta\lam^{-3\kk}(1+(\calE^N(\tau))^{\frac12})\|\pt^2 \calD_i\Theta\|_i\\
        &\le \mathring{\eps}\|\pt^2 \calD_i\Theta\|_i^2 + C(\mathring{\eps})(\delta\lam^{-3\kk})^2,
        \end{split}
        \end{equation}
        where $\mathring{\eps} > 0$ is a small constant to be determined later. We also used that $\calE^N(\tau) < 1$ in the bootstrap regime.

        When $j = 1$, we invoke Proposition \ref{prop:DiE1} to have
        \begin{align*}
        \delta\lam^{-3\kk}(\calD_i E_1,\pt^2\calD_i\Theta)_i &= \delta\lam^{-3\kk}\left(\left((1+\kk)\calF^{-\kk-2}(U^0)^{-4}\left(1+\frac{\Theta}{\ze}\right)^4\calL_i\calD_i\Theta\right)(\calG^{\kk}-1), \pt^2\calD_i\Theta\right)_i\\
        &\quad +\delta\lam^{-3\kk}(R_1^i, \pt^2\calD_i\Theta)_i
        \end{align*}
        The second term above can be estimated straightforwardly using \eqref{est:R1i}:
        \begin{equation}
            \label{line6aux2}
            |\delta\lam^{-3\kk}(R_1^i, \pt^2\calD_i\Theta)_i| \lesssim \mathring{\eps}\|\pt^2 \calD_i\Theta\|_i^2 + C(\mathring{\eps})(\delta\lam^{-3\kk})^2 (1+\calE^N(0))\calE^N(\tau).
        \end{equation}
        The first term above can be treated in an identical way to how we deal with the fourth line of \eqref{eq:Sienergy}, now with
        $$
        g(\tau,\ze) = \calF^{-\kk-2}(U^0)^{-4}\left(1+\frac{\Theta}{\ze}\right)^4(\calG^{\kk}-1).
        $$
        By bootstrap assumptions and Lemma \ref{lem: induced bootstrap}, we also have
        $$
        \|g\|_{L^\infty} + \|\pt g\|_{L^\infty} + \|\pz g\|_{L^\infty} \lesssim 1.
        $$
        Then by an argument which is similar to how we treat the fourth line of \eqref{eq:Sienergy}, we may write
        \begin{equation}
            \label{line6aux3}
            \begin{split}
            &\delta\lam^{-3\kk}\left(\left((1+\kk)\calF^{-\kk-2}(U^0)^{-4}\left(1+\frac{\Theta}{\ze}\right)^4\calL_i\calD_i\Theta\right)(\calG^{\kk}-1), \pt^2\calD_i\Theta\right)_i\\
            &\quad = (1+\kk)\frac{d}{d\tau}\left[\delta \lam^{-3\kk}\int_0^1 w^{\frac1\kk + i + 1}\ze^2\left(\calF^{-\kk-2}(U^0)^{-4}\left(1+\frac{\Theta}{\ze}\right)^4(\calG^{\kk}-1)\right)\calD_{i+1}\Theta \pt\calD_{i+1}\Theta d\ze\right]\\
            &\quad\quad + I_{E_1},
            \end{split}
        \end{equation}
        where $|I_{E_1}| \lesssim \mathring{\eps} \|\pt^2 \calD_i\Theta\|_i^2 + C(\mathring{\eps})(\delta \lam^{-3\kk})^2\calE^N(\tau) + (\delta \lam^{-3\kk})^\frac32\calE^N(\tau)$ for $\mathring{\eps} > 0$ a small constant.

        When $j = 3$, we invoke Proposition \ref{prop:DiE3} to write
        \begin{align*}
            \delta\lam^{-3\kk}(\calD_i E_3, \pt^2 \calD_i\Theta)_i &= \delta\lam^{-3\kk}(\calC^{E_3}_1w\ze\calD_{i+1}\pt\Theta, \pt^2 \calD_i\Theta)_i + \delta\lam^{-3\kk}(\calC_2^{E_3}\calL_i\calD_i\Theta, \pt^2 \calD_i\Theta)_i\\
            &\quad + \delta\lam^{-3\kk}(\calC^{E_3}_3w\ze\calD_{i+1}\pt\Theta(0), \pt^2 \calD_i\Theta)_i + \delta\lam^{-3\kk}(\calC_4^{E_3}\calL_i\calD_i\Theta(0), \pt^2 \calD_i\Theta)_i\\
            &\quad + \delta\lam^{-3\kk}(R_{3}^i, \pt^2 \calD_i\Theta)_i,
        \end{align*}
        where we recall that the coefficients $\calC_k^{E_3}$, $k = 1,2,3,4$ are given in \eqref{E3coeff}. 
        \begin{enumerate}
            \item \textit{Bound for $\delta\lam^{-3\kk}(\calC^{E_3}_1w\ze\calD_{i+1}\pt\Theta, \pt^2 \calD_i\Theta)_i$.} By bootstrap assumptions and Lemma \ref{lem: induced bootstrap}, one observes that $\|\calC_1^{E_3}\|_{L^\infty} \lesssim 1$. Since $i \le N-1$, we have
            \begin{equation}
                \label{line6aux4}
                \begin{split}
                |\delta\lam^{-3\kk}(\calC^{E_3}_1w\ze\calD_{i+1}\pt\Theta, \pt^2 \calD_i\Theta)_i| &\le \delta \lam^{-3\kk}\int_0^1 w^{\frac1\kk + i + 1}\ze^3 |\calC_1^{E_3}\calD_{i+1}\pt\Theta\pt^2 \calD_i\Theta| d\ze\\
                &\lesssim (\delta\lam^{-3\kk})^\frac32\calE^N(\tau)^\frac12\|\pt^2 \calD_i\Theta\|_i\\
                &\lesssim \mathring{\eps}\|\pt^2 \calD_i\Theta\|_i^2 + C(\mathring{\eps})(\delta\lam^{-3\kk})^3\calE^N(\tau).
                \end{split}
            \end{equation}
            \item \textit{Bound for $\delta\lam^{-3\kk}(\calC_2^{E_3}\calL_i\calD_i\Theta, \pt^2 \calD_i\Theta)_i$.} This term can be treated in an almost identical way to that treating the fourth line of \eqref{eq:Sienergy}. Note that by bootstrap assumptions and Lemma \ref{lem: induced bootstrap}, we have $\|\calC_2^{E_2}\|_{L^\infty} + \|\pz \calC_2^{E_2}\|_{L^\infty} +\|\pt \calC_2^{E_2}\|_{L^\infty}  \lesssim \delta\lam^{-3\kk}.$ Then we can write
            \begin{equation}
                \label{line6aux5}
                \begin{split}
                \delta\lam^{-3\kk}(\calC_2^{E_3}\calL_i\calD_i\Theta, \pt^2 \calD_i\Theta)_i = \frac{d}{d\tau}\left[\delta \lam^{-3\kk}\int_0^1 w^{\frac1\kk + i + 1}\ze^2 \calC_2^{E_3} \calD_{i+1}\Theta \pt\calD_{i+1}\Theta d\ze\right]+ I_{E_3}^1,
                \end{split}
            \end{equation}
            where $I_{E_3}^1$ obeys the bound
            $$
            |I_{E_3}^1| \lesssim \mathring{\eps} \|\pt^2 \calD_i\Theta\|_i^2 + C(\mathring{\eps})(\delta \lam^{-3\kk})^4\calE^N(\tau) + (\delta \lam^{-3\kk})^\frac52\calE^N(\tau)
            $$
            for $\mathring{\eps} > 0$ a small constant.

            \item \textit{Bound for $\delta\lam^{-3\kk}(\calC^{E_3}_3w\ze\calD_{i+1}\pt\Theta(0), \pt^2 \calD_i\Theta)_i$.} This term can be estimated in an identical way to item (a) above. Due to $\|\calC_3^{E_3}\|_{L^\infty} \lesssim 1$ by bootstrap assumptions and Lemma \ref{lem: induced bootstrap}, we can bound:
            \begin{equation}
                \label{line6aux6}
                |\delta\lam^{-3\kk}(\calC^{E_3}_3w\ze\calD_{i+1}\pt\Theta(0), \pt^2 \calD_i\Theta)_i| \lesssim \mathring{\eps}\|\pt^2 \calD_i\Theta\|_i^2 + C(\mathring{\eps})\delta^3\lam^{-6\kk}\calE^N(0),
            \end{equation}
            for $\mathring{\eps} > 0$ a small constant.

            \item \textit{Bound for $\delta\lam^{-3\kk}(\calC_4^{E_3}\calL_i\calD_i\Theta(0), \pt^2 \calD_i\Theta)_i$.} Emulating the computation \eqref{eq:line4}, we realize that
            \begin{align*}
                &\delta\lam^{-3\kk}(\calC_4^{E_3}\calL_i\calD_i\Theta(0), \pt^2 \calD_i\Theta)_i = \frac{d}{d\tau}\left[\delta \lam^{-3\kk}\int_0^1 w^{\frac1\kk + i + 1}\ze^2 \calC_4^{E_3} \calD_{i+1}\Theta(0)\pt \calD_{i+1}\Theta d\ze\right]\\
                &\quad + 3\kk \delta \lam^{-3\kk}\lamt \int_0^1 w^{\frac1\kk + i + 1}\ze^2 \calC_4^{E_3} \calD_{i+1}\Theta(0)\pt\calD_{i+1}\Theta d\ze\\
                &\quad - \delta \lam^{-3\kk}\int_0^1 w^{\frac1\kk + i + 1}\ze^2 (\pt \calC_4^{E_3})\calD_{i+1}\Theta(0)\pt\calD_{i+1}\Theta d\ze\\
                &\quad + \delta \lam^{-3\kk}\int_0^1 w^{\frac1\kk + i + 1}\ze^2 (\pz \calC_4^{E_3}) \calD_{i+1}\Theta(0)\pt^2\calD_{i}\Theta d\ze\\
                &=: \frac{d}{d\tau}\left[\delta \lam^{-3\kk}\int_0^1 w^{\frac1\kk + i + 1}\ze^2 \calC_4^{E_3} \calD_{i+1}\Theta(0)\pt \calD_{i+1}\Theta d\ze\right] + I_{E_3}^2.
            \end{align*}
            By bootstrap assumptions and Lemma \ref{lem: induced bootstrap}, it is straightforward to see that
            $$
            \|\calC_4^{E_3}\|_{L^\infty} + \|\pt\calC_4^{E_3}\|_{L^\infty} + \|\pz\calC_4^{E_3}\|_{L^\infty} \lesssim \delta.
            $$
            We may then bound $I_{E_3}^2$ by
            \begin{equation}\label{line6aux7}
            \begin{split}
            |I_{E_3}^2| &\lesssim \delta^2\lam^{-3\kk}\left(\delta^\frac12 \lam^{-\frac32\kk} \calE^N(0)^\frac12 \calE^N(\tau)^\frac12  + \calE^N(0)^\frac12 \|\pt^2\calD_i\Theta\|_i\right)\\
            &\le \delta^2\lam^{-3\kk} \calE^N(\tau) + \mathring{\eps}\|\pt^2\calD_i\Theta\|_i^2 + C(\mathring{\eps})(\delta^2\lam^{-3\kk})^2\calE^N(0).
            \end{split}
            \end{equation}

            \item \textit{Bound for $\delta\lam^{-3\kk}(R_{3}^i, \pt^2 \calD_i\Theta)_i$.} Invoking \eqref{est:R3i} and directly using Cauchy-Schwarz inequality, we have
            \begin{equation}
                \label{line6aux8}
                |\delta\lam^{-3\kk}(R_{3}^i, \pt^2 \calD_i\Theta)_i| \lesssim  \mathring{\eps}\|\pt^2 \calD_i\Theta\|_i^2 + C(\mathring{\eps})(\delta\lam^{-3\kk})^2(1+\calE^N(\tau) + \calE^N(0)),
            \end{equation}
            for $\mathring{\eps} > 0$ a small constant.
        \end{enumerate}
        Finally, when $j = 4$, we invoke Proposition \ref{prop:DiE4} to write
        \begin{align*}
            \delta\lam^{-3\kk}(\calD_i E_4, \pt^2 \calD_i\Theta)_i &= \delta\lam^{-3\kk}\left((\calC^{E_4}_1+\calC^{E_4}_2)w\ze\calD_{i+1}\pt\Theta,\pt^2\calD_i\Theta\right)_i\\
            &\quad +\delta\lam^{-3\kk}\left(\calC^{E_4}_3w\calD_i\pt^2\Theta,\pt^2\calD_i\Theta\right)_i\\
            &\quad + \delta\lam^{-3\kk}\left(R_{4}^i,\pt^2\calD_i\Theta\right)_i,
        \end{align*}
        where we recall that the coefficients $\calC_k^{E_4}$ are given in \eqref{E4coeff}. By bootstrap assumptions and Lemma \ref{lem: induced bootstrap}, we see that $\|\calC_k^{E_4}\|_{L^\infty} \lesssim 1$ in the bootstrap regime. Then we immediately have
        $$
        \left|\delta\lam^{-3\kk}\left((\calC^{E_4}_1+\calC^{E_4}_2)w\ze\calD_{i+1}\pt\Theta,\pt^2\calD_i\Theta\right)_i\right| \lesssim \mathring{\eps}\|\pt^2 \calD_i\Theta\|_i^2 + C(\mathring{\eps})(\delta\lam^{-3\kk})^3\calE^N(\tau),
        $$
        for $\mathring{\eps} > 0$ sufficiently small, and
        $$
        \left| \delta\lam^{-3\kk}\left(\calC^{E_4}_3w\calD_i\pt^2\Theta,\pt^2\calD_i\Theta\right)_i\right| \lesssim \delta\lam^{-3\kk} \|\pt^2\calD_i\Theta\|_i^2.
        $$
        Finally using \eqref{est:R4i}, we have
        \begin{align*}
        \left|\delta\lam^{-3\kk}\left(R_{4}^i,\pt^2\calD_i\Theta\right)_i\right| &\lesssim \delta\lam^{-3\kk}\left(\delta\lam^{-3\kk} + (\delta\lam^{-3\kk})^\frac12 \calE^N(\tau)^\frac12 + \calS^N(\tau)^\frac12\right)\|\pt^2\calD_i\Theta\|_i\\
        &\lesssim \mathring{\eps}\|\pt^2\calD_i\Theta\|_i^2 + C(\mathring{\eps})\left[(\delta\lam^{-3\kk})^4 + (\delta\lam^{-3\kk})^3\calE^N(\tau)\right] + \delta\lam^{-3\kk}\calS^N(\tau)^\frac12\|\pt^2\calD_i\Theta\|_i.
        \end{align*}
    \end{enumerate}
    Now, we take \eqref{eq:Sienergy}, sum over $i = 0,\hdots, N-1$, and integrate over the time interval $[\tau_1,\tau_2]$. By combining all estimates above, we arrive at, for $\mathring{\eps} > 0$ and a constant $C$ depending on $N$:
    \begin{equation}
        \label{est:SN0}
        \begin{split}
        \frac{1}{2}S^N(\tau_2;\tau_1) &\le C(\mathring{\eps}+ \delta\lam^{-3\kk}) S^N(\tau_2;\tau_1) + C(\mathring{\eps})\left[\delta \int_{\tau_1}^{\tau_2}\lam^{-3\kk}\calE^N(\tau)d\tau + \int_{\tau_1}^{\tau_2}(\delta\lam^{-3\kk})^2 + \delta^3\lam^{-6\kk}\calE^N(0) d\tau\right]\\
        & + \left|(1+\kk)\delta \lam^{-3\kk}\int_0^1 w^{\frac1\kk + i + 1}\ze^2\left(\calF^{-\kk-2}(U^0)^{-4}\left(1+\frac{\Theta}{\ze}\right)^4\calG^{\kk}\right)\calD_{i+1}\Theta \pt\calD_{i+1}\Theta d\ze\bigg|_{\tau = \tau_1}\right|\\
        & + \left|(1+\kk)\delta \lam^{-3\kk}\int_0^1 w^{\frac1\kk + i + 1}\ze^2\left(\calF^{-\kk-2}(U^0)^{-4}\left(1+\frac{\Theta}{\ze}\right)^4\calG^{\kk}\right)\calD_{i+1}\Theta \pt\calD_{i+1}\Theta d\ze\bigg|_{\tau = \tau_2}\right|\\
        & + \left|\delta \lam^{-3\kk}\int_0^1 w^{\frac1\kk + i + 1}\ze^2 \calC_2^{E_3} \calD_{i+1}\Theta \pt\calD_{i+1}\Theta d\ze\bigg|_{\tau = \tau_1}\right|\\
        & + \left|\delta \lam^{-3\kk}\int_0^1 w^{\frac1\kk + i + 1}\ze^2 \calC_2^{E_3} \calD_{i+1}\Theta \pt\calD_{i+1}\Theta d\ze\bigg|_{\tau = \tau_2}\right|\\
        & + \left|\delta \lam^{-3\kk}\int_0^1 w^{\frac1\kk + i + 1}\ze^2 \calC_4^{E_3} \calD_{i+1}\Theta(0)\pt \calD_{i+1}\Theta d\ze\bigg|_{\tau = \tau_1}\right|\\
        & + \left|\delta \lam^{-3\kk}\int_0^1 w^{\frac1\kk + i + 1}\ze^2 \calC_4^{E_3} \calD_{i+1}\Theta(0)\pt \calD_{i+1}\Theta d\ze\bigg|_{\tau = \tau_2}\right|.
        \end{split}
    \end{equation}
    First by \eqref{est:lambdaasym}, $\lam^{-6\kk}$ is integrable on $[0,\infty)$. Then we have
    $$
    \int_{\tau_1}^{\tau_2}(\delta\lam^{-3\kk})^2 + \delta^3\lam^{-6\kk}\calE^N(0) d\tau \le C(\delta^2 + \delta^3\calE^N(0)),
    $$
    where $C > 0$ is independent of $\tau_1,\tau_2$.
    Next, since 
    $$
    \left\|\calF^{-\kk-2}(U^0)^{-4}\left(1+\frac{\Theta}{\ze}\right)^4\calG^{\kk}\right\|_{L^\infty} + \|\calC_2^{E_3}\|_{L^\infty} + \|\calC_4^{E_3}\|_{L^\infty} \lesssim 1,
    $$
    uniformly for all $\tau \in [0,\tau_*]$ by the bootstrap assumptions and and Lemma \ref{lem: induced bootstrap}, we may apply Cauchy-Schwarz inequality to bound terms on the temporal boundary $\tau = \tau_{1,2}$ by
    $$
    (\delta\lam^{-3\kk})^\frac32 \calE^N(\tau_1) + (\delta\lam^{-3\kk})^\frac32 \calE^N(\tau_2) + (\delta\lam^{-3\kk})^\frac32 \calE^N(0).
    $$
    Finally, by choosing $\mathring{\eps}, \delta$ sufficiently small such that $C(\mathring{\eps} +\delta\lam^{-3\kk}) \le \frac14$, we conclude, after plugging in all estimates above to \eqref{est:SN0} and rearranging, that
    \begin{align*}
        S^N(\tau_2;\tau_1) &\lesssim (\delta\lam^{-3\kk})^\frac32 \calE^N(\tau_1) + (\delta\lam^{-3\kk})^\frac32 \calE^N(\tau_2) + (\delta\lam^{-3\kk})^\frac32 \calE^N(0) + \delta^2+\delta^3\calE^N(0) + \delta\int_{\tau_1}^{\tau_2} \lam^{-3\kk} \calE^N(\tau) d\tau\\
        &\lesssim \delta^{\frac32}(\calE^N(\tau_1) + \calE^N(\tau_2) + \calE^N(0)) + \delta^2 + \delta\int_{\tau_1}^{\tau_2} \lam^{-3\kk} \calE^N(\tau) d\tau,
    \end{align*}
    which is the desired inequality.
\end{proof}

\section{Energy Estimates V: Main Energy Inequality}\label{sect: Energy Estimate V}
In this section, we establish the main energy inequality by utilizing all estimates established in Section \ref{sect: Energy Estimate 1}--\ref{sect: Energy Estimate 4}. We summarize this crucial inequality in the following proposition:
\begin{prop}[Main Energy Inequality]\label{prop:mainEN0}
    Let $(\Theta,\pt\Theta)$ be the solution to \eqref{eq:momlagmainfull} on $[0,\tau_*]$ verifying the bootstrap assumptions. There exists $\eps'' \in (0,1)$ sufficiently small so that for any $\eps \le \eps''$, $\delta \le \eps^\frac12$, the following statement holds: for any time instances $\tau_0, \tau$ such that $0 \le \tau_0 \le \tau \le \tau_*$, then
    \begin{equation}
        \label{est:mainEN0}
        \begin{split}
        \calE^N(\tau) + \int_{\tau_0}^\tau \mfD^N(\tilde\tau)d\tilde\tau &\le C_0(\calE^N(\tau_0) + \calE^N(0) + \delta^2)\\
        &\quad + C_1 \left[\delta^\frac12\int_{\tau_0}^\tau e^{-\frac32\kk\bar\lam \tilde\tau} (\calE^N(\tilde\tau))^\frac12 d\tilde\tau + \int_{\tau_0}^\tau e^{-\frac32\kk\bar\lam \tilde\tau}\calE^N(\tilde\tau) d\tilde\tau\right],
        \end{split}
    \end{equation}
    where $C_0, C_1$ are positive constants that are independent of parameters $\eps, \delta$, and $M_*$.
\end{prop}
\begin{proof}
    We start by recalling the energy balance \eqref{eq: high-order energy identity}: for $0 \le i \le N$,
    \begin{equation}
        \label{eq:ienergybalance}
        \begin{split}
        \frac{d}{d\tau}\mfE_i(\tau) + \mfD_i(\tau) &=  -(\barcalD_{i-1}\mathfrak{M}, \p_{\tau}\calD_i\Theta)_i - \sum_{j=1}^2 (\calD_i\mathfrak{R}_j(\Theta), \p_\tau\calD_i\Theta)_i-\sum_{j=1}^6(\calD_iE_j, \p_{\tau}\calD_i\Theta)_i\\
        &\quad-\sum_{j=1}^3(\mfC_j, \p_{\tau}\calD_i\Theta)_i -\sum_{j=1}^5\mfI_j.
        \end{split}
    \end{equation}
    Using Proposition \ref{prop:est M}, Proposition \ref{prop:est R}, Proposition \ref{prop:est C12}, and Proposition \ref{prop:est C}, we have the following bound after a use of Cauchy-Schwarz inequality:
    \begin{equation}
        \label{energyineqaux1}
        \begin{split}
        &\left|(\barcalD_{i-1}\mathfrak{M}, \p_{\tau}\calD_i\Theta)_i + \sum_{j=1}^2 (\calD_i\mathfrak{R}_j(\Theta), \p_\tau\calD_i\Theta)_i +\sum_{j=1}^3(\mfC_j, \p_{\tau}\calD_i\Theta)_i +\sum_{j=1}^5\mfI_j\right|\\
        &\quad \lesssim (\delta\lam^{-3\kk})^\frac12\left((\calE^N(\tau))^\frac12 + (\calS^N(\tau))^\frac12\right)(\calE^N(\tau))^\frac12 + \lam^{-\frac32\kk}\calE^N(\tau)\\
        &\quad \lesssim \lam^{-\frac32\kk}\calE^N(\tau) + \calS^N(\tau),
        \end{split}
    \end{equation}
    where we used that $\delta \lam^{-3\kk} + \delta \lesssim 1$. Using Proposition \ref{prop:DiE1}, Proposition \ref{prop:DiE2}, Proposition \ref{prop:DiE3}, Proposition \ref{prop:DiE4}, and Proposition \ref{prop:DiE5+6}, we write
    \begin{equation}
        \label{energyineqaux2}
        \begin{split}
            -\sum_{j=1}^6(\calD_iE_j, \p_{\tau}\calD_i\Theta)_i &= -\left((\calC_1^{E_4} + \calC_2^{E_4} +\calC_1^{E_3})w\ze\calD_{i+1}\pt\Theta, \calD_i\pt\Theta\right)_i\\
            &\quad -\left(\calC_3^{E_4}\calD_{i}\pt^2\Theta, \calD_i\pt\Theta\right)_i - \left(g\calL_{i}\calD_i\Theta, \calD_i\pt\Theta\right)_i\\
            &\quad - \left(\calC_3^{E_3}w\ze\calD_{i+1}\pt\Theta(0), \calD_i\pt\Theta\right)_i - \left(\calC_4^{E_3}\calL_{i}\calD_i\Theta(0), \calD_i\pt\Theta\right)_i\\
            &\quad + (\mfR_E,\calD_i\pt\Theta)_i,
        \end{split}
    \end{equation}
    where the coefficients $\calC_k^{E_j}$ are defined in \eqref{E4coeff} and \eqref{E3coeff}, and the function $g$ is given by
    \begin{equation}\label{gdefn}
    g(\tau,\ze) := (1+\kk)\calF^{-\kk-2}(U^0)^{-4}\left(1+\frac{\Theta}{\ze}\right)^4(\calG^\kk - 1) +\calC_{2}^{E_3}.
    \end{equation}
    Moreover, $\mfR_E$ obeys the following bound for all $\tau \in [0,\tau_*]$:
    \begin{equation}
        \label{energyineqaux3}
        \|\mfR_E\|_i^2 \lesssim 1 + \calE^N(0) + \calS^N(\tau) + \calE^N(\tau).
    \end{equation}
    Then we are left to study the leading order terms appearing in \eqref{energyineqaux2}. 
    \begin{enumerate}
        \item $\left((\calC_1^{E_4} + \calC_2^{E_4} +\calC_1^{E_3})w\ze\calD_{i+1}\pt\Theta, \calD_i\pt\Theta\right)_i$. It suffices to consider the case where $i$ is even, as the $i$ odd case is strictly easier. Writing $\calD_{i+1} = \Dz \calD_i = \pz\calD_i + \frac{2}{\ze}\calD_i$, we have
        \begin{align*}
            &-\left((\calC_1^{E_4} + \calC_2^{E_4} +\calC_1^{E_3})w\ze\calD_{i+1}\pt\Theta, \calD_i\pt\Theta\right)_i = -\frac12\int_0^1 (\calC_1^{E_4} + \calC_2^{E_4} +\calC_1^{E_3}) w^{\frac1\kk + i+1}\ze^3 \pz\left(\calD_i\pt\Theta\right)^2d\ze\\
            &\quad - 2\int_0^1 (\calC_1^{E_4} + \calC_2^{E_4} +\calC_1^{E_3}) w^{\frac1\kk + i+1} \ze^2 \left(\calD_i\pt\Theta\right)^2d\ze\\
            &=\frac12 \int_0^1 \pz\left(\calC_1^{E_4} + \calC_2^{E_4} +\calC_1^{E_3}\right)w^{\frac1\kk + i+1}\ze^3 \left(\calD_i\pt\Theta\right)^2d\ze\\
            &\quad +\frac{\frac1\kk + i + 1}{2} \int_0^1 \left(\calC_1^{E_4} + \calC_2^{E_4} +\calC_1^{E_3}\right)w'w^{\frac1\kk + i}\ze^3 \left(\calD_i\pt\Theta\right)^2d\ze\\
            &\quad -\frac{1}{2}\int_0^1 \left(\calC_1^{E_4} + \calC_2^{E_4} +\calC_1^{E_3}\right)w^{\frac1\kk + i+1}\ze^2 \left(\calD_i\pt\Theta\right)^2d\ze.
        \end{align*}
        Since $\|\calC_1^{E_4}\|_{W^{1,\infty}} + \|\calC_2^{E_4}\|_{W^{1,\infty}} + \|\calC_1^{E_3}\|_{W^{1,\infty}} \lesssim 1$ by bootstrap assumptions and Lemma \ref{lem: induced bootstrap}, we immediately have
        \begin{equation}
            \label{leadingorder1}
            \left|\left((\calC_1^{E_4} + \calC_2^{E_4} +\calC_1^{E_3})w\ze\calD_{i+1}\pt\Theta, \calD_i\pt\Theta\right)_i\right| \lesssim \|\calD_i\pt\Theta\|_i^2 \le \delta\lam^{-3\kk}\calE^N(\tau).
        \end{equation}

        \item $-\left(\calC_3^{E_4}\calD_{i}\pt^2\Theta, \calD_i\pt\Theta\right)_i$. We observe that
        \begin{equation}\label{leadingorder2}
        \begin{split}
            -\left(\calC_3^{E_4}\calD_{i}\pt^2\Theta, \calD_i\pt\Theta\right)_i &= -\frac12\pt\left(\calC_3^{E_4}\calD_i\pt\Theta,\calD_i\pt\Theta\right)_i + \frac12  \left(\pt\calC_3^{E_4} \calD_i\pt\Theta,\calD_i\pt\Theta\right)\\
            &=: -\frac12\pt\left(\calC_3^{E_4}\calD_i\pt\Theta,\calD_i\pt\Theta\right)_i + \mfI_{leading}^1.
            \end{split}
        \end{equation}
        By definition of $\calC_3^{E_4}$, bootstrap assumptions, and Lemma \ref{lem: induced bootstrap}, we have $\|\pt \calC_3^{E_4}\|_{L^\infty} \lesssim 1$. Then the following bound holds:
        $$|\mfI_{leading}^1| \lesssim \delta\lam^{-3\kk}\calE^N(\tau).$$

        \item $- \left(g\calL_{i}\calD_i\Theta, \calD_i\pt\Theta\right)_i$. Following \eqref{ellipticibp}, we have
        \begin{equation}\label{leadingorder3}
        \begin{split}
            -\left(g\calL_{i}\calD_i\Theta, \calD_i\pt\Theta\right)_i &=-\frac{1}{2}\frac{d}{d\tau}\int_0^1w^{1+\frac1\kk+i}\ze^2g|\calD_{i+1}\Theta|^2 + \frac{1}{2}\int_0^1w^{1+\frac1\kk+i}\ze^2(\p_{\tau}g)|\calD_{i+1}\Theta|^2\\
    &-\int_0^1w^{1+\frac1\kk+i}\ze^2(\pz g)(\p_{\tau}\calD_i\Theta)(\calD_{i+1}\Theta)\\
    &=: -\frac{1}{2}\pt (g\calD_{i+1}\Theta, \calD_{i+1}\Theta)_{i+1} + \mfI_{leading}^2.            
        \end{split}
        \end{equation}
        By bootstrap assumptions, Lemma \ref{lem: induced bootstrap}, and the definition of $g$, it is not difficult to see the following estimate:
        $$
        \|\pz g\|_{L^\infty} \lesssim 1,\quad \|\pt g\|_{L^\infty} \lesssim \lam^{-\frac32}.
        $$
        Then an application of Cauchy-Schwarz inequality yields:
        \begin{equation}\label{Ileading2}
            |\mfI_{leading}^2| \lesssim (\delta\lam^{-3\kk})^\frac12 \calE^N(\tau).
        \end{equation}

        \item $- \left(\calC_3^{E_3}w\ze\calD_{i+1}\pt\Theta(0), \calD_i\pt\Theta\right)_i$. It suffices to consider $i$ to be even. We observe that
        \begin{align*}
            &- \left(\calC_3^{E_3}w\ze\calD_{i+1}\pt\Theta(0), \calD_i\pt\Theta\right)_i = \int_0^1 \calC_3^{E_3} w^{\frac1\kk + i + 1}\ze^3\calD_i \pt\Theta(0)\calD_{i+1}\pt\Theta d\ze\\
            &\quad + \int_0^1 \pz\left(\calC_3^{E_3} w^{\frac1\kk + i + 1}\ze^3\right) \calD_i \pt\Theta(0) \calD_i\pt\Theta d\ze-2\int_0^1 \calC_3^{E_3} w^{\frac1\kk + i + 1}\ze^2\calD_i \pt\Theta(0)\calD_i\pt\Theta d\ze\\
            &= \frac{d}{d\tau} \int_0^1 \calC_3^{E_3} w^{\frac1\kk + i + 1}\ze^3\calD_i \pt\Theta(0)\calD_{i+1}\Theta d\ze- \int_0^1 (\pt\calC_3^{E_3}) w^{\frac1\kk + i + 1}\ze^3\calD_i \pt\Theta(0)\calD_{i+1}\Theta d\ze\\
            &\quad + \int_0^1 \pz\left(\calC_3^{E_3} w^{\frac1\kk + i + 1}\ze^3\right) \calD_i \pt\Theta(0) \calD_i\pt\Theta d\ze-2\int_0^1 \calC_3^{E_3} w^{\frac1\kk + i + 1}\ze^2\calD_i \pt\Theta(0)\calD_i\pt\Theta d\ze\\
            &=: \frac{d}{d\tau} \int_0^1 \calC_3^{E_3} w^{\frac1\kk + i + 1}\ze^3\calD_i \pt\Theta(0)\calD_{i+1}\Theta d\ze + \mfI_{leading}^3.
        \end{align*}
        By definition of $\calC_3^{E_3}$, the bootstrap assumptions, and Lemma \ref{lem: induced bootstrap}, we have the following bounds:
        $$
        \|\calC_3^{E_3}\|_{W^{1,\infty}} \lesssim 1,\quad \|\pt \calC_3^{E_3}\|_{L^\infty} \lesssim \lam^{-\frac32\kk},
        $$
        from which we deduce that
        \begin{equation}
            \label{Ileading3}
            |\mfI_{leading}^3| \lesssim (\delta \lam^{-3\kk})^\frac12 (\calE^N(0))^\frac12(\calE^N(\tau))^\frac12 \lesssim (\delta \lam^{-3\kk})^\frac12(\calE^N(0) + \calE^N(\tau)).
        \end{equation}

        \item $- \left(\calC_4^{E_3}\calL_{i}\calD_i\Theta(0), \calD_i\pt\Theta\right)_i$. Following \eqref{ellipticibp}, we note that
        \begin{align*}
            &- \left(\calC_4^{E_3}\calL_{i}\calD_i\Theta(0), \calD_i\pt\Theta\right)_i = -\frac{d}{d\tau}\int_0^1 w^{\frac1\kk + i + 1}\ze^2 \calC_4^{E_3} \calD_{i+1}\Theta(0)\calD_{i+1}\Theta d\ze\\
            &\quad + \int_0^1 w^{\frac1\kk + i + 1}\ze^2(\pt \calC_4^{E_3})\calD_{i+1}\Theta(0)\calD_{i+1}\Theta d\ze- \int_0^1 w^{\frac1\kk + i + 1}\ze^2(\pz \calC_4^{E_3})\calD_{i+1}\Theta(0)\calD_{i}\pt\Theta d\ze\\
            &=: -\frac{d}{d\tau}\int_0^1 w^{\frac1\kk + i + 1}\ze^2 \calC_4^{E_3} \calD_{i+1}\Theta(0)\calD_{i+1}\Theta d\ze + \mfI_{leading}^4.
        \end{align*}
        Again by the definition of $\calC_4^{E_3}$, the bootstrap assumptions and Lemma \ref{lem: induced bootstrap}, we have the bounds
        $$
        \|\calC_4^{E_3}\|_{W^{1,\infty}} \lesssim \delta,\quad \|\pt \calC_4^{E_3}\|_{L^\infty} \lesssim \delta \lam^{-\frac32\kk}.
        $$
        Then we have
        \begin{equation}
            \label{Ileading4}
            |\mfI_{leading}^4| \lesssim \delta \lam^{-\frac32\kk}(\calE^N(0) + \calE^N(\tau)).
        \end{equation}

    \end{enumerate}

    Now, we integrate \eqref{eq:ienergybalance} in time on $[\tau_0,\tau]$, rearranging, and applying \eqref{energyineqaux1}--\eqref{Ileading4}, we obtain
    \begin{equation}\label{est:mfEiaux1}
    \begin{split}
        \mfE_i(\tau) &+ \int_{\tau_0}^\tau \mfD_i(\tilde\tau) d\tilde\tau + \frac12 \left(\calC_3^{E_4}\calD_i\pt\Theta, \calD_i\pt\Theta\right)_i(\tau) + \frac12 \left(g\calD_{i+1}\Theta,\calD_{i+1}\Theta\right)_i(\tau)\\
        &\quad-\int_0^1 \calC_3^{E_3}(\tau) w^{\frac1\kk + i + 1}\ze^3\calD_i \pt\Theta(0)\calD_{i+1}\Theta(\tau) d\ze+\int_0^1 w^{\frac1\kk + i + 1}\ze^2 \calC_4^{E_3}(\tau) \calD_{i+1}\Theta(0)\calD_{i+1}\Theta(\tau) d\ze\\
        &\le \frac12 \left|\left(\calC_3^{E_4}\calD_i\pt\Theta, \calD_i\pt\Theta\right)_i(\tau_0)\right| + \frac12\left|\left(g\calD_{i+1}\Theta,\calD_{i+1}\Theta\right)_i(\tau_0)\right|\\
        &\quad +\left|\int_0^1 \calC_3^{E_3}(\tau_0) w^{\frac1\kk + i + 1}\ze^3\calD_i \pt\Theta(0)\calD_{i+1}\Theta(\tau_0) d\ze\right|+ \left|\int_0^1 w^{\frac1\kk + i + 1}\ze^2 \calC_4^{E_3}(\tau_0) \calD_{i+1}\Theta(0)\calD_{i+1}\Theta(\tau_0) d\ze\right|\\
        &\quad + C\bigg[\delta^\frac12\int_{\tau_0}^\tau \lam^{-\frac32\kk}(\calE^N(\tilde\tau))^\frac12\left(1+\calS^N(\tilde\tau) + \calE^N(\tilde\tau)+\calE^N(0)\right)^\frac12 d\tilde\tau + \delta^\frac12\int_{\tau_0}^\tau \lam^{-\frac32\kk} \calE^N(0) d\tilde\tau\\
        &\quad + \int_{\tau_0}^\tau \lam^{-\frac32\kk}\calE^N(\tilde\tau) d\tilde\tau + \int_{\tau_0}^\tau \calS^N(\tilde\tau) d\tilde\tau\bigg]\\
        &\le \bar{C}\left(\calE^N(\tau_0) + \calE^N(0)\right) + C\left[\delta^\frac12\int_{\tau_0}^\tau e^{-\frac32\kk\bar\lam \tilde\tau} (\calE^N(\tilde\tau))^\frac12 d\tilde\tau + \int_{\tau_0}^\tau \calS^N(\tilde\tau) d\tilde\tau\right]\\
        &\le \bar{C}\left(\delta^\frac32\calE^N(\tau) + \calE^N(\tau_0) + \calE^N(0) + \delta^2\right) + C\left[\delta^\frac12\int_{\tau_0}^\tau e^{-\frac32\kk\bar\lam \tilde\tau} (\calE^N(\tilde\tau))^\frac12 d\tilde\tau+\int_{\tau_0}^\tau e^{-\frac32\kk\bar\lam \tilde\tau}\calE^N(\tilde\tau) d\tilde\tau\right],
        \end{split}
    \end{equation}
    where $C,\bar C > 0$ are positive constants which are independent of $M_*,\eps,\delta$. We remark that we have used bootstrap assumptions, Lemma \ref{lem: induced bootstrap}, the equivalence \eqref{est:lambdaasym}, and the integrability of $\lam^{-\frac32\kk}$ on $[0,\infty]$ in the second inequality above; we also used Proposition \ref{prop:SN} in the final inequality.

    In view of the LHS of \eqref{est:mfEiaux1}, we immediately infer from Lemma \ref{lem: induced bootstrap} that
    \begin{align*}
        \left|\frac12 \left(\calC_3^{E_4}\calD_i\pt\Theta, \calD_i\pt\Theta\right)_i(\tau)\right| &\lesssim \delta\lam^{-3\kk}\calE^N(\tau) \lesssim \delta \calE^N(\tau),\\
        \left|\int_0^1 \calC_3^{E_3}(\tau) w^{\frac1\kk + i + 1}\ze^3\calD_i \pt\Theta(0)\calD_{i+1}\Theta(\tau) d\ze\right| &\lesssim \delta^\frac12(\calE^N(0))^\frac12(\calE^N(\tau))^\frac12 \lesssim \calE^N(0) + \delta\calE^N(\tau),\\
        \left|\int_0^1 w^{\frac1\kk + i + 1}\ze^2 \calC_4^{E_3}(\tau) \calD_{i+1}\Theta(0)\calD_{i+1}\Theta(\tau) d\ze\right| &\lesssim (\calE^N(0))^\frac12(\calE^N(\tau))^\frac12 \le \mathring{\eps}\calE^N(\tau) + C(\mathring{\eps})\calE^N(0),
    \end{align*}
    where $\mathring{\eps} > 0$ is a small parameter to be determined. Using the above bounds, we then bound the LHS of \eqref{est:mfEiaux1} from below as follows:
    \begin{equation}\label{est:LHSEi}
    \begin{split}
    \text{LHS of }\eqref{est:mfEiaux1} &\ge \int_{\tau_0}^\tau \mfD_i(\tilde\tau)d\tilde\tau + \frac{1}{2}\delta^{-1}\lam^{3\kk}\int_0^1 w^{\frac1\kk+i}\ze^2(1+\delta\lam^{-3\kk}w\calF^{-\kk}\calG^{\kk})|\calD_i\p_{\tau}\Theta|^2d\ze \\
    &\quad+ \frac{1}{2}\int_0^1 w^{\frac1\kk+i}\ze^2(1+\delta\lam^{-3\kk}w\calF^{-\kk}\calG^{\kk})|\calD_i\Theta|^2d\ze\\
    &\quad+\frac{1+\kk}{2}\int_0^1w^{1+\frac1\kk+i}\ze^2\left(\calF^{-\kk-2}(U^0)^{-4}\left(1+\frac{\Theta}{\ze}\right)^4\right)|\calD_{i+1}\Theta|^2d\ze\\
    &\quad+\frac12 \left(g\calD_{i+1}\Theta,\calD_{i+1}\Theta\right)_i(\tau) - (\mathring{\eps} + \delta)\calE^N(\tau) - C(\mathring{\eps})\calE^N(0)\\
    &=\int_{\tau_0}^\tau \mfD_i(\tilde\tau)d\tilde\tau + \frac{1}{2}\delta^{-1}\lam^{3\kk}\int_0^1 w^{\frac1\kk+i}\ze^2(1+\delta\lam^{-3\kk}w\calF^{-\kk}\calG^{\kk})|\calD_i\p_{\tau}\Theta|^2d\ze \\
    &\quad+ \frac{1}{2}\int_0^1 w^{\frac1\kk+i}\ze^2(1+\delta\lam^{-3\kk}w\calF^{-\kk}\calG^{\kk})|\calD_i\Theta|^2d\ze\\
    &\quad+\frac{1+\kk}{2}\int_0^1w^{1+\frac1\kk+i}\ze^2\left[\left(\calF^{-\kk-2}(U^0)^{-4}\left(1+\frac{\Theta}{\ze}\right)^4\calG^\kk\right) + \calC_2^{E_3}\right]|\calD_{i+1}\Theta|^2d\ze\\
    &\quad - (\mathring{\eps} + \delta)\calE^N(\tau) - C(\mathring{\eps})\calE^N(0)\\
    &\ge c_0\calE_i(\tau) +\int_{\tau_0}^\tau \mfD_i(\tilde\tau)d\tilde\tau - (\mathring{\eps} + \delta)\calE^N(\tau) - C(\mathring{\eps})\calE^N(0),
    \end{split}
    \end{equation}
    where $c_0 > 0$ is a constant independent of $M_*, \eps,\delta$. Note that we have used the definition of $g$ \eqref{gdefn} in the equality above. In the final inequality above, we used the bootstrap assumption that $\calF, U^0,\calG \approx 1$, the bound $\|\calC_2^{E_3}\|_{L^\infty} \lesssim \delta \ll1$, and the energy-norm equivalence (i.e., Corollary \ref{cor: energy norm equiva}). Now, let us choose $\mathring{\eps} + \delta < \frac{c_0}{2}$. Combining \eqref{est:mfEiaux1} and \eqref{est:LHSEi}, rearranging, and summing over $i = 0,\hdots, N$, we arrive at
    \begin{align*}
    \frac{c_0}{2}\calE^N(\tau) +\int_{\tau_0}^\tau \mfD^N(\tilde\tau)d\tilde\tau &\le \bar{C}N\left(\delta^\frac32\calE^N(\tau) + \calE^N(\tau_0) + \calE^N(0) + \delta^2\right)\\
    &\quad +CN\left[\delta^\frac12\int_{\tau_0}^\tau e^{-\frac32\kk\bar\lam \tilde\tau} (\calE^N(\tilde\tau))^\frac12 d\tilde\tau + \int_{\tau_0}^\tau e^{-\frac32\kk\bar\lam \tilde\tau}\calE^N(\tilde\tau) d\tilde\tau\right]
    \end{align*}
    Choosing $\delta$ even smaller such that $\bar{C}N\delta^\frac32 \le \min\{\frac{c_0}{4},\frac12\}$, we deduce from the above inequality that
    \begin{align*}
    \calE^N(\tau)+\int_{\tau_0}^\tau \mfD^N(\tilde\tau)d\tilde\tau &\le \frac{4\bar C N}{c_0}\left(\calE^N(\tau_0) + \calE^N(0) + \delta^2\right)\\
    &\quad+ \frac{4CN}{c_0}\left[\delta^\frac12\int_{\tau_0}^\tau e^{-\frac32\kk\bar\lam \tilde\tau} (\calE^N(\tilde\tau))^\frac12 d\tilde\tau + \int_{\tau_0}^\tau e^{-\frac32\kk\bar\lam \tilde\tau}\calE^N(\tilde\tau) d\tilde\tau\right]
    \end{align*}
    Then the main result \eqref{est:mainEN0} follows after we fix $C_0 = \frac{4\bar C N}{c_0}$ and $C_1 = \frac{4CN}{c_0}.$
\end{proof}

\section{Local Existence and Uniqueness}
\label{S:LWP}
In this section we state the basic local existence and continuation criterion results concerning the main equation \eqref{eq:momlagmain} that are needed for the Main Theorem stated and proven in Section \ref{sect: mainthm}. In a nutshell, the result follows from \cite{disconzi2022relativistic}, where local well-posedness of the relativistic Euler equations with a physical vacuum boundary condition has been established, along with a continuation criterion. 
In order to precisely appeal to \cite{disconzi2022relativistic}, we first need some definitions.

Consider a domain $\EulerianDomain \subset \mathbb{R}^3$ given by a defining function $\DefFunction \colon \mathbb{R}^3 \rightarrow [0,\infty)$, i.e., $\EulerianDomain = \{ \DefFunction > 0 \}$ and $\left. \partial_\nu \DefFunction \right|_{\partial \EulerianDomain} \neq 0$, where $\nu$ is the inner outer normal to $\partial \EulerianDomain$. Let $k \geq 0$ be an integer and $\sigma > -\frac{1}{2}$ a real number. We define $H_\DefFunction^{k,\sigma}(\EulerianDomain)$ as the space of distributions $\GenericFunction$ on $\EulerianDomain$ whose norm
\begin{align*}
    \| \GenericFunction \|^2_{H_\DefFunction^{k,\sigma}(\EulerianDomain)} := \sum_{|\alpha| \leq k} \| \DefFunction^\sigma \partial^\alpha \GenericFunction \|^2_{L^2(\EulerianDomain)} 
\end{align*}
is finite. 
We note the following embedding from \cite{disconzi2022relativistic}
\begin{align}
    \label{E:Weighted_embedding_DIT}
    H_\DefFunction^{k,\sigma}(\EulerianDomain)
    \subset C^s(\EulerianDomain), \, s < k-\sigma-\frac{3}{2}.
\end{align}
Let us denote
\begin{align*}
    \DITWeightedSpace^{2k}_{\DefFunction}(\EulerianDomain) := 
    H_{\DefFunction}^{2k,k+\frac{1-\kappa}{2\kappa}}(\EulerianDomain) \times 
    H_{\DefFunction}^{2k,k+\frac{1-\kappa}{2\kappa}+\frac{1}{2}}(\EulerianDomain).
\end{align*}

Because we will be appealing to \cite{disconzi2022relativistic}, it is best to introduce the variables used in that work. Given solution variables $(\rho,u^i)$ to \eqref{eq:RE}, define
\begin{subequations}{\label{E:New_sound_and_velocity}}
\begin{align}
    \label{E:New_sound_speed_sq}
    \NewSoundSpeedSq & := \frac{\kappa+1}{\kappa} \rho^\kappa, 
    \\
    \label{E:New_velocity}
    \NewVelocity^i & := (1+\rho^\kappa)^{1+\frac{1}{\kappa}} u^i,
\end{align}
\end{subequations}
and note that $\NewSoundSpeedSq$ is a multiple of the sound speed squared, $c_s^2$, so $\NewSoundSpeedSq$ is comparable to the distance to the boundary in our context. We also observe that the passage from $(\rho,u^i)$ to $(\NewSoundSpeedSq,\NewVelocity^i)$ is invertible and that
\begin{align*}
    \Omega_t := 
    \{ \rho > 0 \} = \{ \NewSoundSpeedSq > 0 \}.
\end{align*}

For a precise statement of continuity relative to the (rescaled) time variable $\tau$, we introduce the following Banach space related to the norm $\calE^N$ given by \eqref{defn:calEN},
\begin{align*}
    \mathcal{X}^{m}:= \Big\{ (\Psi_0,\Psi_1) \colon
    &[0,1] \rightarrow \mathbb{R} \, | \, 
    \\
    &
    \sum_{i = 0}^m \Big(\|\calD_i 
    \Psi_1\|_i^2 + 
    \|\calD_i \Psi_0 \|_i^2 + \|\calD_{i+1}\Psi_0\|_{i+1}^2 \Big) < \infty
    \Big\},
\end{align*}
where $\calD_i$ is given by \eqref{calDj} and the norm $\| \cdot \|_i$ is defined in Definition \ref{def: inner product}. We also note that $\mathcal{X}^m$ can be turned into an inner product since the norms $\| \cdot \|_i$ are inner-produce norms. 

\begin{rmk}
\label{R:Norm_equivalence_calX_calE}
    We see that the $\mathcal{X}^m$ norm is the  $\calE^N$ norm given by \eqref{defn:calEN}
up to the factor $\delta^{-1}\lam(\tau)^{3\kk}$. Recall that $\lambda(\tau)$ is given by Lemma \ref{lem:lambdaODE}, which also ensures that   $\lambda$ does not depend on solutions to \eqref{eq:momlagmain} and is uniformly bounded away from zero on any finite time interval. Therefore, the $\mathcal{X}^{m}$ norm is equivalent to the norm $\calE^N$ given by \eqref{defn:calEN} for each $\tau$.
\end{rmk}

Define the set of initial states $(\mathring{\NewSoundSpeedSq},\mathring{\NewVelocity})$ as
\begin{align*}
    \DITWeightedSpace_0^{2k} := \{ (\mathring{\NewSoundSpeedSq},\mathring{\NewVelocity}^i) \in \DITWeightedSpace^{2k}_{\mathring{\NewSoundSpeedSq}}(\Omega_{\mathring{\NewSoundSpeedSq}}) \},
\end{align*}
where $\Omega_{\mathring{r}} := \{ \mathring{r} > 0 \} = \{ \mathring{\rho} > 0 \}$, with $\mathring{\rho}$ given by \eqref{E:New_sound_speed_sq}. We observe that elements in $\DITWeightedSpace^{2k}_0$ are defined in different domains. Thus, in order to compare elements in this space we need to introduce a suitable topology on $\DITWeightedSpace^{2k}_0$. This is done in \cite{disconzi2022relativistic} and we use the same 
here.

\begin{thm}
    \label{T:LWP}
    Let $(\mathring{\rho},\mathring{u}^i)$ be spherically symmetric initial data for the relativistic Euler equations \eqref{eq:RE} on a bounded domain $\mathring{\EulerianDomain} = B_{\lambda_0}(0) \subset \mathbb{R}^3$, with equation of state \eqref{equation of states}. Assume that $\mathring{\EulerianDomain} = \{ \mathring{\rho} > 0 \}$ and that the physical vacuum boundary condition \eqref{fallout rate} holds on $\mathring{\EulerianDomain}$.  Let $N$ be an integer such that 
    \begin{align}
    \label{E:Fixed_N_bound}
    N \geq \lfloor \frac{1}{\kappa} \rfloor + 10.
    \end{align}
    Suppose that
    \begin{align*}
        (\mathring{\NewSoundSpeedSq},\mathring{\NewVelocity}^i) \in 
        \DITWeightedSpace^{2k}_{\mathring{\NewSoundSpeedSq}}(\mathring{\EulerianDomain}),
    \end{align*}
    where $(\mathring{\NewSoundSpeedSq},\mathring{\NewVelocity}^i)$ are the variables \eqref{E:New_sound_and_velocity} constructed out of $(\mathring{\rho},\mathring{u}^i)$, and $k$ satisfies
    \begin{align}
        \label{E:N_k_relation}
        k & \geq N + 5 + \lceil \frac{1}{2\kappa} \rceil.
    \end{align}        
    Fix parameters $\lambda_1 \in (0,1)$, $\delta >0$, an admissible weight $w$ as in Definition \ref{def: w}, introduce the reference initial data 
    \begin{subequations}{\label{E:Reference_data}}
    \begin{align}
        \label{E:rho_ref}
        \mathring{\rho}_{\text{ref}}(x) &:= \lambda_0^{-3} \big(\delta w(\frac{|x|}{\lambda_0})\big)^\frac{1}{\kappa}, \quad
        \\
        \label{E:u_ref}
        \mathring{u}^i_{\text{ref}}(x) &:= \frac{1}{\sqrt{1-\frac{\lambda_1^2 }{\lambda_0^2}|x|^2}} \frac{\lambda_1}{\lambda_0} x^i,
    \end{align}        
    \end{subequations}
    and let $(\NewSoundSpeedSq_\text{ref},\NewVelocity^i_\text{ref})$ be the variables \eqref{E:New_sound_and_velocity} corresponding to $(\rho_\text{ref},u^i_\text{ref})$.
    Introduce the $\delta$-dependent reference mass
    \begin{subequations}{\label{E:Masses_data}}
    \begin{align}
        \label{E:Reference_mass}
        M_\text{ref}(\delta,w) \equiv M_\text{ref}:=4\pi \int_0^1 (\delta w(\zeta))^\frac{1}{\kappa} \zeta^2  d\zeta,
    \end{align}
    and the total mass of the initial data
    \begin{align}
        \label{E:Initial_mass}
        M(\mathring{\rho}) \equiv 
        M :=4\pi \int_0^{\lambda_0} \mathring{\rho}(r) r^2 dr.
    \end{align}
    \end{subequations}    
    Finally, assume that
    \begin{align}
        \label{E:Mass_constraint}
        M = M_\text{ref}.
    \end{align}

    Then, the following statements hold:
    
    \medskip
    1) The variables $(\Theta_0,U_0)$ in \eqref{mainID} given by the change of variables described in Section \ref{sect: setup}, with $\rho=\mathring{\rho}$ and $u=\mathring{u}$ in \eqref{ansatz radial symmetry}, are well defined. Moreover, $(\Theta_0,U_0)=(0,0)$ if $\mathring{\rho}=\mathring{\rho}_{\text{ref}}$ and $\mathring{u}^i=\mathring{u}^i_{\text{ref}}$.

    \medskip
    2) There exists a $\tau^*>0$ and a unique classical solution $\Theta$ to \eqref{eq:momlagmain} defined on the time interval $[0,\tau^*)$ and taking the data \eqref{mainID}, where again \eqref{eq:momlagmain} is constructed out of $(\mathring{\rho},\mathring{u}^i)$ as in Section \ref{sect: setup}. Furthermore, for any $0 \leq \tau^{**} < \tau^*$, 
    \begin{align}
        \label{E:Regularity_Theta_partial_tau_Theta_closed_time_interval}
        (\Theta,\partial_\tau\Theta) \in C^0([0,\tau^{**}],\mathcal{X}^N),
    \end{align}
    and 
    \begin{align}
    \label{E:Finiteness_E_N_closed_time_interval}
    E^N(\tau^{**}) < \infty.
    \end{align}

    \medskip
    3) Define the set of initial states with compatible mass as 
    \begin{align*}
    \DITWeightedSpace_{\text{ref}}^{2k} := \{ (\mathring{\NewSoundSpeedSq},\mathring{\NewVelocity}^i) \in \DITWeightedSpace^{2k}_0 \, | \,     (\mathring{\NewSoundSpeedSq},\mathring{\NewVelocity}^i) \text{ is spherically symmetric and }
    M(\mathring{r}) = M_\text{ref} \},
    \end{align*}    
    with topology induced from $\DITWeightedSpace_0^{2k}$, where $M(\mathring{r})$ being the corresponding $M(\mathring{\rho})$ given by \eqref{E:Initial_mass}.
    
    Given $\epsilon > 0$, there exists a neighborhood $(\NewSoundSpeedSq_\text{ref},\NewVelocity_\text{ref}) \in \mathcal{U}_\text{ref} \subset \DITWeightedSpace_{\text{ref}}^{2k}$ such that if
    \begin{align*}    (\mathring{\NewSoundSpeedSq},\mathring{\NewVelocity}^i)  \in \mathcal{U}_\text{ref} 
    \end{align*}
    then
    \begin{align*}
        E^N(0) < \epsilon.
    \end{align*}
    
    \medskip
    4) Assume that the bootstrap assumptions \eqref{bootstrap U0} and \eqref{bootstrap FG} hold on $[0,\tau^*)\times[0,1]$. If 
    \begin{align}
        \label{E:Continuation_criterion_Lagrangian}
        \sup_{\tau \in [0,\tau^*)} \Big( 
        \| \Theta(\tau,\cdot) \|_{C^3([0,1])} + \| \partial_\tau \Theta(\tau,\cdot) \|_{C^2([0,1])}
        \Big) < \infty
    \end{align}
    and $\tau^* < \infty$, then the solution $(\Theta(\tau,\cdot),\partial_\tau \Theta(\tau,\cdot))$ given by 2) can be continued past $\tau^*$. 
\end{thm}
\begin{proof}
    The proof follows from carefully translating the result in \cite{disconzi2022relativistic} to the variables used in our setting. Thus, here, we only highlight the main points. Interested readers can check Appendix \ref{S:Proof_of_LWP}, where a more detailed proof is given.

    Statement 1) follows from following the calculations of Section \ref{sect: setup} at the level of the data, whereas local existence and uniqueness of a classical solution to \eqref{eq:momlagmain} follows again by the procedure of Section \ref{sect: setup} applied to the Eulerian solution launched by the data $(\mathring{\rho},\mathring{u}^i)$. We note that our regularity assumptions are more than enough to justify the steps in Section \ref{sect: setup}.

    By \eqref{E:N_k_relation} and \eqref{E:Weighted_embedding_DIT}, we have $\Theta(\tau,\cdot) \in C^{N+2}([0,1])$. From this and the fact that such solutions satisfy the correct center-regularity conditions at $\zeta=0$ coming from spherical symmetry we obtain that their $\mathcal{X}^N$ norm is finite and thus their energy view of Remark \ref{R:Norm_equivalence_calX_calE}. This gives statement 2).

    Statement 3) follows from another judicious inspection of the procedure relating $(\mathring{\rho},\mathring{u}^i)$ to $(\Theta_0,U_0)$, but rephrased in terms of $(\mathring{\NewSoundSpeedSq},\mathring{\NewVelocity}^i)$,  showing that this defines a map $(\mathring{\NewSoundSpeedSq},\mathring{\NewVelocity}^i) \mapsto(\Theta_0,U_0)$ that is continuous relative to the $\DITWeightedSpace^{2k}_\text{ref}$ and $\mathcal{X}^N$ topologies. We note that the restriction to states of same mass singles out the correct set of perturbations in that spherically symmetric solutions to the relativistic Euler equations conserve mass.

    Finally, for statement 4), one verifies that under the corresponding bootstrap assumptions, we can invert back from the $(\Theta,\partial_\tau\Theta)$ variables to $(\rho,u^i)$ and then $(\NewSoundSpeedSq,\NewVelocity^i)$, with uniform control over the change of variables. The bound \eqref{E:Continuation_criterion_Lagrangian} then translate to uniform $C^2$ control of $(\NewSoundSpeedSq,\NewVelocity^i)$ over $[0,T^*)$, where $\tau^*$ and $T^*$ are related by $\tau^* = \int_0^{T^*} \frac{1}{\lambda(t)} dt$. Such control is enough to apply the continuation criterion of \cite{disconzi2022relativistic} and thus continue the solution $(\NewSoundSpeedSq,\NewVelocity^i)$ past $T^*$. Applying the procedure of Section \ref{sect: setup} to this continued-beyond-$T^*$ solution then gives a solution that continues $(\Theta,\partial_\tau\Theta)$ beyond $\tau^*$.
\end{proof}

\begin{rmk}
    It is expected that the passage from Eulerian to Lagrangian variables can lose derivatives. We have not tried to minimize such a loss here, and in fact have chosen the relation between the regularity of the $(\NewSoundSpeedSq,\NewVelocity^i)$ and $(\Theta,\partial_\tau\Theta)$ given by \eqref{E:N_k_relation} so that we have lots of room to spare. This is because we want the easiest route to produce local solutions to \eqref{eq:momlagmain} out of results already established in the literature, so that we can then focus on the most interesting part of the work which is the global existence.
\end{rmk}

\section{The Global Existence Theorem}\label{sect: mainthm}
In this section, we state and prove our main result.
\begin{thm}[Main Theorem: The Precise Version]
\label{thm:mainprecise}
    Let $(\lam_0,\lam_1)\in \R^+\times (0,1)$ and $N\geq \lfloor\frac1\kk\rfloor+10$. 
There exists $\eps_0=\eps_0(\kk, N, \lam_0, \lam_1)$ such that for every $\eps\leq \eps_0$, $\delta\leq \eps$, and initial data $(\Theta_0, U_0)$ verifying 
    $$
    E^N(0) \leq \eps, 
    $$
    the associated solution to the equation \eqref{eq:momlagmain} from Theorem \ref{T:LWP},
    $$
    \tau \mapsto (\Theta(\tau,\cdot), \pt\Theta(\tau,\cdot))
    $$
    uniquely exists for all $\tau>0$. Moreover, there exists $M_*>0$, depending only on $\kk, N, \lam_0, \lam_1$ such that 
    \begin{equation}\label{est: main theorem}
    E^N(\tau) \leq M_*\eps,\quad \tau>0. 
    \end{equation}
    Finally, there exists a $\tau$-independent attractor $\Theta_{\infty}$ such that
    \begin{equation}\label{convergence: main theorem}
    \sum_{j=0}^N \|\calD_j \Theta(\tau,\cdot)-\calD_j\Theta_{\infty}(\cdot)\|_j \rightarrow 0,
    \end{equation}
    as $\tau\to\infty$.
\end{thm}

\begin{rmk}
    We note that the initial data $(\Theta_0,U_0)$ verifying the assumption in Theorem \ref{thm:mainprecise} can be produced by Eulerian data $(\mathring{\rho},\mathring{u}^i)$ from item 3) of Theorem \ref{T:LWP}.
\end{rmk}

By re-interpreting Theorem \ref{thm:mainprecise} in terms of Eulerian variables, we immediately obtain the following corollary:

\begin{cor}
    Fix a reference state $(\mathring{\rho}_\text{ref},\mathring{u}^i_\text{ref})$ as in Theorem \ref{T:LWP}. Initial data $(\mathring{\rho},\mathring{u}^i)$ to the relativistic Euler equations in Eulerian coordinates, i.e., equations \eqref{eq:RE}, that are sufficiently close to $(\mathring{\rho}_\text{ref},\mathring{u}^i_\text{ref})$ in the sense of item 3) of Theorem \ref{T:LWP}, launch solutions that exist globally in time. 
\end{cor}

\subsection{Proof of \eqref{est: main theorem}}
To prove \eqref{est: main theorem}, it suffices to show that the bootstrap assumptions \eqref{bootstrap EN}, \eqref{bootstrap U0}, and \eqref{bootstrap FG} are improved on the bootstrap horizon $[0,\tau_*]$ for an appropriately chosen $M_*$ and $\eps, \delta$ sufficiently small. 
To see why this is the case, let $\tau_\text{max}$ be the maximum interval of existence for our local-in-time solution solution $(\Theta,\partial_\tau\Theta)$. Set
\begin{align*}
    \tau_\text{max boot} := \sup\{ \tau \in (0,\tau_\text{max}) \, | \, \text{the bootstrap assumptions hold on }\, [0,\tau] \}.
\end{align*}
The quantities on the LHS of 
on the LHS of  \eqref{bootstrap U0} and \eqref{bootstrap FG}, as well as $E^N(\tau)$, are continuous functions of $\tau$. Thus the set $\{ \tau \, | \, \text{the bootstrap assumptions hold on }\, [0,\tau]\}$ is closed under increasing limits. Therefore, if $\tau_\text{max boot} < \tau_\text{max}$, the bootstrap bounds hold on $[0,\tau_\text{max boot}]$ and  at least one of the inequalities \eqref{bootstrap EN}, \eqref{bootstrap U0}, or \eqref{bootstrap FG} is saturated at $\tau_\text{max boot}$, otherwise their validity would extend slightly past $\tau_\text{max boot}$ by continuity. Moreover, by our assumptions on the initial data and again continuity in $\tau$, we have that $\tau_\text{max boot}>0$.

Therefore, if we show that the bootstraps assumptions are improved over $[0,\tau_*]$, this means that their validity and improvement over each $[0,\tau]$, $\tau<\tau_\text{max boot}$, allows us to pass to the limit $\tau\rightarrow\tau_\text{max boot}^-$ and conclude that the bootstrap assumptions are improved over $[0,\tau_\text{boot max}]$, contradicting the maximality of $\tau_\text{max boot}$ unless $\tau_\text{max} < \infty$. But if $\tau_\text{max} < \infty$, then the bootstraps assumptions hold on $[0,\tau]$ for any $\tau < \tau_\text{max}$ and we can then invoke our continuation criterion to extend the solution past $\tau_\text{max}$. Thus $\tau_{\max} = \infty$.

Thus, it suffices to prove the following theorem:

\begin{thm}\label{thm:improved}
    Let $(\Theta,\pt\Theta)$ be a solution to \eqref{eq:momlagmainfull} on $[0,\tau_*]$ verifying the bootstrap assumptions. There exists $\eps_0 > 0$ sufficiently small such that for any $\eps \le \eps_0$ and $\delta \le \eps$, the following improved estimates hold:
    \begin{subequations}
    \label{est:improved}
    \begin{equation}\label{est:improvedEN}
        E^N(\tau_*) \le \frac{M_*\eps}{2},
        \end{equation}
        \begin{equation}\label{est:improvedU0}
        \frac23(1-\bar\lam^2\ze^2) \le (U^0(\tau,\ze))^{-2} \le \frac{4}{3}(1-\bar\lam^2\ze^2),\quad (\tau,\ze) \in [0,\tau_*]\times [0,1],
        \end{equation}
        \begin{equation}\label{est:improvedFG}
        \|\calF(\tau) - 1\|_{L^\infty} + \|\kk\barcalG(\tau) - 1\|_{L^\infty} \le \frac{1}{200},\quad \tau \in [0,\tau_*].
        \end{equation}
    \end{subequations}
\end{thm}
We start with the proof of the improved bound \eqref{est:improvedEN}. Let $\eps \le \eps''$ and $\delta \le \eps^\frac12$, where $\eps''$ is chosen as in Proposition \ref{prop:mainEN0}. We also fix time instances $\tau_0, \tau$ such that $0 \le \tau_0 \le \tau \le \tau_*$. Taking supremum over time interval $[\tau_0,\tau]$ on both sides of \eqref{est:mainEN0}, we obtain
\begin{equation}
    \label{est:mainEN1}
    \begin{split}
    E^N(\tau;\tau_0) + \int_{\tau_0}^\tau \mfD^N(\tilde\tau) d\tilde\tau &\le C_0(\calE^N(\tau_0) + \calE^N(0) + \delta^2)\\
        &\quad + C_1 \left[\delta^\frac12\int_{\tau_0}^\tau e^{-\frac32\kk\bar\lam \tilde\tau} (E^N(\tilde\tau;\tau_0))^\frac12 d\tilde\tau + \int_{\tau_0}^\tau e^{-\frac32\kk\bar\lam \tilde\tau}E^N(\tilde\tau;\tau_0) d\tilde\tau\right]\\
        &\le C_2(\calE^N(\tau_0) + \calE^N(0) + \delta) + C_3\int_{\tau_0}^\tau e^{-\frac32\kk\bar\lam \tilde\tau}E^N(\tilde\tau;\tau_0) d\tilde\tau,
        \end{split}
\end{equation}
where $C_2, C_3$ are positive constants that are independent of $M_*, \eps, \delta$. Note that we have used Cauchy-Schwarz inequality and the integrability of $e^{-\frac32\kk\bar\lam \tilde\tau}$ on $(0,\infty)$ to bound:
$$
\delta^\frac12\int_{\tau_0}^\tau e^{-\frac32\kk\bar\lam \tilde\tau} (E^N(\tilde\tau;\tau_0))^\frac12 d\tilde\tau \lesssim \delta + \int_{\tau_0}^\tau e^{-\frac32\kk\bar\lam \tilde\tau} E^N(\tilde\tau;\tau_0) d\tilde\tau.
$$
Given $\bar\eps \in (0,1)$ sufficiently small, the local well-posedness theory guarantees the existence of an $\tilde\eps > 0$ and $\tilde C > 0$ such that if $2\calE^N(0) + \delta = 2\eps + \delta < \tilde\eps$, the (unique) solution $(\Theta,\pt\Theta)$ exists on $[0,\tau_{loc}]$ such that $\tau_{loc} \ge \frac{1}{\frac32\kk\bar\lam}\left|\log(\frac32\kk\bar\lam\bar\eps)\right|$, and it holds that
\begin{equation}
    \label{est:initiallayer}
    E^N(\tau)+ \int_{0}^\tau \mfD^N(\tilde\tau) d\tilde\tau \le \tilde C\left(2\calE^N(0) + \delta\right),\quad \tau \in [0,\tau_{loc}].
\end{equation}
For any $\tau \in (\tau_{loc}, \tau_*)$, we obtain from \eqref{est:mainEN1} that
\begin{align*}
    E^N(\tau;\tau_{loc}) &\le C_2(\calE^N(\tau_{loc}) + \calE^N(0) + \delta) + C_3\int_{\tau_{loc}}^\tau e^{-\frac32\kk\bar\lam \tilde\tau}E^N(\tilde\tau;\tau_{loc}) d\tilde\tau\\
    &\le C_2(\tilde C\left(2\calE^N(0) + \delta\right) + \calE^N(0) + \delta) + C_3\frac{1}{\frac32\kk\bar\lam}e^{-\frac32\kk \bar\lam \tau_{loc}} E^N(\tau;\tau_{loc})\\
    &\le C_2(\tilde C(2\eps + \delta) + \eps + \delta) + C_3\bar\eps E^N(\tau;\tau_{loc}).
\end{align*}
Choosing $\bar \eps < \frac{1}{2C_3}$, $\delta \le \eps$, and $M_* \ge 4C_2(3\tilde C + 2),$ the above inequality can be rearranged as:
$$
E^N(\tau;\tau_{loc}) \le 2C_2(\tilde C(2\eps + \delta) + \eps + \delta) \le 2C_2(3\tilde C + 2)\eps \le \frac{M_*}{2}\eps.
$$
By further choosing $M_* \ge 6\tilde C$ if necessary, the estimate \eqref{est:initiallayer} together with the above inequality yields $E^N(\tau_*) \le \frac{M_*}{2}\eps$, which concludes the proof of \eqref{est:improvedEN}.

\noindent To show the improved bound \eqref{est:improvedU0}, we recall that
\begin{equation}
    \label{defnU0}
    \begin{split}
        &(U^0)^{-2} - (1- \bar\lam^2 \ze^2) = \bar\lam^2 \ze^2 - \left(\pt\Theta + \lamt(\ze + \Theta)\right)^2\\
        &\quad=  - \ze^2\left[\left(\frac{\pt\Theta}{\ze}\right)^2 + (\lamt^2 - \bar\lam^2 ) + 2\lamt^2 \frac{\Theta}{\ze} + \lamt^2 \left(\frac{\Theta}{\ze}\right)^2 + 2\lamt\left(1+\frac{\Theta}{\ze}\right)\left(\frac{\pt \Theta}{\ze}\right)\right].
    \end{split}
\end{equation}
By \eqref{bootstrap EN}, we know from Hardy inequalities \eqref{est:LinftyHardy1} and \eqref{est:LinftyHardy2} that for $\tau \in [0,\tau_*]$,
$$
\delta^{-1}\lam^{3\kk}\|\frac{\pt\Theta}{\ze}\|_{L^\infty}^2 + \|\frac{\Theta}{\ze}\|_{L^\infty}^2 \lesssim \eps.
$$
Moreover, by \eqref{est:difflamt}, $|\lamt^2 - \bar\lam^2| \lesssim \delta \le \eps$. Then the above two estimates with \eqref{defnU0} implies that for some $C > 0$,
$$
|(U^0)^{-2} - (1- \bar\lam^2 \ze^2)| \le C\eps^\frac12\ze^2,
$$
which verifies \eqref{est:improvedU0} after choosing $\eps$ sufficiently small.

Next, we show the improved bound for $\|\calF(\tau)-1\|_{L^{\infty}}$ in \eqref{est:improvedFG}. Invoking Lemma \ref{lem:calF-1} and Lemma \ref{lem:Fnonlinear}, and then the Hardy's inequality \eqref{est:LinftyHardy1}, we obtain
\begin{align*}
|\calF(\tau) - 1|^2 \leq C\Big(|\Dz \Theta(\tau)|^2+ \prod_{\substack{A_j\in \calP_{i_j},\,i_j\leq 1\\\sum_{j=1}^2i_j=2}}|A_j\Theta(\tau)|^2+ \prod_{\substack{A_j\in \calP_{i_j},\,i_j\leq 1\\\sum_{j=1}^3i_j=2}}|A_j\Theta(\tau)|^2\Big)
\leq C\calE^N(\tau),
\end{align*}
for some positive generic constant $C$. By virtual of \eqref{est:improvedEN}, this suggests $\|\calF(\tau) - 1\|_{L^\infty}\leq \frac{1}{200}$ after choosing $\epsilon$ sufficiently small.

Finally,  the improved bound for $\|\kk\barcalG(\tau)-1\|_{L^{\infty}}$ follows from the lemma below after selecting $\eps$ sufficiently small.
\begin{lem}
    \label{prop:barGimproved}
    Given the assumptions of Theorem \ref{thm:improved}, the following estimate holds:
    \begin{equation}
        \label{est:barGimproved}
        \sup_{0\le \tau \le \tau_*}\|\kk\barcalG(\tau) - 1\|_{L^\infty} \leq C(M_*^{\frac12}+1)\epsilon,
    \end{equation}
    where $C>0$ is a generic constant. 
\end{lem}
\begin{proof}
We begin by converting \eqref{eq:dtbarcalG} into
\begin{equation}
\label{eq:dtbarcalG_rewrite}
\begin{split}
\kk \p_\tau \barcalG &= \mathcal{C}_0 \pt (U^0)^{\kk}- \delta w\pt(\lam^{-3\kk}\calF^{-\kk}),
\end{split}
\end{equation}
where
\begin{equation}
    \nonumber
    \mathcal{C}_0:= (U^0)^{-\kk}(0)(1+ \delta\lambda^{-3\kk}(0)w\calF^{-\kk}(0)).
\end{equation}
Integrating \eqref{eq:dtbarcalG_rewrite} and using $\left. \barcalG \right|_{\tau=0}=\frac{1}{\kk}$, we thus find
\begin{equation}
\nonumber
    \begin{split}
        \kk \barcalG - 1= \mathcal{C}_0 \int_0^\tau \pt (U^0)^{\kk} d\tilde{\tau}
        - \delta w \int_0^\tau\pt (\lam^{-3\kk}\calF^{-\kk})d\tilde{\tau}
        =: T_1 - T_2.
    \end{split}
\end{equation}

    First, we estimate $T_1$. Recall $(U^0)^{-2} = 1-(\p_\tau\Theta+\lamt (\Theta+\ze))^2$. Then, 
    \begin{equation}\label{dtU0a}
    \begin{split}
        \pt (U^0)^{\kk} &= \kk(U^0)^{\kk+2} \left(\pt\Theta + \lamt(\Theta+\ze)\right)\left(\pt^2\Theta + \pt\lamt(\Theta+\ze) + \lamt\pt\Theta\right)\\
        &= \kk(U^0)^{\kk+2}\bigg[\lamt\pt\lamt(\Theta + \ze)^2 + \pt^2\Theta\left(\pt\Theta + \lamt(\Theta+\ze)\right) + \calO(\pt\Theta)\bigg].
    \end{split}
    \end{equation}
    This implies, together with \eqref{est:improvedU0}, that
    \begin{align*}
        |T_1| \lesssim& \int_0^\tau |\lamt\pt\lamt|(\Theta + \ze)^2 d\tilde\tau + \left|\int_0^\tau \pt^2\Theta(\pt\Theta + \lamt(\Theta+\ze)) d\tilde\tau\right| + \int_0^\tau |\calO(\pt\Theta)|d\tilde\tau\\
        &=: T_{11}+T_{12}+T_{13}.
    \end{align*}
    Since $|\lamt| \approx \bar\lam$, we obtain from the Hardy's the embedding \eqref{est:LinftyHardy1} and \eqref{est:ptlamt} to obtain that $T_{11} \lesssim \delta$. Also, by applying \eqref{est:LinftyHardy1} to $\pt\Theta$ and because $\lambda^{-\frac32\kk}$ is integrable, we have 
    $$
    T_{13} \lesssim \int_0^\tau \delta^{\frac12}\lam^{-\frac32\kk}(\calE^{N})^{1/2}  \lesssim (M_*)^{\frac12} \epsilon^{\frac12} \delta^{\frac12},
    $$ 
    where we used \eqref{est:improvedEN} in the last inequality. To bound $T_{12}$, we integrate $\pt$ by parts to obtain that
    \begin{align*}
        T_{12} &\le \frac12 \left|(\pt\Theta(\tau)^2 - \pt\Theta(0)^2)\right| + \left|\lamt\pt\Theta(\Theta + \ze) - \lamt(0)\pt\Theta(0)(\Theta(0) + \ze)\right|\\
        &\quad + \int_0^\tau |\pt\Theta|\left|\pt\lamt(\Theta + \ze) + \lamt\pt\Theta\right|d\tilde\tau\\
        &\lesssim \delta^{\frac12}\left((\calE^N(0))^{1/2} + (\calE^N(\tau))^{1/2}\right) + \int_0^\tau \delta\lam^{- 3\kk}\|\pt\Theta\|_{L^\infty} + \lamt\|\pt\Theta\|_{L^\infty}^2 d\tilde\tau\\
        &\lesssim \delta^{\frac12}\left((\calE^N(0))^{1/2} + (\calE^N(\tau))^{1/2}\right) + \delta^{\frac32}\left(E^N(\tau_*)\right)^{1/2} + \delta^{\frac32}E^N(\tau_*)\\
        &\lesssim M_*^{\frac12}\eps^{\frac12}\delta^{\frac12},
    \end{align*}
    where we used \eqref{est:LinftyHardy1} and \eqref{est:lambdaasym} in the third inequality, as well as \eqref{est:improvedEN} in the final inequality. Summarizing all the estimates above, we have
    \begin{equation}
        \label{est:U0T1}
        |T_1| \lesssim 
        \delta+
        M_*^{\frac12}\epsilon^{\frac12}\delta^{\frac12}.
    \end{equation}
     Lastly, we control $T_2$. By decomposing
    \begin{align*}
        T_2 = \delta w\int_0^\tau \left((-3\kk\lam^{-3\kk}\lamt) \calF^{-\kk} + \lam^{-3\kk}(-\kk\calF^{-\kk-1}\pt\calF)\right)d\tilde\tau
        =: T_{21} + T_{22}, 
    \end{align*}
    we see that $|T_{21}| \lesssim \delta$ from \eqref{est:lambdaasym}, \eqref{est:improvedU0}, and the bound $|\calF-1|\leq \frac{1}{200}$. Additionally, we recall that
    $$
    \pt\calF = \pt\left(\Dz\Theta + \Dz\left(\frac{\Theta^2}{\ze} + \frac{\Theta^3}{3\ze^2}\right)\right).
    $$
    Then an application of Hardy's inequality \eqref{est:LinftyHardy1}, \eqref{est:improvedU0}, the bound $|\calF-1|\leq \frac{1}{200}$, and \eqref{est:lambdaasym} yields the bound $|T_{22}| \lesssim M_*^{\frac12}\eps^{\frac12}\delta$. In summation, we have 
    \begin{align}|T_2| \lesssim \delta+M_*^{\frac12}\eps^{\frac12}\delta.
    \end{align}
    By combining this with \eqref{est:U0T1}, we conclude that
    \begin{align}
    \begin{aligned}
        |\kk\barcalG-1| \leq& C\left(\delta+M_*^{\frac12}\eps^{\frac12}\delta^{\frac12}\right)
        \leq C(M_*^{\frac12}+1)\eps,
        \end{aligned}
    \end{align}
    where we used $\delta\le \eps$ in the final inequality.
\end{proof}
\subsection{Proof of \eqref{convergence: main theorem}}
From the global-in-time boundedness of $E^N$ \eqref{est: main theorem}, there exists a weak limit $\Theta_{\infty}$ independent of $\tau$ such that $\|\calD_i \Theta_{\infty}(\cdot)\|_{i}^2 + \|\calD_{i+1} \Theta_{\infty}(\cdot)\|_{i+1}^2 \lesssim E^N(\tau) \leq M_*\epsilon$. To prove the strong convergence of $\Theta(\tau, \cdot)$, we observe that for any $0<\tau_2<\tau_1$, 
\begin{align*}
    \|\calD_i\Theta(\tau_1,\cdot)-\calD_i\Theta(\tau_2,\cdot)\|_i^2 =& \int_0^1\ze^2w^{\frac1\kk+i}\left|\int_{\tau_2}^{\tau_1}\calD_i \pt \Theta(\tau,\cdot)d\tau\right|^2d\ze\\
    \lesssim& \left(\int_{\tau_2}^{\tau_1}\lam^{-\frac32\kk}d\tau\right)\left(\int_{\tau_2}^{\tau_1}\lam^{\frac32\kk}\int_0^1\ze^2w^{\frac1\kk+i}|\calD_i\pt\Theta(\tau,\cdot)|^2d\ze d\tau\right)\\
     \lesssim& \left(\int_{\tau_2}^{\tau_1}\lam^{-\frac32\kk}d\tau\right)\left(\int_{\tau_2}^{\tau_1}\lam^{-\frac32\kk}\left[\lam^{3\kk}\int_0^1\ze^2w^{\frac1\kk+i}|\calD_i\pt\Theta(\tau,\cdot)|^2d\ze\right] d\tau\right)\\
    \lesssim& M_*\epsilon\delta (e^{-\frac32\kk\bar\lam\tau_2}-e^{-\frac32\kk\bar\lam\tau_1})^2.
\end{align*}
Here, the final inequality follows from \eqref{est:lambdaasym} and $$\sup_{\tau\in[\tau_2, \tau_1]}\delta^{-1}\lam^{3\kk}\int_0^1\ze^2w^{\frac1\kk+i}|\calD_i\pt\Theta(\tau,\cdot)|^2d\ze \leq M_* \epsilon.$$
Consequently, for any strictly increasing sequence $\{\tau_n\}_{n=1}^{\infty}$ satisfying $\tau_n\to \infty$, the sequence $\{\calD_i \Theta(\tau_n, \cdot)\}_{n=1}^{\infty}$ is Cauchy with respect to the $\|\cdot\|_i$ norm provided that $i\leq N$. This completes the proof of \eqref{convergence: main theorem}.

\begin{appendices}
\section{Toolbox}
In this section, we provide a collection of functional-analytic tools that we use frequently throughout this work. We start with a number of calculus lemmas concerning vector fields of classes $\calP_i$ and $\barcalP_i$. We then supplement Hardy-type inequalities that are adapted to our weighted functional setup. We remark that most of the results presented in this section have been proved in the Appendix of \cite{guo2021continued}, so we are content with only proving the statements which are different from those appearing in \cite{guo2021continued}.
\subsection{Coercivity}
An important fact that we use throughout this paper is that vectorfield $\calD_i$ dominates other $\calP_i$ vector fields in energy estimates. This observation is summarized in the following lemma
\begin{lem}[Lemma A.1, \cite{guo2021continued}]
    \label{lem:Dicoercivity}
    Suppose $\calD_i f \in L^2\left([0,\frac34], r^2dr\right)$. Then the following estimate holds:
    \begin{equation}
        \label{est:Dicoercivity}
        \sum_{P_i \in \calP_i}\int_0^1 |P_i f|^2 \chi(r) r^2dr \lesssim \int_0^1 |\calD_i f|^2 \chi(r) r^2 dr,
    \end{equation}
    where $\chi$ is a bump function supported on $[0,\frac34]$.
\end{lem}
\subsection{Commuting $\calD_i$ with elliptic operator $L_{0}$}
Define $\mathcal{P}$ (resp. $\mathcal{Q}$) to be the class of smooth functions on $[0,1]$ such that their Taylor expansion at $\zeta = 0$ only contains odd (resp. even)-order terms. 

With the definitions above, we prove the following crucial lemma. 
\begin{lem}\label{lem:Lcommute}
    For any $i > 0$ a positive integer and $f$ a smooth function, there exist smooth functions $\xi_{i,j}(\zeta)$ with $j \ge 0$ such that 
    \begin{equation}
        \label{maincommutation}
        \calD_i L_0 f = \calL_{i}\calD_i f + \sum_{j=0}^{i-1} \xi_{i,j}\calD_{i-j}f,
    \end{equation}
    where
    $$
    \xi_{i,j} \in \begin{cases}
        \mathcal{P},& j \text{ is odd},\\
        \mathcal{Q},& j \text{ is even}.
    \end{cases}
    $$
\end{lem}
\begin{rmk}
The readers might find this lemma somewhat surprising especially when compared to the Lemma B.1. of \cite{guo2021continued}, in that the coefficients $\xi_{i,j}$ are smooth functions without any negative powers of $\ze$. The reason that Lemma \ref{lem:Lcommute} holds crucially relies on the fact that $w: [0,1] \to \R_{\ge 0}$ can be evenly extended to a sufficiently regular function on $[-1,1]$.
\end{rmk}
\begin{proof}
We first recall the following product rule:
\begin{equation}\label{eq:basicproduct}
\Dz(fg) = \Dz fg + f\dz g.
\end{equation}
Following \cite[p. 538]{guo2021continued}, we also have
\begin{equation}\label{DzLk}
    \Dz L_k f = L^*_{1+k}\Dz f- (1+k)(w'' + \frac{2}{\zeta}w')\Dz f,
\end{equation}

\begin{equation}\label{dzLk*}
    \dz L_k^* f = L_{1+k}\dz f- (1+k)(w'' - \frac{2}{\zeta}w')\dz f.
\end{equation}
    Now, we prove \eqref{maincommutation} by induction. For notational convenience, we will denote $p(\zeta)$ as a generic member of $\mathcal{P}$, and $q(\zeta)$ a generic member of $\mathcal{Q}$. The exact expression of $p,q$ might change from line to line.
    \begin{enumerate}
        \item $i = 1$. This case is clear upon using \eqref{DzLk} and realizing that $\Dz = \calD_1$, $w'' \pm 2w'/\zeta \in \mathcal{Q}$:
        $$
        \calD_1L_0 f = L^*_{1}\calD_1 f - (1+k)(w'' + 2w'/\zeta)\calD_1 f = L^*_{1}\calD_1 f + q\calD_1f.
        $$
        \item $i > 1$. We first discuss the case where $i$ is even. The following elementary fact is useful for the rest of the proof:
        $$
        \dz p,\Dz p \in \mathcal{Q},\quad \dz q \in \mathcal{P}.
        $$
        We remark that $\Dz q$ is extremely dangerous as $\Dz 1 = 2/\zeta$. However, we will see that such adverse derivative will never appear thanks to the vanishing properties of weight $w$. By induction hypothesis, we have
        \begin{align*}
            \calD_i L_0 f &= \dz (\calD_{i-1}L_0 f ) = \dz\left(\calL_{i-1}\calD_{i-1}f + \sum_{j=0}^{i-2}\xi_{i-1,j}\calD_{i-1-j}f\right)\\
            &= \dz\left(L_{i-1}^*\calD_{i-1}f\right) + \sum_{j=0}^{i-2}\dz\left(\xi_{i-1,j}\calD_{i-1-j}f\right)\\
            &= L_{i}\calD_i f - i(w''-2w'/\zeta)\calD_if + \sum_{j=0}^{i-2}\dz\left(\xi_{i-1,j}\calD_{i-1-j}f\right).
        \end{align*}
        It is clear that the second term above has the required structure since $w'' - 2w'/\zeta \in \mathcal{Q}$. Then it suffices for us to consider the final term above. For $0\le j \le i - 2$ being even, we write
        \begin{equation}\label{lemcommuteaux1}
        \begin{split}
        \dz\left(\xi_{i-1,j}\calD_{i-1-j}f\right) &= \dz(\xi_{i-1,j})\calD_{i-1-j}f + \xi_{i-1,j}\dz\calD_{i-1-j}f\\
        &= \dz(\xi_{i-1,j})\calD_{i-1-j}f + \xi_{i-1,j}\calD_{i-j}f,
        \end{split}
        \end{equation}
        where we used that $\dz \calD_{i-j-1} = \calD_{i-j}$ when $j$ is even (i.e. $i-j-1$ is odd). Moreover, because
        \begin{equation}\label{dzxi}
        \p_\zeta \xi_{i-1,j} \in \begin{cases}
        \mathcal{P},& j \text{ is even},\\
        \mathcal{Q},& j \text{ is odd},
    \end{cases}
        \end{equation}
        the RHS to \eqref{lemcommuteaux1} has the desired structure. When $j$ is odd, we recall that $\xi_{i-1,j} \in \mathcal{P}$ by the induction hypothesis. Hence, we may write
        $$
        \xi_{i-1,j} = \zeta q_{i-1,j},\quad q_{i-1,j} \in \mathcal{Q}.
        $$
        Then
        \begin{align*}
            \dz\left(\xi_{i-1,j}\calD_{i-1-j}f\right) &= \dz(\zeta q_{i-1,j}\calD_{i-1-j}f) = q_{i-1,j}\calD_{i-1-j}f + \zeta\p_\zeta (q_{i-1,j}\calD_{i-1-j}f).
        \end{align*}
        But then we recall that for a function $h$,
        $$
        \zeta \Dz h = \zeta\dz h + 2h.
        $$
        Then combining the two pieces of computations above, we have
        \begin{align*}
        \dz\left(\xi_{i-1,j}\calD_{i-1-j}f\right) &= -q_{i-1,j}\calD_{i-1-j}f + \zeta \Dz(q_{i-1,j}\calD_{i-1-j}f)\\
        &= -q_{i-1,j}\calD_{i-1-j}f + \zeta(q_{i-1,j}\calD_{i-j}f + \dz q_{i-1,j}\calD_{i-j-1}f).
        \end{align*}
        This case is thus proved after noticing that $\zeta q_{i-1,j} \in \mathcal{P}$ and $\zeta\dz q_{i-1,j} \in \mathcal{Q}$. The proof for $i$ odd is similar, where one needs to quote \eqref{DzLk} in place of \eqref{dzLk*}.
    \end{enumerate}
\end{proof}
\subsection{Product Rules}
In this section, we record a general product rule concerning a general vectorfield of class $\calP_i$ or $\barcalP_i$, which are useful in analyzing nonlinear terms which are not of top order. Moreover, we will also prove some specific product rules that possess more structures than the general version. In fact, these special product rules will be crucial when we perform the top order calculations, where a structural symmetry is required to avoid a loss of derivatives.

We start with a general product rule, which follows from the basic version \eqref{eq:basicproduct}:
\begin{lem}[Lemma A.4, \cite{guo2021continued}]
    \label{lem:generalproductrule}
Let $ i \in \N$ be given.
    \begin{enumerate}
        \item For any $P \in \calP_{i}$, the following identity holds:
        \begin{equation}
            \label{Pproduct}
            P(fg) = \sum_{k=0}^i \sum_{\substack{A \in \calP_{k}\\ B \in \barcalP_{i-k}}}c^{AB}_k(Af)(Bg),
        \end{equation}
        where $c^{AB}_k$ are real constants.

        \item For any $\bar P \in \barcalP_{i}$, the following identities hold:
        \begin{equation}
            \label{Pbarproduct}
            \bar P(fg) = \sum_{k=0}^i \sum_{\substack{A \in \barcalP_{k}\\ B \in \barcalP_{i-k}}}c^{AB}_k(Af)(Bg),
        \end{equation}
        and
        \begin{equation}
            \label{Pbarproduct2}
            \bar P(fg) = \sum_{k=0}^i \sum_{\substack{A \in \calP_{k}\\ B \in \calP_{i-k}}}\bar{c}^{AB}_k(Af)(Bg),
        \end{equation}
        where $c^{AB}_k, \bar{c}^{AB}_k$ are real constants.
    \end{enumerate}
\end{lem}
\begin{proof}
 Invoking the definition of $\Dz = \pz + \frac{2}{\ze}$, the identity \eqref{eq:basicproduct} follows from the usual product rule.  
    The identities \eqref{Pproduct} and \eqref{Pbarproduct} follow from \cite[Lemma A.4]{guo2021continued}. Then it suffices for us to prove \eqref{Pbarproduct2}. Given $\bar P \in \barcalP_i$, we can write $\bar P = P\pz$, where $P \in \calP_{i-1}$. We thus have
    \begin{align*}
        \bar P(fg) &= P(\pz f g + f\pz g) =P\left(\Dz fg + f\Dz g - \frac{4f}{\ze}\cdot g\right)\\
        &= \sum_{\substack{A_{1,2} \in \calP_{l_1,l_2},\\l_1 + l_2 = i}}\bar{c}^{l_1,l_2}(A_1 f)(A_2 g),
    \end{align*}
    where we used \eqref{Pproduct} in the final equality above.
\end{proof}

Next, we present a few more specific product rules which are instrumental in analyzing the highest order terms.
\begin{lem}\label{lem:highestspatial}
    The following identity holds:
    $$
\calD_i(f\Dz g) = f\calD_{i+1}g + \sum_{\substack{A_{1,2} \in \calP_{l_1,l_2},\\l_1+l_2 = i+1,\\l_1, l_2 \le i}}c^{A_1A_2} (A_1f)(A_2g),
$$
where $c^{A_1A_2}$ are real constants. 
\end{lem}
\begin{proof}
    We prove by induction. If $i = 1$, we see that
    \begin{align*}
        \calD_1 (f\Dz g) = \Dz (f\Dz g) = \Dz f \Dz g + f \calD_2 g = \calD_1 f\calD_1 g + f \calD_2 g,
    \end{align*}
    and we are done.

    To complete the inductive step, we assume $i$ is even. Then $\calD_{i+1} = \Dz \calD_{i}$, and we have
    \begin{align*}
        \calD_{i+1} (f\Dz g) &= \Dz\left(f\calD_{i+1}g + \sum_{k \ge 1,A \in \calP_k,B\in\calP_{i-k+1}}^{i} (Af)(Bg)\right)\\
        &= \calD_1 f \calD_{i+1}g + f\calD_{i+2}g + \Dz \left(\sum_{k \ge 1,A \in \calP_k,B\in\calP_{i-k+1}}^{i} (Af)(Bg)\right).
    \end{align*}
    Since $i+1$ is odd, then $k$ and $i-k+1$ have different parity. We may without loss assume that $k$ is even and hence $i-k+1$ is odd. Then
    $$
    \Dz ((Af)(Bg) ) = \Dz (Af) (Bg) + Af \dz(Bg) = \calA f Bg + Af\mathcal{B}g,
    $$
    where $\calA\in \calP_{k+1}$, $\mathcal{B} \in \calP_{i-k+2}$. Thus we conclude the proof for $i$ even case. The case where $i$ is odd is similar.
\end{proof}
\begin{lem}\label{lem:leadingorder}
    The following identity holds:
    \begin{align}\label{eq: leading order symmetry}
        \calD_i (fg) = (\calD_i f) g + \sum_{k=0}^{i-1}\sum_{\substack{A\in \calP_k\\B\in\bar{\calP}_{i-k}}}c_k^{AB} (Af)(Bg),
    \end{align}
    where $c_k^{AB}$ are real constants. 
    As an immediate corollary, we have
    \begin{equation}\label{commutator [calDi g]f}
        [\calD_i, g] f = \sum_{k=0}^{i-1}\sum_{\substack{A\in \calP_k\\B\in\bar{\calP}_{i-k}}}c_k^{AB} (Af)(Bg).
    \end{equation}
    \end{lem}
\begin{proof}
   We prove by induction. It is clear that \eqref{eq: leading order symmetry} is true for $i=1$ since $\calD_1 = \Dz$. Suppose that \eqref{eq: leading order symmetry} holds for all $i\leq \ell$. We will show that it is true for $i=\ell+1$. 
   
   When $\ell$ is even, then $\calD_{\ell+1} = \Dz \calD_{\ell}$ or $\frac{1}{\zeta}\calD_{\ell}$. From the inductive hypothesis, it suffices to show that $\Dz((\calD_\ell f) g)$, $\frac{1}{\zeta}(\calD_\ell f) g$, as well as $\Dz((Af)(Bg))$, $\frac{1}{\zeta}(Af)(Bg)$, with $A\in \calP_k$ and $B\in \bar\calP_{\ell-k}$, are of the correct form. The first two are easy to handle:
    $$
    \Dz((\calD_\ell f) g) = (\Dz\calD_\ell f)g + (\calD_\ell f)(\pz g),
    $$
    and
    $$
    \frac{1}{\zeta}(\calD_\ell f) g = \left(\frac{1}{\zeta}\calD_\ell f\right) g.
    $$
    Additionally, if $k$ is even, then $\ell-k$ is even. Thus, we can write
    $$
    \Dz((Af)(Bg)) = (\Dz A f)(Bg) + (Af)(\pz Bg),  
    $$
    and 
    $$
    \frac{1}{\zeta}(Af)(Bg) =\left(\frac{1}{\zeta} Af\right) (Bg). 
    $$
   We note that both $\Dz A$ and $\frac{1}{\zeta} A\in \calP_{k+1}$ with $k+1\leq \ell$. If $k$ is odd, then $\ell-k$ is odd. Thus, we can write
        $$
    \Dz((Af)(Bg)) = (\dz A f)(Bg) + (Af)(\Dz Bg). 
    $$
    and 
    $$
    \frac{1}{\zeta}(Af)(Bg) = (Af)\left(\frac{1}{\zeta}Bg\right). 
    $$
  Here,  $\dz A\in \calP_{k+1}$ with $k+1\leq \ell$. 
    This completes the induction when $\ell$ is even. 

Let us assume $\ell$ is odd now. In this case, $\calD_{\ell+1} = \pz \calD_{\ell}$. Hence, we need to show that $\dz((\calD_\ell f) g)$ as well as $\dz((Af)(Bg))$ with $A\in \calP_k$ and $B\in \bar\calP_{\ell-k}$, are of the correct form. Note that
$$
\dz((\calD_\ell f) g)  = (\dz \calD_\ell f)g + (\calD_{\ell} f)(\pz g), 
$$
where $\dz \calD_\ell = \calD_{\ell+1}$. As before, we consider $k$ is even and odd separately to justify if $\dz((Af)(Bg))$ is of the correct form. If $k$ is even, then $\ell-k$ is odd. We write
$$
\dz((Af)(Bg)) = (\Dz Af)(Bf) + (Af)(\Dz Bf) - (Af)\left(\frac{4}{\zeta}B g\right),
$$
where $\Dz A\in\calP_{k+1}$ with $k+1\leq\ell$. 
If $k$ is odd, then $\ell-k$ is even. 
We write
$$
\dz((Af)(Bg)) = (\pz Af)(Bg)+ (Af)(\pz Bg). 
$$
This completes the proof since $\pz A\in \calP_{k+1}$. 
\end{proof}
Next, we present a lemma concerning derivatives of a nonlinear term arising from $\mathfrak{R}_1(\Theta)$:
\begin{lem}\label{lem: DiZ}
    Let $Z=\frac{1}{\ze}\left(\ze\pz\left(\frac{\Theta}{\ze}\right)\right)^2$. Then $P_i Z$ verifies the following identity: 
\begin{align}\label{eq:DiZ}
P_i Z = \sum_{1\leq k\leq i}\sum_{\substack{A\in\calP_{k+1}\\B\in\calP_{i-k+2}}}c_i^{AB}A\Theta B\Theta, \quad P_i\in \calP_i.  
\end{align}
\end{lem}
\begin{proof}
Using $\ze\pz\left(\frac{\Theta}{\ze}\right) = \Dz\Theta-3\frac{\Theta}{\ze}$, we see that
$$
Z = \frac{1}{\ze} \left(\Dz\Theta- 3\frac{\Theta}{\ze}\right)^2.
$$
We are now ready to show \eqref{eq:DiZ} by induction. When $P_1=\Dz$, and since $\Dz\frac{1}{\ze} = \frac{1}{\ze^2}$, we have
\begin{equation*}
    \begin{aligned}
\Dz Z =&\frac{2}{\ze}\left(\Dz\Theta- 3\frac{\Theta} {\ze}\right)\pz \left(\Dz\Theta- 3\frac{\Theta} {\ze}\right) + \frac{1}{\ze^2}\left(\Dz\Theta- 3\frac{\Theta} {\ze}\right)^2\\
=&2\pz \left(\frac{\Theta}{\ze}\right)\pz \left(\Dz\Theta- 3\frac{\Theta} {\ze}\right)+\left(\pz\left(\frac{\Theta}{\ze}\right)\right)^2\\
=&2\pz \left(\frac{\Theta}{\ze}\right)\pz\Dz\Theta - 5\left(\pz\left(\frac{\Theta}{\ze}\right)\right)^2. 
\end{aligned}
\end{equation*}
Additionally, when $P_1=\frac{1}{\ze}$, 
\begin{equation*}
  \frac{1}{\ze} Z =   \left(\pz\left(\frac{\Theta}{\ze}\right)\right)^2. 
\end{equation*}
These verify \eqref{eq:DiZ} when $i=1$. Suppose now that \eqref{eq:DiZ} holds for $\ell=1,\cdots,i-1$. If $i$ is odd, $P_iZ = \Dz P_{i-1}Z$ or $\frac{1}{\ze}P_{i-1}Z$. In the former case, 
\begin{align*}
    P_i Z &= \Dz \sum_{1\leq k\leq i-1}\sum_{\substack{A\in\calP_{k+1}\\B\in\calP_{i-k+1}}}c_i^{AB}A\Theta B\Theta\\
   &=
   \begin{cases}
\sum_{1\leq k\leq i-1}\sum_{\substack{A\in\calP_{k+1}\\B\in\calP_{i-k+1}}}c_i^{AB}(\pz A\Theta B\Theta+A\Theta \Dz B\Theta),\quad k\text{ is even};\\
\sum_{1\leq k\leq i-1}\sum_{\substack{A\in\calP_{k+1}\\B\in\calP_{i-k+1}}}c_i^{AB}(\Dz A\Theta B\Theta+A\Theta \pz B\Theta),\quad k\text{ is odd}.
   \end{cases}
\end{align*}
In the latter case, 
\begin{align*}
    P_i Z &= \frac{1}{\ze}\sum_{1\leq k\leq i-1}\sum_{\substack{A\in\calP_{k+1}\\B\in\calP_{i-k+1}}}c_i^{AB}A\Theta B\Theta\\
   &=
   \begin{cases}
\sum_{1\leq k\leq i-1}\sum_{\substack{A\in\calP_{k+1}\\B\in\calP_{i-k+1}}}c_i^{AB}A\Theta\left(\frac{1}{\ze}B\Theta\right),\quad k\text{ is even};\\
\sum_{1\leq k\leq i-1}\sum_{\substack{A\in\calP_{k+1}\\B\in\calP_{i-k+1}}}c_i^{AB}\left(\frac{1}{\ze}A\Theta\right)B\Theta,\quad k\text{ is odd}.
   \end{cases}
\end{align*}
On the other hand, if $i$ is even, then 
\begin{align*}
    \calD_i Z &= \pz \sum_{1\leq k\leq i-1}\sum_{\substack{A\in\calP_{k+1}\\B\in\calP_{i-k+1}}}c_i^{AB}A\Theta B\Theta\\
   &=
   \begin{cases}
\sum_{1\leq k\leq i-1}\sum_{\substack{A\in\calP_{k+1}\\B\in\calP_{i-k+1}}}c_i^{AB}(\pz A\Theta B\Theta+A\Theta \pz B\Theta),\quad k\text{ is even};\\
\sum_{1\leq k\leq i-1}\sum_{\substack{A\in\calP_{k+1}\\B\in\calP_{i-k+1}}}c_i^{AB}\left(\left(\Dz-\frac{2}{\ze}\right) A\Theta B\Theta+A\Theta \left(\Dz-\frac{2}{\ze}\right)B\Theta\right),\quad k\text{ is odd}.
   \end{cases}
\end{align*}
This completes the induction. 
\end{proof}
The following lemma concerns derivatives of nonlinear terms which arise from $\calF$:
\begin{lem}
    \label{lem:Fnonlinear}
    Let $f$ be a smooth function and $P_i \in \calP_i$ for $i \ge 1$. Then the following identities hold:
    \begin{equation}
        \label{eq:thetasquare}
        P_i\left(\frac{f^2}{\zeta}\right) = \sum_{\substack{A_{1,2} \in \calP_{l_1,l_2},\\l_1 + l_2 = i+1,\\l_1,l_2 \le i}}c^{l_1, l_2}_{i}(A_1 f)(A_2f),
    \end{equation}
    \begin{equation}
        \label{eq:thetacube}
        P_i\left(\frac{f^3}{\zeta^2}\right) = \sum_{\substack{A_{1,2,3} \in \calP_{l_1,l_2,l_3},\\l_1 + l_2 + l_3 = i+1,\\l_1,l_2,l_3 \le i}}c^{l_1, l_2, l_3}_{i}(A_1 f)(A_2f)(A_3f).
    \end{equation}
\end{lem}
\begin{proof}
    We prove \eqref{eq:thetasquare} by induction, and the proof of \eqref{eq:thetacube} will be similar. If $i = 1$, we either have $P_1 = \Dz$ or $P_1 = \ze^{-1}$. If the former case hold, then we have
    \begin{align*}
    \Dz\left(\frac{f^2}{\zeta}\right) &= \Dz f \frac{f}{\ze} + f\pz\left(\frac{f}{\ze}\right)= 2\Dz f\frac{f}{\ze} -3\left(\frac{f}{\ze}\right)^2,
    \end{align*}
    which verifies \eqref{eq:thetasquare} when $i = 1$. If the case $P_1 = \ze^{-1}$ holds, then easily $P_1(f^2/\ze) = (f/\ze)^2$, which also verfies the base case.

    To complete the inductive step, we assume that $i \ge 1$ is even. Then $P_{i+1} = VP_{i}$, where $V = \Dz$ or $\ze^{-1}$. Then
    \begin{align*}
    P_{i+1}\left(\frac{f^2}{\ze}\right) &= V\sum_{\substack{A_{1,2} \in \calP_{l_1,l_2},\\l_1 + l_2 = i+1,\\l_1,l_2 \le i}}(A_1 f)(A_2f).
    \end{align*}
    Since $i$ is even and $l_1 + l_2 = i+1$, we may without loss assume that $l_1$ is even and $l_2$ is odd, which implies that $l_1 \le i$ and $l_2 \le i-1$. If $V = \Dz$, then
    \begin{align*}
    P_{i+1}\left(\frac{f^2}{\ze}\right) &= \sum_{\substack{A_{1,2} \in \calP_{l_1,l_2},\\l_1 + l_2 = i+1,\\l_1,l_2 \le i}}(\Dz A_1 f)(A_2f) + (A_1 f)(\pz A_2 f),
    \end{align*}
    which completes the inductive step since $\Dz A_1 \in \calP_{l_1 + 1}$, $\dz A_2 \in \calP_{l_2 + 1}$. If $V = \ze^{-1}$, then
    \begin{align*}
    P_{i+1}\left(\frac{f^2}{\ze}\right) &= \sum_{\substack{A_{1,2} \in \calP_{l_1,l_2},\\l_1 + l_2 = i+1,\\l_1,l_2 \le i}}(\ze^{-1}A_1 f)(A_2f),
    \end{align*}
    which also verfies the inductive step since $\ze^{-1} A_1 \in \calP_{l_1 + 1}$. The case where $i$ is odd is similar.
\end{proof}

\subsection{Chain Rule}
We record a chain rule concerning vector fields of class $\barcalP_i$:
\begin{lem}[Lemma A.5, \cite{guo2021continued}]
    \label{lem:chainrule}
    Let $a \in \R$, $i \in \N$ be given and fix a vectorfield $W \in \barcalP_i$. Then the following holds for any smooth function $f$:
    \begin{equation}
        \label{eq:chainrule}
        W(f^a) = \sum_{k=1}^i f^{a-k}\sum_{\substack{i_1 +\hdots i_k = i\\W_{j} \in \barcalP_{i_j}}}c_{k,i_i,\hdots,i_k}\prod_{j=1}^k W_{j} f.
    \end{equation}
\end{lem}

\subsection{Hardy Inequalities}
Since we work in a weighted energy framework, Hardy-type inequalities and embeddings are fundamental for us to close the nonlinear estimates. We record these inequalities in this section. For the rest of this section, we will always fix $N \ge \lfloor\alpha\rfloor + 10$, where $\alpha := \frac{1}{\kk}$.

The first set of basic Hardy inequalities concerns $L^\infty$ estimates close to the origin.
\begin{lem}[Lemma C.2, \cite{guo2021continued}]
    \label{lem:hardy0}
    Let $m$ be a positive integer, and $u: B_1(0) \to \R$ be a smooth function. Then the following inequalities hold:
    \begin{equation}
        \label{est:Linftyu}
        \|u\|_{L^\infty}^2 \lesssim \sum_{i=1}^2 \int_0^\frac34 |B_i u|^2 r^2dr + \sum_{i = 0}^{m+1} \int_\frac14^1 w^{\alpha - \lfloor\alpha\rfloor + 2m}|B_i u|^2 dr,
    \end{equation}
    and
    \begin{equation}
        \label{est:Linftyur}
        \|ru\|_{L^\infty}^2 \lesssim \sum_{i=0}^1 \int_0^\frac34 |B_i u|^2 r^2dr + \sum_{i = 0}^{m+1} \int_\frac14^1 w^{\alpha - \lfloor\alpha\rfloor + 2m}|B_i u|^2 dr,
    \end{equation}
    where $B_i = \calD_i$ or $\barcalD_i$, $i = 1,\hdots, m+1$. Moreover,
    \begin{equation}
        \label{est:Linftyu/r}
        \|\frac{u}{r}\|_{L^\infty}^2 \lesssim \sum_{i=2}^3 \int_0^\frac34 |\calD_i u|^2 r^2dr + \sum_{i = 0}^{m+1} \int_\frac14^1 w^{\alpha - \lfloor\alpha\rfloor + 2m}|\calD_i u|^2 dr,
    \end{equation}
\end{lem}
Consider the Hilbert space $B^N$ afforded by norm
$$
\|f\|_{B^N}^2 := \sum_{i = 0}^N \|\calD_i f\|_i.
$$
We present various $L^2$ and $L^\infty$ based Hardy inequalities concerning functions in class $B^N$. The first of which is an $L^2$ based embedding result:
\begin{lem}[Lemma C.3, \cite{guo2021continued}]
    Let $\Theta, \pt\Theta \in B^N$, $\pt^2\Theta \in B^{N-1}$. Then the following holds:
    \begin{enumerate}
        \item For $\frac{N-\alpha}{2} \le k \le N$, $P_k \in \calP_k$,
        \begin{equation}
        \label{est:L2Hardy}
        \delta^{-1}\lam^{3\kk}\int_0^1 w^{\alpha + 2k - N}|P_k\pt\Theta|^2 r^2dr + \int_0^1 w^{\alpha + 2k - N}|P_k \Theta|^2 r^2 dr \lesssim \calE^N.
    \end{equation}
    
    \item If we further have $\sum_{j = 0}^N \int_0^1 w^{\alpha + j + 1}|\calD_j \Theta|^2 r^2 dr < \infty$, then for $\frac{N-\alpha - 1}{2} \le k \le N$, $P_k \in \calP_k$:
    \begin{equation}
        \label{est:L2Hardy2}
        \int_0^1 w^{\alpha + 2k + 1 - N} |P_{k+1}\Theta|^2 r^2 dr \lesssim \calE^N.
    \end{equation}
    \item For $\frac{N-\alpha - 1}{2} \le k \le N-1$, $P_k \in \calP_k$,
    \begin{equation}
        \label{est:L2Hardypt2}
        \int_0^1 w^{\alpha + 2k + 1 - N}|P_k\pt^2\Theta|^2 r^2dr \lesssim \calS^N.
    \end{equation}
    \end{enumerate}
\end{lem}
\begin{proof}
    The proof of the first two items is in a similar spirit to that of Lemma C.3 in \cite{guo2021continued}. The only key difference is that the differential operators appeared in \eqref{est:L2Hardy} and \eqref{est:L2Hardy2} are general $\calP_k$ vector fields instead of $\calD_k$, which is allowed in view of the coercivity estimate \eqref{est:Dicoercivity}. The proof of \eqref{est:L2Hardypt2} follows similarly. Note that $k$ cannot exceed $N-1$ due to the definition of $S^N$.
\end{proof}

The following two lemmas concern $L^\infty$ Hardy embeddings with or without weight $w$.
\begin{lem}[Lemma C.4, \cite{guo2021continued}]
    \label{lem:LinftyHardy}
    Let $\Theta, \pt\Theta \in B^N$ and $\pt^2\Theta \in B^{N-1}$. Then the following holds:
    \begin{enumerate}
        \item For any nonnegative integer $k$ such that $k \le \frac{N - \lfloor\alpha\rfloor-2}{2}$ and given $P_k \in \calP_k$, we have
        \begin{equation}
            \label{est:LinftyHardy1}
            \delta^{-1}\lambda^{3\kk}\|P_k\pt\Theta\|_{L^\infty}^2 + \|P_k \Theta\|_{L^\infty}^2 \lesssim \calE^N.
        \end{equation}

        \item For any nonnegative integer $k$ such that $k \le \frac{N - \lfloor\alpha\rfloor-2}{2}$ and given $\bar P_k \in \barcalP_k$, we have
        \begin{equation}
            \label{est:LinftyHardy2}
            \delta^{-1}\lambda^{3\kk}\|\bar P_k\pt\left(\frac{\Theta}{\ze}\right)\|_{L^\infty}^2 + \|\bar P_k \left(\frac{\Theta}{\ze}\right)\|_{L^\infty}^2 \lesssim \calE^N.
        \end{equation}

        \item For any nonnegative integer $k$ such that $k \le \frac{N - \lfloor\alpha\rfloor-3}{2}$ and given $P_k \in \calP_k$, we have
        \begin{equation}
            \label{est:LinftyHardypt2}
            \|P_k\pt^2\Theta\|_{L^\infty}^2  \lesssim \calS^N.
        \end{equation}
    \end{enumerate}
\end{lem}
\begin{proof}
    The proof follows similarly to that of Lemma C.4 in \cite{guo2021continued}, so we are content with sketching the modified proof here. Invoking \eqref{est:Linftyu} with $u = P_k \Theta$ and $m = \lfloor \alpha \rfloor + k + 1$, we arrive at
    \begin{align*}
        \|P_k \Theta\|_{L^\infty} &\lesssim \sum_{i=1}^2 \int_0^\frac34 |B_i P_k \Theta|^2 r^2dr + \sum_{i = 0}^{\lfloor \alpha \rfloor + k + 2} \int_\frac14^1 w^{\alpha - \lfloor\alpha\rfloor + 2(\lfloor \alpha \rfloor + k + 1)}|B_i P_k \Theta|^2 dr
    \end{align*}
    Selecting $B_i = \calD_i$ if $k$ is even and $B_i = \barcalD_i$ if $k$ is odd, we have $B_iP_k \in \calP_{k+i}$. Since $i + k \le 2 + \frac{N-\lfloor\alpha\rfloor - 2}{2} \le N$ and $w \gtrsim 1$ in $[0,\frac34]$, the first sum is bounded by $\calE^N$ after invoking the coercivity estimate \eqref{est:Dicoercivity}. On the other hand, since $w \lesssim 1$, then $w^{\alpha - \lfloor\alpha\rfloor + 2(\lfloor \alpha \rfloor + k + 1)} \lesssim w^{\alpha + i + k}$ for $i \le \lfloor \alpha \rfloor + k + 2$. Then the second sum is also bounded by $\calE^N$ after invoking coercivity \eqref{est:Dicoercivity} again. Other estimates follow similarly.
\end{proof}

\begin{lem}[Lemma C.5, \cite{guo2021continued}]\label{lem:LinftyHardyweighted}
Let $\Theta,\pt\Theta \in B^N$ and $\pt^2\Theta \in B^{N-1}$. Then the following holds:
\begin{enumerate}
    \item Given any $0 \le k \le N - 4$ and $P_{k+2} \in \calP_{k+2}$, we have
    \begin{equation}
        \label{est:LinftyHardyweighted1N-4}
        \delta^{-1}\lambda^{3\kk}\|w^kP_{k+2}\pt\Theta \|_{L^\infty}^2 + \|w^kP_{k+2}\Theta\|_{L^\infty}^2 \lesssim \calE^N.
    \end{equation}
    For $k = N-3$ and $P_{N-1} \in \calP_{N-1}$, we have
    \begin{equation}
        \label{est:LinftyHardyweighted1N-3}
        \delta^{-1}\lambda^{3\kk}\|\ze w^{N-3}P_{N-1}\pt\Theta \|_{L^\infty}^2 + \|\ze w^{N-3}P_{N-1}\Theta\|_{L^\infty}^2 \lesssim \calE^N.
    \end{equation}
    
    \item Given any $0 \le k \le N - 5$ and $P_{k+2} \in \calP_{k+2}$, we have
    \begin{equation}
        \label{est:LinftyHardyweightedpt2}
        \|w^kP_{k+2}\pt^2\Theta \|_{L^\infty}^2 \lesssim \calS^N.
    \end{equation}
    
    \item For any $0 \le k \le N-5$ and $\bar P_{k+2} \in \barcalP_{k+2}$,
    \begin{equation}
        \label{est:LinftyHardyweighted2N-5}
        \delta^{-1}\lambda^{3\kk}\|w^k\bar P_{k+2}\pt\left(\frac{\Theta}{\ze}\right) \|_{L^\infty}^2 + \|w^k\bar P_{k+2}\left(\frac{\Theta}{\ze}\right)\|_{L^\infty}^2 \lesssim \calE^N.
    \end{equation}
    For $k = N-4$ and $\bar P_{N-2} \in \barcalP_{N-2}$, we have
    \begin{equation}
        \label{est:LinftyHardyweighted2N-4}
        \delta^{-1}\lambda^{3\kk}\|\ze w^{N-4}\bar P_{N-2}\pt\left(\frac{\Theta}{\ze}\right) \|_{L^\infty}^2 + \|\ze w^{N-4}\bar P_{N-2}\left(\frac{\Theta}{\ze}\right)\|_{L^\infty}^2 \lesssim \calE^N.
    \end{equation}
\end{enumerate}
\end{lem}
\begin{proof}
    We prove this lemma here since the proof proceeds slightly differently from that of Lemma C.5 of \cite{guo2021continued}. We first prove \eqref{est:LinftyHardyweighted1N-4}. Fix $k \le N-4$ and $P_{k+2} \in \calP_{k+2}$. Invoking \eqref{est:Linftyu} with $u = w^k P_{k+2}\Theta$ and $m = 3 + \lfloor\alpha\rfloor - k$, we have
    \begin{align*}
        \|w^k P_{k+2}\Theta\|_{L^\infty} &\lesssim \sum_{i=1}^2 \int_0^\frac34 |B_i(w^k P_{k+2}\Theta)|^2 r^2dr\\
        &\quad+ \sum_{i = 0}^{4 + \lfloor\alpha\rfloor - k} \int_\frac14^1 w^{\alpha - \lfloor\alpha\rfloor + 2(3 + \lfloor\alpha\rfloor - k)}|B_i(w^k P_{k+2}\Theta)| r^2 dr.
    \end{align*}
    We also let $B_i = \calD_i$ for $k$ even or $B_i = \barcalD_i$ for $k$ odd. To treat the first sum, we invoke the product rule \eqref{Pproduct} (or \eqref{Pbarproduct}) to deduce that
    $$
    B_i(w^k P_{k+2}\Theta) = \sum_{j=0}^i\sum_{A_1 \in \barcalP_j, A_2 \in \calP_{k+2+i-j}} c_{jk}A_1(w^k)A_2\Theta.
    $$
    By the chain rule \eqref{eq:chainrule} as well as Lemma \ref{lem:wderivative}, we know that $A_1(w^k)$ is a smooth function near $\ze = 0$, and thus $|A_1(w^k)| \lesssim 1$. Since $k+2 + i \le k+4 \le N$, we conclude that the first sum is controlled by $\calE^N$ after using \eqref{est:Dicoercivity}. 

    As for the second sum, a similar reasoning to above would yield 
    \begin{align*}
        &\sum_{i = 0}^{4 + \lfloor\alpha\rfloor - k} \int_\frac14^1 w^{\alpha - \lfloor\alpha\rfloor + 2(3 + \lfloor\alpha\rfloor - k)}|B_i(w^k P_{k+2}\Theta)| r^2 dr \\
        &\lesssim \sum_{i = 0}^{4 + \lfloor\alpha\rfloor - k}\sum_{j=0}^i\sum_{A_1 \in \barcalP_j, A_2 \in \calP_{k+2+i-j}} \int_\frac14^1 w^{\alpha - \lfloor\alpha\rfloor + 2(3 + \lfloor\alpha\rfloor - k)} |A_1(w^k)A_2\Theta|^2r^2 dr\\
        &\lesssim \sum_{i = 0}^{4 + \lfloor\alpha\rfloor - k}\sum_{j=0}^i\sum_{A_2 \in \calP_{k+2+i-j}} \int_\frac14^1 w^{\alpha - \lfloor\alpha\rfloor + 2(3 + \lfloor\alpha\rfloor - k)} w^{2k-2j}|A_2\Theta|^2r^2 dr\\
        &\lesssim \sum_{i = 0}^{4 + \lfloor\alpha\rfloor - k}\sum_{j=0}^i\sum_{A_2 \in \calP_{k+2+i-j}} \int_\frac14^1 w^{\alpha + \lfloor\alpha\rfloor + 6-2j} |A_2\Theta|^2r^2 dr
    \end{align*}
    where we used the chain rule and Lemma \ref{lem:wderivative} in the second to the last inequality. Now, note that $w^{\alpha + \lfloor\alpha\rfloor + 6-2j} \lesssim w^{\alpha + k + 2 + i-j}$ since $j \le i \le 4 + \lfloor\alpha\rfloor - k$. Moreover, the term which contains the most derivatives have $k + 2 + i \le 6 + \lfloor\alpha\rfloor \le N$ derivatives. We thus conclude that the second sum is also bounded by $\calE^N$ after using \eqref{est:Dicoercivity}. The bound for $\delta^{-1}\lambda^{-3\kk}\|w^kP_{k+2}\pt\Theta \|_{L^\infty}^2$ is similar. 

    The proof for other inequalities follow from similar arguments, where the proof for \eqref{est:LinftyHardyweighted1N-3} and \eqref{est:LinftyHardyweighted2N-4} utilizes \eqref{est:Linftyur} in place of \eqref{est:Linftyu}.
\end{proof}

\subsection{Miscellaneous Statements}

\begin{lem}
    \label{L:Inverse_odd_function}
    Let $\psi \colon [0,a]\rightarrow \mathbb{R}$ be invertible, $C^1$, and have the form
    \begin{align*}
        \psi(\zeta) = \zeta \mathbf{F}(\zeta^2),
    \end{align*}
    where $\mathbf{F}(0) \neq 0$. Then, near the origin, we can write 
    \begin{align*}
        \psi^{-1}(\zeta) = \zeta \mathbf{G}(\zeta^2),
    \end{align*}
    where $\mathbf{G}(0)\neq 0$.
\end{lem}
\begin{proof}
    Set $s=\zeta^2$ and $\Psi(s) = s (\mathbf{F}(s))^2$. Then $\Psi^\prime(0) >0$ and $\Psi$ is a $C^1$ diffeomorphism in a neighborhood of the origin. Let $\alpha=\psi(\zeta)=\zeta \mathbf{F}(\zeta^2)$, so $\alpha^2=\zeta^2 (\mathbf{F}(\zeta^2))^2=\Psi(\zeta^2)$, thus $\Psi^{-1}(\alpha^2) = \zeta^2$ for small $\zeta$. Therefore
    \begin{align*}
        \psi^{-1}(\alpha) = \zeta = \frac{\alpha}{\mathbf{F}(\zeta^2)} = \frac{\alpha}{\mathbf{F}(\Psi^{-1}(\alpha^2))} = \alpha \mathbf{G}(\alpha^2)
    \end{align*}
    where $\mathbf{G}(s) = \frac{1}{\mathbf{F}(\Psi^{-1}(s))}$ satisfies $\mathbf{G}(0) \neq 0$.
\end{proof}

\begin{lem}
    \label{L:calD_odd_functions}
    Let $\GenericFunction \in C^{k}([0,1])$ satisfy $\partial_\zeta^{2\ell}\GenericFunction(0)=0$, $2\ell \leq k$. Then
\begin{align*}
    \calD_{2\ell} \GenericFunction(0) & = 0, \, 2\ell \leq k,
    \\
    \calD_{2\ell+1} \GenericFunction(0) & = \text{finite value }, \, 2\ell+1 \leq k,
\end{align*}
where we recall that $\calD_\ell$ is given by 
\eqref{calDj}. Moreover,
\begin{align*}
    \| \calD_\ell \GenericFunction \|_{L^\infty([0,1])} \lesssim \| \GenericFunction \|_{C^\ell([0,1]), }\, \ell \leq k.
\end{align*}
\end{lem}
\begin{proof}
For notational convenience, let us write $\mathsf{L} := \partial_\zeta \Dz$, where we recall from \eqref{E:Primitive_weighted_derivative} that $\Dz = \frac{1}{\ze^2}\pz (\ze^2\cdot) = \pz + \frac{2}{\ze}$. Then $\calD_{2\ell} = \mathsf{L}^\ell$ and $\calD_{2\ell+1}= \Dz \mathsf{L}^\ell$.

Near the origin, we can write
\begin{align*}
    \GenericFunction(\zeta) = \zeta \mathbf{F}(\zeta^2),
\end{align*}
for some $C^k$ function $\mathbf{F}$. With $s=\zeta^2$, we have $\zeta^2 \GenericFunction(\zeta) = \zeta^3 \mathbf{F}(s)$ and then
\begin{align*}
    \partial_\zeta (\zeta^2 \GenericFunction(\zeta)) = \partial_\zeta(
    \zeta^3 \mathbf{F}(s) )
    = 3\zeta^2 \mathbf{F}(s) + 2\zeta^4 \mathbf{F}^\prime(s),
\end{align*}
so that dividing by $\zeta^2$ the operator on LHS becomes $\Dz$ and we get 
\begin{align}
    \label{E:Generic_Dz}
    \Dz \GenericFunction(\zeta) = 3\mathbf{F}(s) + 2s\mathbf{F}^\prime(s),
\end{align}
which is a function only of $s=\zeta^2$, i.e., even in $\zeta$, thus it extends continuously to $\zeta=0$. Hence
$\Dz \GenericFunction(0)= 3\mathbf{F}(0)$. Applying $\partial_\zeta$ to \eqref{E:Generic_Dz} we then obtain 
\begin{align}
    \label{E:Generic_L}
    \mathsf{L} \GenericFunction(\zeta) = \zeta (10 \mathbf{F}^\prime(s) + 4s\mathbf{F}^{\prime\prime}(s) ).
\end{align}
The RHS of \eqref{E:Generic_L} is of the form $\zeta$ times a function of $s=\zeta^2$, thus it extends continuously to $\zeta=0$ so that $\mathsf{L}\GenericFunction(0) = 0$.

Define recursively
\begin{align*}
    \mathbf{F}_{\ell+1} & := \mathsf{L} \mathbf{F}_\ell
    \\
    &= 10 \mathbf{F}^\prime_\ell + 4\zeta^2\mathbf{F}^{\prime\prime}_\ell,
\end{align*}
with $\mathbf{F}_0 := \mathbf{F}$. Then \eqref{E:Generic_L} says that
\begin{align*}
    \mathsf{L} \GenericFunction(\zeta) = \zeta \mathbf{F}_1(\zeta^2).
\end{align*}
Applying to $\mathsf{L} \GenericFunction$ the same reasoning that led to \eqref{E:Generic_L} and proceeding inductively gives
\begin{align*}
    \mathsf{L}^\ell\GenericFunction(\zeta) = \zeta \mathbf{F}_\ell(\zeta^2)
\end{align*}
and $\mathsf{L}^\ell\GenericFunction(0)=0$. Also, applying $\Dz$ to $\mathsf{L}^\ell\GenericFunction$ gives, as in \eqref{E:Generic_Dz},
\begin{align*}
    \Dz (\mathsf{L}^\ell\GenericFunction)(\zeta) = 3 \mathbf{F}_\ell(\zeta^2) + 2\zeta^2\mathbf{F}^\prime_\ell(\zeta^2),
\end{align*}
and thus $\Dz (\mathsf{L}^\ell\GenericFunction)(0)$ is a well-defined finite value. The norm statement follows from the RHS of the above expressions for $\mathsf{L}^\ell$ and $\Dz (\mathsf{L}^\ell)$.
\end{proof}

\section{Proof of Theorem \ref{T:LWP}}
\label{S:Proof_of_LWP}

In this appendix we establish local existence and uniqueness and a continuation criterion for equations \eqref{eq:momlagmain}. 

We begin by recalling the definition of the topology of $\DITWeightedSpace_0^{2k}$. 
A sequence $(\mathring{\NewSoundSpeedSq}_n,\mathring{\NewVelocity}_n^i) \rightarrow (\mathring{\NewSoundSpeedSq},\mathring{\NewVelocity}^i)$
in $\DITWeightedSpace_0^{2k}$ if  (i) $|\nabla \mathring{\NewSoundSpeedSq}_n| \geq c >0$ (this is a uniform degeneracy condition); (ii) $\| \mathring{\NewSoundSpeedSq}_n - \mathring{\NewSoundSpeedSq}\|_{Lip(\mathbb{R}^3)} \rightarrow 0$, where $\mathring{\NewSoundSpeedSq}_n$ and $\mathring{\NewSoundSpeedSq}$ are extended as zero outside their domains, resulting in Lipschitz functions on $\mathbb{R}^3$ (this is a domain convergence condition); (iii) given $\epsilon >0$, there exist smooth functions $\mathring{\NewSoundSpeedSq}_\text{smooth}$ and $\mathring{\NewVelocity}_\text{smooth}^i$ defined in a neighborhood of $\Omega_{\mathring{\NewSoundSpeedSq}}$ such that $\| (\mathring{\NewSoundSpeedSq},\mathring{\NewVelocity}^i) - (\mathring{\NewSoundSpeedSq}_\text{smooth}, \mathring{\NewVelocity}_\text{smooth}^i) \|_{\DITWeightedSpace^{2k}_{\mathring{\NewSoundSpeedSq}}(\Omega_{\mathring{\NewSoundSpeedSq}})} \leq \epsilon$ and
$\limsup_{n\rightarrow\infty}\| (\mathring{\NewSoundSpeedSq}_n,\mathring{\NewVelocity}^i_n) - (\mathring{\NewSoundSpeedSq}_\text{smooth}, \mathring{\NewVelocity}_\text{smooth}^i) \|_{\DITWeightedSpace^{2k}_{\mathring{\NewSoundSpeedSq}_n}(\Omega_{\mathring{\NewSoundSpeedSq}_n})} \leq \epsilon$ (this condition serves three purposes: it provides a uniform bound and an equicontinuity-type of property on $(\mathring{\NewSoundSpeedSq}_n,\mathring{\NewVelocity}^i_n)$, which in turn ensures that we cannot have ever larger amounts of the norm of the sequence concentrating on very small neighborhoods of the boundary, and it gives triangle-inequality type of comparison, with $(\mathring{\NewSoundSpeedSq}_\text{smooth}, \mathring{\NewVelocity}_\text{smooth}^i)$ serving as the middle term since it is defined on both $\Omega_{\mathring{\NewSoundSpeedSq}}$ and $\Omega_{\mathring{\NewSoundSpeedSq}_n}$ for large $n$).

\begin{proof}[Proof of Theorem \ref{T:LWP}] The proof is given in several steps.

\medskip
\noindent\textbf{Step 1} (Radially symmetric data  launches radially symmetric solutions). Since the variables in \eqref{eq:momlagmain} assume a spherically symmetric solution to \eqref{eq:RE}, we first need to verify that such solutions are spherically symmetric if the initial data is. The argument is standard (see \cite{MR1724661}), but since it has not been explicitly written in a physical vacuum boundary context, we provide it here.

By \cite{disconzi2022relativistic}, there exists a classical solution $(\rho,u^i)$ to equations \eqref{eq:RE} defined on
\begin{align}
    \label{E:Moving_domain_spherically_symmetric_LWP}
    \bigcup_{0 \leq t < T^*} \{ t \} \times \Omega_t,
\end{align}
for some $T^*>0$, where $\Omega_t$ is given by \eqref{def: moving domain Omega_t} and taking the given data $(\mathring{\rho},\mathring{u}^i)$ at $t=0$, provided that
\begin{align}
        \label{E:k_DIT_LWP}
        2k &> 5 +\frac{1}{\kappa}.
\end{align}
Condition \eqref{E:k_DIT_LWP} is satisfied in view of \eqref{E:Fixed_N_bound} and \eqref{E:N_k_relation}. Moreover, such solution satisfies the physical vacuum boundary condition.

Let $\mathsf{R} \in SO(3)$ and define
\begin{align*}
    \rho_\mathsf{R}(t,x) &:= \rho(t, \mathsf{R}^{-1} x),
    \\
    u^i_\mathsf{R}(t,x) & := \mathsf{R}\indices{^i_j}  u^j(t,\mathsf{R}^{-1} x), 
\end{align*}
which are defined in the rotated domain
\begin{align*}
    & \bigcup_{0 \leq t < T^*} \{ t \} \times \Omega_t^\mathsf{R},
    \\
    & \Omega_t^\mathsf{R} := \mathsf{R}(\Omega_t) := \{ \mathsf{R}x \, | \, x \in \Omega_t \}.
\end{align*}
The time component of the four velocity transforms as 
\begin{align*}
    u^0_\mathsf{R}(t,x) := u^0(t,\mathsf{R}^{-1}x), 
\end{align*}
so that \eqref{eq: RE_normalization} is preserved in view of the invariance of the Minkowski metric under rotations. Since $SO(3) \subset \text{Lorentz group}$ and the Lorentz group maps solutions to relativistic Euler equations into solutions, we have that $(\rho_\mathsf{R},u^i_\mathsf{R})$ also solves \eqref{eq:RE}. 

Next, we verify that $(\rho_\mathsf{R},u^i_\mathsf{R})$ satisfies the physical vacuum boundary condition in the rotated domain. First, observe that if $x \in \partial \Omega_t^\mathsf{R}$, then $\mathsf{R}^{-1}x \in \partial \Omega_t$ and therefore $\rho_\mathsf{R}(t,x) = \rho(t,\mathsf{R}^{-1}x) =0$, so that $\rho_\mathsf{R}$ vanishes on the boundary of the rotated domain; we also see that $\rho_\mathsf{R}$ remains positive away from the boundary. The sound speed squared in the rotated domain is given by $c_{s,\mathsf{R}}^2 = (1+\kappa) \rho_\mathsf{R}^\kappa$. Because $\mathsf{R}$ is an isometry of $\mathbb{R}^3$, we have that $\text{dist(}x,\partial\Omega_t^\mathsf{R})=\text{dist}(\mathsf{R}^{-1}x,\partial\Omega_t)$. Therefore
\begin{align*}
    c_{s,\mathsf{R}}^2(t,x) = (1+\kappa) \rho^\kappa_\mathsf{R}(t,x) = (1+\kappa)\rho^\kappa(t,\mathsf{R}^{-1}x) = c_s^2(t,\mathsf{R}^{-1}x) \approx 
    \text{dist}(\mathsf{R}^{-1}x,\partial\Omega_t)
    =\text{dist(}x,\partial\Omega_t^\mathsf{R}),
\end{align*}
i.e., $c_{s,\mathsf{R}}^2$ is comparable to the distance to the boundary in the rotated domain. Finally, to verify the kinematic boundary condition that the four-velocity is tangent to the free-boundary in spacetime for the rotated quantities (or equivalently, that the rotated domain is advected by the rotated velocity), we  observe that if $(t,\gamma(t)) \subset \cup_{0\leq t < T^*} \{t\} \times \partial\Omega_t$ is a fluid-particle trajectory, so that
\begin{align*}
    \frac{d}{dt}\gamma(t) = \frac{u^i}{u^0}(t,\gamma(t)),
\end{align*}
then the rotated trajectory $(t,\gamma_\mathsf{R}(t)) = (t, \mathsf{R} \gamma(t))$ is such that $(t,\gamma_\mathsf{R}(t)) \subset  \cup_{0\leq t < T^*} \{t\} \times \partial\Omega^\mathsf{R}_t$ and
\begin{align*}
     \frac{d}{dt}\gamma_\mathsf{R}(t) = \mathsf{R} \frac{d}{dt} \gamma(t) = \left( \mathsf{R} \frac{u^i}{u^0} \right)(t,\gamma(t)) = \frac{\mathsf{R}\indices{^i_j} u^j(t,\gamma(t))}{u^0(t,\gamma(t))}  
     =
     \frac{\mathsf{R}\indices{^i_j} u^j(t,\mathsf{R}^{-1}\gamma_\mathsf{R}(t))}{u^0(t,\mathsf{R}^{-1}\gamma_\mathsf{R}(t))}
     =\frac{u^i_\mathsf{R}(t,\gamma_\mathsf{R}(t))}{u^0_\mathsf{R}(t,\gamma_\mathsf{R}(t))},
\end{align*}
as needed.

At $t=0$ we have
\begin{align*}
    \rho_\mathsf{R}(0,x) &= \rho(0,\mathsf{R}^{-1}x) = \mathring{\rho}(\mathsf{R}^{-1}x)  = \mathring{\rho}(x),
    \\
    u^i_\mathsf{R}(0,x) & := \mathsf{R}\indices{^i_j}  u^j(0,\mathsf{R}^{-1} x)
    = \mathsf{R}\indices{^i_j}  \mathring{u}^j(\mathsf{R}^{-1} x)
    =
    \mathring{u}^j(x),
\end{align*}
where the last equality on each line follows from the assumption that the data is spherically symmetric. Thus, $(\rho_\mathsf{R},u^i_\mathsf{R})$ and $(\rho,u^i)$ are two solutions to \eqref{eq:RE} satisfying the physical vacuum boundary condition and taking the same data at time zero. By uniqueness (also established in \cite{disconzi2022relativistic}), they must agree, so that for all $0 \leq t < T^*$, $\Omega_t = \Omega_t^\mathsf{R}$, $\rho_\mathsf{R}(t,x) =\rho(t,x)$, $u_\mathsf{R}^i(t,x)=u^i(t,x)$, thus in particular
\begin{align*}
    \rho(t, \mathsf{R}^{-1} x) = \rho(t,x),
    \\
    \mathsf{R}\indices{^i_j}  u^j(t,\mathsf{R}^{-1} x) = u^i(t,x), 
\end{align*}
i.e., $(\rho,u^i)$ is spherically symmetric. (In particular, each $\Omega_t$ must be a ball.)

\medskip
\noindent\textbf{Step 2} (Definiteness of the change of variables at the level of data). Here, we follow the arguments of Section \ref{sect: setup} and check that the procedure is well defined at the level of the data. Throughout, we will refer to expressions from Section \ref{sect: setup} that involve $\rho$, $u^i$, and $u^0$, but it should be understood that they are to be taken with $\mathring{\rho}$, $\mathring{u}^i$, and $\mathring{u}^0$ instead. We will refer to expressions as being well defined to mean that they are determined uniquely in terms of the data $(\mathring{\rho},\mathring{u}^i)$ and the fixed parameters in the statement of the Theorem.
We will not track the regularity of the change of variables in this step because this will be done further below for spacetime quantities, from which in particular we can verify the regularity of the data $(\Theta_0,U_0)$ constructed out of $(\mathring{\rho},\mathring{u}^i)$.

In view of our assumption on the data, \eqref{ansatz radial symmetry} and \eqref{u00} are well defined.
The ODE defining $\lambda$ in \eqref{lambdaODEt}
depends only on $\delta$, $\kappa$, and the initial conditions \eqref{lambdaODEt_data} involve only the parameters $\lambda_0$ and $\lambda_1$. Thus, the change of variables \eqref{defn:scaling} is well-defined, as are the modified velocity $V$ in \eqref{modvel} and the expression for the time component of $\mathring{u}^0$ in \eqref{u0}.
It follows that $\partial_s\eta$ in \eqref{E:d_eta_d_s} is well defined at $s=0$ in terms of the data and $\eta_0$. It follows that $f$, $U$, and $U^0$ in \eqref{E:f_and_U_of_s} are also well defined at $s=0$, as is  $\calF$ in \eqref{E:Jac_det_def}. By \eqref{E:d_tau_d_s} and $\partial_\tau = \lambda^{-\frac{3}{2}\kk}\partial_s$, all these quantities are also well defined at $\tau=0$, thus in terms of the data and $\eta_0$. 

To fix $\eta_0$, we invoke \eqref{gauge fixing}, but also note that \eqref{gauge fixing} imposes a constraint. Integrating \eqref{gauge fixing} and recalling that $w_\delta = \delta w$,
\begin{align*}
   \int_0^1 \calF_0 f_0 \zeta^2 d\zeta = \int_0^1 (\delta w)^{\frac1\kk} \zeta^2 d\zeta.
\end{align*}
The LHS is $\frac{1}{4\pi}M_\text{ref}$ by \eqref{E:rho_ref} and \eqref{E:Reference_mass}, whereas the RHS is $\frac{1}{4\pi}M$ in view of \eqref{E:f_and_U_of_s}, \eqref{E:Jac_det_def},  \eqref{E:d_tau_d_s}, a change of variables, and \eqref{E:Initial_mass}. The constraint $M=M_\text{ref}$ is satisfied by \eqref{E:Mass_constraint}.

It now follows from the above that $\Theta(0,\cdot)$ given by \eqref{E:Theta_def} and $\partial_\tau \Theta(0,\cdot)$ obtained from \eqref{E:d_eta_d_s}, and \eqref{E:Theta_def}, \eqref{E:d_tau_d_s}, are uniquely determined in terms of the data (and the gauge choice \eqref{gauge fixing}).

To show that $(\Theta_0,U_0)=(0,0)$ if $\mathring{\rho}=\mathring{\rho}_{\text{ref}}$ and $\mathring{u}^i=\mathring{u}^i_{\text{ref}}$, it suffices to verify that for such data we have $V(0,\cdot)=0$ and that \eqref{gauge fixing} is satisfied by taking $\eta_0(\zeta)=\zeta$. Both statements are a direct consequence of following the above steps in the construction of $(\Theta_0,U_0)$ using $(\mathring{\rho},\mathring{u}^i)=(\mathring{\rho}_{\text{ref}},\mathring{u}^i_{\text{ref}})$.

This establishes item 1) of the Theorem.

\medskip
\noindent\textbf{Step 3} (Spherically symmetric solutions to \eqref{eq:RE} give rise to solutions to \eqref{eq:momlagmain}). By Step 1 and \cite{disconzi2022relativistic}, there exists a unique spherically symmetric classical solution ($\rho,u^i)$ to \eqref{eq:RE} defined on \eqref{E:Moving_domain_spherically_symmetric_LWP} for some $T>0$, and taking the given data at $t=0$. In view of \eqref{E:d_tau_d_s}, if we set
\begin{align}
    \label{E:Time_existence_Lagrangian}
    \tau^* := \int_0^{T^*} \frac{1}{\lambda(t)} dt,
\end{align}
then the procedure in Section \ref{sect: setup} immediately produces a solution $(\Theta,\partial_\tau\Theta)$ defined for $0\leq \tau < \tau^*$ and taking the data $(\Theta_0,U_0)$ produced in Step 2, provided all steps in Section \ref{sect: setup} are justified. This, in turn, is a matter of verifying that there is enough regularity to carry out all steps. The regularity of $(\Theta,\partial_\tau\Theta)$ is investigated in the next step, which shows in particular that $(\Theta,\partial_\tau\Theta)$ is at least (in fact, much better) $C^2([0,\tau^*)\times[0,1])\times C^1([0,\tau^*)\times[0,1])$, hence a classical solution.

Uniqueness of \eqref{eq:momlagmain} for the class of data $(\Theta_0,U_0)$ constructed out of $(\mathring{\rho},\mathring{u}^i)$ in Step 2 follows from uniqueness of \eqref{eq:RE} (established in \cite{disconzi2022relativistic}).

This establishes the existence and uniqueness part in item 2) of the Theorem.

\medskip
\noindent\textbf{Step 4} (Spacetime and radial regularity). The solution $(\NewSoundSpeedSq,\NewVelocity^i)$ given by \cite{disconzi2022relativistic} has regularity
\begin{align*}
    (\NewSoundSpeedSq(t,\cdot),\NewVelocity^i(t,\cdot)) & \in \DITWeightedSpace_\NewSoundSpeedSq^{2k}(\Omega_t), 
    \\
    (\partial_t\NewSoundSpeedSq(t,\cdot),\partial_t\NewVelocity^i(t,\cdot)) & \in \DITWeightedSpace_\NewSoundSpeedSq^{2k-1}(\Omega_t),
\end{align*}
from which follows, in light of \eqref{E:N_k_relation} and \eqref{E:Weighted_embedding_DIT}, that
\begin{align*}
    (\NewSoundSpeedSq(t,\cdot),\NewVelocity^i(t,\cdot)) &\in C^{N+2}(\overline{\Omega_t}),
    \\
    (\partial_t\NewSoundSpeedSq(t,\cdot),\partial_t \NewVelocity^i(t,\cdot)) &\in C^{N+1}(\overline{\Omega_t}). 
\end{align*}
From the equations of motion we then obtain standard regularity for time derivatives, where each $\partial_t$ corresponds to one lower spatial regularity. From \eqref{E:New_sound_and_velocity} and the physical vacuum boundary condition it follows that
\begin{align*}
    u^i(t,\cdot) & \in C^{N+2}(\overline{\Omega_t}), \quad
    \partial_t u^i(t,\cdot) \in C^{N+1}(\overline{\Omega_t}),  
    \\
    \rho(t,\cdot) & \in C^{N+2}_{loc}(\Omega_t), \quad
    \partial_t \rho(t,\cdot) \in C^{N+1}_{loc}(\Omega_t),     
\end{align*}
again with higher time derivatives having regularity determined by the equations of motion as above. Crucially, these all refer to the regularity relative to rectangular coordinates in $\mathbb{R}^3$, which therefore implies the usual compatibility conditions at the origin for spherically symmetric solutions, so that 
the variables \eqref{ansatz radial symmetry} satisfy
\begin{align*}
    \partial_r^{2\ell+1} \rho(t,0) & = 0,\, 
    2\ell+1 \leq N+2,
    \\
    \partial_r^{2\ell}v(t,0) & = 0, \, 
    2\ell \leq N+2,
\end{align*}
The variables $\tilde{\rho}$ and $\tilde{v}$ in \eqref{defn:scaling} inherit the regularity of $\rho$ and $u^i$ and thus also satisfy the same compatibility conditions in $r$ at the origin. Thus we can summarize the regularity of the relevant variables as
\begin{subequations}{\label{E:Regularity_rescaled_symmetric_variables}}
\begin{align}
    \label{E:Regularity_rescaled_v}
   & \tilde{v}(t,\cdot) \in C^{N+2}([0,1]), \quad
    \partial_t \tilde{v}(t,\cdot) \in C^{N+1}([0,1]),  
    \\
    \label{E:Regularity_rescaled_rho}
    &\tilde{\rho}(t,\cdot)  \in C^{N+2}_{loc}([0,1)), \quad
    \partial_t \rho(t,\cdot) \in C^{N+1}_{loc}([0,1)), 
    \\
    \label{E:Regularity_rescaled_sound_speed}
    &\tilde{\rho^\kappa}(t,\cdot)  \in C^{N+2}([0,1]), \quad
    \partial_t \rho^\kappa(t,\cdot) \in C^{N+1}([0,1]), 
    \\
    \label{E:Regularity_rescaled_v_at_zero}
    & \partial_r^{2\ell}\tilde{v}(t,0)  = 0, \, 
    2\ell \leq N+2,
    \\
    \label{E:Regularity_rescaled_rho_at_zero}    
    &\partial_r^{2\ell+1} \tilde{\rho}(t,0)  = 0,\, 
    2\ell+1 \leq N+2,
\end{align}
\end{subequations}
and, again, each further time derivative costs another spatial derivative.

\medskip
\noindent\textbf{Step 5} (Regularity of $\eta_0$).
Define
\begin{subequations}{\label{E:Mass_functions}}
\begin{align}
    \label{E:Mass_function}
    m(\zeta) & := \int_0^\zeta \tilde{\rho}_0(z) z^2 dz,
    \\
    \label{E:Reference_mass_function}
    m_\text{ref}(\zeta) &:= \int_0^\zeta (\delta w(z))^\frac{1}{\kappa} z^2 dz.
\end{align}
\end{subequations}
We have that $m \colon [0,1] \rightarrow [0,m(1)]$ and $m_\text{ref} \colon [0,1] \rightarrow [0,m_\text{ref}(1)]$ are continuous and strictly increasing maps, thus invertible. By  \eqref{E:Reference_mass}, \eqref{E:Initial_mass}, and \eqref{E:Mass_constraint},
$m(1)=M=M_\text{ref}=m_\text{ref}(1)$, and therefore we have a well-defined orientation preserving homeomorphism
\begin{align*}
    &\eta_0 \colon [0,1] \rightarrow [0,1],
    \\
    & \eta_0(\zeta) := m^{-1}(m_\text{ref}(\zeta)).
\end{align*}

The expressions defining $m$ and $m_\text{ref}$ show that these quantities are differentiable. We further show below that $\eta_0$ is in fact a $C^{N+2}$ diffeomorphism, which then justifies the following step: differentiating the identity
\begin{align}
\label{E:Pre_gauge_identity_in_terms_of_masses}
    m(\eta_0(\zeta)) = m_\text{ref}(\zeta)
\end{align}
gives
\begin{align*}
    \tilde{\rho}_0(\eta_0(\zeta)) (\eta_0(\zeta))^2 \partial_\zeta \eta_0(\zeta) = (\delta w(\zeta))^\frac{1}{\kappa} \zeta^2,
\end{align*}
which is \eqref{gauge fixing} (recall \eqref{E:Jac_det_def}). 

An orientation preserving diffeomorphism $\hat{\eta}$ on $[0,1]$ must be increasing. If it satisfies \eqref{gauge fixing} (with $\hat{\eta}$ in place of $\eta_0$), then upon integration it must satisfy $m(\hat{\eta}(\zeta))=m_\text{ref}(\zeta)$, hence $\hat{\eta} = \eta_0$. This shows that in our setting the diffeomorphism given by Dacorogna--Moser \cite{dacorogna1990partial} in \eqref{gauge fixing} agrees with the one we just constructed.

By \eqref{E:Regularity_rescaled_rho} and Definition \ref{def: w}, we obtain that $m,m_\text{ref} \in C^{N+3}([\epsilon, 1-\epsilon])$ for any $\epsilon > 0$, and since $\partial_\zeta m(\zeta) > 0$, we have $m^{-1} \in C^{N+3}([\epsilon,1-\epsilon])$ by the inverse function theorem. Thus, $\eta_0 \in C^{N+3}([\epsilon,1-\epsilon])$ for any $\epsilon>0$, i.e., $\eta_0 \in C^{N+3}_{loc}((0,1))$.

Let us next investigate regularity at $\zeta=0$. By 
\eqref{E:Regularity_rescaled_rho} and \eqref{E:Regularity_rescaled_rho_at_zero} we have
\begin{align*}
    \tilde{\rho}_0(\zeta)=\sum_{\ell=0}^{\lfloor \frac{N+1}{2} \rfloor}  \frac{1}{(2\ell)!} \partial_\zeta^{2\ell} \tilde{\rho}(0) \zeta^{2\ell}
    +O(\zeta^{2 \lfloor \frac{N+1}{2} \rfloor+2}),
\end{align*}
so that
\begin{align}
    \label{E:Mass_function_near_zero}
    m(\zeta) = \frac{1}{3}\tilde{\rho}(0) \zeta^3 + \frac{1}{10}\partial_\zeta^2 \tilde{\rho}(0)\zeta^5 + \dots = \zeta^3 \GenericFunction_m(\zeta^2),
\end{align}
for some function $\GenericFunction_m$ that is $C^{N+2}$ in a neighborhood of $\zeta=0$ and satisfies $\GenericFunction_m(0)=\frac{1}{3}\tilde{\rho}(0)>0$. Similarly, in view of the properties of $w$ stated in Definition \ref{def: w}, we have
\begin{align}
    \label{E:Reference_mass_function_near_zero}    
    m_\text{ref}(\zeta) = \zeta^3 \GenericFunction_{m_\text{ref}}(\zeta^2)
\end{align}
for some function $\GenericFunction_{m_\text{ref}}$ that is $C^{N+2}$ in a neighborhood of $\zeta=0$ and satisfies $\GenericFunction_{m_\text{ref}}(0)>0$.

The maps $(m(\zeta))^\frac{1}{3}=\zeta (\GenericFunction_m(\zeta^2))^\frac{1}{3}$ and $(m_\text{ref})^\frac{1}{3} = \zeta (\GenericFunction_{m_\text{ref}}(\zeta^2))^\frac{1}{3}$ are $C^{N+2}$ in a neighborhood of $\zeta=0$ because $\GenericFunction_m(0)>0$ and $\GenericFunction_{m_\text{ref}}(0)>0$. Since 
\begin{align*}
    \left. \partial_\zeta [(m(\zeta))^\frac{1}{3}]\right|_{\zeta=0} = (\GenericFunction_m(0))^\frac{1}{3}>0,
\end{align*}
by the inverse function theorem, $m^\frac{1}{3}$ is a $C^{N+2}$ diffeomorphism in a neighborhood of $\zeta=0$. But near zero we have that 
$(m(\eta_0(\zeta)))^\frac{1}{3} := (m_\text{ref}(\zeta))^\frac{1}{3}$ if and only if $m(\eta_0(\zeta)) := m_\text{ref}(\zeta)$, and we conclude that 
\begin{align}
    \label{E:eta_0_near_zero}
    \eta_0(\zeta) = (m^\frac{1}{3})^{-1}(m_\text{ref}(\zeta))^\frac{1}{3}),
\end{align}
thus $\eta_0$ is $C^{N+2}$ at $\zeta=0$.

We finally move to establish regularity at $\zeta=1$. Since $w(1)=0$, we can write, for $0<\zeta< 1$,
\begin{align*}
    w(1-\zeta) &= w(1-\zeta) - w(1) =\int_1^{1-\zeta} \partial_z w(z) dz
    \\
    & = \int_{1-\zeta}^1 (-\partial_z w(z) )dz
    \\
    & =
    \zeta \int_0^1 (-\partial_z w(1-\zeta z))dz,
\end{align*}
where we changed variables $z \mapsto 1-\zeta z$ in the last step. Thus we can write
\begin{align*}
    w_\delta(1-\zeta) = \zeta W_\delta(\zeta),
\end{align*}
where 
\begin{align*}
    W_\delta(\zeta) = \delta  \int_0^1 (-\partial_z w(1-\zeta z))dz
\end{align*}
is smooth in a neighborhood of $\zeta=0$. Since  $\partial_\zeta w(1) < 0$, we also have that $W_\delta$ is positive in a neighborhood of $\zeta=0$. Similarly, by the physical vacuum boundary condition and \eqref{E:Regularity_rescaled_sound_speed}, we can write 
\begin{align}
    \label{E:Identity_boundary_regularity_eta_0_1}
    \tilde{\rho}^\kappa_0(1-\zeta) = \zeta  R(\zeta),
\end{align}
where $R$ is $C^{N+2}$ and positive in a neighborhood of zero. Thus, there exists an $\epsilon>0$ such that $W_\delta,R \in C^{N+2}([0,\epsilon])$ and $W_\delta, R > 0$ on $[0,\epsilon]$. In what follows, we may need to shrink $\epsilon$ as we will be making several statements valid in a neighborhood of $\zeta=0$.

Introduce
\begin{subequations}{\label{E:Endpoint_mass_functions}}
\begin{align}
    \label{E:Endpoint_mass_function}    
    m^-(\zeta) & := \int_\zeta^1 \tilde{\rho}_0(z) z^2 dz,
    \\
    \label{E:Endpoint_reference_mass_function}     
    m^-_\text{ref}(\zeta) &:= \int_\zeta^1 (\delta w(z))^\frac{1}{\kappa} z^2 dz.
\end{align}
\end{subequations}
Then
\begin{align*}
     m^-(1-\zeta) & = \int_{1-\zeta}^1 (\tilde{\rho}^\kappa_0(z))^\frac{1}{\kappa} z^2 dz   
     \\
     &= \int_0^\zeta
     (\tilde{\rho}^\kappa_0(1-z))^\frac{1}{\kappa} (1-z)^2 dz
     \\
     &=
     \int_0^\zeta z^\frac{1}{\kappa} \tilde{R}(z) dz,
\end{align*}
where we used \eqref{E:Identity_boundary_regularity_eta_0_1}, changed variables $z\mapsto 1-z$ in the second step, and defined
\begin{align*}
    \tilde{R}(\zeta) := R^\frac{1}{\kappa}(\zeta)(1-\zeta)^2.
\end{align*}
By the above construction of $R$ in \eqref{E:Identity_boundary_regularity_eta_0_1}, we have that $\tilde{R}\in C^{N+2}([0,\epsilon])$ and $\tilde{R}>0$ on $[0,\epsilon ]$. 
Changing variables once more, $z \mapsto \zeta z$,
\begin{align*}
     m^-(1-\zeta) & = \zeta^{1+\frac{1}{\kappa}}
     \int_0^1 z^\frac{1}{\kappa} \tilde{R}(\zeta z) dz
     \\
     &=
     \zeta^{1+\frac{1}{\kappa}} \hat{R}(\zeta),
\end{align*}
where 
\begin{align*}
    \hat{R}(\zeta) :=     \int_0^1 z^\frac{1}{\kappa} \tilde{R}(\zeta z) dz
\end{align*}
belongs to $C^{N+2}([0,\epsilon])$ and is positive on $[0,\epsilon]$. Similarly, we can write
\begin{align*}
    m^-_\text{ref}(1-\zeta) = \zeta^{1+\frac{1}{\kappa}} \hat{W}_\delta(\zeta),
\end{align*}
where $\hat{W}_\delta \in C^{N+2}([0,\epsilon])$ and is positive on $[0,\epsilon]$.

Define
\begin{subequations}{\label{E:Endpoint_masses_to_a_power}}
\begin{align}
    \mathsf{m}^-(\zeta) & := (m^-(\zeta))^{\frac{\kappa}{1+\kappa}},
    \\
    \mathsf{m}^-_\text{ref}(\zeta) & := (m^-_\text{ref}(\zeta))^\frac{\kappa}{1+\kappa},
\end{align}
\end{subequations}
so that 
\begin{align*}
    \mathsf{m}^-(1-\zeta) = \zeta (\hat{R}(\zeta))^\frac{\kappa}{1+\kappa}.
\end{align*}
Because $\hat{R} \in C^{N+2}([0,\epsilon])$ and $\hat{R}(0) > 0$, we have that $\left. \partial_\zeta \mathsf{m}^-(1-\zeta)\right|_{\zeta=0} >0$. By the inverse function theorem, $\mathsf{m}^-(1-\zeta)$ is a $C^{N+2}$ diffeomorphism in a neighborhood of $\zeta=0$, or equivalently  $\mathsf{m}^-$ is a $C^{N+2}$ diffeomorphism in a neighborhood of $\zeta=1$. Similarly, $\mathsf{m}^-_\text{ref}$ is a $C^{N+2}$ diffeomorphism in a neighborhood of $\zeta=1$.

In view of the definitions \eqref{E:Masses_data}, \eqref{E:Mass_functions}, and \eqref{E:Endpoint_mass_functions}, we can write
\begin{subequations}{\label{E:Masses_and_endpoint_masses_relations}}
\begin{align}
    m^-(\zeta) &= M - m(\zeta),\\
    m^-_\text{ref}(\zeta) & = M_\text{ref} - m_\text{ref}(\zeta).
\end{align}
\end{subequations}
Identities \eqref{E:Pre_gauge_identity_in_terms_of_masses} and \eqref{E:Masses_and_endpoint_masses_relations} imply that $M-m^-(\eta_0(\zeta)) = M_\text{ref} - m^-_\text{ref}(\zeta)$, and thus
\begin{align}
    \label{E:Pre_gauge_identity_in_terms_of_endpoint_masses} 
    m^-(\eta_0(\zeta)) = m^-_\text{ref}(\zeta)    
\end{align}
in view of assumption \eqref{E:Mass_constraint}. Reciprocally, if \eqref{E:Pre_gauge_identity_in_terms_of_endpoint_masses} holds so does \eqref{E:Pre_gauge_identity_in_terms_of_masses} using again \eqref{E:Masses_and_endpoint_masses_relations}
and \eqref{E:Mass_constraint}.
In other words, \eqref{E:Pre_gauge_identity_in_terms_of_endpoint_masses} and \eqref{E:Masses_and_endpoint_masses_relations}
are equivalent, and since the latter holds by our construction of $\eta_0$ so does the former. Raising \eqref{E:Pre_gauge_identity_in_terms_of_endpoint_masses} to the power $\frac{\kappa}{1+\kappa}$ and invoking \eqref{E:Endpoint_masses_to_a_power}
gives
\begin{align}
    \mathsf{m}^-(\eta_0(\zeta)) = \mathsf{m}^-_\text{ref}(\zeta)    
\end{align}
By the foregoing, $\mathsf{m}^-$ and $\mathsf{m}_\text{ref}^-$ are $C^{N+2}$ diffeomorphisms in a neighborhood of $\zeta=1$, and thus is $\eta_0 = (\mathsf{m}^{-})^{-1} \circ \mathsf{m}^-_\text{ref}$. In particular $\eta_0 \in C^{N+2}([1-\epsilon,1])$. 

Considering the three regions investigated above, namely, a neighborhood of $\zeta=0$, a neighborhood of $\zeta=1$, and $(0,1)$, we conclude that $\eta_0 \in C^{N+2}([0,1])$ and $\eta_0^{-1} \in C^{N+2}([0,1])$.

\medskip
\noindent\textbf{Step 6} (Regularity of $(\Theta,\partial_\tau\Theta)$).
By \eqref{E:Regularity_rescaled_v} and \eqref{modvel}, we have $V(\tau,\cdot) \in C^{N+2}([0,1])$ and $\partial_\tau V(\tau,\cdot) \in C^{N+1}([0,1])$. We showed above that $\eta_0 \in C^{N+2}([0,1])$. Therefore,  by \eqref{E:d_eta_d_s},
we conclude that $\eta(\tau,\cdot)$ and $\partial_\tau\eta(\tau,\cdot)$ are in $C^{N+2}([0,1])$, and thus the same holds for $\Theta(\tau,\cdot)$ and $\partial_\tau\Theta(\tau,\cdot)$ in view of \eqref{E:Theta_def}, where we recall that the change of variables from $s$ to $\tau$ does not lose derivatives.

\medskip
\noindent\textbf{Step 7} (Center regularity of $(\Theta,\partial_\tau\Theta)$ and boundedness of $E^N$). We now show that $\Theta$ and $\partial_\tau \Theta$ are genuinely spherically symmetric functions. For this, we need to verify that they satisfy
\begin{align*}
    \partial_\zeta^{2\ell} \Theta(\tau,0) &=0, \, 2\ell \leq N+2, 
    \\
    \partial_\zeta^{2\ell} \partial_\tau \Theta(\tau,0) &=0, \, 2\ell \leq N+2.   
\end{align*}

We will refer to functions whose even derivatives vanish at $\zeta=0$ as odd, and those whose odd derivatives vanish at $\zeta=0$ as even. This is a slight abuse of notation in that our maps are not defined for $\zeta < 0$, but the important point is that the standard parity algebra $(\text{odd}) \cdot (\text{odd})= (\text{even})$, $(\text{odd}) \cdot (\text{even})= (\text{odd})$, etc., continues to hold. Here we are interested in parity in $\zeta$ for functions of $\tau$ and $\zeta$, while we will appeal to expressions of $s$ and $\zeta$ where $s$ is as in \eqref{defn:scaling}. The change from $s$ to $\tau$, however, does not affect the parity in $\zeta$.

Let us show that $V(\tau,\zeta)$ is an odd function of $\zeta$. Consider \eqref{modvel}. The term $\frac{\lambda^\prime}{\lambda}\zeta$ is odd, as is $\tilde{v}$ by \eqref{E:Regularity_rescaled_v_at_zero}. $u^0$ is even by \eqref{u00}. Thus, $V$ is odd.

From \eqref{E:Mass_function_near_zero},
\eqref{E:Reference_mass_function_near_zero}, and \eqref{E:eta_0_near_zero}, we have that near $\zeta=0$ we can write
\begin{align*}
    \eta_0(\zeta) = \zeta \mathbf{F}(\zeta^2),
\end{align*}
with $\mathbf{F}(0) \neq 0$,
where we used an invertible map of the form, $\mathbf{F}(\zeta^2)$, $\mathbf{F}(0) \neq 0$,  also has the form 
$\zeta \mathbf{F}(\zeta^2)$, $\mathbf{F}(0) \neq 0$ by Lemma \ref{L:Inverse_odd_function}. 
Thus, $\eta_0$ is odd.

$V$ being odd in $\zeta$, i.e., having all even derivatives vanishing at $\zeta=0$, implies that $V$ can be extended from $[0,1]$ to an odd function on $[-1,1]$ which is of the same regularity, $C^{N+2}$. Similarly for $\eta_0$. By \eqref{E:d_eta_d_s}, $\eta$ is the flow of an odd vector field with odd initial condition. A standard parity-preserving argument based on uniqueness shows that $\eta(\tau,\zeta)$ is then odd in $\zeta$. Upon restricting to $\zeta \in [0,1]$, we have $\partial_\zeta^{2\ell}\eta(\tau,0)=0$, $2\ell \leq N+2$. Thus the same holds for $\Theta(\tau,\zeta)$ by \eqref{E:Theta_def}.

For $\partial_\tau \Theta$, we use  \eqref{E:Theta_def} and \eqref{E:d_eta_d_s} to get
\begin{align*}
    \partial_\tau \Theta(\tau,\zeta) = \frac{ds}{d\tau} \partial_s \Theta(s,\zeta) = 
    \frac{ds}{d\tau} \partial_s \eta(s,\zeta)
    = \frac{ds}{d\tau} V(s,\eta(s,\zeta)),
\end{align*}
which is odd in $\zeta$.

For any odd function $\GenericFunction \in C^{N+2}([0,1])$, Lemma \ref{L:calD_odd_functions} gives that
\begin{subequations}{\label{E:calDj_on_odd_functions}}
\begin{align}
    \calD_{2\ell} \GenericFunction(0) & = 0, \, 2\ell \leq N+2,
    \\
    \calD_{2\ell+1} \GenericFunction(0) & = \text{finite value }, \, 2\ell+1 \leq N+2,
\end{align}
\end{subequations}
and also a bound on the respective norms,
where we recall that $\calD_\ell$ is given by 
\eqref{calDj}.

With \eqref{E:calDj_on_odd_functions} at hand and the regularity of $\Theta$ and $\partial_\tau \Theta$ established in Step 6, 
we then immediately obtain \eqref{E:Regularity_Theta_partial_tau_Theta_closed_time_interval}; \eqref{E:Finiteness_E_N_closed_time_interval} then follows from \eqref{E:Regularity_Theta_partial_tau_Theta_closed_time_interval} and Remark \ref{R:Norm_equivalence_calX_calE}.

This completes the proof of item 2) of the Theorem.

\medskip
\noindent\textbf{Step 8} (Perturbations of reference data give rise to small $E^N$ data). 
This is essentially the statement that our association $(\mathring{\NewSoundSpeedSq},\mathring{\NewVelocity}^i)\mapsto (\Theta_0,U_0)$ is continuous at $(\NewSoundSpeedSq_\text{ref},\NewVelocity_\text{ref}^i)$. We have already seen that this association is well defined, but now we need to be more quantitative about it.

Define the set 
\begin{align*}
    \mathcal{U} := \{ (\mathring{\NewSoundSpeedSq},\mathring{\NewVelocity}^i) \in \DITWeightedSpace^{2k}_\text{ref} \, \, | \, \text{the construction of } (\Theta_0,U_0) \, \text{ out of } \, (\mathring{\NewSoundSpeedSq},\mathring{\NewVelocity}^i) \,\text{ in Step 2 is well defined} \},
\end{align*}
and endow $\mathcal{U}$ with the topology inherited from $\DITWeightedSpace^{2k}_\text{ref}$. In practice, recall that the ingredients needed for $(\Theta_0,U_0)$ to be well defined are the physical vacuum boundary condition, the gauge fixing of $\eta_0$, subluminality, enough regularity for the variables, and the center-regularity conditions. 
These conditions are stable under convergence in $\DITWeightedSpace^{2k}_\text{ref}$ and thus $\mathcal{U}$ is open.
By construction, on $\mathcal{U}$ we have a well-defined Euler-to-Lagrangian data map $\mathbb{D} \colon (\mathring{\NewSoundSpeedSq},\mathring{\NewVelocity}) \mapsto (\Theta_0,U_0)$ and we need to ascertain the continuity of this map, where the topology on the image is given by the $\mathcal{X}^N$ topology.

Let $\{ (\mathring{\NewSoundSpeedSq}_n,\mathring{\NewVelocity}^i_n)\}_{n=0}^\infty \subset \mathcal{U}$ converge to $(\mathring{\NewSoundSpeedSq},\mathring{\NewVelocity}^i)$.
By definition of convergence in $\DITWeightedSpace^{2k}_0$, we have that $\Omega_{\mathring{\NewSoundSpeedSq}_n}$ converges to $\Omega_{\mathring{\NewSoundSpeedSq}}$. Since 
$\Omega_{\mathring{\NewSoundSpeedSq}_n}=B_{\lambda_{0,n}}(0)$ and $\Omega_{\mathring{\NewSoundSpeedSq}}=B_{\lambda_0}(0)$ by spherical symmetry, we have $\lambda_{0,n} \rightarrow \lambda_0$. 

Set
\begin{align*}
    \mathring{\NewSoundSpeedSq}^L_n(\zeta) & := \mathring{\NewSoundSpeedSq}_n(\lambda_{0,n}\zeta),
    \\
    \mathring{\NewSoundSpeedSq}^L(\zeta) & := \mathring{\NewSoundSpeedSq}(\lambda_{0}\zeta),   \\
    (\mathring{\NewVelocity}^L_n)^i(\zeta) & := \mathring{\NewVelocity}_n^i(\lambda_{0,n}\zeta),\\ 
    (\mathring{\NewVelocity}^L)^i(\zeta) & := \mathring{\NewVelocity}^i(\lambda_{0}\zeta), 
\end{align*}
$\zeta \in [0,1]$. By definition of convergence in 
$\DITWeightedSpace^{2k}_0$, given $\epsilon$ we have a smooth $(\mathring{\NewSoundSpeedSq}_\text{smooth}, \mathring{\NewVelocity}_\text{smooth}^i)$, which we can assume to be spherically symmetric (this is because the relevant domains and norms are invariant under $SO(3)$; if $(\mathring{\NewSoundSpeedSq}_\text{smooth}, \mathring{\NewVelocity}_\text{smooth}^i)$ were not spherically symmetric, we could replace it by its $SO(3)$ average without worsening the approximation), defined in a neighborhood $B_{\lambda_\text{smooth}}(0)$ of $\Omega_{\mathring{\NewSoundSpeedSq}}$ such that
$\| (\mathring{\NewSoundSpeedSq},\mathring{\NewVelocity}^i) - (\mathring{\NewSoundSpeedSq}_\text{smooth}, \mathring{\NewVelocity}_\text{smooth}^i) \|_{\DITWeightedSpace^{2k}_{\mathring{\NewSoundSpeedSq}}(\Omega_{\mathring{\NewSoundSpeedSq}})} \leq \epsilon$ and
$\limsup_{n\rightarrow\infty}\| (\mathring{\NewSoundSpeedSq}_n,\mathring{\NewVelocity}^i_n) - (\mathring{\NewSoundSpeedSq}_\text{smooth}, \mathring{\NewVelocity}_\text{smooth}^i) \|_{\DITWeightedSpace^{2k}_{\mathring{\NewSoundSpeedSq}_n}(\Omega_{\mathring{\NewSoundSpeedSq}_n})} \leq \epsilon$. From this and the non-degeneracy of $\mathring{\NewSoundSpeedSq}_n$ it follows that 
\begin{align*}
    \| (\mathring{\NewSoundSpeedSq}^L,(\mathring{\NewVelocity}^L)^i) - (\mathring{\NewSoundSpeedSq}^L_\text{smooth}, (\mathring{\NewVelocity}^L)^i_\text{smooth}) \|_{\DITWeightedSpace^{2k}_{w}([0,1])} \lesssim \epsilon
\end{align*}
and
\begin{align*}
    \limsup_{n\rightarrow\infty}
    \| (\mathring{\NewSoundSpeedSq}^L_n,(\mathring{\NewVelocity}^L_n)^i) - (\mathring{\NewSoundSpeedSq}^L_\text{smooth}, (\mathring{\NewVelocity}^L)^i_\text{smooth}) \|_{\DITWeightedSpace^{2k}_{w}([0,1])} \lesssim \epsilon,
\end{align*}
where $\mathring{\NewSoundSpeedSq}^L_\text{smooth}(\zeta):=\mathring{\NewSoundSpeedSq}_\text{smooth}(\lambda_\text{smooth}\zeta)$ and $(\mathring{\NewVelocity}^L)^i_\text{smooth}(\zeta) := \mathring{\NewVelocity}_\text{smooth}^i(\lambda_\text{smooth}\zeta)$. This implies that $(\mathring{\NewSoundSpeedSq}^L_n,(\mathring{\NewVelocity}^L_n)^i) \rightarrow (\mathring{\NewSoundSpeedSq}^L,(\mathring{\NewVelocity}^L)^i)$ in $\DITWeightedSpace^{2k}_w([0,1])$.

From \eqref{E:New_sound_and_velocity} we can write $u^i = (1+\frac{\kappa}{1+\kappa}r)^{-(1+\frac{1}{\kappa})} \NewVelocity^i$, so that $u^i$ is related to $\NewVelocity^i$ by a smooth function of $\NewSoundSpeedSq$ that is uniformly bounded away from zero. Thus, for the Eulerian velocities $\mathring{u}^i$ and $\mathring{u}^i_n$ corresponding to $\mathring{\NewVelocity}^i$ and $\mathring{\NewVelocity}^i_n$, respectively, if we set $(\mathring{u}^L)^i(\zeta) := \mathring{u}^i(\lambda_{0}\zeta)$ and $(\mathring{u}^L_n)^i(\zeta) := \mathring{u}_n^i(\lambda_{0}\zeta)$, then we also have $(\mathring{\NewSoundSpeedSq}^L_n,(\mathring{u}^L_n)^i) \rightarrow (\mathring{\NewSoundSpeedSq}^L,(\mathring{u}^L)^i)$ in $\DITWeightedSpace^{2k}_w([0,1])$. By \eqref{E:N_k_relation} and \eqref{E:Weighted_embedding_DIT} we then obtain 
$(\mathring{\NewSoundSpeedSq}^L_n,(\mathring{u}^L_n)^i) \rightarrow (\mathring{\NewSoundSpeedSq}^L,(\mathring{u}^L)^i)$ in $C^{N+2}([0,1])$.

Next, by \eqref{E:New_sound_speed_sq} we have $\| \mathring{\rho}_n - \mathring{\rho} \|_{L^\infty(\mathbb{R})} \rightarrow 0$ (with $\mathring{\rho}_n,\mathring{\rho}$ extended to zero outside the domains). This combined with \eqref{defn:scaling} and the dominated convergence theorem gives that $m_n$ converges uniformly to $m$, where $m_n$ is \eqref{E:Mass_function} constructed out of $\mathring{\rho}_n$ and $m$ out of $\mathring{\rho}$. Each $m_n$ is strictly increasing and uniformly non-degenerate away from the end points, and therefore $m_n^{-1}$ also converges uniformly to $m^{-1}$. Thus, $\eta_{0,n} = m_n^{-1} \circ m_\text{ref}$ converges uniformly to $\eta_0 = m^{-1} \circ m_\text{ref}$. Analyzing the regularity and invertibility of $\eta_{0,n}$ near the endpoints and away from the endpoints as in Step 5, and recalling \eqref{E:N_k_relation}, we can show that in fact $\eta_{0,n} \rightarrow \eta_0$ in $C^{N+1}([0,1])$. From this, we then obtain $C^{N+1}([0,1])$ convergence of $\Theta_n := \eta_{0,n} - \zeta$ to $\Theta = \eta_0 -\zeta$. 

From the convergences already established, we can  obtain $C^{N+1}([0,1])$ convergence of $U_{0,n}(\zeta) := V_n(0,\eta_{0,n}(\zeta))$ to $U_0(\zeta) := V(0,\eta_0(\zeta))$, where $V_n$ and $V$ are given in terms of $\mathring{u}_n^i$ and $\mathring{u}^i$, respectively, and  \eqref{defn:scaling} in view of \eqref{modvel}.

Finally, as in Step 8, the center-regularity conditions ensure that the above $C^{N+1}([0,1])$ convergence implies $\mathcal{X}^N$ convergence and thus $E^N(0)$ convergence by Remark \ref{R:Norm_equivalence_calX_calE}. Recalling that $(\Theta_0,U_0)=(0,0)$ for the reference state, we obtain item 3) of the Theorem.

\medskip
\noindent\textbf{Step 9} (Continuation criterion). The solutions $(\Theta,\partial_\tau\Theta)$ to \eqref{eq:momlagmain} that we have constructed are obtained out of solutions $(\rho,u^i)$ to \eqref{eq:RE}, which in turn are obtained from the solutions $(\NewSoundSpeedSq,\NewVelocity)$ obtained in \cite{disconzi2022relativistic}. The passage from $(\NewSoundSpeedSq,\NewVelocity)$ to $(\rho,u^i)$ is given by \eqref{E:New_sound_and_velocity}, whereas the passage from $(\rho,u^i)$ to $(\Theta,\partial_\tau\Theta)$ is given by the procedure of Section \ref{sect: setup}, which we justified through Steps 1--3 above. From this procedure we have that if $(\NewSoundSpeedSq,\NewVelocity)$ is defined on the time interval $[0,T^*)$ then $(\Theta,\partial_\tau\Theta)$ is defined on the time interval $[0,\tau^*)$, where $\tau^*$ is given by \eqref{E:Time_existence_Lagrangian}. Thus, if $(\NewSoundSpeedSq,\NewVelocity)$ is defined instead on the time interval $[0,T^*+ \epsilon)$, $\epsilon > 0$, then $(\Theta,\partial_\tau\Theta)$ is defined on the time interval $[0,\tau^*+\epsilon^\prime)$, where $\epsilon^\prime$ is given by 
\begin{align*}
    \epsilon^\prime := \int_{T^*}^{T^*+\epsilon} \frac{1}{\lambda(t)} dt.
\end{align*}

By the continuation criterion established in \cite{disconzi2022relativistic}, a subluminal solution $(\NewSoundSpeedSq,\NewVelocity)$ defined on $[0,T^*)$ and uniformly satisfying the physical vacuum boundary condition can be continued past $T^*$ if
\begin{align}
    \label{E:DIT_continuation_criterion}
    \sup_{0\leq t < T^*}
    \| (\NewSoundSpeedSq,\NewVelocity)\|_{C^2(\Omega_t)} < \infty.
\end{align}
(The continuation criterion in \cite{disconzi2022relativistic} requires in fact control of a norm weaker than $C^2$, but $C^2$ suffices for our purposes.)
Therefore, in order to show item 4) of the Theorem, it suffices to show that under the bootstrap assumptions \eqref{bootstrap U0} and \eqref{bootstrap FG}, \eqref{E:Continuation_criterion_Lagrangian} implies \eqref{E:DIT_continuation_criterion}.

Recalling \eqref{E:U_0_Theta_and_bar_lambda}, we have the subluminality condition
\begin{align*}
  0 <  (U^0)^{-2} = 1 - (\p_\tau \Theta + \lamt (\Theta + \ze))^2,  
\end{align*}
The bootstrap assumption \eqref{bootstrap U0} is stronger than subluminality, giving a uniform positive lower bound for $U^0$ on $[0,\tau^*)$, i.e., there exists a $C_0 >0$ such that 
\begin{align*}
    (U^0)^{-2} = 1 - (\p_\tau \Theta + \lamt (\Theta + \ze))^2 \geq C_0
    \, \text{ on } [0,\tau^*).
\end{align*}

Recalling the computations surrounding \eqref{E:Jac_det_def}, we have
\begin{align}
    \calF = \left(\frac{\eta}{\zeta}\right)^2\partial_\zeta\eta.
\end{align}
We showed in Step 7 that $\eta$ is odd in $\zeta$, thus $\eta(\tau,\zeta) = \zeta + O(\zeta^3)$ near $\zeta=0$ for $\tau \in [0,\tau^*)$. Therefore, 
the bootstrap assumption \eqref{bootstrap FG} on $\calF$ gives that $\partial_\zeta \eta(\tau,\cdot)$ remains uniformly bounded away from zero and thus $\eta = \Theta + \zeta$ remains uniformly non-degenerate on $[0,\tau^*)$, i.e., there exists a $C_1>0$ such that
\begin{align}
    \label{E:eta_remains_diffeo_for_continuation_criterion}
    \partial_\zeta \eta = 1+\partial_\zeta \Theta \geq C_1 \, \text{ on } [0,\tau^*).
\end{align}
This, combined with 
\eqref{E:f_and_U_of_s}, 
\eqref{defn:scaling}, and the bootstrap assumption \eqref{bootstrap FG} on $\bar{\calG}$, gives that there is no interior vacuum formation in the corresponding Eulerian variables solutions, i.e., if we define
\begin{align*}
    \overline{\Omega_{T^*} } &:=    \Big(\overline{  \bigcup_{0 \leq t < T^*} \{ t \} \times \Omega_t} \Big)\backslash
     \Big(\bigcup_{0 \leq t < T^*} \{ t \} \times \Omega_t \Big),
     \\
     \Omega_{T^*} &:=\text{interior}(\overline{\Omega_{T^*} }),
\end{align*}
then for each $x \in \Omega_{T^*}$, it holds that
\begin{align}
    \liminf_{t \rightarrow (T^{*})^-} \rho(t,x) > 0.
\end{align}
Furthermore, \eqref{E:eta_remains_diffeo_for_continuation_criterion} and  \eqref{bootstrap FG} also give that $f^\kappa(\tau,\zeta) \approx w_\delta(\zeta) \approx 1-\zeta$ uniformly over $[0,\tau^*)$, from which we conclude that the physical vacuum boundary condition is satisfied uniformly on the Eulerian variables over $[0,T^*)$, i.e., there exists $C_2>1$ such that for all $(t,x) \in \cup_{0 \leq t < T^*} \{ t \} \times \Omega_t$
\begin{align*}
    \frac{1}{C_2}\dist(x,\partial\Omega_t) \leq c_s^2(t,x) \leq C_2 \dist(x,\partial \Omega_t).
\end{align*}

With the above ingredients, we can now verify that \eqref {E:DIT_continuation_criterion} holds if \eqref{E:Continuation_criterion_Lagrangian} does so. We note that $\eta = \zeta + \Theta$ is $C^3$ if $\Theta$ is $C^3$; $\calF = (\frac{\eta}{\zeta})^2 \partial_\zeta \eta = (1+\frac{\Theta}{\zeta})^2 (1+\partial_\zeta \Theta) $ is $C^2$ if $\Theta$ is $C^3$; $U=\lambda^{\frac{3}{2 }\kappa}\partial_\tau \Theta$ is $C^2$ if $\partial_\tau \Theta$ is $C^2$; $U^0 = 1 - (\p_\tau \Theta + \lamt (\Theta + \ze))^2$ is $C^2$ if $\Theta$ is $C^2$ and $\partial_\tau \Theta$ is $C^2$. These observations combined with the fact that $\eta(\tau,\cdot) = \eta(\tau(t),\cdot) \colon [0,1] \rightarrow \Omega_t$, $\tau \in [0,\tau^*)$, is a diffeomorphism, imply the claim. This establishes item 4) of the Theorem.

\end{proof}
\end{appendices}

\bibliographystyle{abbrv}
\bibliography{references.bib}
\end{document}